\documentclass[12pt,reqno]{amsart}
\usepackage{times}
\usepackage{amsmath,amsfonts,amstext,amssymb,amsbsy,amsopn,amsthm,eucal}
\usepackage{txfonts}
\usepackage{dsfont}
\usepackage{graphicx}   
\usepackage{hyperref}
\usepackage{accents}
\usepackage{enumerate}

\newcommand{\Cyl}{\text{Cyl}}
\newcommand{\Lip}{\text{Lip}}

\newcommand{\Ric}{\text{Ric}}
\newcommand{\Ent}{\text{Ent}}

\newcommand{\inj}{\text{inj}}

\newcommand{\dC}{\mathds{C}}

\newcommand{\dN}{\mathds{N}}

\newcommand{\dQ}{\mathds{Q}}
\newcommand{\dR}{\mathds{R}}

\newcommand{\dZ}{\mathds{Z}}

\newcommand{\bt}{\text{\bf{t}}}
\newcommand{\bv}{\text{\bf{v}}}

\newcommand{\bw}{\text{\bf{w}}}
\newcommand{\bx}{\text{\bf{x}}}
\newcommand{\by}{\text{\bf{y}}}

\newcommand{\cB}{\mathcal{B}}

\newcommand{\cD}{\mathcal{D}}
\newcommand{\cF}{\mathcal{F}}

\newcommand{\cH}{\mathcal{H}}

\newcommand{\cP}{\mathcal{P}}

\newcommand{\cR}{\mathcal{R}}
\newcommand{\cS}{\mathcal{S}}

\newtheorem{theorem}{Theorem}[section]

\newtheorem{proposition}[theorem]{Proposition}
\newtheorem{lemma}[theorem]{Lemma}
\newtheorem{corollary}[theorem]{Corollary}
\theoremstyle{definition}
\newtheorem{definition}[theorem]{Definition}
\theoremstyle{remark}
\newtheorem{remark}{Remark}[section]
\theoremstyle{remark}
\newtheorem{example}{Example}[section]
\theoremstyle{remark}

\theoremstyle{remark}

\theoremstyle{remark}

\begin{document}

\title{Characterizations of Bounded Ricci Curvature\\ on Smooth and NonSmooth Spaces}

\author{Aaron Naber}

\date{\today}
\maketitle

\begin{abstract}
There are two primary goals to this paper.  In the first part of the paper we study smooth metric measure spaces $(M^n,g,e^{-f}dv_g)$ and give several ways of characterizing bounds $-\kappa g\leq \Ric+\nabla^2f\leq \kappa g$ on the Ricci curvature of the manifold.  In particular, we see how bounded Ricci curvature on $M$ controls the analysis of path space $P(M)$ in a manner analogous to how lower Ricci curvature controls the analysis on $M$.  In the second part of the paper we develop the analytic tools needed to in order to use these new characterizations to give a definition of bounded Ricci curvature on general metric measure spaces $(X,d,m)$.  We show that on such spaces many of the properties of smooth spaces with bounded Ricci curvature continue to hold on metric-measure spaces with bounded Ricci curvature.

In more detail, in this paper we see that bounded Ricci curvature can be characterized in terms of the metric-measure geometry of path space $P(M)$.  The correct notion of geometry on path space is the one induced by what we call the parallel gradient, and the measures on path space of interest are the classical Wiener measures.  Our first characterization shows that bounds on the Ricci curvature are equivalent to certain parallel gradient estimates on path space.  These turn out to be infinite dimensional analogues of the Bakry-Emery gradient estimates.  Our second characterization relates bounded Ricci curvature to the stochastic analysis of path space.  In particular, we see that bounds on the Ricci curvature are equivalent to the appropriate $C^{\frac{1}{2}}$-time regularity of martingales on $P(M)$.  Our final characterization of bounded Ricci curvature relates Ricci curvature to the analysis on path space.  Specifically, we study the Ornstein-Uhlenbeck operator, a form on infinite dimensional laplacian on path space, and some twisted generalizations of it.  We prove sharp spectral gap and log-Sobolev estimates under the assumption of bounded Ricci curvature for these operators.  These estimates again turn out to be equivalent to bounds on the Ricci curvature.  We have analogous results for $d$-dimensional bounded Ricci curvature.

In the second part of the paper we study metric measure spaces $(X,d,m)$ and use the structure of the first part of the paper to define the notion of bounded Ricci curvature.  A primary technical difficulty is to describe the notion of the parallel gradient in such a setting.  Even in the smooth case one requires some deep ideas from stochastic analysis, namely the stochastic parallel translation map, to deal with this.  Our replacement for this allows us to sidestep the need for the stochastic parallel translation map, and in particular works on an arbitrary metric space.  After this is introduced and studied we spend the rest of the paper proving various structural properties of metric-measure spaces with bounded Ricci curvature.  Among others, we will see that spaces with Ricci curvature bounded by $\kappa$ have lower Ricci curvature bounded from below by $-\kappa$ in the sense of Lott-Villani-Sturm.  We will see that spaces with bounded Ricci curvature continue to have well behaved martingales.  Further, we will see that not only can one define the Ornstein-Uhlenbeck operator on path space, which still behaves as an infinite dimensional laplacian on path space, but that on spaces with bounded Ricci curvature these operators still enjoy poincare and log-sobolev estimates.  In particular, these tools allow us to do analysis on the path space of metric-measure spaces.
\end{abstract}
\newpage

\tableofcontents

\section{Introduction}\label{s:Intro}

The purpose of this series of papers is to give new characterizations of bounded Ricci curvature on a smooth manifold, and to introduce the necessary techniques in order to use these as motivation for a definition of bounded Ricci curvature in the case of nonsmooth spaces.  \\

More specifically, given a $n$-dimensional smooth metric measure space $(M^n,g,e^{-f}dv_g)$ recall that the Ricci curvature, or Bakry-Emery-Ricci curvature, of the metric measure space is defined to be the tensor
\begin{align}\label{e:ricci_curvature}
\Ric+\nabla^2f\, .
\end{align}
More generally, one often considers the $d$-dimensional Ricci tensor defined by
\begin{align}\label{e:d_ricci_curvature}
\Ric+\nabla^2f-\frac{1}{d-n}\nabla f\otimes \nabla f\, .
\end{align}

A good deal of effort has been spent on understanding these objects, and specifically on the geometric and analytic consequences of lower or upper Ricci curvature bounds on the manifold.  In particular, there has been a lot of work in recent years on understanding equivalences between lower bounds on the Ricci curvature and other geometric estimates on $M$.  In the context of the heat flow, the most natural point of view is the Bakry-Emery criteria and the corresponding gradient estimates.  In short, these may be used to equate a lower Ricci curvature bound of the space to certain gradient estimates of the heat flow.  More recently there has been a slew of work \cite{McCann_BBL_inequality},\cite{LV_OptimalRicci},\cite{Sturm_GeomMetricMeasSpace} in relating the Ricci curvature to the geometry of the space of probability measures on $M$.  In particular, Lott-Villani \cite{LV_OptimalRicci} and Sturm \cite{Sturm_GeomMetricMeasSpace} have been able to use these ideas to provide very reasonable definitions of lower Ricci curvature bounds on nonsmooth metric measure spaces, and prove many properties about such spaces.  More recently the work of \cite{Ambrosio_Ricci},\cite{Sturm_KE_LVS} has shown that in reasonable situations these notions of a lower Ricci curvature bound are the same even on nonsmooth spaces.  Since they play an important role in what we do here, we will review more completely the ideas involved in understanding lower Ricci curvature in Section \ref{s:smooth_lower_ricci}.\\

On the other hand, what has not been studied essentially at all at this point are ways of characterizing two sided bounds on the Ricci curvature.  Philosophically, such equivalences are quite important.  On the most basic level, they give rise to new understanding of the meaning of Ricci curvature.  More practically, such equivalences give rise to new tools and structures which may be used to study such spaces.  However, when it comes to understanding the structure of spaces with bounded Ricci curvature there are essentially only two tools available to bounded Ricci curvature that are not available to lower Ricci curvature.  The first are $\epsilon$-regularity theorems \cite{Anderson_Einstein},\cite{CCT_eps_reg}, which play a crucial role in the regularity theory of spaces with bounded Ricci curvature, see \cite{CheegerNaber_Ricci}.  The second tool, available only in the K\"ahler case, identifies the Ricci with the first chern class of the canonical line bundle.  This has been exploited quite deeply in many ways, see \cite{Tian_EinsteinSurfaces}, \cite{Donaldson_PositiveEinsteinI}, \cite{Tian_PositiveEinstein}.\\

This first paper in the series therefore focuses on smooth metric measure spaces $(M^n,g,e^{-f}dv_g)$, and asks the question if there are estimates analogous to those for lower Ricci curvature which characterize bounded Ricci curvature.  In Section \ref{s:smooth_bounded_ricci_intro} we answer this question in the affirmative and introduce several such characterizations of bounded Ricci curvature.  These characterizations are in several forms, but all come down to a better understanding of the metric-measure geometry of path space $P(M)\equiv C^0([0,\infty),M)$ of $M$.  Our first characterization of bounded Ricci curvature in Section \ref{sss:char_ricci_gradient} directly relates the functional analysis of path space $P(M)$ to the functional analysis of $M$.  Specifically, once we have defined the correct metric-measure geometry on path space we will see a space has bounded Ricci curvature if and only if a certain functional analytic estimate holds.  We will see how these directly generalize the Bakry-Emery gradient estimate to an infinite dimensional setting.  In our second characterization of bounded Ricci curvature in Section \ref{sss:char_ricci_quad_intro}, we relate bounded Ricci curvature to the stochastic analysis of path space $P(M)$.  To be a little more precise, we will see that bounded Ricci curvature on $M$ is characterized by the time regularity of martingales on path space.  In particular a typical martingale, viewed as a one parameter family $F^t$ of $L^2$ functions on path space, will be precisely $C^{\frac{1}{2}}$-H\"older with estimates precisely characterized by the Ricci curvature of $M$.  For our third characterization of bounded Ricci curvature we study the analysis of path space.  Specifically, one can define the Ornstein-Uhlenbeck operator on path space, which is a form of infinite dimensional laplacian, as well as some related operators which we refer to as twisted Ornstein-Uhlenbeck operators.  We will show that bounded Ricci curvature is equivalent to the existence of a spectral gap or log-Sobolev inequality for these operators.  We will prove analogous results in order to characterize smooth manifolds with bounded $d$-dimensional Ricci curvature.  See Section \ref{s:smooth_bounded_ricci_intro} for a more complete introduction to this first paper.\\

In the second part of the paper we analyze a more general class of metric measure spaces $(X,d,m)$.  We will use the characterizations of bounded Ricci curvature given in the first part of the paper to motivate definitions of bounded Ricci curvature in the more general setup.  As we will see, a key difficulty is to define the notion of the parallel gradient on a general metric space.  Even on a smooth manifold this is a subtle point which requires stochastic analysis ideas of Malliavin in \cite{Malliavin_Geometrie} by using the stochastic parallel translation map \cite{Driver_Survey}.  We will see in Section \ref{s:parallel_gradient_nonsmooth} how go about the construction in a manner which avoids such tools and allows us to generalize the construction to an essentially arbitrary metric space.  These ideas will also allow us to make sense of the Ornstein-Uhlenbeck operator on an almost arbitrary metric-measure space.  Once we have made rigorous sense of bounded Ricci curvature on a metric-measure space, we spend the rest of the paper proving properties about such spaces.  As the most basic result we will see that a metric-measure space with bounded Ricci curvature has a lower Ricci curvature bound in the sense of either Bakry-Emery or Lott-Villani-Sturm.  In fact, such spaces will have the even stronger lower Ricci curvature bound in the sense of \cite{Ambrosio_Ricci}.  More generally, we will see that many results proved in the first part of the paper for smooth spaces with bounded Ricci curvature continue to hold for metric-measure spaces.  In particular, we will see that metric-measure space with bounded Ricci curvature have a well behaved martingales, and we will be able to define and study generalizations of the Ornstein-Uhlenbeck operator on path space.  See Section \ref{s:nonsmooth_bounded_ricci_intro} for a more complete introduction to this part of the paper.  \\

\section{Introduction to Part I: The Smooth Bounded Ricci Case}\label{s:smooth_bounded_ricci_intro}

In this section we consider a smooth metric measure space 
\begin{align}
\big(M^n,g,e^{-f}dv_g\big)\, ,
\end{align}
and discuss the main results of the first part of the paper, which give various characterizations for bounds on the Ricci curvature tensor.  The outline of this Section is as follows.  In Section \ref{ss:prelim_intro} we discuss a few preliminaries at their most basic level.  This will be expanded on in Section \ref{s:prelim_smooth}, as we describe the preliminaries only enough here in order to state our results precisely.  In Section \ref{ss:char_ricci_intro} we discuss our main results as they pertain to the Ricci curvature $\Ric+\nabla^2f$ of a metric measure space, and in Section \ref{ss:char_d_ricci_intro} we describe our main results with respect to the $d$-Ricci curvature operator $\Ric+\nabla^2f-\frac{1}{d-n}\nabla f\otimes \nabla f$.  The main results of this part of the paper are summarized by Theorem \ref{t:smooth_bounded_ricci} and Theorem \ref{t:smooth_bounded_d_ricci}.

\subsection{Preliminaries}\label{ss:prelim_intro}

Because what is done in this paper lies in the intersection of several areas, the notation can be intense.  Often times notational standards of one area differ slightly from others, so we begin here by clarifying our terminology on some relatively standard ideas for use throughout the paper.  This section is relatively brief, we refer to Section \ref{s:prelim_smooth} for a more complete description of our notation.

\subsubsection{The $f$-Laplace Operator and Heat Flow.}\label{sss:f_laplace_intro}
Given a smooth metric measure space $(M^n,g,e^{-f}dv_g)$ let us begin by remarking that there is a canonical geometric differential operator associated to the triple given by the $f$-laplacian 
\begin{align}\label{d:f_laplace}
\Delta_f u \equiv \Delta u - \langle \nabla f,\nabla u\rangle\, .
\end{align}
Notice of course that $\Delta_f$ is a self adjoint operator on the Hilbert space of functions $L^2(M,e^{-f}dv_g)$, which is defined by the inner product
\begin{align}
\langle u,v\rangle_{L^2_f} \equiv \int_M uv\, e^{-f}dv_g \, .
\end{align}

We can define the heat flow $H_t:L^2(M,e^{-f}dv_g)\to L^2(M,e^{-f}dv_g)$ as the flow generated by the operator $\frac{1}{2}\Delta_f$.  The choice of $\frac{1}{2}$ conflicts with most papers on geometric analysis, though is consistent with that of stochastic analysis.  With respect to the heat flow we have the heat kernel $\rho_t(x,dy)\equiv \rho_t(x,y) e^{-f}dv_g(y)$ which is defined by
\begin{align}
H_tu(x) = \int_{M} u(y) \rho_t(x,dy)\, .
\end{align}

\subsubsection{Path Space on $M$ and the Diffusion Measures.}\label{sss:Diffusion_Measure_intro}

Now we introduce the path space $P(M)$ on a smooth manifold and the construction of the diffusion measures on $P(M)$ through the Wiener construction.  Again, we only give a brief overview here, we will discuss these issues more carefully in Section \ref{s:prelim_smooth}.  To begin with, let us introduce the path space of a manifold.  In fact there are many variants that will be interesting to us, the most broad version of path space on $M$ is the total path space given by 
\begin{align}
P(M)\equiv C^0([0,\infty),M)\, .
\end{align}

Notice that this is the collection of continuous unbased paths in $M$.  Our goal will be to do analysis on path space.  To do this will require two basic tools, measures over path space and a geometry over path space.  The geometry we will need over path space is the one induced by the parallel gradient and is not standard, and so we will introduce it in Section \ref{sss:parallel_grad_intro}.  The measures we will study are the diffusion measures.  As we will see in Section \ref{ss:diffusion_measures} these are the measures on path space most naturally associated to the heat flow on $M$.  \\

Recall that path space comes equipped with a very canonical collection of mappings, namely the evaluation mappings.   Specifically, given any partition of times
\begin{align}
\bt=\{0\leq t_1<t_2<\cdots<t_k<\infty\}\, ,
\end{align}
there exists the corresponding evaluation map 
\begin{align}
e_\bt:P(M)\to M\times \cdots \times M = M^{k}\equiv M^{|\bt|}\, ,
\end{align}
given by
\begin{align}
e_\bt(\gamma) = (\gamma(t_1),\ldots,\gamma(t_k))\, .
\end{align}

It is not hard to show that the $\sigma$-algebra generated by all the evaluation maps is the standard Borel $\sigma$-algebra on $P(M)$.  More generally, for each interval $[t,T]$ we could look at the $\sigma$-algebra $\cF^T_t$ generated by the evaluations maps $e_\bt$, where $\bt$ is a partition of $[t,T]$.  In the case where $t=0$ we write $\cF^T$ for the $\sigma$-algebra induced by the partitions of $[0,T]$.  This family of $\sigma$-algebras will play an important role.  Note that a function $F$ on $P(M)$ which is measurable with respect to $\cF_t^T$ only depends on the curve $\gamma$ in the region $[t,T]$.  In particular, such a function can be viewed as living on time restricted path space $P^T(M)\equiv C^0([0,T],M)$.\\

We will define the diffusion measures in general in Section \ref{ss:diffusion_measures}, here we will simply introduce the Wiener measures.  The Wiener measures are a family of measures $\Gamma_x$ on $P(M)$ indexed by $x\in M$.  As we will see, $\Gamma_x$ are the unique measures on $P(M)$ such that for each partition $\bt$ we have that the pushforward measures $e_{\bt,*}\Gamma_x$ are given by
\begin{align}\label{e:WM}
e_{\bt,*}\Gamma_x = \rho_{t_1}(x,dy_1)\rho_{t_2-t_1}(y_1,dy_2)\cdots\rho_{t_k-t_{k-1}}(y_{k-1},dy_k)\, .
\end{align}
See Section \ref{s:prelim_smooth} for more details.  Let us observe that for $x\in M$ fixed, then for $\Gamma_x$-a.e. $\gamma\in P(M)$ we have that $\gamma(0)=x$.  In particular, if we equipped $P(M)$ with the measure $\Gamma_x$ then we have focused ourselves on the based path space $P_x(M)\equiv \{\gamma\in P(M):\gamma(0)=x\}$.  Finally let us denote the total Wiener measure $\Gamma_f$ on path space $P(M)$ defined by $\Gamma_f(U)\equiv \int_M \Gamma_x(U)\,dv_g$, where $U$ is Borel.

\subsubsection{Cylinder Functions}\label{sss:cylinder_fun_intro}

To do analysis on path space $P(M)$ the starting point is to have a natural and easy class of functions that one can work with and are dense in the various function spaces.    In this way one can consider most functional analytic constructions on this subspace and then extend by continuity to more general functions.  In our context an especially natural collection of functions on path space $P(M)$ are the smooth cylinder functions.  These are the functions $F:P(M)\to \dR$ of the form
\begin{align}
F\equiv e_\bt^*u\, ,
\end{align}
where $e_\bt:P(M)\to M^{|\bt|}$ is an evaluation map and $u:M^{|\bt|}\to\dR$ is a smooth function with compact support.  One can imagine from the definition of the Wiener measures (\ref{e:WM}) why the cylinder functions are especially natural choices.  The collection of smooth cylinder functions are dense in essentially every function space we will be interested in throughout this paper.  We will discuss this more in Section \ref{ss:functions_pathspace}.

\subsection{Characterizations of Ricci Curvature:}\label{ss:char_ricci_intro}

In this Section we describe our first main results, which give characterizations of bounds on $\Ric+\nabla^2f$ in terms of the geometry and measure theory of $P(M)$.  These characterizations will fall into several distinct categories, but all will require understanding the relationship between the metric-measure geometry on $M$ and the metric-measure geometry on $P(M)$.  In Section \ref{ss:char_d_ricci_intro} we describe analogous characterizations for the $d$-dimensional Ricci curvature tensor.

The relationship between the measures on $M$ and the measures on $P(M)$ that we will consider was described previously in Section \ref{sss:Diffusion_Measure_intro}, and is given through the construction of the diffusion measures.  In order to discuss the relationship of Ricci curvature with the geometry on path space we still need to describe the correct notion of a gradient for functions on path space.  This will turn out to be what we call the parallel gradient, which acts as an almost finite dimensional gradient on the infinite dimensional path space, and we will give a brief description of it in Section \ref{sss:parallel_grad_intro}.  In Section \ref{ss:parallel_gradient} we will discuss this notion more completely and describe some of its properties.  In particular, we will see in Section \ref{ss:parallel_gradient} how to recover from the parallel gradients the more standard $H^1_0$-gradient on path space.

In Section \ref{sss:char_ricci_gradient} we will be in a position to give our first characterization of bounded Ricci curvature, which relates bounds on the Ricci curvature to the functional analysis of $P(M)$.  Speficially, using the parallel gradient we will see that the eigenvalue bounds $|\Ric+\nabla^2f|\leq \kappa$ are equivalent to the appropriate gradient estimate on path space.  We will see that this gradient estimate acts as the infinite dimensional generalization of the Bakry-Emery gradient estimates on $M$.  In Section \ref{sss:char_ricci_gradient} we will give several versions of the gradient estimate.

In Section \ref{sss:char_ricci_quad_intro} we give our second characterization of bounded Ricci curvature, which relates the bounds on the Ricci curvature to the stochastic analysis of $P(M)$.  To do this we recall the notion of a martingale and its quadratic variation, and then see how using the parallel gradient one can equate bounds on the Ricci curvature with bounds on the quadratic variation of a martingale on $P(M)$.

In Section \ref{sss:char_ricci_OU} we discuss our third characterization of bounded Ricci curvature, which describes the Ricci curvature in terms of the infinite dimensional analysis of $P(M)$.  We begin by recalling briefly the construction of the Ornstein-Uhlenbeck operator $L_x$ on path space, and then we will discuss the construction of the twisted Ornstein-Uhlenbeck operators $L_{t_0,\kappa}^{t_1}$.  We will then see that the eigenvalue bound $|\Ric+\nabla^2f|\leq \kappa$ is equivalent to the spectral gaps for these operators.  We will see in particular that this implies the spectral gap $\lambda_1(L_{x})\geq 2\big(e^{\kappa T}+1\big)^{-1}$ for the classical Ornstein-Uhlenbeck operator.  In fact, we will see that the Ricci curvature bound $|\Ric+\nabla^2f|\leq \kappa$ holds if and only if the twisted Ornstein-Uhlenbeck operators have a log-Sobolev bound of $2e^{\frac{\kappa}{2}\big(T-t_0\big)}$, see Section \ref{sss:char_ricci_OU} for more precision.

\subsubsection{The Parallel Gradients on $P(M)$}\label{sss:parallel_grad_intro}

In order to give our characterizations of bounded Ricci curvature we need to introduce the correct geometry on path space, and in particular we need to discuss the notion of a gradient for functions on path space.  The construction of a gradient on the path space of $\dR^n$ is a very straightforward matter.  However, even on a smooth {\it non-flat} manifold $M$ there are nontrivial issues involved, see \cite{Driver_CM},\cite{Malliavin_StocAnal}, \cite{Hsu_book}, \cite{Stroock_book} and Section \ref{ss:parallel_gradient}.  For simplicity, in the introduction we will mostly skirt these issues and give only semi-rigorous definitions.  In Section \ref{ss:parallel_gradient} we will give on a smooth manifold more rigorous definitions that are line with classical constructions, while in Part II of the paper we will introduce an entirely new approach which will allow us to handle the constructions on nonsmooth metric spaces.\\

There are several notions of a gradient on path space $P(M)$ that will play a role.  The most fundamental of these for us is are the parallel gradients of a function.  It's definition is similar in spirit to the more commonly used $H^1_0$-gradient, though it has a more finite dimensional flavor to it.  We will discuss it and the $H^1_0$-gradient in full detail in Section \ref{ss:H1_gradient}.  In particular, we will show there how to recover the infinite dimensional $H^1_0$-gradient from the {\it family} of 'finite dimensional' parallel gradients.  This will be especially important on nonsmooth metric measure spaces.

To define the parallel gradient first recall that $P(M)$ is a smooth Banach manifold, and that the tangent space of a curve $\gamma\in P(M)$ can be naturally identified with the continuous vector fields $V\in C^0(\gamma^*TM)$.  Thus, if we are given a reasonable mapping $F:P(M)\to \dR$, for instance a smooth cylinder function as in Section \ref{sss:cylinder_fun_intro}, then we can define its partial derivative $D_VF(\gamma)$ at a curve $\gamma$ in the direction $V$.  The norm of the parallel gradient of $F$ is then defined by
\begin{align}\label{e:parallelgradient_0}
|\nabla_0 F|(\gamma)\equiv \sup\big\{D_VF: |\nabla_{\dot\gamma} V| = 0\,\, , |V|(0)=1\big\}\, .
\end{align}
That is, we are taking the supremum over the directional derivatives of $F$ in directions that are parallel translation invariant, which is an $n$-dimensional subspace of the vector fields on $\gamma$.  

More generally, we will define a family of gradients $\nabla_s$.  We will introduce this more carefully in Section \ref{ss:parallel_gradient}, however in short let us notice that for a smooth cylinder function $F$ the directional derivatives $D_VF$ are well defined even for only right continuous vector fields.  In particular, to define $|\nabla_s F|$, instead of maximizing over all parallel translation invariant vector fields, we can maximize over all vector fields $V(t)$ which vanish for $t<s$ and are parallel translation invariant for $t\geq s$.  That is,

\begin{align}\label{e:parallelgradient_s}
|\nabla_s F|(\gamma)\equiv \sup\big\{D_VF: V(t)=0\text{ if }t<s\text{ and otherwise } |\nabla_{\dot\gamma} V| = 0\,\, , |V|(s)=1\big\}\, .
\end{align}

Now let us quickly point out the extreme, but standard, subtlety in these definitions.  Namely, we have that $\gamma\in P(M)$ is only a continuous curve, and in general it will not be anywhere differentiable.  Thus, it is not at all clear what is meant by the parallel translation invariant condition $|\nabla_{\dot\gamma} V|=0$.  In fact, this is not an easy point.  It was first circumvented by Malliavin in \cite{Malliavin_Geometrie} with the use of the stochastic parallel translation map in order to define the $H^1_0$-gradient.  We will discuss the stochastic parallel translation map in Section \ref{ss:stoc_par_trans}, and then use Malliavin's technique to make rigorous the above definition.  We will also handle this issue in a very different manner in Part II of the paper, without the use of the stochastic parallel translation map.  This will allow us to define the parallel and $H^1_0$-gradients of a function on even nonsmooth metric spaces.\\

\subsubsection{Characterizing Bounded Ricci Curvature and Gradient Estimates}\label{sss:char_ricci_gradient}

Now we are in a position to discuss our first characterization of bounded Ricci curvature on $M$.  Let us begin by recalling the classic gradient estimates of Bakry-Emery on the heat flow.  Their estimates tell us that the lower Ricci curvature bound $\Ric+\nabla^2f\geq -\kappa$ is equivalent to the gradient estimate on the heat flow given by
\begin{align}\label{e:BE_gradient}
|\nabla H_t u|\leq e^{\frac{\kappa}{2}t}H_t|\nabla u|\, ,
\end{align}
where $H_t$ is the heat flow associated to the operator $\frac{1}{2}\Delta_f$ on $M$.  We will construct in this Section a path space version of this estimate which gives rise to a characterization of bounded Ricci curvature on $M$.  We will show in Section \ref{ss:r1_r2} how we may recover (\ref{e:BE_gradient}) by applying the path space estimate to essentially the simplest type of function on path space.\\

To describe the characterization let $F\in C^0(P(M))$ be a continuous function on path space, for instance a smooth cylinder function, and let us observe that by letting the diffusion measures $\Gamma_x$ act on $F$ we can construct a continuous function on $M$ by considering
\begin{align}
\int_{P(M)} F \,d\Gamma_x\, ,
\end{align}
as a function of $x$.  This method takes continuous functions on $P(M)$ to continuous functions on $M$, and it is reasonable to ask what else we know about $\int F \,d\Gamma_x$ as a function on $M$ in terms of $F$ as a function on $P(M)$.  In particular, when is it a lipschitz function on $M$, and can we control the gradient of $\int F \,d\Gamma_x$ as a function on $M$ in terms of the gradient of $F$ as a function on $P(M)$.  In this case it of course matters a great deal what we mean by {\it gradient} of $F$ on $P(M)$.  It turns out that if we mean the parallel gradient, as defined in Section \ref{sss:parallel_grad_intro}, then the estimate
\begin{align}\label{e:ricciflat_gradient}
|\nabla \int_{P(M)} F \,d\Gamma_x|\leq \int_{P(M)}|\nabla_0 F|\, d\Gamma_x\, ,
\end{align}
is equivalent to the smooth metric measure space being Ricci flat, that is, we will show that (\ref{e:ricciflat_gradient}) holds if and only if $\Ric+\nabla^2f=0$.  More generally, we will see in Theorem \ref{t:smooth_bounded_ricci} that $|\Ric+\nabla^2 f|\leq \kappa$ if and only if
\begin{align}\label{e:ricci_gradient}
|\nabla \int_{P(M)} F \,d\Gamma_x|\leq \int_{P(M)}|\nabla_0 F|+\int_0^\infty \frac{\kappa}{2}e^{\frac{\kappa}{2}s}|\nabla_s F|\, ds\cdot d\Gamma_x\, .
\end{align}

In Section \ref{ss:r1_r2} we see how to apply the above estimate to the simplest function on path space, namely a cylinder function $F(\gamma)\equiv u(\gamma(t))$ where $u$ is a smooth function on $M$ and $t>0$ is fixed, to recover (\ref{e:BE_gradient}).

It will also be useful to consider a quadratic version of the (\ref{e:ricci_gradient}).  In this case we have, again recorded in Theorem \ref{t:smooth_bounded_ricci}, that the Ricci curvature bound $|\Ric+\nabla^2 f|\leq \kappa$ holds if and only if for every $\cF^T$ measurable function $F$ on $P(M)$ we have the estimate
\begin{align}
|\nabla_x \int_{P(M)} F \,d\Gamma_x|^2 \leq e^{\frac{\kappa}{2}T}\int_{P(M)}|\nabla_0 F|^2+\int_0^T \frac{\kappa}{2}e^{\frac{\kappa}{2}s}|\nabla_s F|^2\,ds\cdot d\Gamma_x\, .
\end{align}
\vspace{.5 cm}

\subsubsection{Characterizing Bounded Ricci Curvature and the Quadratic Variation}\label{sss:char_ricci_quad_intro}

Our second characterization of bounded Ricci curvature relates the bounds on the Ricci curvature to stochastic analysis on $M$.  In order to state the results let us briefly review the notion of a martingale on based path space $P_x(M)$, and its associated quadratic variation.  For simplicity we only consider $L^2$-martingales here, however we will review these notions more carefully in Section \ref{ss:mart_quad}, and they will play an especially important role in the structure theory in the second part of the paper.\\

Let us begin with an arbitrary function $F\in L^2(P_x(M),\Gamma_x)$ which is square integrable.  Recall from Section \ref{sss:Diffusion_Measure_intro}, discussed more completely in Section \ref{ss:pathspace_basics}, that path space comes equipped with a canonical family of $\sigma$-algebras $\cF^t$.  Thus for each $t>0$ we can consider the closed subspace $L^2(P^t_x(M),\Gamma_x)\subseteq L^2(P_x(M),\Gamma_x)$ formed by those functions which are $\cF^t$-measurable.  By the remarks of Section \ref{sss:Diffusion_Measure_intro} this closed subspace may be naturally identified with the space of $L^2$ functions on time restricted path space $P^t_x(M)\equiv \{\gamma\in C^0([0,t]:M):\gamma(0)=x\}$.  Now if $F\in L^2(P_x(M),\Gamma_x)$ is an arbitrary function we can denote by $F^t$ the one parameter family of functions obtained by projecting $F$ to the subspace of $\cF^t$-measurable functions.  This decomposition of $F$ into a one parameter family of functions $F^t$ is called the {\it martingale} on $P_x(M)$ generated by $F$, see Section \ref{ss:mart_quad} for more details.

The martingale $F^t$ is a measurement of how much of $F$, as a function on path space, depends only on first $[0,t]$ of a curve.  As a family of functions $F^t$ is highly nondifferentiable.  To see this note that for any partition $\bt=\{0\leq t_1<\cdots<t_{|\bt|}<\infty\}$ we have the identity
\begin{align}\label{e:quadratic_isometry}
||F||_{L^2}^2 = \sum ||F^{t_{k+1}}-F^{t_k}||^2_{L^2}\, .
\end{align}
From this it is clear that not only is the family $F^t$ not differentiable in the $t$-variable, but what we may hope to converge is the quadratic limit 
\begin{align}
\lim_{s\to 0} \frac{\big(F^{t+s}-F^{t}\big)^2}{s} = [dF^t]\, ,
\end{align}
in $L^1$, for at least $a.e.$ $t>0$.  For a general martingale one has to be a little careful about such limits, however from the right perspective this turns out to be true \cite{Kuo_book}, and in fact if $F$ is well behaved, say a smooth cylinder function, then the limit exists for every time $t$.  The infinitesimal quadratic variation $[dF^t]$ is the correct replacement for the time derivative of $F^t$.  Using (\ref{e:quadratic_isometry}) it is not surprising that we have the isometric identity
\begin{align}
\int_{P(M)}\big|F^t\big|^2\, d\Gamma_x -\bigg(\int_{P(M)}F\,d\Gamma_x\bigg)^2= \int_{P(M)}\bigg(\int_0^t[dF^s]\bigg)\, d\Gamma_x \equiv \int_{P(M)}[F^t]\, d\Gamma_x\, . 
\end{align}
We call the family of maps $[F^t]$ the quadratic variation of $F$.  A reasonable question would be what properties of $F$ control the quadratic variation $[F^t]$ and its infinitesimal $[dF^t]$.  Our main result in this Section is that the estimate
\begin{align}\label{e:var_est_intro0}
\int_{P(M)}\sqrt{[dF^t]}\, d\Gamma_x \leq \int_{P(M)} |\nabla_t F|+\int_t^T \frac{\kappa}{2}e^{\frac{\kappa}{2}(s-t)}|\nabla_s F|\, d\Gamma_x\, ,
\end{align}
for every $\cF^T$-measurable function $F$ and $t<T$ is equivalent to the Ricci curvature eigenvalue bound $|\Ric+\nabla^2f|\leq \kappa$.  As with the gradient estimate it will be useful, especially when considering the dimensional Ricci curvature, to consider the quadratic version of the above which states that the estimate
\begin{align}\label{e:var_est_intro}
\int_{P(M)}[dF^t]\, d\Gamma_x \leq e^{\frac{\kappa}{2}(T-t)}\int_{P(M)} |\nabla_t F|^2+\int_t^T \frac{\kappa}{2}e^{\frac{\kappa}{2}(s-t)}|\nabla_s F|^2\, d\Gamma_x\, ,
\end{align}
for every $\cF^T$-measurable function $F$ and $t<T$ is also equivalent to the Ricci curvature bound $-\kappa g\leq \Ric+\nabla^2f\leq \kappa g$.  \\

In fact, it is not hard to see that (\ref{e:var_est_intro}) implies a seemingly stronger pointwise estimate.  If $\gamma\in P_x(M)$ and $t\geq 0$ are fixed, then if we denote by $F_{\gamma_t}:P_{\gamma(t)}(M)\to\dR$ the function defined by $F_{\gamma_t}(\sigma)\equiv F(\gamma_{[0,t]}\circ\sigma)$, where $\circ$ denotes the concatenation of paths, then we have the pointwise estimate
\begin{align}
[dF^t](\gamma)&\leq e^{\frac{\kappa}{2}(T-t)}\int_{P(M)} |\nabla_t F|^2(\gamma_{[0,t]}\circ\sigma)+\int_t^{T}\frac{\kappa}{2}e^{\frac{\kappa}{2}(s-t)}|\nabla_{s} F|^2(\gamma_{[0,t]}\circ\sigma)\,d\Gamma_{\gamma(t)}\notag\\
&=e^{\frac{\kappa}{2}(T-t)}\int_{P(M)} |\nabla_0 F_{\gamma_t}|^2+\int_0^{T-t}\frac{\kappa}{2}e^{\frac{\kappa}{2}s}|\nabla_{s} F_{\gamma_t}|^2\,d\Gamma_{\gamma(t)}\notag\\
\end{align}
for $a.e.$ $\gamma\in P_x(M)$, see Section \ref{ss:r3_r5} for details.  In particular this tells us that for spaces with bounded Ricci curvature that a martingale $F^t$ induced by a well behaved function $F$, for instance a smooth cylinder function, is exactly $C^{\frac{1}{2}}$-H\"older when viewed as a family in $L^2(P_x(M),\Gamma_x)$.  

Let us also discuss the pointwise regularity $F^t(\gamma)$ of a martingale and its relationship to Ricci curvature.  In Theorem \ref{t:smooth_martingale_lowerricci} we will show that only under a lower Ricci curvature assumption that $F^t(\gamma)$ is a continuous function of time for each fixed $\gamma$.  In Theorem \ref{t:smooth_martingale_holder} we will show with a bound on the Ricci curvature that $F^t(\gamma)$ is $C^\alpha$-H\"older continuous in time for all $\alpha<\frac{1}{2}$.  These results will be especially important in the second part of the paper.

As a last remark let us compare these estimates to the lower Ricci curvature context, and in particular let us note that (\ref{e:var_est_intro}) implies a generalization of the Bakry-Emery gradient estimate when applied to the simplest functions on path space.  Specifically, when we apply (\ref{e:var_est_intro}) to the functions of the form $F(\gamma)\equiv u(\gamma(t))$, where $u$ is a smooth function on $M$ and $t$ is fixed, then we will get the estimate
\begin{align}\label{e:BE_gradient2}
H_t|\nabla H_{T-t}u|^2(x)\leq e^{\kappa(T-t)}H_T|\nabla u|^2(x)\, ,
\end{align}
for every smooth $u$ and all times $0\leq t\leq T$.  It is not hard to see that (\ref{e:BE_gradient2}) is equivalent to (\ref{e:BE_gradient}), and in particular is itself equivalent to the Ricci curvature lower bound $\Ric+\nabla^2 f\geq -\kappa g$.\\

\subsubsection{Characterizing Bounded Ricci Curvature and the Ornstein Uhlenbeck Operators}\label{sss:char_ricci_OU}

Our third characterization of bounded Ricci curvature shows how to equate bounds on the Ricci curvature of a smooth metric-measure space with the analysis on path space.  Specifically, we will define below the Ornstein-Uhlenbeck operators as well as its twisted variations, which are infinite dimensional laplacians on path space, and see how the spectral properties of these operators are equivalent to bounds on the Ricci tensor.\\

Spectral gap and log-Sobolev inequalities for the Ornstein-Uhlenbeck operator on path space have a long history.  In the context of path space on $\dR^n$ they were first proved by Gross \cite{Gross_LogSob}.  In this case one can approximate in a very strong sense the Ornstein-Uhlenbeck operator by finite dimensional operators and thus prove the estimate rather directly by more classical arguments.  In the case of path space on a smooth Riemannian manifold the Ornstein-Uhlenbeck operator was first defined in \cite{DriverRockner_Diff}, and its spectral gap and log-Sobolev properties were first studied in \cite{Fang_Poin}, \cite{Hsu_LogSob} and \cite{AidaElworthy_logsob}.  In \cite{AidaElworthy_logsob} it was proven that such estimates existed for an arbitrary compact Riemannian manifold.  To prove the result the manifold was isometrically embedded in Euclidean space, and therefore the spectral gap itself depended on the embedding.  In \cite{Fang_Poin},\cite{Hsu_LogSob} it was first understood that Ricci curvature could also be used to control the spectral gap and log-Sobolev inequalities.    The proof in \cite{Fang_Poin} was based on a clever manipulation of the martingale representation formula for manifolds, which itself was based on a combination of the classic Clark-Ocone-Haussmann formula and Driver's integration by parts formula for the Malliavin gradient \cite{Driver_CM}.  The proof in \cite{Hsu_LogSob} is based on a more inductive procedure.  We refer the reader to the useful book \cite{Hsu_book} for a more complete reference.\\

In this section we define a new class of operators, which include in them the classical Ornstein-Uhlenbeck operator.  Our main goal for this section is to see that bounds on the Ricci curvature are equivalent to a form of spectral gap and log-Sobolev estimates on these operators.  In fact, we can deduce from them estimates on the classical Ornstein-Uhlenbeck operator which are sharper than those currently in the literature.  These improved estimates are actually vital for our purposes, as the {\it equivalence} between these estimates and the bounds on the Ricci curvature fail for the weaker estimates currently in the literature.  The techniques will also generalize in the second part of the paper to allow us to prove the corresponding estimates for nonsmooth metric measure spaces which have generalized Ricci curvature bounds.  In this case the Ornstein-Uhlenbeck operators may {\it apriori} not even be linear operators, none-the-less we will be able to prove the correct log-Sobolev and Poincare estimates on them.\\

To explain all of this more carefully let us briefly discuss the $H^1_x$-gradient on path space, which was first introduced in \cite{Malliavin_Geometrie}.  We will be interested in what's to come in studying functions $F$ which are defined on based path space $P_x(M)$.  Normally it is easier to consider the constructions on smooth cylinder functions first, and then to extend more arbitrarily.  Classically, one defines the $H^1_x$-gradient on based path space in a manner similar to the parallel gradient (\ref{e:parallelgradient_0}) by
\begin{align}\label{e:H1gradient}
|\nabla F|_{H^1_x}(\gamma)\equiv \sup\big\{D_VF: \int_\gamma |\dot V|^2 = 1, V(0)=0\big\}\, .
\end{align}
Again, as with the parallel gradient we remind the reader of the subtlety of the definition, which we will discuss more carefully in Section \ref{ss:H1_gradient}.  In fact, it will often be more convenient for us to express the $H^1_x$-gradient in terms of the parallel gradient by the formula
\begin{align}
|\nabla F|^2_{H^1_x}(\gamma) \equiv \int_0^\infty |\nabla_s F|^2\,ds\, ,
\end{align}
see Section \ref{ss:H1_gradient} for the proof of the equivalence of the two definitions. 

Now on based path space $P_{x}(M)$ we have introduced both a natural geometry given by the $H^1_x$-gradient, and a canonical measure given by the Wiener measure $\Gamma_x$.  This allows us to define a Dirichlet form, from which the Ornstein-Uhlenbeck operator will be defined.  Namely, we define the closed symmetric bilinear form on $L^2(P_{x}(M),\Gamma_x)$ by the formula
\begin{align}
E[F]\equiv \frac{1}{2}\int_{P_{x}(M)} |\nabla F|_{H^1_x}^2 d\Gamma_x\, .
\end{align}
In fact, we have that the energy functional $E[F]$ is a Dirichlet form, see \cite{DriverRockner_Diff} for the smooth case and the second part of the paper for the nonsmooth case.   In particular, by the standard theory of Dirichlet forms \cite{MaRockner_DirichletForms}, there exists a unique, closed, nonnegative, self-adjoint operator
\begin{align}
L_{x}:L^2(P_{x}(M),\Gamma_x)\to L^2(P_{x}(M),\Gamma_x)\, ,
\end{align}
such that
\begin{align}
E[F] = \int_{P_{x}(M)} \langle F,L_{x}F\rangle\, d\Gamma_x\, .
\end{align}
The operator $L_{x}$ is the Ornstein-Uhlenbeck operator on $P_{x}(M)$.  Let us now generalize this to define the twisted Ornstein-Uhlenbeck operators.  These will in essence be the part of the Ornstein-Uhlenbeck operator restricted to the time interval $[t_0,t_1]$.  To be more precise, let us pick times $t_0<t_1$ as well as fix $\kappa\geq 0$, then we define the Dirichlet energies
\begin{align}
E^{t_1}_{t_0,\kappa}[F]\equiv \int_{P_x(M)}\Bigg( \int_{t_0}^{t_1}\cosh\big(\frac{\kappa}{2}(s-t_1)\big)|\nabla_s F|^2 +\big(1-e^{-\frac{\kappa}{2}(t_2-t_1)}\big)\int^{\infty}_{t_2} e^{\frac{\kappa}{2}(s-t_1)}|\nabla_s F|^2\Bigg)\,d\Gamma_x\, ,
\end{align}
and from this we define the twisted Ornstein-Uhlenbeck operators given by
\begin{align}
E^{t_1}_{t_0,\kappa}[F] = \int_{P_{x}(M)} \langle F,L_{t_0,\kappa}^{t_1}F\rangle\, d\Gamma_x\, .
\end{align}

Before continuing let us discuss these operators and their meaning.  For intuition purposes it is best to begin with the $\kappa\equiv 0$ case to see that
\begin{align}
E^{t_1}_{t_0,0}[F]\equiv \int_{P_x(M)}\Bigg( \int_{t_0}^{t_1}|\nabla_s F|^2\Bigg)\,d\Gamma_x\, .
\end{align}
Thus, we see that $L_{t_0}^{t_1}\equiv L_{t_0,0}^{t_1}$ is literally the part of $L_x$ which only looks at the piece of the gradient in the time range $[t_0,t_1]$.  In particular, we have that $L_{0}^{\infty}\equiv L_x$ is the classical Ornstein-Uhlenbeck operator itself. \\ 

Note that the kernel of $L_{t_0}^{t_1}$ contains the subspace 
\begin{align}
L^2(P^{t_0}_x(M),\Gamma_x)\subseteq \ker L_{t_0}^{t_1}\, ,
\end{align}
where recall $L^2(P^{t_0}_x(M),\Gamma_x)$ is the collection of $\cF^{t_0}$-measurable functions.  One might ask if this is the whole kernel, or even more if there is a spectral gap between the kernel and the rest of the spectrum of $L_{t_0}^{t_1}$.  What we will see in Theorem \ref{t:smooth_bounded_ricci} is that the weak spectral gap estimate
\begin{align}
\int_{P_xM} |F^{t_1}-F^{t_0}|^2 \leq \int_{P_xM} \langle F,L_{t_0}^{t_1} F\rangle\, ,
\end{align}
holds if and only if $\Ric+\nabla^2f\equiv 0$.  In the case where $t_1=\infty$, then this gives rise to a classical spectral gap for the operator $L^{\infty}_{t_0}$, as in this case we have $F^\infty=F$.  One can go further, and see that the log Sobolev estimate
\begin{align}
\int_{P_{x}(M)} |F^{2}|^{t_1}\ln |F^{2}|^{t_1}\,d\Gamma_x-\int_{P_{x}(M)} |F^{2}|^{t_0}\ln |F^{2}|^{t_0}\,d\Gamma_x  \leq 2\int_{P_{x}(M)}\langle F, L_{t_0}^{t_1} F\rangle\,d\Gamma_x\, ,
\end{align}
holds if and only if $\Ric+\nabla^2 f=0$, where $|F^2|^t$ is the martingale induced by $F^2$.  Note in particular that these imply spectral gap and log-sobolev estimates on the classical Ornstein-Uhlenbeck operator by using the observation $L_{0}^{\infty}=L_x$.\\

More generally, the operators $L^{t_1}_{t_0,\kappa}$ contain an additional twisting term to account for additional Ricci curvature, but should in principle be viewed in the same manner as $L^{t_1}_{t_0}$.  We will see that these operators control Ricci curvature bounded by $\kappa$ is the same manner that $L^{t_1}_{t_0}$ controls Ricci flat manifolds.  The main result of Theorem \ref{t:smooth_bounded_ricci} is that the spectral gap
\begin{align}\label{e:intro:OU:spectral_gap_k}
\int_{P_xM} |F^{t_1}-F^{t_0}|^2\leq e^{\frac{\kappa}{2}(T-t_0)}\int_{P_xM} \langle F,L_{t_0,\kappa}^{t_1} F\rangle\,d\Gamma_x\, ,
\end{align}
where $F$ is $\cF^T$-measurable, holds if and only if $|\Ric+\nabla^2f|\leq \kappa$.  More generally, the log-Sobolev
\begin{align}\label{e:intro:OU:log_sob_k}
\int_{P_{x}(M)} \big(F^2\big)^{t_1}\ln \big(F^2\big)^{t_1}\,d\Gamma_x-\int_{P_{x}(M)} \big(F^2\big)^{t_0}\ln \big(F^2\big)^{t_0}\,d\Gamma_x  \leq 2e^{\frac{\kappa}{2}(T-t_0)}\int_{P_{x}(M)}\langle F, L_{t_0,\kappa}^{t_1} F\rangle\,d\Gamma_x\, ,
\end{align}
holds for all $F\in L^2(P^T_x(M),\Gamma_x)$ if and only if $|\Ric+\nabla^2f|\leq \kappa$.  Let us remark that on time restricted path space $P^T_x(M)$ we have the estimate
\begin{align}
L_{0,\kappa}^{T}\leq \cosh\big(\frac{\kappa}{2}T\big)L_x\, ,
\end{align}
where recall $L_x$ is the standard Ornstein-Uhlenbeck operator.  In particular, the estimates \eqref{e:intro:OU:spectral_gap_k} and \eqref{e:intro:OU:log_sob_k} imply the following spectral gap and log-sobolev on the classical Ornstein-Uhlenbeck operators:
\begin{align}\label{e:OU_spectralgap}
&\int_{P_xM} |F-\int F|^2\leq \frac{1}{2}\big(e^{\kappa T}+1\big)\int_{P_xM} \langle F,L_x F\rangle\,d\Gamma_x\, ,
\end{align}
\begin{align}\label{e:OU_logsob}
&\int_{P_{x}(M)} |F|^2\ln |F|^2\,d\Gamma_x- \Big(\int F^2\Big) \ln \Big(\int F^2\Big)  \leq \big(e^{\kappa T}+1\big)\int_{P_{x}(M)}\langle F, L_{x} F\rangle\,d\Gamma_x\, . \\ \notag
\end{align}

We end the Section by discussing the relationship between this estimate and the lower Ricci curvature version.  To do this precisely we need to discuss the heat kernel laplacian.  Specifically, we have discussed in Section \ref{sss:f_laplace_intro} how given a metric measure space one can naturally associate a laplace operator.  Now with $x\in M$ fixed and $t>0$ one can consider the metric-measure space $(M^n,g,\rho_t(x,dy))$, where $\rho_t(x,dy)$ is the heat kernel measure.  The laplace operator associated to this triple is the heat kernel laplacian
\begin{align}\label{d:hk_laplace}
\Delta_{x,t} u \equiv \Delta u + \langle \nabla \ln\big(\rho_{x,t}e^{-f}\big),\nabla u\rangle\, .
\end{align}

We will see in Section \ref{ss:r4_r7} that when one applies (\ref{e:OU_spectralgap}) or (\ref{e:OU_logsob}) to the simplest functions on path space, namely functions of the form $F(\gamma)=u(\gamma(t))$, then one obtains the spectral gap 
\begin{align}\label{e:hk_laplace_spectralgap}
\lambda_1(-\Delta_{x,t})\geq \kappa\left( e^{\kappa t}-1\right)^{-1}\, ,
\end{align}
for every $x\in M$ and $t>0$, and the log-Sobolev
\begin{align}\label{e:hk_laplace_logsob}
\int_M u^2\ln u^2 \rho_t(x,dy) \leq 2 \kappa^{-1}\left(e^{\kappa t}-1\right)\int_M |\nabla u|^2 \rho_t(x,dy)\, ,
\end{align}
where $u$ is any function such that $\int_M u^2(y) \rho_t(x,dy)=1$.  A consequence of \cite{BakryLedoux_logSov} is that these estimates are themselves equivalent to the lower Ricci bound $\Ric+\nabla^2 f\geq -\kappa g$, and therefore we have again recovered the lower Ricci curvature from the path space estimate.\\

\subsubsection{Summary of Results}

Let us record the main statements of Section \ref{ss:char_ricci_intro} and some of the easy corollaries.  For the notation we refer back to Section \ref{ss:char_ricci_intro}.

\begin{theorem}\label{t:smooth_bounded_ricci}
Let $(M^n,g,e^{-f}dv_g)$ be a smooth metrically complete metric measure space, then the following are equivalent:
\begin{enumerate}
\item[(R1)] The Ricci curvature satisfies the bound 
\begin{align}
-\kappa g\leq \Ric+\nabla^2 f\leq \kappa g\, .
\end{align}

\item[(R2)] For any function $F\in L^2(P(M),\Gamma_{f})$ on the total path space $P(M)$ we have the estimate
\begin{align}
\big|\nabla \int_{P(M)} F\,d\Gamma_{x}\big| \leq \int_{P(M)} \bigg(|\nabla_0 F|+\int_0^\infty\frac{\kappa}{2}e^{\frac{\kappa}{2}s}\,|\nabla_s F|\,ds\bigg) d\Gamma_{x}\, .
\end{align}

\item[(R3)] For any function $F\in L^2(P(M),\Gamma_{f})$ on the total path space $P(M)$ which is $\cF^T$-measurable we have the estimate
\begin{align}
|\nabla \int_{P(M)} F \,d\Gamma_x|^2 \leq e^{\frac{\kappa}{2}T}\int_{P(M)}\,|\nabla_0 F|^2+\int_0^T \frac{\kappa}{2}e^{\frac{\kappa}{2}s}|\nabla_s F|^2\, ds\cdot d\Gamma_x\, .
\end{align} 

\item[(R4)] For any function $F\in L^2(P(M),\Gamma_{x})$ on based path space $P_x(M)$ we have the estimate
\begin{align}
\int_{P(M)}\sqrt{[dF^t]}\, d\Gamma_x \leq \int_{P(M)} |\nabla_t F|+\int_t^T \frac{\kappa}{2}e^{\frac{\kappa}{2}(s-t)}|\nabla_s F|\, d\Gamma_x\, ,
\end{align}

\item[(R5)] For any function $F\in L^2(P(M),\Gamma_{x})$ on based path space $P_x(M)$ which is $\cF^T$-measurable we have the estimate
\begin{align}
\int_{P(M)}[dF^t]\, d\Gamma_x \leq e^{\frac{\kappa}{2}(T-t)}\int_{P(M)} |\nabla_t F|^2+\int_t^T \frac{\kappa}{2}e^{\frac{\kappa}{2}(s-t)}|\nabla_s F|^2\, d\Gamma_x\, ,
\end{align}

\item[(R6)] The twisted Ornstein-Uhlenbeck operator $L_{t_0,\kappa}^{t_1}:L^2(P_{x}^T(M),\Gamma_x)\to L^2(P_{x}^T(M),\Gamma_x)$ on based path space $P_{x}^T(M)$ satisfies the spectral gap estimate 
\begin{align}
\int_{P_xM} |F^{t_1}-F^{t_0}|^2\leq e^{\frac{\kappa}{2}(T-t_0)}\int_{P_xM} \langle F,L_{t_0,\kappa}^{t_1} F\rangle\,d\Gamma_x\, .
\end{align}
In particular, the standard Ornstein-Uhlenbeck operator $L_x$ satisfies the spectral gap\newline $\int_{P_xM} |F|^2\leq \frac{1}{2}\big(e^{\kappa T}+1\big)\int_{P_xM} \langle F,L_{x} F\rangle\,d\Gamma_x\,$, for each $\int F=0$.
\item[(R7)] The twisted Ornstein-Uhlenbeck operator $L_{t_0,\kappa}^{t_1}:L^2(P_{x}^T(M),\Gamma_x)\to L^2(P_{x}^T(M),\Gamma_x)$ on based path space $P_{x}^T(M)$ satisfies the log-Sobolev estimate 
\begin{align}
\int_{P_{x}(M)} |F^2|^{t_1}\ln |F^2|^{t_1}\,d\Gamma_x-\int_{P_{x}(M)} |F^2|^{t_0}\ln |F^2|^{t_0}\,d\Gamma_x  \leq 2e^{\frac{\kappa}{2}(T-t_0)}\int_{P_{x}(M)}\langle F, L_{t_0,\kappa}^{t_1} F\rangle\,d\Gamma_x\, .
\end{align}
In particular, the standard Ornstein-Uhlenbeck operator $L_x$ satisfies the log-Sobolev\newline $\int_{P_xM} |F|^2\ln |F|^2\leq \big(e^{\kappa T}+1\big)\int_{P_xM} \langle F,L_{x} F\rangle\,d\Gamma_x\, ,$ for each $\int F^2=1$.
\end{enumerate}
\end{theorem} 
\begin{remark}
The assumption of completeness here only refers to the metric completeness.   Stochastic completeness, which is to say that $\Gamma_x$ is a probability measure, is then a consequence of any of the conditions $(R1)-(R7)$.  More precisely, once one knows the lower Ricci bound then stochastic completeness follows.
\end{remark}
\vspace{.5 cm}

An obvious but interesting corollary of the above is the following characterization of Ricci flat manifolds.

\begin{corollary}
Let $(M^n,g,e^{-f}dv_g)$ be a smooth metrically complete metric measure space, then the following are equivalent:
\begin{enumerate}
\item The space is Ricci flat, that is,  $\Ric+\nabla^2f=0$.

\item For any function $F$ on the total path space $P(M)$ we have the estimate $\big|\nabla \int_{P(M)} F\,d\Gamma_{x}\big| \leq \int_{P(M)} |\nabla_0 F|\, d\Gamma_{x}$.

\item For any function $F$ on based path space $P_x(M)$ we have the estimate $\int_{P_x(M)}\sqrt{[dF^t]}\, d\Gamma_x \leq \int_{P_x(M)} |\nabla_t F|\, d\Gamma_x$.

\item The twisted Ornstein-Uhlenbeck operator $L_{t_0}^{t_1}:L^2(P_{x}^T(M),\Gamma_x)\to L^2(P_{x}^T(M),\Gamma_x)$ on based path space $P_{x}^T(M)$ satisfies the spectral gap estimate $\int_{P_xM} |F^{t_1}-F^{t_0}|^2\leq \int_{P_xM} \langle F,L_{t_0}^{t_1} F\rangle\,d\Gamma_x$.

\item The twisted Ornstein-Uhlenbeck operator on based path space satisfies the log-Sobolev estimate $\int_{P_{x}(M)} |F^2|^{t_1}\ln |F^2|^{t_1}\,d\Gamma_x-\int_{P_{x}(M)} |F^2|^{t_0}\ln |F^2|^{t_0}\,d\Gamma_x  \leq 2\int_{P_{x}(M)}\langle F, L_{t_0}^{t_1} F\rangle\,d\Gamma_x$.
\end{enumerate}
\end{corollary}
\begin{remark}
In fact, by slightly changing the Wiener measure $\Gamma$ so that it is induced by the kernel of the operator $\big(\frac{d}{dt}-\frac{1}{2}\Delta +\frac{\kappa}{2}\big)u=0$, one can instead characterize solutions of the equation $\Ric+\nabla^2f=\kappa g$.
\end{remark}
\vspace{.5 cm}


\subsection{Characterizations of $d$-dimensional Ricci Curvature}\label{ss:char_d_ricci_intro}

Recall that in the context of lower Ricci curvature often the best estimates come not just from the lower Ricci curvature bound $\Ric+\nabla^2 f\geq -\kappa g$, but from the $d$-dimensional lower Ricci curvature bound $\Ric+\nabla^2f-\frac{1}{d-n}\nabla f\otimes \nabla f\geq -\kappa g$.  Now while a lower bound on this tensor is an improvement of a lower bound on the Ricci tensor, an upper bound on this tensor is strictly weaker than an upper bound on the Ricci tensor.  Therefore, we say a smooth metric-measure space $(M^n,g,e^{-f}dv_g)$ has $d$-dimensional Ricci curvature bounded by $\kappa$ if we have the bounds
\begin{align}
-\kappa g+ \frac{1}{d-n}\nabla f\otimes \nabla f\leq \Ric+\nabla^2f\leq \kappa g\, .
\end{align}

In this Section we extend the results of the previous Section to consider the case of the $d$-dimensional Ricci curvature bounds.  We use heavily the notation and ideas already introduced in Section \ref{ss:char_ricci_intro}.  

We saw in Section \ref{sss:char_ricci_gradient} that a bound on the Ricci curvature is tied with control over the gradient of functions of the form $\int F\,d\Gamma_x$.  When one controls bounds on the dimensional Ricci curvature, then our main results show that such bounds are equivalent to control over not only the gradient of $\int F\,d\Gamma_x$, but also of its laplacian.  To obtain such control we need to consider functions $F$ which are $\cF^T_t$-measurable, and typically the estimates will depend on $T$ and $t$. 

The following version of Theorem \ref{t:smooth_bounded_ricci} for the $d$-dimensional Ricci curvature is a relatively simple consequence of Theorem \ref{t:smooth_bounded_ricci} and the corresponding results on lower Ricci curvature, see \cite{BakryLedoux_logSov} and Section \ref{s:smooth_lower_ricci}.  It will be proved in Section \ref{s:smooth_d_ricci}.

\begin{theorem}\label{t:smooth_bounded_d_ricci}
Let $(M^n,g,e^{-f}dv_g)$ be a smooth metrically complete metric measure space, then the following are equivalent:
\begin{enumerate}
\item The $d$-dimensional Ricci curvature satisfies the bound 
\begin{align}
-\kappa g+\frac{1}{d-n}\nabla f\otimes \nabla f \leq \Ric+\nabla^2f\leq \kappa g\, .
\end{align}

\item For any function $F\in L^2(P(M),\Gamma_{f})$ on the total path space $P(M)$ which is $\cF^T_t$-measurable we have the estimate
\begin{align}
|\nabla_x \int_{P(M)} F \,d\Gamma_x|^2+\frac{e^{\kappa t}-1}{\kappa d}\,\big|\Delta_f \int_{P(M)}F\,d\Gamma_x\big|^2 \leq e^{\frac{\kappa}{2}T}\int_{P(M)}\bigg(|\nabla_0 F|^2+\int_0^T \frac{\kappa}{2}e^{\frac{\kappa}{2}s}|\nabla_s F|^2\, ds\bigg) d\Gamma_x\, .
\end{align} 

\item If $F\in L^2(P(M),\Gamma_x)$ is $\cF^T$-measurable, then for $\gamma\in P_x(M)$ if we denote by $t_-\geq 0$ the maximum $s$ such that $F_{\gamma_t}$ is $\cF^T_s$-measurable, then we have the estimate
\begin{align}
[dF^t](\gamma)+\frac{e^{\kappa t_-}-1}{\kappa d}\,\big|\Delta_f \int_{P(M)}F_{\gamma_t}\,d\Gamma_{\gamma(t)}\big|^2&\leq e^{\frac{\kappa}{2}(T-t)}\int_{P(M)} |\nabla_0 F_{\gamma_t}|^2+\int_0^{T-t}\frac{\kappa}{2}e^{\frac{\kappa}{2}s}|\nabla_{s} F_{\gamma_t}|^2\,d\Gamma_{\gamma(t)}
\end{align}

\end{enumerate}
\end{theorem} 
\begin{remark}
By using $(3)$ there are many variations of $(R6)$ and $(R7)$ which are provable.\\
\end{remark}

\section{Introduction to Part II: The NonSmooth Bounded Ricci Case}\label{s:nonsmooth_bounded_ricci_intro}

In this Section we will outline how to use the results of the first part of the paper in order to define the notion of bounded Ricci curvature on a metric measure space $(X,d,m)$.  Throughout the second part of the paper the minimal requires we make on the metric-measure space is that

\begin{align}\label{e:mms_assumptions_intro}
(X,&d,m) \text{ is a locally compact, complete length space such that}\notag\\ 
&\text{$m$ is a locally finite, $\sigma$-finite Borel measure with supp}\,m=X\, .
\end{align}

In fact, it is quite possible that these assumptions, in particular the local compactness, may be weakened, but we do not worry about this here.  A primary complication in defining the notion of bounded Ricci curvature is the construction of the geometry on path space, and in particular the parallel gradient, which was subtle even for a smooth manifold and required the stochastic parallel translation map.  After introducing some preliminaries in Section \ref{ss:prelim_intro_nonsmooth}, which will be discussed in more detail in Section \ref{s:prelim_nonsmooth}, we will give a brief introduction to the construction of the parallel gradient on $P(X)$ in Section \ref{ss:parallel_grad_intro_nonsmooth}.  We will only outline enough of it here in order to make the definition of bounded Ricci curvature and state the main theorems of the paper.  We will discuss it in more detail in Section \ref{s:parallel_gradient_nonsmooth}.  

In Section \ref{ss:ricci_intro_nonsmooth} we give our definitions of bounded Ricci curvature on a metric-measure space, and discuss some of the basic properties of such spaces.  In particular, we will see in Section \ref{ss:bounded_implies_lower_intro_nonsmooth} that a metric-measure space with Ricci curvature bounded by $\kappa$ have their lower Ricci curvatures bounded from below by $-\kappa$ in either the sense of Bakry-Emery or Lott-Villani-Sturm.  In fact, we will see that an even stronger notion of a lower Ricci curvature bound holds, in that a space with Ricci curvature bounded by $\kappa$ will be a $RCD(\kappa,\infty)$ space, see \cite{Ambrosio_Ricci} and Section \ref{s:lowerricci_nonsmooth}.

In Sections \ref{ss:boundedricci_martingales_intro} and \ref{ss:ricci_analysis_intro_nonsmooth} we will study properties on metric measure spaces with bounded Ricci curvature which we had studied in the first part for smooth spaces.  In particular in Section \ref{ss:boundedricci_martingales_intro} we see that  bounded Ricci curvature is tied to the regularity of martingales on $P(X)$.  On the other hand, in Section \ref{ss:ricci_analysis_intro_nonsmooth} we discuss more carefully the implications of bounded Ricci curvature on the analysis of path space $P(X)$.  Using the ideas of Section \ref{ss:parallel_grad_intro_nonsmooth} we show there exists an Ornstein-Uhlenbeck operator on path space, which still acts as an infinite dimensional laplacian.  Note however, that unlike the smooth case this operator may {\it apriori} no longer be a linear operator.  Regardless, as in the smooth case, we show that on metric-measure spaces with bounded Ricci curvature, that this operator has a spectral gap and log-sobolev inequality.

Finally in Section \ref{ss:examples_intro} we discuss various examples of metric-measure spaces which do and do not have bounds on their Ricci curvature.

\subsection{Preliminaries}\label{ss:prelim_intro_nonsmooth}

In this Section we briefly review a few concepts which will play an important role in the results of the second part of the paper.  We review these ideas more carefully in Section \ref{s:prelim_nonsmooth}.  In Section \ref{sss:weakly_riemannian_intro} we introduce the notion of a weakly Riemannian metric-measure space.  Roughly, on every metric-measure space one can define a laplace operator, see Section \ref{sss:cheeger_energy}, and a weakly Riemannian space is one for which this operator is linear.  In Section \ref{sss:diffusion_measure_intro} we see how the linearity of the laplace operator is equivalent to the existence of the diffusion measures on $X$.

\subsubsection{Gradients and Sobolev Spaces on $X$}\label{sss:sobolev_intro}

Given a function $u$ on $X$, a fundamental point is deciding what the gradient of $u$ should be.  Following \cite{Cheeger_DiffLipFun}, one good starting point is through the fundamental theorem calculus by defining the slope $|\partial u|$ to be the smallest function such that for every absolutely continuous curve $\gamma$ connecting $x,y\in X$ we have the inequality
\begin{align}
|u(x)-u(y)|\leq\int_{\gamma}|\partial u|\cdot|\dot\gamma|\, dt\, .
\end{align}
On a smooth manifold this is a good definition of gradient, however in general this breaks down as a definition because if one were to then consider the energy functional $\int_X |\partial u|^2 dm$ on $L^2(X,m)$, then unlike the smooth case the energy functional need not be lower semicontinuous.  In \cite{Cheeger_DiffLipFun} it was then decided to take the gradient $|\nabla u|$ to be the lower semicontinuous refinement of the slope $|\partial u|$.  The resulting energy functional $E_X[u]\equiv \int_X |\nabla u|^2 dm$ is sometimes called the Cheeger energy, and is lower semicontinuous and convex.  See Section \ref{sss:cheeger_energy} for more details.

From the Cheeger energy we can define the Sobolev space $W^{1,2}(X,m)$ as the complete Banach Space of functions in $L^2$ such that $|\nabla u|$ is also in $L^2$.  Since the energy $E_X[u]$ is convex and lower semicontinuous, standard function space theory tells us there is a densely defined gradient operator in $L^2$, which we denote by $\Delta_X :\cD(\Delta_X)\subseteq L^2(X,m)\to L^2(X,m)$ and call the laplace operator on $X$.

\subsubsection{Weakly Riemannian Metric-Measure Spaces}\label{sss:weakly_riemannian_intro}

Unlike the case of a smooth metric-measure space the Sobolev space $W^{1,2}(X,m)$ may in principal be only a Banach space, not a Hilbert space.  A weakly Riemannian metric-measure space is by definition a metric-measure space for which $W^{1,2}(X,m)$ is a Hilbert space, that is, the energy functional satisfies the parallelgram law.  We discuss such spaces more completely in Section \ref{sss:weakly_riemannian}, however let us remark that an equivalent condition for $X$ to be weakly Riemannian is that the laplacian $\Delta_X$ introduced in the last section is linear.  Equivalently, the heat flow $H_t:L^2(X,m)\to L^2(X,m)$ of $\frac{1}{2}\Delta_X$ is linear.  It is clear that every Riemannian manifold is weakly Riemannian, less trivial is that every Gromov Hausdorff limit of smooth Riemannian manifolds with lower Ricci curvature bounds is weakly Riemannian.  

An improvement on the notion of a weakly Riemannian space is that of an {\it almost} Riemannian space.  Namely, it is possible that a metric-measure space $(X,d,m)$ is weakly Riemannian for trivial reasons in that the laplacian $\Delta_X\equiv 0$ is identically zero, see Section \ref{sss:almost_riemannian_spaces} for an example.  More fundamentally, degeneracies like this can occur because if one were to consider for a lipschitz function $u$ the lipschitz slope
\begin{align}
|\Lip\, u|(x)\equiv \limsup_{y\to x}\frac{|u(x)-u(y)|}{d(x,y)}\, ,
\end{align}
then unlike for a smooth space, we may have that $|\nabla u|(x)\neq |\Lip\, u|(x)$ a.e.  We call a weakly Riemannian space $X$ an almost Riemannian space if for every lipschitz function $u$ we have that the slope and gradient agree $|\nabla u|=|\Lip\, u|$ a.e.  This will play an important role as we will show in Theorem \ref{t:br_basic_properties} that metric-measure spaces with bounded Ricci curvature are almost Riemannian.  See Section \ref{sss:almost_riemannian_spaces} for more on this.

\subsubsection{The Diffusion Measures on Path Space}\label{sss:diffusion_measure_intro}

As we saw in the first part of the paper, one of the key ingredients in characterizing bounded Ricci curvature involves the existence of the diffusion measures on path space $P(X)$.  On a general metric-measure space it is clear that the diffusion measures need not always exist.  In fact, it is not hard to check, and we will do this in Section \ref{ss:Diff_Meas_weaklyRiemannian}, that the existence of the diffusion measures is equivalent to the metric-measure space being weakly Riemannian.  

More precisely, the key point is that in this case the energy functional $E[u]$ becomes a regular Dirichlet form, see Section \ref{sss:cheeger_energy}.  The first implication of this is that in this case the heat flow $H_t$ can be written in terms of kernels.  Namely, for a continuous function $f\in C_c(X)$ we have 
\begin{align}
H_tf(x) = \int_X f(y)\rho_t(x,dy)\, ,
\end{align}
where $\rho_t:X\times \cB(X)\to \dR^+$ is such that $\rho_t(x,\cdot)$ is a measure for each $x\in X$ and $\rho_t(\cdot,U)$ is a measurable function for each Borel set $U\in \cB(X)$.  

Recall from Section \ref{ss:prelim_intro} that path space $P(X)$ is equipped the family of evaluation maps $e_\bt:P(X)\to X^{|\bt|}$, where $\bt$ is a finite partition of $[0,\infty)$.  Further, using these maps recall that we can construct the bi-family of $\sigma$-algebras $\cF^T_t$ generated by the evaluation maps $e_\bt$ with $\bt$ a partition of $[t,T]$.  Now exactly as in (\ref{e:WM}) we can associate to each measure $\mu$ on $X$ the associated diffusion measure $\Gamma_\mu$ on path space $P(X)$, which is uniquely determined by the formula
\begin{align}
e_{\bt,*}\Gamma_\mu = \int_M \rho_{t_1}(x,dy_1)\rho_{t_2-t_1}(y_1,dy_2)\cdots\rho_{t_k-t_{k-1}}(y_{k-1},dy_k) d\mu(x)\, .
\end{align}
See Section \ref{ss:Diff_Meas_weaklyRiemannian} for more details, and see \cite{Fukushima_DirichletForms} for a complete introduction to the subject of Dirichlet forms and diffusion measures.  The most common choices of diffusion measures that we will be using will come from either choosing $\mu\equiv \delta_x$ to be a dirac delta at a point, in which case $\Gamma_x\equiv \Gamma_{\delta_x}$ is the classical Wiener measure supported on based path space $P_x(X)$, or we will take $\mu\equiv m$ to obtain the diffusion measure $\Gamma_m$ on path space $P(X)$.

\subsection{The Parallel Gradient on Path Space}\label{ss:parallel_grad_intro_nonsmooth}

In Section \ref{sss:parallel_grad_intro}, and more carefully in Section \ref{ss:parallel_gradient}, we introduced the parallel gradient operators $\nabla_s$ on the path space of a smooth manifold.  The construction required two ingredients.  First by using the stochastic parallel translation map we identified parallel translation invariant vector fields $V(t)$ along a curve $\gamma$, and then we defined $|\nabla_0 F|(\gamma)$ to be the supremum of all directional derivatives over all such vector fields with $|V|(0)=1$.  In this Section we briefly discuss the tools needed to make sense of this construct on a general metric-measure space.  We do this more precisely in Sections \ref{s:variation} and \ref{s:parallel_gradient_nonsmooth}.

\subsubsection{Variations of a Curve}\label{sss:variations_intro}

In Section \ref{s:variation} we introduce the notion of a variation of a curve, which will take the place of a vector field.  On a smooth manifold a vector field $V$ along a curve $\gamma$ represents an element of the tangent space $TP(X)$, that is, it represents an infinitesmal deformation of the curve.  It is therefore natural to replace a vector field along $\gamma$ with a form of infinitesmal variation of $\gamma$.  In essence, a variation of a curve $\gamma$ will be an assignment to each point $\gamma(t)$ a Cauchy sequence which converges to $\gamma(t)$.  There will be an important equivalence relation defined on this set which will be particularly useful for studying regularity issues, and crucial in the study of parallel variations on a smooth space.

Once the variation of a curve is introduced, we will define in Sections \ref{ss:parallel_variation_rect_curv} the notion of a parallel variation and $s$-parallel variation on rectifiable curves.  These will take the place of parallel translation invariant vector fields along curves in a smooth space.  We will prove a variety of structure about such variations.  On a smooth space we will see in Section \ref{ss:parallel_variation_rect_curv} that the parallel variations, though defined in a completely geometric manner, agree up to equivalence at least for a.e. curve with the stochastic parallel translation invariant vector fields. 

\subsubsection{The Parallel Gradients}\label{sss:parallel_intro_nonsmooth}

Having defined the notions of the $s$-parallel variations in Section \ref{s:variation}, in Section \ref{s:parallel_gradient_nonsmooth} we use these ideas to define the parallel gradients.  As in the case of defining a gradient on $X$, we must first define the appropriate notion of the parallel slope, and then take the lower semicontinuous refinement in order to construct the parallel gradient.  However, unlike the slope on $X$, that the definition of the parallel slope agrees with the standard one on a smooth space is {\it apriori} not at all clear.

We will worry only about first defining the parallel gradients on a dense subset of $L^2(P(X),\Gamma_m)$, and then we will extend by standard methods.  On a general metric space the collection of functions on $P(X)$ that are best to work with in this context are the cylinder functions $Cyl(X)\subseteq C^0(P(X))$, which are the functions on path space given by the form $F\equiv e_{\bt}^*u$, where $e_\bt:P(X)\to X^{|\bt|}$ is an evaluation map and $u\in Lip_c(X^{|\bt|})$ is a lipschitz function with compact support.  

On this collection we will see in Section \ref{s:parallel_gradient_nonsmooth} how to give the directional derivative $D_VF$ of $F$ in the direction of a parallel variation $V$ a canonically well defined meaning along a piecewise geodesic curve.  If $\gamma\in P(X)$ is any continuous curve and $\bt\equiv\{0\leq t_1<t_2<\cdots<t_N<\infty\}$ is a partition, we call a piecewise geodesic $\gamma_\bt$ with vertices at $\bt$ a $\bt$-approximation of $\gamma$ if $\gamma_{\bt}(\bt)=\gamma(\bt)$.  Then we will define in Section \ref{ss:parallel_slope} the parallel slope by

\begin{align}\label{e:parallel_slope_intro}
|\partial_s F|(\gamma)\equiv \limsup_{\bt} |D_VF|(\gamma_\bt)\, ,
\end{align}
where the supremum is over all $s$-parallel variations $V$ of a $\bt$-approximation $\gamma_\bt$, and the limit is as $\bt$ becomes increasingly dense.  See Section \ref{ss:parallel_slope} for a precise statement.

Finally, in Section \ref{ss:parallel_gradient_nonsmooth} we follow ideas inspired by \cite{Cheeger_DiffLipFun},\cite{Ambrosio_Calculus_Ricci} in order to use the parallel slopes to define the corresponding parallel gradients $|\nabla_s F|$ of a general function $F\in L^2(P(X),\Gamma_{m})$.  One must be a little careful, because on a smooth manifold the slope (\ref{e:parallel_slope_intro}) may be taken directly as the definition of the parallel gradient.  As in the case of gradient on $X$, the fundamental issue with this is that the associated energy functions $\int_{P(X)} |\partial_s F|^2 \,d\Gamma_m$ need not be lower semi-continuous.  To fix this, in Section \ref{ss:parallel_gradient_nonsmooth} we define the parallel gradient $|\nabla_s F|$ to be the lower semi-continuous refinement of $|\partial_s F|$.  

We end by remarking on the following, possibly confusing, notational convention in the paper.  We will often work with expressions which involve integrals or other combinations of parallel gradients.  For instance, in Definition \ref{d:bounded_ricci_mms} we have $|\nabla_0 F|+\int_0^\infty\frac{\kappa}{2}e^{\frac{\kappa}{2}s}|\nabla_s F|$.  There are two ways to interpret such a formula.  One may interpret this directly as the integral of of the parallel gradients, or one may interpret this as the lower semi-continuous refinement of the expression $|\partial_0 F|+\int_0^\infty\frac{\kappa}{2}e^{\frac{\kappa}{2}s}|\partial_s F|$.  On a smooth metric-measure space the two are the same, and in general the first interpretation is always less than or equal to the second.  Our convention in this paper is to use the second convention, namely that such expressions are always the lower semi-continuous refinements of the corresponding slopes, see Section \ref{sss:expressions_parallel} for more on this.

\subsection{Bounded Ricci Curvature and Basic Properties}\label{ss:ricci_intro_nonsmooth}

Having discussed the metric-measure geometry of path space $P(X)$ in Section \ref{ss:prelim_intro_nonsmooth} and \ref{ss:parallel_grad_intro_nonsmooth} we are in a position to use the results of Part \ref{part:smooth} to make our definition of bounded Ricci curvature on a metric-measure space:\\

\begin{definition}\label{d:bounded_ricci_mms}
Let $(X,d,m)$ be a metric measure space which satisfies (\ref{e:mms_assumptions_intro}) and which is weakly Riemannian.  Then we say $X$ is a $BR(\kappa,\infty)$ space if for every function $F\in L^2(P(X),\Gamma_m)$ we have the inequality
\begin{align}\label{e:bounded_ricci}
|\Lip_x \int_X F\,d\Gamma_x|\leq \int_{P(X)} |\nabla_0 F|+\int_0^\infty \frac{\kappa}{2}e^{\frac{\kappa}{2}s}|\nabla_s F|\,d\Gamma_x\, ,
\end{align}
for $a.e.$ $x\in X$.
\end{definition}
\begin{remark}
Recall that $|\Lip_x\cdot|$ is the lipschitz slope as in Section \ref{sss:weakly_riemannian_intro}, and see Section \ref{sss:expressions_parallel} for the precise construction of $|\nabla_0 F|+\int_0^\infty \frac{\kappa}{2}e^{\frac{\kappa}{2}s}|\nabla_s F|$ .
\end{remark}

Let us begin with a few comments on the definition.  To begin with, it follows from \cite{Cheeger_DiffLipFun} that (\ref{e:bounded_ricci}) immediately implies the inequality
\begin{align}\label{e:bounded_ricci_weak}
|\nabla_x \int_X F\,d\Gamma_x|\leq \int_{P(X)} |\nabla_0 F|+\int_0^\infty \frac{\kappa}{2}e^{\frac{\kappa}{2}s}|\nabla_s F|\,d\Gamma_x\, ,
\end{align}
where $|\nabla_x\cdot|$ is the cheeger gradient as in Section \ref{sss:sobolev_intro}.  In fact, we will see in Theorem \ref{t:br_basic_properties} that a space $X$ with bounded Ricci curvature is almost Riemannian, and thus the cheeger gradient and lipschitz slope agree, and so it will not matter if we are talking about the lipschitz slope or cheeger gradient of a function.  In particular, the two inequalities (\ref{e:bounded_ricci}), (\ref{e:bounded_ricci_weak}) turn out to be the same inequality on such spaces.

Now from the definition it is easy to see, computing as in Section \ref{ss:r2_r3}, that for any function $F\in L^2(P(X),\Gamma_{m})$ which is $\cF^T$-measurable we have the estimate
\begin{align}
|\nabla \int_{P(X)} F \,d\Gamma_x|^2 \leq e^{\frac{\kappa}{2}T}\int_{P(X)}\,|\nabla_0 F|^2+\int_0^T \frac{\kappa}{2}e^{\frac{\kappa}{2}s}|\nabla_s F|^2\, ds\cdot d\Gamma_x\, ,
\end{align} 
which corresponds to the estimate $(R3)$ in Part I.  The rest of this Section will be devoted to stating and explaining the basic results obtained in the paper about such spaces.  Let us begin with a listing of some basic properties of the geometry and heat flow on such a space.  In fact, many of these properties will follow or be improved on in later theorems, but regardless the properties are of such a basic and important nature that they are worth listing separately.

\begin{theorem}\label{t:br_basic_properties}
Let $(X,d,m)$ be a $BR(\kappa,\infty)$-space, then the following hold
\begin{enumerate}
\item (Almost Riemannian) For $u\in W^{1,2}(X,m)$ we have that $|\nabla u|(x)=|\Lip\, u|(x)$ for a.e. $x\in X$.
\item (Stochastic Completeness)  If $\mu$ is a probability measure on $X$ then $\Gamma_\mu$ is a probability measure on $P(X)$.
\item (Strong Feller Property) If $f\in L^2(X)$, then for each $t>0$ we have that $|\nabla H_tf|$ is uniformly bounded.
\item (Continuous Martingale Property) Let $F^t\in L^1(P(X),\Gamma_m)$ be a martingale, then for a representative of $F^t$ we have for every $\gamma\in P(X)$ that $F^t(\gamma)$ is a continuous function of $t$.
\item (H\"older Martingale Property) Let $F^t\in L^2(P(X),\Gamma_m)$ be a martingale induced by a function $F\in L^2(P(X),\Gamma_m)$ with uniformly bounded parallel gradients, then for a representative of $F^t$ we have for every $\gamma\in P(X)$ that $F^t(\gamma)$ is $C^\alpha$-H\"older continuous in $t$ for every $\alpha<\frac{1}{2}$.
\item (Existence of Parallel Translation Invariant Variations) Given any lipschitz function $f:X\to \dR$ let $F(\gamma)\equiv f(\gamma(0))$.  Then for $m-a.e.$ $x\in X$ and $\Gamma_x-a.e.$ $\gamma\in P_x(X)$ we have that $|\nabla f|(x)=|\nabla_0 F|(\gamma)$.
\end{enumerate}
\end{theorem}

The first four properties above are fundamentally properties of lower Ricci curvature, see in particular Sections \ref{ss:bounded_implies_lower} and \ref{ss:martingale_lowerricci}.  The last two properties stated however are fundamentally properties of bounded Ricci curvature.  A typical application of Theorem \ref{t:br_basic_properties}.5 is to a cylinder function $F$.  Recall then that the induced martingale $F^t\in L^2(P(X),\Gamma_m)$ is the decomposition of $F$ obtained by projecting $F$ to its $\cF^t$-measurable pieces.  Then Theorem \ref{t:br_basic_properties}.5 tells us that pointwise $F^t$ is H\"older continuous.  In fact we will see in Theorem \ref{t:ricci_quad_nonsmooth} that viewing $F^t$ as a mapping into $L^2(P(X),\Gamma_m)$ that the mapping is exactly $\frac{1}{2}$-H\"older continuous.  This continuity in time of such families of functions is a key property of bounded Ricci curvature, and for a typical metric-measure space is not true.  Theorem \ref{t:br_basic_properties}.6 can be interpreted roughly as the statement that for $a.e.$ curve $\gamma$ and almost every variation $V(0)$ of $\gamma(0)$, there exists a parallel translation invariant variation $V$ of $\gamma$ which extends $V(0)$.  This is clear on a smooth space, but in general highly nontrue on a general metric-measure space.

We have listed a few basic properties of bounded Ricci curvature, let us also remark that in Section \ref{ss:bounded_implies_lower_intro_nonsmooth} we see how bounded Ricci curvature implies a lower bound on the Ricci curvature in the sense of either Lott-Villani-Sturm or Bakry-Emery.  In fact we will see spaces with bounded Ricci curvature satisfy the stronger $RCD$ condition of \cite{Ambrosio_Ricci}.  Before we get there we will spend some time discussing the other characterizations of bounded Ricci curvature given in Part \ref{part:smooth} of the paper.

\subsection{Bounded Ricci Curvature, Martingales and Quadratic Variation}\label{ss:boundedricci_martingales_intro}

In Section \ref{s:ricci_martingales} we study the relationship between bounded Ricci curvature on a metric-measure space and the regularity of martingales.  Let us begin by observing that in the smooth case we focused on martingales on based path space $P_x(X)$.  In the nonsmooth case we will focus on martingales on the total path space $P(X)$.  There is in fact no fundamental difference in that the restriction of a martingale $F^t$ on $P(X)$ to each fiber $P_x(X)$ induces a martingale on $P_x(X)$.   However, because of the measure theoretic aspect it is apriori more appropriate to study martingales on the full path space. 

Recall that if $F\in L^1(P(X),\Gamma_m)$ then we can consider the one parameter family of functions $F^t$ obtained by the formula
\begin{align}
F^t(\gamma)\equiv \int_{P(M)} F_{\gamma_t}\,d\Gamma_{\gamma(t)}\equiv \int_{P(M)}F(\gamma_{[0,t]}\circ \sigma) d\Gamma_{\gamma(t)}\, ,
\end{align}
which is by definition the martingale induced by $F$.  If $F\in L^2(P(X),\Gamma_m)$ then $F^t$ agrees with the projection of $F$ into the closed subspace $L^2(P^t(M),\Gamma_m)$ of $\cF^t$-measurable functions.  Our first structural theorem in Section \ref{ss:martingale_lowerricci} is to show that under only a lower Ricci curvature assumption, in the sense of \cite{Ambrosio_Ricci}, that any martingale $F^t$ is a continuous function of time.  In particular, in combination with Section \ref{ss:bounded_implies_lower_intro_nonsmooth} we will prove Theorem \ref{t:br_basic_properties}.4.\\

Now for any martingale, recall that one defines the infinitesmal quadratic variation
\begin{align}
[dF^t]=\lim_{s\to 0} \frac{\big(F^{t+s}-F^{t}\big)^2}{s}\, ,
\end{align}
which exists at least in measure for almost every $t$.  We saw in Part \ref{part:smooth} of the paper how bounds on the Ricci curvature could be used to estimate the quadratic variation.  We see the same estimates hold in the nonsmooth case, and in fact are still equivalent to the definition of bounded Ricci curvature on a metric-measure space.  Specifically:

\begin{theorem}\label{t:ricci_quad_nonsmooth}
Let $(X,d,m)$ be a metric-measure space.  If $X$ is a $BR(\kappa,\infty)$ space then for each $F\in L^2(P(X),\Gamma_m)$ and a.e. $\gamma\in P(X)$ we have 
\begin{align}\label{e:ricci_quad_est}
\sqrt{[dF^t]}(\gamma)&\leq \int_{P(X)} |\nabla_t F|(\gamma_{[0,t]}\circ \sigma)+\int_t^{\infty}\frac{\kappa}{2}e^{\frac{\kappa}{2}(s-t)}|\nabla_{s} F|(\gamma_{[0,t]}\circ \sigma)\,d\Gamma_{\gamma(t)}\, .
\end{align}
Further, if $X$ is an almost Riemannian metric-measure space then the converse holds.  That is, if (\ref{e:ricci_quad_est}) holds then $X$ is a $BR(\kappa,\infty)$ space.
\end{theorem}

Note that in the smooth case the above corresponds to Theorem \ref{t:pointwise_r4}, which was the pointwise version of $(R4)$.  From this one can immediately conclude for general metric measure spaces the estimates $(R4)$, $(R5)$ and the pointwise version 
\begin{align}
[dF^t](\gamma)&\leq e^{\frac{\kappa}{2}(T-t)}\int_{P(X)} |\nabla_t F|^2+\int_t^{\infty}\frac{\kappa}{2}e^{\frac{\kappa}{2}(s-t)}|\nabla_{s} F|^2\,d\Gamma_{\gamma(t)}\, .
\end{align}

Notice the interesting application that this implies that $F^t$, viewed as a mapping into $L^2(P(X),\Gamma_m)$ is a $C^{\frac{1}{2}}$-H\"older mapping.  It is not hard to check that this is sharp.

We prove the above Theorem in Section \ref{ss:martingale_quad}, while in Section \ref{ss:martingale_holder} we use the result to prove Theorem \ref{t:br_basic_properties}.5.  That is, for a martingale on $P(X)$ induced by a sufficiently nice function on $P(X)$, in particular a cylinder function, we will see that the martingale is not only continuous in time but H\"older continuous in time.

\subsection{Bounded Ricci Curvature and the Analysis on Path Space}\label{ss:ricci_analysis_intro_nonsmooth}

In Section \ref{s:bounded_ricci_analysis} we analyze how bounded Ricci curvature in the sense of Definition \ref{d:bounded_ricci_mms} can be used to do analysis on path space, and specifically we will define the Ornstein-Uhlenbeck operator on general metric-measure spaces and see how the Ricci curvature controls the operator in a way analogous to the smooth case.  As in the case of martingales, it will be more convenient to view the Ornstein-Uhlenbeck operator $L$ as acting on total path space $P(X)$, as opposed to based path space $P_x(X)$.  Also as in the martingale case, there is no fundamental difference as the restriction of $L$ to each fiber $P_x(X)$ will agree with the based path space Ornstein-Uhlenbeck operator.\\

Now recall in Section \ref{s:smooth_bounded_ricci_intro} we showed on a smooth metric measure space that bounded Ricci curvature is equivalent to spectral gap and log-sobolev inequalities of the Ornstein-Uhlenbeck operator.  The primary goal of Section \ref{s:bounded_ricci_analysis} is to prove that if $X$ is a $BR(\kappa,\infty)$-space, then this still implies the same spectral gap and log-sobolev inequalities on the Ornstein-Uhlenbeck operator on path space.

Of course, the first point we must address in Section \ref{s:bounded_ricci_analysis} is how to construct the Ornstein-Uhlenbeck operator on the path space of a general metric measure space.  Such an operator has only been constructed on smooth metric measure spaces.  The key technical point to this construction is the need to construct the $H^1_0$-gradient on path space.  We proceed in a manner which is motivated by Section \ref{ss:parallel_grad_intro_nonsmooth} and Proposition \ref{p:parallel_H1_relation}.  Namely, we begin in Section \ref{sss:H1_slope} by defining the $H^1_0$-slope of a cylinder function on path space $P(X)$ by the formula
\begin{align}
|\partial F|^2_{H^1_0}(\gamma) \equiv \int_0^\infty |\partial_s F|^2 ds\, ,
\end{align}
where $|\partial_s F|$ is the parallel slope defined in Section \ref{ss:parallel_slope}.  To define the $H^1_0$-gradient we then let $|\nabla F|_{H^1_0}$ be the lower semi-continuous refinement of $|\partial F|_{H^1_0}$ in $L^2(P(X),\Gamma_m)$, see Section \ref{sss:H10_gradient}.  Thus, we can define the energy function
\begin{align}
E[F]\equiv \int_{P(X)} |\nabla F|^2_{H^1_0} d\Gamma_m\, ,
\end{align}
on $L^2(P(X),\Gamma_m)$.  We see in Theorem \ref{t:Dirichlet_form_pathspace} that among other properties this defines a closed, convex and lower-semicontinuous Dirichlet form on $L^2(P(X),\Gamma_m)$.  In particular, by standard theory there is a dense subset $\cD(L)$ such that we can define the minimal gradient $\nabla E\equiv L$.  This defines for us the Ornstein-Uhlenbeck operator on $P(X)$, see Section \ref{ss:OU_nonsmooth_definition} for more details.  Notice that on a general metric measure space the Ornstein-Uhlenbeck operator may not be linear.  More specifically, linearity of the Ornstein-Uhlenbeck operator is equivalent to the energy functional $E[F]$ satisfying the parallelogram law.  In Section \ref{ss:OU_nonsmooth_definition} we will show that on a smooth metric-measure space this operator does in fact agree with the classic Ornstein-Uhlenbeck operator.

Now we can state the main theorem of this Section, namely that on a metric-measure space with bounded Ricci curvature the Ornstein-Uhlenbeck operator satisfies a spectral gap and log-sobolev inequality.

\begin{theorem}\label{t:OU_boundedricci}
Let $(X,d,m)$ be a $BR(\kappa,\infty)$ space, then the following hold:
\begin{enumerate}
\item (Spectral Gap)  Let $F\in L^2(P(X),\Gamma_m)$ be a $\cF^T$-measurable function, then we have for a.e. $x\in X$ the Poincare Estimate
\begin{align}
\int_{P_x(X)}\bigg(F-\int_{P_x(X)} F\bigg)^2 d\Gamma_x \leq e^{\frac{\kappa}{2}T}\int_{P_x(X)}\bigg(\int_0^T \cosh(\frac{\kappa}{2}t)|\nabla_t F|^2\, dt\bigg)\, d\Gamma_x\leq \frac{e^{\kappa T}+1}{2} \int_{P(X)} |\nabla F|^2_{H^1_0} d\Gamma_x\, .
\end{align}

\item (Log-Sobolev) Let $F\in L^2(P(X),\Gamma_m)$ be a $\cF^T$-measurable function, then for a.e. $x\in X$ we have the log-Sobolev estimate
\begin{align}
\int_{P_x(X)}F^2\ln F^2 d\Gamma_x - \bigg(\int_{P_x(X)} F\bigg)^2\ln\bigg(\int_{P_x(X)} F\bigg)^2 &\leq 2e^{\frac{\kappa}{2}T}\int_{P_x(X)}\bigg(\int_0^T \cosh(\frac{\kappa}{2}t)|\nabla_t F|^2\, dt\bigg)\, d\Gamma_x\notag\\
&\leq \big(e^{\kappa T}+1\big) \int_{P_x(X)} |\nabla F|^2_{H^1_0} d\Gamma_x\, .
\end{align}
\end{enumerate}
\end{theorem}

See also Theorem \ref{t:OU_global_logsob} for the global version of the above Theorem.

\subsection{Bounded Ricci Curvature Implies Lower Ricci Curvature}\label{ss:bounded_implies_lower_intro_nonsmooth}

With the basic properties established in Theorem \ref{t:br_basic_properties}, the next reasonable question is the about the relationship on metric measure spaces of bounded Ricci curvature and lower Ricci curvature.  It a consequence of Theorem \ref{t:smooth_bounded_ricci} that on a smooth metric measure space that bounded Ricci curvature in the sense of Definition \ref{d:bounded_ricci_mms} implies a lower Ricci curvature bound.  We ask in this Section if a general metric measure space with bounded Ricci curvature in the sense of Definition \ref{d:bounded_ricci_mms} has lower Ricci curvature bounds in the sense of Bakry-Emery \cite{BakryEmery_diffusions}, Lott-Villani-Sturm \cite{LV_OptimalRicci},\cite{Sturm_GeomMetricMeasSpace} or more generally Ambrosio-Gigli-Savare \cite{Ambrosio_Ricci}.  Our main result in this direction is an affirmative answer to all.  Let us begin with the basic results for the Bakry-Emery estimates:
\begin{theorem}\label{t:boundedricci_implies_BE}
Let $(X,d,m)$ be a $BR(\kappa,\infty)$-space, then $W^{1,\infty}(X,m)\cap \cD(\Delta_X)$ is dense in $L^2(X,m)$ and for all $u,w\in W^{1,\infty}(X,m)\cap \cD(\Delta_X)$ with $w\geq 0$ we have:
\begin{enumerate}
\item $\int_X \Delta_X w\cdot |\nabla u|^2\,dm\geq -2\kappa\int_X w|\nabla u|^2\,dm$.
\item $|\nabla H_t u|\leq e^{\frac{\kappa}{2}t}H_t|\nabla u|$.
\item $|\nabla H_t u|^2(x) \leq \frac{e^{\kappa t}}{\kappa^{-1}(e^{\kappa t}-1)}\bigg(H_tu^2(x)-(H_tu)^2(x)\bigg)$.
\item $H_tu^2(x)-(H_tu)^2(x) \leq \kappa^{-1}\left(e^{\kappa t}-1\right)H_t|\nabla u|^2(x)$.
\item $\int_M u^2\ln u^2 \rho_t(x,dy) \leq 2 \kappa^{-1}\left(e^{\kappa t}-1\right)\int_M |\nabla u|^2 \rho_t(x,dy)$ if $\int_M u^2\,\rho_t=1$.
\end{enumerate}
\end{theorem}
\begin{remark}
Recall that we are using the convention that the heat flow $H_t$ is the flow generated by the infinitesmal generator $\frac{1}{2}\Delta_X$.  This notational convention is in contrast with most papers which discuss lower Ricci curvature, though is consistent with most papers that discuss the stochastic properties of path space.
\end{remark}

We will define more carefully in Section \ref{s:lowerricci_nonsmooth} the notion of lower Ricci curvature as introduced in \cite{Ambrosio_Ricci}, and denoted by $RCD(\kappa,\infty)$.  For now it will suffice to say that is related to the {\it strong} convexity of the entropy functional $Ent_m[\rho m]\equiv \int_X \rho\ln\rho\,dm$ on the space of probability measures $\cP_2(X)$ on $X$, and that it is a strictly stronger notion of a lower Ricci curvature bound than that of Lott-Villani-Sturm introduced in \cite{LV_OptimalRicci},\cite{Sturm_GeomMetricMeasSpace}.  Now using \cite{Ambrosio_BE_vs_LVS}, Theorem \ref{t:boundedricci_implies_BE}, and Theorem \ref{t:br_basic_properties}.1 we can conclude the following:

\begin{theorem}\label{t:boundedricci_implies_LVS}
Let $(X,d,m)$ be a $BR(\kappa,\infty)$ space, then we have that $X$ is a $RCD(-\kappa,\infty)$ space.
\end{theorem}

\subsection{$d$-dimensional Bounded Ricci Curvature}\label{ss:d_ricci_intro_nonsmooth}

In analogy with Section \ref{ss:char_d_ricci_intro} we define and study the notion of a $d$-dimensional Ricci curvature bound for a metric measure space.  Motivated by Theorem \ref{t:smooth_bounded_d_ricci} and using the structure of Section \ref{ss:parallel_grad_intro_nonsmooth} we make the following definition:

\begin{definition}\label{d:bounded_d_ricci_mms}
Let $(X,d,m)$ be a metric measure space which satisfies (\ref{e:mms_assumptions_intro}) and which is weakly Riemannian.  Then we say that $X$ is a $BR(\kappa,d)$ space if for every function $F\in L^2(P(X),\Gamma_m)$ which is $\cF^T_t$-measurable we have the inequality
\begin{align}
|\Lip_x \int_{P(X)} F \,d\Gamma_x|^2+\frac{e^{\kappa t}-1}{\kappa d}\,\big|\Delta_X \int_{P(X)}F\,d\Gamma_x\big|^2 \leq e^{\frac{\kappa}{2}T}\int_{P(X)}\bigg(|\nabla_0 F|^2+\int_0^T \frac{\kappa}{2}e^{\frac{\kappa}{2}s}|\nabla_s F|^2\, ds\bigg) d\Gamma_x\, .
\end{align}
for $a.e.$ $x\in X$.
\end{definition}

It is clear immediately from the definition that if $X$ has $d$-dimensional Ricci curvature bounded by $\kappa$, then $X$ has Ricci curvature bounded by $\kappa$ in the sense of Definition \ref{d:bounded_ricci_mms}.  More generally, in this case it is clear we have that $X$ has $d'$-dimensional Ricci curvature bounded by $\kappa$ for all $d'\geq d$.

Let us summarize the basic estimates of a space with $d$-dimensional Ricci curvature bounded.  The proof of the following will end up being the same as the proof of $d=\infty$ case combined with the modifications involved in the proof of Theorem \ref{t:smooth_bounded_d_ricci}.  We rely heavily in the next Theorem on the notation developed throughout the introduction:

\begin{theorem}\label{t:nonsmooth_bounded_d_ricci}
Let $(X,d,m)$ be a $BR(\kappa,d)$ space.  Then the following estimates hold:
\begin{enumerate}

\item For any function $F\in L^2(P(M),\Gamma_{m})$ on the total path space $P(X)$ which is $\cF^T_t$-measurable we have the estimate
\begin{align}
|\nabla_x \int_{P(X)} F \,d\Gamma_x|^2+\frac{e^{\kappa t}-1}{\kappa d}\,\big|\Delta_X \int_{P(X)}F\,d\Gamma_x\big|^2 \leq e^{\frac{\kappa}{2}T}\int_{P(M)}\bigg(|\nabla_0 F|^2+\int_0^T \frac{\kappa}{2}e^{\frac{\kappa}{2}s}|\nabla_s F|^2\, ds\bigg) d\Gamma_x\, .
\end{align} 

\item If $F\in L^2(P(M),\Gamma_m)$ is $\cF^T$-measurable, then for $a.e.$ $\gamma\in P(X)$ if $t^*$ is such that $F_{\gamma_t}$ is $\cF^{T-t}_{t^*}$-measurable, then we have the estimate
\begin{align}
[dF^t](\gamma)+\frac{e^{\kappa t^*}-1}{\kappa d}\,\big|\Delta_X \int_{P(X)}F_{\gamma_t}\,d\Gamma_{\gamma(t)}\big|^2&\leq e^{\frac{\kappa}{2}(T-t)}\int_{P(X)} |\nabla_0 F_{\gamma_t}|^2+\int_0^{T-s}\frac{\kappa}{2}e^{\frac{\kappa}{2}s}|\nabla_{s} F_{\gamma_t}|^2\,d\Gamma_{\gamma(t)}\, .
\end{align}

\item If $F$ is a $\cF^T_t$ measurable function on $P(X)$, then we have the Poincare estimate
\begin{align}
\int_{P_{x}(X)} \big(F-\int F\,d\Gamma_x\big)^2\,d\Gamma_x+\frac{e^{\kappa t}-1-\kappa t}{n\kappa^2}\,\big|\Delta_X \int_{P_x(X)}F\,d\Gamma_x\big|^2 &\leq e^{\frac{\kappa}{2}T}\int_{P_x(X)}\bigg(\int_0^T \cosh(\frac{\kappa}{2}t)|\nabla_t F|^2\, dt\bigg)\, d\Gamma_x \notag\\
&\leq \frac{e^{\kappa T}+1}{2}\int_{P_{x}(X)} |\nabla F|^2_{H^1_x}\,d\Gamma_x\, .
\end{align}
\end{enumerate}

\end{theorem} 
\begin{remark}
By using $(1)$ there are many variations of $(2)$ and $(3)$ which are provable.
\end{remark}

Now let us end this Section by remarking as in Section \ref{ss:bounded_implies_lower_intro_nonsmooth} that on a metric measure space with $d$-dimensional Ricci curvature bounded by $\kappa$, we have in particular that the Ricci curvature is bounded from below by $-\kappa$ is the sense of Bakry-Emery, Lott-Villani-Sturm, and more importantly is a $RCD(-\kappa,d)$ space, see Section \ref{s:lowerricci_nonsmooth} and \cite{Sturm_KE_LVS} for a definition:

\begin{theorem}\label{t:lower_implies_bounded_d}
Let $(X,d,m)$ be a $BR(\kappa,d)$ space.  Then $\{|\nabla u|\in L^\infty(X,m)\}\cap \{\Delta_X u\in L^\infty(X,m)\}$ is dense in $L^2(X,m)$, and for all $u,w\in \{|\nabla u|\in L^\infty(X,m)\}\cap \{\Delta_X u\in L^\infty(X,m)\}$ with $w\geq 0$ we have that:
\begin{enumerate}
\item $\int_X \Delta_X w\cdot |\nabla u|^2\,dm\geq \frac{2}{d}\int_X w|\Delta_X u|^2\,dm-2\kappa\int_X w|\nabla u|^2\,dm$.

\item $\kappa^{-1}(1-e^{-\kappa t})|\nabla H_t u|^2(x)+\frac{1}{d}\kappa^{-2}(1-\kappa t -e^{-\kappa t})|\Delta_X H_tu|^2(x) \leq H_tu^2(x)-(H_tu)^2(x)$.
\item $H_tu^2(x)-(H_tu)^2(x) \leq \kappa^{-1}\left(e^{\kappa t}-1\right)H_t|\nabla u|^2(x)$.
\item $X$ is a $RCD(-\kappa,d)$ space.
\end{enumerate}
\end{theorem}

\subsection{Examples and Counter Examples}\label{ss:examples_intro}

Finally we would like to discuss examples of metric measure spaces with bounded Ricci curvature.  We begin in Section \ref{sss:cone_spaces} by studying metric cone spaces $C(N)$.  In Section \ref{sss:smooth_quotient} we consider a generalization of Theorems \ref{t:smooth_bounded_ricci} and \ref{t:smooth_bounded_d_ricci}.  Namely, we consider metric-measure spaces which are locally quotients of smooth manifolds, and we prove a classification theorem for which are $BR(\kappa,d)$-spaces.  We postpone the proofs of the results of this Section to an upcoming paper, which discusses a much more general context.

\subsubsection{Cone Spaces}\label{sss:cone_spaces}

We begin by studying singular spaces which are cone spaces, namely if $(N,h)$ is a Riemannian manifold then we consider the topological cone $C(N)$ equipped with the cone metric $g\equiv dr^2+r^2dv_h$.  Such examples appear frequently as singularity dilations.  We have two basic points to study, the first is the following:

\begin{theorem}
Let $(N^{n-1},h)$ be an Einstein manifold with $\Ric\equiv (n-2)h$ and $n\geq 3$, then $(C(N),g,dv_g)$ is a $BR(0,n)$-space.
\end{theorem}

In the above we studied cones over spaces which were at least $2$ dimensional.  In the one dimensional case a little more care is needed.  In the next theorem we denote by $S^1(\ell)$ the one dimensional circle of length $\ell$:

\begin{theorem}
Consider the metric-measure space $X\equiv \big(C\big(S^1(\ell)\big),g,dv_g\big)$.  Then $X$ is a $BR(0,2)$-space iff $\ell\leq 2\pi$.
\end{theorem}

\subsubsection{Smooth Quotient Spaces}\label{sss:smooth_quotient}
Theorem \ref{t:smooth_bounded_ricci} tells us that a smooth metric measure space has bounded Ricci tensor iff it is a $BR(\kappa,\infty)$-space.  We begin in this subsection by showing a generalization of this point.  Once one moves into the world of singular spaces the next nicest collection of metric spaces are those which are locally isometric to quotients of smooth manifolds under isometric actions.  More precisely we have the following:

\begin{definition}
We say a metric-measure space $(X,d,m)$ is a smooth quotient space if for each $x\in X$ there exists a neighborhood $x\in U$ and a quadruple $(M,g,e^{-f}dv_g,G)$, where $(M,g,e^{-f}dv_g)$ is a smooth metric-measure space and $G$ is a compact Lie Group which acts isometrically on $M$ and preserves the volume form, such that $U\equiv M/G$ as metric-measure spaces.
\end{definition}

The nicest nonmanifold example of a smooth quotient space is a Riemannian orbifold, so that a neighborhood of each point is isometric to a finite quotient of $\dR^n$.  Smooth quotient spaces arise in a particularly natural way in Riemannian geometry, for instance the Gromov-Hausdorff limit of Riemannian manifolds with bounded sectional curvature are smooth quotient spaces by the work of Fukaya \cite{Fukaya_Haus_Conv}.  Note that if $(X,d,\mu)$ is a smooth quotient space then there exists an open dense subset $\cR(X)\subseteq X$ which is in fact a smooth metric-measure space.  Our basic result in this section is the following.

\begin{theorem}
Let $(X,d,m)$ be a complete smooth quotient space, then the following are equivalent:
\begin{enumerate}
\item The metric measure space $(X,d,m)$ is a $BR(\kappa,d)$ space.
\item $X$ is a smooth Riemannian orbifold such that $-\kappa g+\frac{1}{d-n}\nabla f\otimes\nabla f\leq \Ric+\nabla^2 f\leq \kappa g$ on the regular part $\cR(X)$.
\end{enumerate}
\end{theorem}
\begin{remark}
One can view the above theorem as a generalization of Theorem \ref{t:smooth_bounded_ricci}, which gives the similar statement for smooth metric measure spaces.
\end{remark}
\begin{remark}
Note a topological implication of the above.  In principle the singularities of a smooth quotent space $X$ may be much worse than orbifold, but a consequence of the Theorem is that if we already know the singularities are at worst quotient in nature, then those quotients are finite.
\end{remark}

{\it Acknowledgements:}  The author would like to that Robert Haslhofer for many corrections and comments, as well as the identification of errors in the first version.

\vspace{1cm}

\part{The Case of Smooth Metric-Measure Spaces}\label{part:smooth}

In this part of the paper we focus on smooth metric measure spaces 
$$(M^n,g,e^{-f}dv_g)\, ,$$ 
where $(M^n,g)$ is a complete Riemannian manifold and $f$ is a smooth function.  The primary goal of this part of the paper is to prove the results of Section \ref{s:smooth_bounded_ricci_intro}.\\

The outline of this part of the paper is as follows.  We begin in Section \ref{s:smooth_lower_ricci} by recalling the main results from \cite{BakryEmery_diffusions},\cite{BakryLedoux_logSov}, which give various characterizations of lower Ricci curvature in terms of the analysis on $M$.  We will use these as a point of comparison for the bounded Ricci curvature case, and we will also prove some very mild extensions of known classifications that will be used later.  Section \ref{s:prelim_smooth} is dedicated to a variety of preliminaries which will be needed to discuss the bounded Ricci curvature case.  Section \ref{s:smooth_bounded_ricci_gradient} is dedicated to proving our first characterization of bounded Ricci curvature by the gradient estimates $(R2)$,$(R3)$ of Theorem \ref{t:smooth_bounded_ricci}.  In Section \ref{s:bounded_ricci_stoc_anal_smooth} we discuss the stochastic analysis of bounded Ricci curvature and prove $(R4),(R5)$ of Theorem \ref{t:smooth_bounded_ricci}, while in Section \ref{s:smooth_br_OU} we focus on the analysis description of bounded Ricci curvature and prove $(R6),(R7)$ of Theorem \ref{t:smooth_bounded_ricci}.  Finally in Section \ref{s:finish_maintheorem} we finish the proof of Theorem \ref{t:smooth_bounded_ricci} by showing that any of the estimates of Theorem \ref{t:smooth_bounded_ricci} itself implies the correct corresponding bound on the Ricci curvature.  In Section \ref{s:smooth_d_ricci} we prove the $d$-dimensional of the main results, namely Theorem \ref{t:smooth_bounded_d_ricci}.

\section{Lower Ricci Curvature on Smooth Metric-Measure Spaces}\label{s:smooth_lower_ricci}

In this Section we discuss some analytic methods for characterizing lower Ricci curvature.  For those familiar with these results this Section may be skipped entirely, it is presented for convenience since we will want to compare directly the conditions on lower Ricci curvature with those on bounded Ricci curvature which were presented in Section \ref{s:smooth_bounded_ricci_intro}.  Many of these estimates were first observed on $\dR^n$ by Gross \cite{Gross_LogSob} as he studied gaussian measures on Euclidean space.  The generalizations of these results to more general metric-measure spaces go back primarily to \cite{BakryEmery_diffusions} and \cite{BakryLedoux_logSov}.  \\

Let us recall our notation from Section \ref{s:smooth_bounded_ricci_intro} that $\Delta_f u =\Delta u -\langle\nabla f,\nabla u\rangle$ is the $f$-laplacian associated to our smooth metric-measure space, $H_t:L^2(M,e^{-f}dv_g)\to L^2(M,e^{-f}dv_g)$ is the heat flow associated to $\frac{1}{2}\Delta_f$, and $\rho_t(x,dy)$ is the heat kernel measure associated to this flow.  Finally, for $x\in M$ and $t>0$ fixed if we view $(M^n,g,\rho_t(x,dy))$ as a metric measure space then we can denote by $\Delta_{x,t}=div_{x,t}\circ\nabla$ the associated laplace operator, where the divergence of $\nabla$ is with respect to the measure $\rho_t(x,dy)$.\\

\subsection{The Lower Ricci Bound $\Ric+\nabla^2f\geq -\kappa g$}

To understand the role of the Bakry-Emery Ricci curvature tensor on the analysis of $M$ one begins with a simple computation with the $f$-laplacian to obtain the Bochner formula
\begin{align}
\Delta_f |\nabla u|^2 = \langle\nabla u,\nabla\Delta_f u\rangle +2|\nabla^2 u|^2+2\big(\Ric+\nabla^2f\big)(\nabla u,\nabla u)\, ,
\end{align}
from which if we assume the lower Ricci bound $\Ric+\nabla^2f\geq -\kappa g$ we get the Bochner inequality
\begin{align}\label{e:Bochner}
\Delta_f |\nabla u|^2 \geq \langle\nabla u,\nabla\Delta_f u\rangle -2\kappa|\nabla u|^2\, .
\end{align}

These inequalities are equivalent to the lower bounds on the Ricci curvatures, and are the basis for the definition of lower Ricci curvature given by Bakry-Emery \cite{BakryEmery_diffusions}.  It should be pointed out that their precise condition applies to a much broader situation.  

From the Bochner formula many important estimates on the heat flow $H_t$ can be proved, which themselves turn out to be equivalent to the lower Ricci bound.  We summarize the results of \cite{BakryEmery_diffusions},\cite{BakryLedoux_logSov} in the following theorem.

\begin{theorem}[\cite{BakryEmery_diffusions},\cite{BakryLedoux_logSov}]\label{t:lower_ricci}
 Let $(M^n,g,e^{-f}dv_g)$ be a smooth metric-measure space, then the following are equivalent:
\begin{enumerate}
\item $\text{Ric}+\nabla^2 f\geq -\kappa g$.
\item $|\nabla H_t u|\leq e^{\frac{\kappa}{2}t}H_t|\nabla u|$.
\item $|\nabla H_t u|^2\leq e^{\kappa t}H_t|\nabla u|^2$.
\item $\lambda_1(-\Delta_{x,t})\geq \kappa (e^{\kappa t}-1)^{-1}$. 
\item $\int_M u^2\ln u^2 \rho_t(x,dy) \leq 2 \kappa^{-1}\left(e^{\kappa t}-1\right)\int_M |\nabla u|^2 \rho_t(x,dy)$ if $\int_M u^2\,\rho_t=1$.
\end{enumerate}
\end{theorem}

The implication
\begin{align}
\Ric+\nabla^2f\geq \kappa g \implies \lambda_1(\Delta_{x,t})\geq \frac{\kappa}{e^{\kappa t}-1}\, ,
\end{align}
is well known \cite{BakryLedoux_logSov}.  The proof was first done by Gross in $\dR^n$ (in this case the estimate reduces to understanding a Poincare inequality for the gaussian measure).  The proof was generalized by Bakry-Ledoux \cite{BakryLedoux_logSov} to the more general case.  Other versions of this have been proved in the parabolic setting for the Ricci flow in \cite{HeinNaber_logSob}.  The proof of each of these cases is essentially the same.  

We will briefly describe how to prove the statement $\lambda_1(-\Delta_{x,t})\geq \frac{\kappa}{e^{\kappa t}-1}\implies \text{Ric}+\nabla^2 f\geq -\kappa g$.  Primarily, this gives us an excuse to introduce a little structure which will be useful later in the paper, in particular in the construction of test functions.\\

To prove the statement let us begin by introducing the following variation of the Almgren frequency function.  Namely, given a smooth function $u:M\to \dR$ we define for each $x\in M$ and $t>0$ the frequency function
\begin{align}\label{d:frequency1}
N^u(x,t)\equiv \frac{t\int_M |\nabla u|^2\, \rho_t(x,dy)}{\int_M u^2\, \rho_t(x,dy)- \big(\int_M u\, \rho_t(x,dy)\big)^2}\, .
\end{align}

The following sums up the properties of $N^u$ that we will require:

\begin{lemma}\label{l:lower_ricci:1}
Let $u$ be a smooth function on $M$ and $x\in M$ such that $|\nabla u|(x)=1$, then we have that
\begin{align}
\frac{d}{dt}N^u(x,0) = \frac{1}{2}\left(|\nabla^2 u|^2+(\Ric+\nabla^2 f)(\nabla u,\nabla u)\right)\, .
\end{align}
\end{lemma}
\begin{proof}
A quick computation gives that
\begin{align}\label{e:lower_ricci:1}
\frac{d}{dt}N^u(x,t) = \frac{\frac{1}{2}t\int \Delta_f|\nabla u|^2\, dv_{x,t}}{\int_M u^2 \,dv_{x,t}- \big(\int_M u\, dv_{x,t}\big)^2} - N(x,t)\Bigg(\frac{N(x,t)-1}{t}+\frac{\int u\, dv_{x,t}\int \Delta_f u\, dv_{x,t}- \int u\Delta_f u\, dv_{x,t}}{\int_M u^2\, dv_{x,t}- \big(\int_M u\, dv_{x,t}\big)^2}\Bigg)\, .
\end{align}

Now in the case where $|\nabla u|(x)=1$ we have the estimates
\begin{align}
&N(x,t)\to 1\text{ as }t\to 0\, ,\notag\\
&\int_M u^2\, dv_{x,t}- \big(\int_M u\, dv_{x,t}\big)^2\approx t\text{ as }t\to 0\, .
\end{align}
Further, a similar computation yields the estimate
\begin{align}
&\frac{d}{dt}\bigg(\int u\, dv_{x,t}\int \Delta_f u\, dv_{x,t}- \int u\Delta_f u\, dv_{x,t}\bigg) \notag\\
&= \bigg(\int\Delta_f u\, dv_{x,t}\bigg)^2 + \int u\, dv_{x,t}\int \Delta_f\Delta_f u \, dv_{x,t} -\int \bigg((\Delta_f u)^2+ u\Delta_f\Delta_f u +\langle \nabla u,\nabla \Delta_f u\rangle\bigg)\, dv_{x,t}\, .
\end{align}

Combining all of this with (\ref{e:lower_ricci:1}) and letting $t\to 0$ gives us
\begin{align}
\frac{d}{dt}N(x,0) = \frac{1}{2}\Delta_f u(x)-\langle\nabla u,\nabla\Delta_f u\rangle - \frac{d}{dt}N(x,0)\, .
\end{align}
Rearranging and using the Bochner formula gives the result.
\end{proof}

We mention one more easy lemma.

\begin{lemma}\label{l:lower_ricci:2}
For each point $x\in M$ and unit vector $v\in T_xM$ there exists a smooth function $u$ with compact support in a neighborhood of $x$ such that $u(x)=0$, $\nabla u(x)=v$ and $\nabla^2 u(x)=0$.
\end{lemma}
\begin{proof}
Consider exponential coordinates in a neighborhood of $x$.  Let $u$ be a linear combination of the coordinate functions, multiplied by a cutoff function.
\end{proof}

Now we can easily prove the desired lower Ricci curvature bound:

\begin{proof}[Proof that $\lambda_1(-\Delta_{x,t})\geq \frac{\kappa}{e^{\kappa t}-1}\implies \text{Ric}+\nabla^2 f\geq -\kappa g$]
For any $x\in M$ and unit vector $v\in T_xM$ let $u$ be as in Lemma \ref{l:lower_ricci:2}.  By the spectral gap estimate we have the estimate
\begin{align}
N^u(x,t)\geq 1-\frac{\kappa}{2} t+o(t) \, ,
\end{align}
for each $t>0$.  Since $|\nabla u|(x)=1$ we have that $N^u(x,0)=1$, and therefore we have the estimate
\begin{align}
\frac{d}{dt}N(x,0)\geq -\frac{\kappa}{2}\, .
\end{align}
However by Lemma \ref{l:lower_ricci:1} and Lemma \ref{l:lower_ricci:2} we have
\begin{align}
\frac{d}{dt}N(x,0) = \frac{1}{2}\bigg(\Ric+\nabla^2 f\bigg)(v,v)\, .
\end{align}
Combining these and using that $x$ and $|v|=1$ were arbitrary gives the result.

\end{proof}
\vspace{.5 cm}

\subsection{The $d$-dimensional Lower Ricci Bound}

It is the dimensional form of a lower Ricci curvature bound that is needed for the most powerful applications, for instance the Harnack inequalities.  As an extension of the Bochner formula (\ref{e:Bochner}) it was shown in \cite{BakryEmery_diffusions} that the lower bound $\Ric+\nabla^2f-\frac{1}{d-n}\nabla f\otimes \nabla f\geq -\kappa$ is equivalent to the dimensional Bochner formula
\begin{align}
\Delta_f |\nabla u|^2 \geq \langle\nabla u,\nabla\Delta_f u\rangle+\frac{2}{d}|\Delta_f u|^2 -2\kappa|\nabla u|^2\, .
\end{align}
Again we have that many important estimates on the heat flow follow, which are themselves equivalent to the $d$-dimensional lower bound when said precisely.  The following summarizes these estimates:

\begin{theorem}[\cite{BakryEmery_diffusions},\cite{BakryLedoux_logSov}]\label{t:lower_d_ricci}
 Let $(M^n,g,e^{-f}dv_g)$ be a smooth metric-measure space, then the following are equivalent:
\begin{enumerate}
\item $\text{Ric}+\nabla^2 f-\frac{1}{d-n}\nabla f\otimes \nabla f\geq -\kappa g$.

\item $|\nabla H_t u|^2+\frac{e^{\kappa t}-1}{d\kappa}\big(\Delta_f H_t u\big)^2\leq e^{\kappa t}H_t|\nabla u|^2$.

\item $H_tu^2(x)-(H_tu(x))^2 + \frac{e^{\kappa t}-1-\kappa t}{d \kappa^2}|\Delta_fH_t u|\leq \frac{e^{\kappa t}-1}{\kappa}H_t|\nabla u|^2(x)$.

\end{enumerate}
\end{theorem}
\vspace{1cm}

\section{Preliminaries and Notation on Path Space}\label{s:prelim_smooth}

In this Section we discuss some of the basic structure of path spaces on a smooth manifold.  Many of the constructions of this Section are standard, though often we will need slight generalizations and extensions of the constructions to fit them better to the current paper.  Additionally since the topics of this paper cover multiple areas, which at times have conflicting notational norms, it seemed reasonable to collect together any possible points of notational confusion into one area.\\

The outline of this Section is as follows.  In Section \ref{ss:pathspace_basics} we recall the basic definitions and constructions on path space.  Section \ref{ss:function_spaces} is dedicated to briefly defining the function spaces on path space, and the most basic definition of a martingale.  The construction of useful functions on path space is an important topic, and in Section \ref{ss:functions_pathspace} we give several methods for constructing functions on path space which will be useful in the paper.  In Section \ref{ss:diffusion_measures} we define the general notion of a diffusion measure.  Finally in Section \ref{ss:stoc_par_trans} we briefly introduce the stochastic parallel translation map, which will be especially important in Section \ref{ss:parallel_gradient}.

\subsection{Path Space Basics}\label{ss:pathspace_basics}

We will generally be most interested in the space
\begin{align}
&P(M)\equiv C^0([0,\infty),M)\, ,
\end{align}
of continuous unbased paths in $M$.  Recall from Section \ref{sss:Diffusion_Measure_intro} that for each partition of
\begin{align}
\bt\equiv\{0\leq t_1<\cdots<t_{|\bt|}<\infty\}\, ,
\end{align}
that there is the corresponding evaluation map $e_\bt:P(M)\to M^{|\bt|}$ given by 
\begin{align}
e_\bt(\gamma)\equiv (\gamma(t_1),\ldots, \gamma(t_{|\bt|}))\, .
\end{align}
Further, for each interval $[t,T]$ we have the $\sigma$-algebra $\cF_t^T$ on $P(M)$ generated by the collection of evaluation maps $e_\bt$ whose associated partitions $\bt$ are partitions of $[t,T]$, that is $t_j\in [t,T]$ for each $t_j\in \bt$.  If we consider the time restricted path space
\begin{align}
&P^T(M)\equiv C^0([0,T],M)\, ,
\end{align}
then we see the measurable functions $F$ on $P^T(M)$ are in one-to-one correspondence with the $\cF^T\equiv \cF^T_0$-measurable functions on $P(M)$.  Similar statements of course hold for the based path spaces
\begin{align}
&P_x(M)\equiv \{\gamma\in P(M): \gamma(0)=x\}\, ,\notag\\
&P_{x}^T(M)\equiv \{\gamma\in P^T(M): \gamma(0)=x\}\, .
\end{align}

\subsection{Function Spaces and Martingales}\label{ss:function_spaces}

The structure of this Section holds for any measure $\Gamma$ on $P(M)$, however in principle we will only be interested in the diffusion measures which we will introduce in Section \ref{ss:diffusion_measures}.  Now given any measure $\Gamma$ we can canonically associate to $P(M)$ the Hilbert space
\begin{align}
L^2(P(M),\Gamma)\, .
\end{align}
Recall that the Borel $\sigma$-algebra on $P(M)$ comes equipped with the canonical families of subalgebras given by $\cF^T_t$.  In particular the family $\cF^T$ forms a filtration of $\sigma$-algebras on $P(M)$.  Associated to each $\sigma$-algebra $\cF^T_t$ is the closed Hilbert subspace
\begin{align}
L^2(P(M),\Gamma;\cF^T_t)\subseteq L^2(P(M),\Gamma)\, ,
\end{align}
of $L^2$ functions which are $\cF^T_t$-measurable.  Because $L^2(P(M),\Gamma;\cF^T_t)$ is a closed subspace we may naturally project to it, and if $F\in L^2(P(M),\Gamma)$ is a $L^2$ function then we denote by $F^T_t\in L^2(P(M),\Gamma;\cF^T_t)$ its projection.  As usual if $t\equiv 0$ then we write $F^T\equiv F^T_0$.  The function $F^T_t$ is characterized uniquely by the property that
\begin{align}
\int_U F^T_t d\Gamma = \int_U F d\Gamma\, ,
\end{align}
for every $\cF^T_t$ measurable subset $U\subseteq P(M)$.  That is, from the probability point of view $F^T_t$ is the expectation of $F$ given the $\sigma$-algebra $\cF^T_t$.  Notice however, that characterized in this fashion $F$ only needs to be a $L^1$-function in order to define $F^T_t$.

Let us end this Section with a standard definition.  We say a one parameter family of $F^t\in L^1(P^t(M),\Gamma)=L^1(P(M),\Gamma;\cF^t)$ is a {\it martingale} if for every $s<t$ we have that $\big(F^t\big)^s \equiv F^s$.  Note that given any function $F\in L^1(P(M),\Gamma)$ there is the associated martingale given by $F^t$. 

\subsection{Function Theory on Path Space}\label{ss:functions_pathspace}

To do analysis on path space $P(M)$, or indeed any space, it is important to have a class of functions with which one can work with especially easily, and which will be dense in the various function spaces.  In this way one can consider most functional analytic constructions on these subspaces, and then extend by continuity to a more general class of functions.

In this Section we will introduce three classes of functions on path space $P(M)$ which will play a role at some point in this paper.

\subsubsection{Cylinder Functions}\label{sss:cylinder_functions}

In the context of analysis an especially natural collection of functions on path space $P(M)$ are the smooth cylinder functions.  These are the functions $F:P(M)\to \dR$ of the form
\begin{align}
F\equiv e_\bt^*u\, ,
\end{align}
where $e_\bt:P(M)\to M^{|\bt|}$ is an evaluation map and $u:M^{|\bt|}\to\dR$ is a smooth function with compact support.  The collection of smooth cylinder functions are dense in $L^2(P(M),\Gamma_\mu)$ for any of the diffusion measures $\Gamma_\mu$ (briefly introduced in Section \ref{sss:Diffusion_Measure_intro}, and discussed more completely in the next Section).  

Let us make a few remarks about the cylinder functions.  By definition a cylinder function $F$ on path space only depends on the value of a curve at a fixed, finite number of times.  In particular, if the partition $\bt$ associated to the cylinder function is a subset of $[t,T]$, then $F$ is $\cF^T_t$-measurable.  The cylinder functions will be the primary functions on $P(M)$ that we compute with.  Namely, most constructions in this paper will be begin on the smooth cylinder functions, and then will be extended to a broader class though a continuity argument.

\subsubsection{Path Integral Functions}\label{sss:pathintegral_functions}

The cylinder functions from the previous Section have the property that they are derived from smooth functions on finite dimensional spaces.  In this section we consider other ways to associate functions on $P(M)$ from functions on $M$.  The next simplest manner is to integrate along a curve.  Namely, given a smooth function $u$ on $M$ and a fixed interval $[t,T]$ one can associate the function $F:P(M)\to \dR$ by
\begin{align}\label{e:pi1}
F(\gamma) \equiv \int_t^T u(\gamma(s))ds\, .
\end{align}
It is clear from the definition that $F$ is $\cF_t^T$-measurable.  

One can generalize the above construction as follows.  Fix smooth bounded functions $u,v$ on $M$ and an interval $[t,T]$.  For each continuous curve $\gamma\in P(M)$ let $\phi(s) = \phi_\gamma(s):[t,T]\to \dR$ be the solution to the ode
\begin{align}\label{e:pi2}
&\frac{d}{ds}\phi(s) = v(\gamma(s))\phi(s)+u(\gamma(s))\, ,\notag\\
&\phi(t)=0\, .
\end{align}
Then we define a function $F:P(M)\to \dR$ on path space by the formula
\begin{align}
F(\gamma)\equiv \phi_\gamma(T)\, .
\end{align}
If course if one has $v\equiv 0$ then this reduces to (\ref{e:pi1}).  One can check easily that $F$ is a continuous bounded function on $P(M)$.  The path integral functions, and certain generalizations that will be introduced once we have defined stochastic parallel translation, arise naturally in the context of studying pde's on $M$.

\subsubsection{Stochastic Integrals}\label{sss:Ito_integrals}

In comparison to the cylinder and path integral functions, the Ito and Stratonovich integal is much more subtle, though they form the backbone of stochastic analysis.  We will give a definition here, but refer the reader to \cite{Kuo_book},\cite{StrVar_book} for a more complete understanding.

We consider path space $P(M)$ equipped with a measure $\Gamma$.  We call a family of measurable mappings $X^t:P(M)\to \dR$ a stochastic process if $X^t$ is $\cF^t$-measurable, and recall from Section \ref{ss:function_spaces} we call $X^t$ a $L^2$-martingale if for every $t<T$ we have that $(X^T)^t\equiv X^t$, where $(X^T)^t$ is the $L^2$ projection of $X^T$ to the $\cF^t$-measurable functions. 

Now given a martingale $X^t$ and a stochastic process $Y^t$ with
\begin{align}
\int_{P(M)} \bigg(\int_0^T |Y^t|^2dt\bigg) d\Gamma<\infty\, ,
\end{align}
we define the Ito integral by the limit
\begin{align}\label{d:Ito_integral}
\int_0^\infty Y^t\, dX^t \equiv \lim_{\Delta t\to 0} \sum Y^{t_i}\big(X^{t_{i+1}}-X^{t_i}\big)\in L^2(P(M),\Gamma)\, ,
\end{align}
where $\Delta t\equiv\max |t_{k+1}-t_k|$ is the maximum step size of the partition.  It turns out the limit does exist in $L^2$.  Similarly, we may define the Stratonovich integral
\begin{align}\label{d:Strat_integral}
\int_0^\infty Y_t\circ dX_t \equiv \lim_{\Delta t\to 0} \sum \frac{1}{2}\big(Y_{t_{i+1}}+Y_{t_{i}}\big)\big(X_{t_{i+1}}-X_{t_i}\big)\in L^2(P(M),\Gamma)\, ,
\end{align}
where again the limit converges in $L^2(P(M),\Gamma)$.  It is a subtle point that these two limits are not equal.

\subsection{Diffusion Measures}\label{ss:diffusion_measures}

In this Section we discuss the construction of a canonical family of measures on path space $P(M)$ known as the diffusion measures, and state some well known properties about such measures.  Specifically, let us recall that given any Borel measure $\mu$ on $M$ there exists a unique measure $\Gamma_\mu$ on $P(M)$ whose pushforward's by the evaluation maps $e_\bt:P(M)\to M^{|\bt|}$ are given by 
\begin{align}\label{e:DM}
e_{\bt,*}\Gamma_\mu = \int_M\rho_{t_1}(x,dy_1)\rho_{t_2-t_1}(y_1,dy_2)\cdots\rho_{t_k-t_{k-1}}(y_{k-1},dy_k)d\mu(x)\, ,
\end{align}
where $\rho_t(x,dy)$ is the heat kernel measure associated to $\frac{1}{2}\Delta_f$.  Uniqueness of such a measure is not hard to see, for a proof to the existence of $\Gamma_\mu$ we refer the reader to \cite{Stroock_book}, or indeed any beginning text on stochastic analysis.  We call the measure $\Gamma_\mu$ a diffusion measure with initial probability $\mu$.  Let us recall the subtlety, first observed by Wiener in his original construction, that this measure exists on continuous path space but not even on $H^1$-path space.  \\

There are two primary families of diffusion measures which will interest us.  The first family are Wiener measures given by $\Gamma_x\equiv \Gamma_{\delta_x}$.  Notice that the measures $\Gamma_x$ are concentrated on the based paths $P_x(M)\subseteq P(M)$.  In particular the Hilbert spaces $L^2(P(M),\Gamma_x)$ and $L^2(P_x(M),\Gamma_x)$ are canonically isomorphic.  We will not always distinguish between the two.  The other diffusion measure of particular interest in this paper is given by $\Gamma_f\equiv \Gamma_{e^{-f}dv_g}$.  Notice that we can also interpret $\Gamma_f$ as the measure on $P(M)$ given by
\begin{align}
\Gamma_f \equiv \int_M \Gamma_x\, e^{-f}dv_g\, .
\end{align}\\

Now the formula (\ref{e:DM}) well defines the diffusion measure, however the next result gives a characterization of the diffusion measures in terms of their expectations on the $\sigma$-algebras $\cF^T$ as in Section \ref{ss:function_spaces}.  For a proof of the following we refer the reader to \cite{Hsu_book}:

\begin{theorem}\label{t:diffusion_measure_characterization}
Let $\Gamma$ be a measure on $P(M)$.  Then $\Gamma$ is a diffusion measure $\Gamma_\mu$ with respect to some measure $\mu$ on $M$ if and only if the family of functions 
$$
F^T(\gamma) = u(\gamma(T))-u(\gamma(0))-\frac{1}{2}\int_0^T\Delta_f u\,(\gamma(s))\,ds\, ,
$$
is a martingale for every smooth function $u:M\to \dR$.
\end{theorem}

\subsection{Stochastic Parallel Translation}\label{ss:stoc_par_trans}

One of the challenges to doing analysis on path space, even on a smooth manifold other than $\dR^n$, is that one must consider {\it nice} variations of very irregular curves.  To put this into proper perspective let us note that given a based path space $P_{x}M$ equipped with the Wiener measure $\Gamma_x$, it is well known that almost every curve $\gamma\in P_{x}M$ has the property that the set $\gamma([0,T])$ has Hausdorff dimension 2 in $M$ for each $T>0$.  In particular, $\gamma$ is highly non-differentiable.  Nonetheless, one wants to consider along $\gamma$ vector fields $V(t)\in T_{\gamma(t)}M$ which are parallel translation invariant.  Apriori, on an arbitrary continuous curve it is not reasonable or possible to consider such a nice class of vector fields. 

The key technical tool on a smooth manifold to handle this issue is the stochastic parallel translation map, which is briefly introduced in this Section.  We refer the reader to \cite{Emery_StocAnalMan},\cite{Hsu_book}, and \cite{Stroock_book} for a more rigorous introduction and various interpretations.


Intuitively, one wants to take the parallel translation map for piecewise smooth curves, and by approximating a continuous curve by such curves, limit the resulting parallel translation maps to define a parallel translation map for the continuous curve.  More precisely, for each partition $\bt$ let $P^\bt_x(M)$ be the collection of piecewise geodesics in $M$ with vertices given by $\bt$ (not necessarily minimizing geodesics).  Notice that $P^\bt_x(M)$ is a smooth submanifold of $P_x(M)$, in fact is canonically diffeomorphic to $\dR^{n\cdot |\bt|}$.  Notice also that there is a canonical projection map $P_x(M)\to P_x^\bt(M)$ which maps each curve $\gamma$ to the piecewise geodesic curve $\gamma_\bt$ with vertices $\bt$ and $\gamma(\bt)=\gamma_\bt(\bt)$.  This projection mapping is well defined away from a set of measure zero.

To define the stochastic parallel translation map we define the stochastic horizontal lifting map $H:P_x(M)\to P_{\tilde x}(FM)$, where $FM$ is the orthonormal frame bundle and $\tilde x\equiv (x,F_x)$ is any fixed lifting of $x$.  Note for each partition $\bt$ there is the standard such lifting of $P^\bt_x(M)$, and that by composing with the projection map $P_x(M)\to P^\bt(M)$ we have for each partition the approximate horizontal lifting map
\begin{align}
H^\bt:P_x(M)\to P_{\tilde x}(FM)\, .
\end{align}
The main result is that there exists a mapping $H:P_x(M)\to P_{\tilde x}(FM)$ such that for any sequence of increasing dense partitions $\bt^j$ we have that that $H^{\bt^j}\to H$ in measure.  In particular a subsequence converges pointwise a.e.  See the second part of the paper for some more refined statements, and see \cite{Stroock_book} for more details.

Finally, to define the stochastic parallel translation maps fix $t>0$ and let $T_tP_x(M)=e^*_t TM$ be the vector bundle over $P_x(M)$ given by the pullback of the tangent bundle $TM$, where $e_t:P(M)\to M$ is the evaluation map at time $t$.  Thus a section of $T_tP_x(M)$ assigns to each curve $\gamma$ a vector in $T_{\gamma(t)}M$.  Then we define $P_t:T_tP_x(M)\to T_0 P_x(M)$ by the formula
\begin{align}
P_t(\gamma) = H(\gamma)(0)\cdot H(\gamma)^{-1}(t):T_{\gamma(t)}M\to T_{\gamma(0)}M\, .
\end{align}

Let us end this Section with the following observation.  The isometry $P_t(\gamma):T_{\gamma(t)}M\to T_x M$ induces corresponding isometries between the higher tensor spaces $P_t: T^{p,q}_{\gamma(t)}M\to T^{p,q}_xM$.  Given a tensor $A\in T^{p,q}_{\gamma(t)}M$ we will write $P_t A\in T^{p,q}_xM$ for the corresponding tensor above $x$.

\section{Bounded Ricci Curvature and Gradient Estimates}\label{s:smooth_bounded_ricci_gradient}

In this Section we discuss characterizations of bounded Ricci in terms of gradient estimates on path space.  The first point toward this end is the introduction of the parallel gradient in Section \ref{ss:parallel_gradient}.  With this in hand we prove the gradient estimate $(R2)$ of Theorem \ref{t:smooth_bounded_ricci} in Section \ref{ss:r1_r2}.  The converse statement will not be proved until Section \ref{s:finish_maintheorem}.  Finally in Section \ref{ss:r2_r3} we prove the gradient estimate $(R3)$.  The dimensional versions of these estimates will be discussed in Section \ref{s:smooth_d_ricci}.\\

\subsection{The Parallel Gradient}\label{ss:parallel_gradient}

The parallel gradient operators act as a form of finite dimensional gradient operators on the infinite dimensional path space.  In this Section we define a one parameter family of gradients $\nabla_s:L^2(P_x(M),\Gamma_x)\to L^2(T_0P_x(M),\Gamma_x)$.  This one parameter family of gradients will arise in many ways throughout the paper, and will be particularly important in various estimates.  Though each has a finite dimensional flavor to it, we will see in Section \ref{ss:H1_gradient} how as a family they recover the infinite dimensional $H^1_0$-gradient on path space.  We will begin in Section \ref{sss:0_parallel_gradient} by introducing the $0$-parallel gradient, which is the easiest to describe.  We will extend the construction in Section \ref{sss:s_parallel_gradient} to the the general $s$-parallel gradient operators.\\

\subsubsection{The $0$-Parallel Gradient.}\label{sss:0_parallel_gradient}

Given based path space $P_x(M)$ we consider the (trivial) vector bundle over $P_x(M)$ defined by
\begin{align}
T_0P_x(M)\equiv e_0^*TM\, ,
\end{align}
where $e_0:P(M)\to M$ given by $e_0(\gamma)=\gamma(0)$ is the evaluation map given by the partition $\bt=\{0\}$.  That is, a section of $T_0P_x(M)$ is a continuous mapping $P_x(M)\to T_{x}M$.  The bundle $T_0P_x(M)$ comes naturally equipped with an inner product, and therefore we have the canonical Hilbert space $L^2(T_0P_x(M),\Gamma_x)$ of $L^2$ sections of $T_0P_x(M)$ with respect to the diffusion measure $\Gamma_x$.\\

The parallel gradient is an unbounded, closed operator
\begin{align}
\nabla_0:L^2(P_x(M),\Gamma_x)\to L^2(T_0P_x(M),\Gamma_x)\, .
\end{align}

We will first define it on smooth cylinder functions, and then extend it to the rest of $L^2(P_x(M),\Gamma_x)$.  Now if $F:P_x(M)\to \dR$ is a smooth cylinder function, then for any $\gamma\in P_x(M)$ and any vector field $V(t)$ along $\gamma$ the directional derivative $D_VF$ is well defined.  We define the parallel gradient of $F$ at a curve $\gamma$ to be the unique vector $\nabla_0 F(\gamma) \in T_{x}M$ such that for every parallel translation invariant vector field $V(t)\equiv P_t^{-1}V_0$ along $\gamma$ we have that
\begin{align}
D_VF = \langle\nabla_0 F(\gamma), V_0\rangle_{T_{x}M}\, .
\end{align}
Recall that the stochastic parallel translation map $P_t$ is defined for $a.e.$ curve $\gamma$, and therefore the parallel gradient is well defined $a.e.$ in $P(M)$.  

The next Theorem tells us that the parallel gradient extends to a closed operator on $L^2$.  We postpone the proof until the next Section where we prove the more general statement for the $s$-parallel gradients. 

\begin{theorem}
The parallel gradient $\nabla_0 F$ extends to a closed operator $\nabla_0:L^2(P_x(M),\Gamma_x)\to L^2(T_0P_x(M),\Gamma_x)$ such that the smooth cylinder functions are dense in the domain $\cD(\nabla_0)$.
\end{theorem}

\subsubsection{The $s$-Parallel Gradient}\label{sss:s_parallel_gradient}

As with the $0$-parallel gradient we begin by defining the $s$-parallel gradient on the cylinder functions.  So let $s>0$ be fixed and let $F=e_\bt^*u$ be a smooth cylinder function on $P_x(M)$.  Let us begin with the observation that for a cylinder function $F$ the directional derivative $D_V F$ is well defined not only for continuous vector fields but also for left continuous vector fields.  So for each $V_x\in T_{x}M$ we can define the $s$-parallel translation invariant vector field $V_s(t)$ given by $V_s(t)=0$ if $t<s$ and $V_s(t)=P_t^{-1}V_x$ if $t\geq s$.  The partial derivatives $D_{V_s}F$ is therefore well defined for each such vector field.

Now define the $s$-parallel gradient $\nabla_s F$ of $F$ at a curve $\gamma$ to be the unique vector $\nabla_s F(\gamma) \in T_{x}M$ such that for every $s$-parallel translation invariant vector field $V_s(t)$ along $\gamma$ we have that
\begin{align}
D_{V_s}F = \langle\nabla_s F(\gamma), V_x\rangle_{T_{x}M}\, .
\end{align}
Note in particular that we can write for a.e. $\gamma\in P(M)$ that
\begin{align}\label{e:parallel_slope_smooth}
|\nabla_s F|(\gamma)\equiv \sup\{|D_V F|: V\text{ is a }s\text{-parallel variation with }|V|(s)=1\}.
\end{align}

The next lemma is an integration by parts formula for the parallel gradients.  The statement of the result requires two standard tools from stochastic analysis which, however, will not be discussed with any care until later in the paper, namely the quadratic variation $[,]$ in Section \ref{ss:mart_quad} and the Brownian motion map $W^s$.  The lemma will be used for Theorem \ref{t:s_parallel_grad_closed} in order to prove that $\nabla_s$ extends to a closed operator, but otherwise is not required for the rest of the paper and may be skipped on a first read.

\begin{lemma}\label{l:parallel_ibp}
Let $F,G$ be smooth cylinder functions which are $\cF^T$-measurable with $V_x\in T_x M$.  Then for $s\geq 0$ if $V_s$ is the vector field on $P_x(M)$ such that for each $\gamma$ we have that $V_s(\gamma)$ is the $s$-parallel vector field induced by $V_x$, then 
\begin{align}
\int_{P_x(M)} D_{V_s}F\cdot G\,d\Gamma_x = \int_{P_x(M)} F\cdot\bigg(-D_{V_s}G+d[G^s,W^s]-\frac{1}{2}G\int_s^T \big\langle (\Ric+\nabla^2f)(V_s),dW^s\big\rangle\bigg)\,d\Gamma_x\,  .
\end{align}
\end{lemma}
\begin{proof}
This is an application of Driver's integration by parts formula \cite{Driver_CM}.  Namely, let $y(t)\in T_xM$ be  a $H^1_0$-curve in $T_xM$ with $Y$ the vector field on $P_x(M)$ defined by $Y(t)=P_t^{-1}y(t)$.  Then in \cite{Driver_CM} it was proved that
\begin{align}
\int_{P_x(M)} D_{Y}F\cdot G\,d\Gamma_x = \int_{P_x(M)} F\cdot\bigg(-D_{Y}G+G\int_0^T \big\langle \dot y-\frac{1}{2}(\Ric+\nabla^2f)(y(t)),dW^s\big\rangle\bigg)\,d\Gamma_x\,  ,
\end{align}
where $W^s$ is the brownian motion map and as in Section \ref{ss:functions_pathspace} $\int\langle, dW^s\rangle$ is the associated Ito integral.  Now for $V_x\in T_xM$ fixed and each $\epsilon>0$ let $y_\epsilon(t)$ be defined by $y_\epsilon(t)=0$ for $t\leq s$, $y_{\epsilon}(t)=V_x\,\epsilon^{-1}(t-s)$ for $s\leq t\leq s+\epsilon$ and $y_\epsilon(t)=V_x$ otherwise.  Computing gives
\begin{align}
\int_{P_x(M)} &F\cdot\bigg(G\int_0^T \big\langle \dot y_\epsilon-\frac{1}{2}(\Ric+\nabla^2f)(y_\epsilon(t)),dW^s\big\rangle\bigg)\,d\Gamma_x\,\notag\\
&=\int_{P_x(M)} F\cdot\bigg(G\frac{W^{s+\epsilon}-W^s}{\epsilon}+ G\int_s^{T} \big\langle -\frac{1}{2}(\Ric+\nabla^2f)(y_\epsilon(t)),dW^s\big\rangle\bigg)\,d\Gamma_x\notag\\
&=\int_{P_x(M)} F\cdot\bigg(\frac{(G^{s+\epsilon}-G^s)(W^{s+\epsilon}-W^s)}{\epsilon}+ G\int_s^{T} \big\langle -\frac{1}{2}(\Ric+\nabla^2f)(y_\epsilon(t)),dW^s\big\rangle\bigg)\,d\Gamma_x\notag\\
&\to \int_{P_x(M)} F\cdot\bigg(d[G^s,W^s]+ G\int_s^{T} \big\langle -\frac{1}{2}(\Ric+\nabla^2f)(V_s),dW^s\big\rangle\bigg)\,d\Gamma_x\, ,
\end{align}
as claimed.
\end{proof}

Using the above we immediately have the following, which tells us that the $s$-parallel gradient operators extend to closed operators in $L^2$:

\begin{theorem}\label{t:s_parallel_grad_closed}
For each $s\geq 0$ the $s$-parallel gradient $\nabla_s F$ extends to a closed operator $\nabla_s:L^2(P_x(M),\Gamma_x)\to L^2(T_0P_x(M),\Gamma_x)$ such that the smooth cylinder functions are dense in the domain $\cD(\nabla_s)$.
\end{theorem}

Let us remark on the following.  Given the above we can define the Dirichlet form on path space given by
\begin{align}
E_{x,s}[F,G] \equiv \int_{P_xM} \langle\nabla_s F,\nabla_s G\rangle\, d\Gamma_x\, ,
\end{align}
from which we can define the $s$-laplacian $\Delta_s:L^2(P_xM)\to L^2(P_xM)$, which is an unbounded operator defined uniquely by
\begin{align}
E_{x,s}[F,G] \equiv \int_{P_xM} \langle \Delta_s F, G\rangle\, d\Gamma_x\, .
\end{align}
Likewise, it clear by similar arguments that we may consider the $s$-parallel gradient and $s$-laplacian as closed unbounded operators $\nabla_s,\Delta_s:L^2(P(M),\Gamma_f)\to L^2(T_0P(M),\Gamma_f)$ on unbased path space.\\

The following is almost a tautology, however it is a sufficiently useful formula for computing the singular parallel gradient of a smooth cylinder function that we record it.

\begin{proposition}\label{p:parallel_gradient}
Let $F=e_\bt^*u$ be a smooth cylinder function on $P(M)$.  Then the parallel gradient $\nabla_s F$ is given by
\begin{align}
\nabla_s F(\gamma) = \sum_{t_j\geq s} P_{t_j}\nabla_j u\, ,
\end{align}
where $P_t:T_{\gamma(t)}M\to T_{\gamma(0)}M$ is the stochastic parallel translation map.
\end{proposition}
\begin{proof}
Let $\gamma\in P(M)$ with $V_s(t)$ such that $V_s(t)\equiv 0$ is $t<s$ and such that $V_s(t)=P_t^{-1}v$ is parallel translation invariant for $t\geq s$.  Then we have that
\begin{align}
D_{V_s}F &= \sum \langle V(t_j),\nabla_j u\rangle = \sum_{t_j\geq s} \langle P^{-1}_{t_j}v,\nabla_j u\rangle\, ,\notag\\
&=\big\langle v,\sum P_{t_j}\nabla_j u\big\rangle\, ,
\end{align}
from which the result follows.
\end{proof}

\subsection{Proof that $(R1)\implies (R2)$}\label{ss:r1_r2}

Recall that for every $F\in L^2(P(M),\Gamma_f)$ that we have an induced $L^2(M,e^{-f}dv_g)$ function on $M$ given by $\int_{P(M)}F\,d\Gamma_x$ .  The question naturally arose about understanding the properties of $\int F\,d\Gamma_x$ as a function on $M$ in terms of the properties of $F$ as a function on $P(M)$.  If we consider the parallel gradient on $P(M)$, then the main result of this Section will be to prove the estimate
\begin{align}\label{e:ss:BR2}
\big|\nabla \int_{P(M)} F\,d\Gamma_{x}\big| \leq \int_{P(M)} \bigg(|\nabla_0 F|+\frac{\kappa}{2}\int_0^\infty e^{\frac{\kappa}{2}s}\,|\nabla_s F|\,ds\bigg) \,d\Gamma_{x}\, ,
\end{align}
under the assumption of the Ricci curvature bound 
\begin{align}
-\kappa g\leq \Ric+\nabla^2 f\leq \kappa g\, .
\end{align}\\

The proof is essentially an application of the Bochner formula in combination with the stochastic analogue of a vector bundle Feyman-Kac formula in infinite dimensions, which originally goes back to a host of authors including Bismut \cite{Bismut_Malliavin} and Stroock \cite{Stroock_book}.  To understand this we begin with the next lemma.  

\begin{lemma}\label{l:gradient_estimate}
Let $F:P(M)\to \dR$ be a smooth cylinder function, then we have that
\begin{align}
\nabla_x \int_{P(M)} F\,d\Gamma_x = \int_{P(M)} \nabla_0F + \int^{\infty}_{0} {\tiny \frac{d}{ds}}\phi_s \cdot\nabla_s F\, ds\, d\Gamma_x\, ,
\end{align}
where $\phi_t=\phi_t(\gamma):T_{\gamma(0)}M\to T_{\gamma(0)}M$ solves the ode $\frac{d}{dt}\phi = -\frac{1}{2}\phi P_{t}\big(\Ric+\nabla^2 f\big)P^{-1}_{t}$ with $\phi(0)=Id$.
\end{lemma}

\begin{proof}

Let $e_\bt:P(M)\to M^{|\bt|}$ be an evaluation map and $F\equiv e_\bt^*u$ a smooth cylinder function.  The proof is by induction on $|\bt|$, see \cite{Hsu_LogSob} for related arguments.  

For $|\bt|=1$ we have that $F(\gamma)=u(\gamma(t))$ for some $t\geq 0$.  In this case we have that
\begin{align}
\int_{P(M)} F\,d\Gamma_x = \int_M u(y)\rho_t(x,dy) = H_tu(x)\, ,
\end{align}
as a function on $M$ determines the heat flow of $u$ at time $t$.  In essence this is the stochastic analogue of the Feyman-Kac formula.  Now the standard Weizenbrock formula tells us that $\nabla H_t u(x)$ solves the equation
\begin{align}\label{e:grad1}
\frac{d}{dt}\nabla H_tu(x) = \Delta_f \big(\nabla H_tu\big)+\frac{1}{2}\big(\Ric+\nabla^2 f\big)(\nabla H_tu)\, ,
\end{align}
which tells us that $\nabla H_tu=\tilde H_t\nabla u$, where $\tilde H_t$ is the heat flow operator associated to (\ref{e:grad1}).  The stochastic Feynman-Kac formula for vector bundles \cite{Bismut_Malliavin},\cite{Stroock_book} allows us to therefore write
\begin{align}
\nabla H_tu(x) = \int_{P(M)}\phi_t\cdot P_t\nabla u(\gamma(t)) d\Gamma_x\, ,
\end{align}
where $\phi_t=\phi(\gamma,t):T_xM\to T_xM$ solves $\frac{d}{dt}\phi = -\phi\frac{1}{2}P_{t}\big(\Ric+\nabla^2 f\big)P^{-1}_{t}$ along $\gamma$ with $\phi(0)=Id$.
Combining all of this and rewriting using Proposition \ref{p:parallel_gradient} gives us
\begin{align}
\nabla \int_{P(M)} F\,d\Gamma_x &= \int_{P(M)} (I+\int_0^t\frac{d}{ds}\phi_s)P_t\nabla u(\gamma(t))\,ds\,d\Gamma_x\, \notag\\
&= \int_{P(M)} \nabla_0 F+\int_0^\infty \frac{d}{ds}\phi_s\cdot \nabla_s F\,ds\,d\Gamma_x\, ,
\end{align}
as claimed.

Now for the inductive step we assume the result holds for all cylinder functions with order $|\bt|< N$, and let us denote by 
\begin{align}
F(\gamma)=e_{\bt}^*u(\gamma) = u(\gamma(t_1),\ldots,\gamma(t_N))\, ,
\end{align}
a smooth cylinder function of order $|\bt|=N$.  Now for $y\in M$ fixed let us define the smooth cylinder function $F_y(\gamma) \equiv u(y,\gamma(t_2-t_1),\ldots,\gamma(t_N-t_1))$.  Note then that we may rewrite
\begin{align}
\int_{P(M)} F\, d\Gamma_x &= \int_{M}\bigg(\int_{P(M)}F_y\,d\Gamma_y\bigg)\rho_{t_1}(x,dy)
\end{align}

Viewing $\int_{P(M)}F_y\,d\Gamma_y$ as a function on $M$ we can then apply Feynman-Kac formula for bundles again to write
\begin{align}
\nabla_x \int_{P(M)} F\, d\Gamma_x &= \nabla_x\int_{M}\bigg(\int_{P(M)}F_y\,d\Gamma_y\bigg)\rho_{t_1}(x,dy)\notag\\
&=\int_{P(M)}\phi_{t_1}\cdot P_{t_1}\nabla_y\bigg(\int_{P(M)}F_y\,d\Gamma_y\bigg)\,d\Gamma_x\, \notag\\
\end{align}
By viewing $F_y$ as a function on path space we can then use our inductive hypothesis to compute
\begin{align}
&=\int_{P(M)}\phi_{t_1}\cdot P_{t_1}\bigg(\int_{P(M)}\nabla_0F_y+\int_0^\infty\frac{d}{ds}\phi_s\cdot\nabla_s F_y\,d\Gamma_y\bigg)\,d\Gamma_x\notag\\
&=\int_{P(M)}\phi_{t_1}\cdot P_{t_1}\bigg(\int_{P(M)}\nabla_{t_1}F+\int_{t_1}^\infty\phi^{-1}_{t_1}\frac{d}{ds}\phi_s\cdot\nabla_s F\,d\Gamma_y\bigg)\,d\Gamma_x\notag\\
&= \int_{P(M)} \nabla_0 F+\int_0^\infty \frac{d}{ds}\phi_s\cdot \nabla_s F\,ds\,d\Gamma_x\, ,
\end{align}
which is the desired equality.

\end{proof}

Using the previous lemma we are in a position to prove the main statement of this Section:

\begin{proof}[Proof that Theorem \ref{t:smooth_bounded_ricci}.R2 $\implies$ Theorem \ref{t:smooth_bounded_ricci}.R3 ]

Let us note that if the eigenvalues of the Ricci curvature tensor satisfy the estimate
\begin{align}
-\kappa g\leq \Ric+\nabla^2f\leq \kappa g\, ,
\end{align}
then a standard application of Gronwall's inequality tells us that the solution $\phi$ of $$\frac{d}{dt}\phi = -\phi\frac{1}{2}P_{t}\big(\Ric+\nabla^2 f\big)P^{-1}_{t}$$ with $\phi(0)=I$ satisfies the eigenvalue estimate
\begin{align}
||\phi(t)||_{max} \leq e^{\frac{\kappa}{2} t}\, ,
\end{align}
where $||\cdot||_{max}$ is the maximum eigenvalue norm.  Plugging this back into the equation gives us the estimate
\begin{align}
||\frac{d}{dt}\phi(t)||_{max} \leq \frac{\kappa}{2}e^{\frac{\kappa}{2} t}\, ,
\end{align}
on $\frac{d}{dt}\phi$.  Applying Lemma \ref{l:gradient_estimate} immediately gives the result
\begin{align}
\big|\nabla_x \int_{P(M)} F\, d\Gamma_x\big| &\leq \int_{P(M)}\,|\nabla_0 F| + \int_0^\infty ||\frac{d}{dt}\phi||_{max}\cdot|\nabla_t F|\, dt\, d\Gamma_x\, ,\notag\\
& = \int_{P(M)}\,|\nabla_0 F| + \int^{\infty}_{0}\frac{\kappa}{2}e^{\frac{\kappa}{2}t}|\nabla_t F|\, dt \, d\Gamma_x\, ,
\end{align}
as claimed.
\end{proof}

The next Theorem applies the estimate $(R2)$ to the simplest functions on path space.  From this we will see how to recover the Bakry-Emery gradient estimate, and hence a lower Ricci curvature bound on $M$.

\begin{theorem}\label{t:boundedricci_BE_implies_lowerricci}
If the estimate
\begin{align}\label{e:boundedricci_BE}
\big|\nabla_x \int_{P(M)} F\, d\Gamma_x\big| &\leq \int_{P(M)}\,|\nabla_0 F| + \int^{\infty}_{0}\frac{\kappa}{2}e^{\frac{\kappa}{2}t}|\nabla_t F|\, dt\, ,
\end{align}
holds for every smooth cylinder function $F$ on $P(M)$, then for every smooth function $u$ on $M$ the estimate
\begin{align}\label{e:lowerricci_BE}
|\nabla H_t u|\leq e^{\frac{\kappa}{2}t}H_t|\nabla u|\, .
\end{align}
holds by applying (\ref{e:boundedricci_BE}) to function $F(\gamma)\equiv u(\gamma(t))$. 
\end{theorem}

\begin{proof}
Consider the function $F:P(M)\to\dR$ defined by
\begin{align}
F(\gamma)\equiv u(\gamma(t))\, ,
\end{align}
where $u$ is a smooth function on $M$.  Then by the definition of the Wiener measure in Section \ref{ss:diffusion_measures} we have the identity
\begin{align}
\int_{P(M)} F\, d\Gamma_x = \int_M u(y)\rho_t(x,dy) = H_tu(x)\, ,
\end{align}
and using Proposition \ref{p:parallel_gradient} we have for $s\leq t$ that
\begin{align}
|\nabla_s F|(\gamma) = |\nabla u|(\gamma(t))\, ,
\end{align}
with $|\nabla_s F| = 0$ for $s>t$.  Plugging these into (\ref{e:boundedricci_BE}) gives the estimate
\begin{align}
|\nabla H_tu|(x)\leq \int_M \bigg(|\nabla u|(y)+\big(e^{\frac{\kappa}{2}t}-1\big)|\nabla u|(y)\bigg)\,  \rho_t(x,dy) = e^{\frac{\kappa}{2}t}H_t|\nabla u|(x)\, ,
\end{align}
as claimed.
\end{proof}

\subsection{Proof that $(R2)\implies (R3)$}\label{ss:r2_r3}

In this Section we prove the quadratic gradient estimate, based on the assumption of a bound on the Ricci curvature tensor.

\begin{proof}[Proof that Theorem \ref{t:smooth_bounded_ricci}.R2 $\implies$ Theorem \ref{t:smooth_bounded_ricci}.R3 ]

The estimate is nothing more than a careful application of H\"older's inequality on (R2).  Specifically let $F$ be $\cF^T$ measurable, then we have

\begin{align}\label{e:br2_br3:1}
\big|\nabla_x \int_{P(M)} F&\, d\Gamma_x\big|^2 \leq \int_{P(M)}\bigg|\,|\nabla_0 F| + \int^{\infty}_{0}\frac{\kappa}{2}e^{\frac{\kappa}{2}t}|\nabla_t F|\, dt \bigg|^2 \, d\Gamma_x\, ,\notag\\
&= \int_{P(M)}\,|\nabla_0 F|^2 + 2 \,|\nabla_0 F|\,\bigg(\int^{T}_{0}\frac{\kappa}{2}e^{\frac{\kappa}{2}t}|\nabla_t F|\, dt\bigg) +\bigg(\int^{T}_{0}\frac{\kappa}{2}e^{\frac{\kappa}{2}t}|\nabla_t F|\, dt\bigg)^2 \, d\Gamma_x\, .\notag\\
\end{align}

We estimate the second term by
\begin{align}
2 \,|\nabla_0 F|&\,\bigg(\int^{T}_{0}\frac{\kappa}{2}e^{\frac{\kappa}{2}t}|\nabla_t F|\, dt\bigg)\leq 2\,|\nabla_0 F|\,\sqrt{e^{\frac{\kappa}{2}T}-1}\sqrt{\int_0^T\frac{\kappa}{2}e^{\frac{\kappa}{2}t}|\nabla_t F|^2\, dt }\,  ,\notag\\
&\leq \big(e^{\frac{\kappa}{2}T}-1\big)|\nabla_0 F|^2+\int_0^T\frac{\kappa}{2}e^{\frac{\kappa}{2}t}|\nabla_t F|^2\, dt \, .
\end{align}

Similarly we can estimate the third term of (\ref{e:br2_br3:1}) by
\begin{align}
\bigg(\int^{T}_{0}\frac{\kappa}{2}e^{\frac{\kappa}{2}t}|\nabla_t F|\, dt\bigg)^2 \leq \big(e^{\frac{\kappa}{2}T}-1\big)\int_0^T\frac{\kappa}{2}e^{\frac{\kappa}{2}t}|\nabla_t F|^2\, dt\, .
\end{align}
Combining these gives the estimate
\begin{align}
\big|\nabla_x \int_{P(M)} F\, d\Gamma_x\big|^2 \leq e^{\frac{\kappa}{2}T}\int_{P(M)}\,|\nabla_0 F|^2+\int_0^T \frac{\kappa}{2}e^{\frac{\kappa}{2}s}|\nabla_s F|^2\, ds\cdot d\Gamma_x\, ,
\end{align}
as claimed.
\end{proof}

\section{Bounded Ricci Curvature and Stochastic Analysis on $P(M)$}\label{s:bounded_ricci_stoc_anal_smooth}

Stochastic analysis already appeared in the proof of $(R2)$, however with a little work the estimate itself may be understood without it (see for instance the second paper).  In this Section we understand Ricci curvature in terms of the stochastic analysis of $M$ more completely by relating bounds on the Ricci curvature to the regularity of martingales on $P(M)$.  Specifically we will see how to relate bounded Ricci curvature to estimates on the quadratic variation of a martingale on $P(M)$ by proving the estimates $(R5),(R6)$ of Theorem \ref{t:smooth_bounded_ricci}, as well as some pointwise versions.  We begin in Section \ref{ss:mart_quad} by reviewing martingales and their quadratic variations.  In Sections \ref{ss:r3_r5}, \ref{ss:r2_r4} we prove the estimates $(R4),(R5)$.  Finally in Section \ref{ss:smooth_mart_cont} we study as an application of the estimates of this Section the continuity properties of martingales.  This will be especially interesting in the nonsmooth case.

\subsection{Martingales and Quadratic Variation}\label{ss:mart_quad}

We already briefly introduced martingales on $P_x(M)$ in Section \ref{ss:function_spaces}, regardless we will begin this Section with another interpretation of martingales on $P_x(M)$ (which only holds for the diffusion measures) and will be particularly useful later.  We will then introduce the quadratic variation and its infinitesimal.\\

If $F\in L^2(P_x(M),\Gamma_x)$ then in Section \ref{ss:function_spaces} we described the martingale induced by $F$ as the family of maps $F^t\in L^2(P^t_x(M),\Gamma_x)$, where $F^t$ is the projection of $F$ to the closed subspace $L^2(P^t_x(M),\Gamma_x)\subseteq L^2(P_x(M),\Gamma_x)$.  If $F$ is only $L^1$ then we can still define $F^t$ as the $\cF^t$-expectation of $F$.  Equivalently, we can write for every $t\geq 0$ and a.e. $\gamma\in P_x(M)$ that
\begin{align}\label{e:martingale}
F^t(\gamma) = \int_{P_x(M)} F(\gamma_{[0,t]}\circ\sigma)\,d\Gamma_{\gamma(t)}\equiv \int_{P_x(M)} F_{\gamma_t}(\sigma)\,d\Gamma_{\gamma(t)}\, ,
\end{align}
 where as before $F_{\gamma_t}:P_{\gamma(t)}(M)\to \dR$ is defined as above by $F_{\gamma(t)}(\sigma)=F(\gamma_{[0,t]}\circ\sigma)$.\\
 
Now a martingale $F^t$ a canonical decomposition of a function into pieces which are $\cF^t$-measurable.  A way of representing the size of the pieces is through the {\it quadratic variation} $[F^t]$ defined as the limit
\begin{align}
[F^t]\equiv \lim_{\bt\subseteq [0,t]} \sum \big(F^{t_{k+1}}-F^{t_k}\big)^2\, ,
\end{align}
where the limit is over partitions $\bt$ of $[0,t]$ with $\Delta\bt\equiv \sup|t_{k+1}-t_k|\to 0$.  The limit exists in measure by standard methods as in \cite{Kuo_book}, and under stronger assumptions on $F^t$ the limit exists in $L^p$ spaces.  Notice the quadratic variation is nonnegative, increasing in $t$, and has the property that
\begin{align}
\int_{P(M)} |F^t|^2\, d\Gamma_x = \int_{P(M)} [F^t]\,d\Gamma_x\, .
\end{align}
Note that for a martingale the quadratic variation is an absolutely continuous process.  In particular, one can construct from this for $t>0$ the $\cF^t$-measure infinitesimal quadratic variation $[dF^t]$ given by the nonnegative function
\begin{align}
[dF^t]\equiv \lim_{s\to 0} \frac{[F^{t+s}]-[F^t]}{s}\, .
\end{align}
Note then for a martingale that because $[F^t]$ is absolutely continuous in $t$ we have that $[dF^t]$ is $\cF^t$-measurable.  The infinitesimal quadratic variation is the appropriate replacement as a measurement of the time rate of change of $F^t$.  Note that the quadratic variation can be extended to a bilinear mapping on pairs of martingales $F^t$, $G^t$ by
\begin{align}
[F^t,G^t]\equiv \lim_{\bt\subseteq [0,t]} \sum \big(F^{t_{k+1}}-F^{t_k}\big)\big(G^{t_{k+1}}-G^{t_k}\big)\, .
\end{align}
In this case we of course have that $[F^t]\equiv [F^t,F^t]$.  We can still consider the infinitesimal $d[F^t,G^t]$, defined in the analogous manner.

\subsection{Proof that $(R3)\Leftrightarrow (R5)$}\label{ss:r3_r5}

In this Section we consider the quadratic variation of a function $F\in L^2(P_x(M),\Gamma_x)$, and study its relationship with bounded Ricci curvature.  Specifically, we prove estimate $(R5)$ in Theorem \ref{t:smooth_bounded_ricci}.  That is, under the assumption of the Ricci bound $-\kappa g\leq \Ric+\nabla^2f\leq \kappa g$ we prove the estimates
\begin{align}\label{e:ricci_quadvar2}
\int_{P(M)}[dF^t]\, d\Gamma_x \leq e^{\frac{\kappa}{2}(T-t)}\int_{P(M)} |\nabla_t F|^2+\int_t^T \frac{\kappa}{2}e^{\frac{\kappa}{2}(s-t)}|\nabla_s F|^2\, d\Gamma_x\, .
\end{align}
for every $\cF^T$-measurable function $F$ and $t<T$.  In fact, we will use directly the gradient estimate $(R3)$, which was proved in the previous Section, to prove the above estimates.  \\

We begin by proving $(R5)$ at $t=0$.

\begin{lemma}\label{l:r4_r5_t0}
For $F\in L^2(P(M),\Gamma_x)$ we have that the gradient estimate $(R3)$ is equivalent to 
$$[dF^0](x)\leq e^{\frac{\kappa}{2}T}\int_{P(M)} |\nabla_0 F|^2+\int_0^T \frac{\kappa}{2}e^{\frac{\kappa}{2}s}|\nabla_s F|^2\, d\Gamma_x\, .
$$
\end{lemma}

\begin{proof}
To prove the Lemma it is enough to study smooth cylinder functions, then the result follows for arbitrary functions in $L^2(P_x(M),\Gamma_x)$ through extension.  So let $F=e_\bt^*u$ be a smooth cylinder function given by
\begin{align}
F(\gamma)\equiv u(\gamma(t_1),\ldots,\gamma(t_N))\, .
\end{align}
Then for $t<t_1$ we can use (\ref{e:martingale}) to write the projection of $F$ to the $\cF^t$-measurable functions by
\begin{align}
F^t(\gamma) = \int_{M^{N}}u(y_1,\ldots,y_N)\rho_{t_{1}-t}(\gamma(t),dy_{1})\cdots\rho_{t_N-t_{N-1}}(y_{N-1},dy_{N})\, .
\end{align}
Note in particular that $F^t\equiv e_t^*v$ is itself a smooth cylinder function with $$v(y)\equiv \int_{M^{N}}u(y_1,\ldots,y_N)\rho_{t_{1}-t}(y,dy_{1})\cdots\rho_{t_N-t_{N-1}}(y_{N-1},dy_{N})\, .$$

Now we can compute
\begin{align}\label{e:r3_r5:equivalence}
[dF^0]&=\int_{P(M)}[dF^0]\,d\Gamma_x = \lim\int_{P(M)}\frac{\big(F^s-F^{0}\big)^2}{s}\,d\Gamma_x = \lim\int_{P(M)}\frac{\big(F^{s}-(F^{s})^0\big)^2}{s}\,d\Gamma_x\notag\\
&=\lim\frac{1}{s}\int_M\bigg(v(y_s)-\int_M v(z_s)\rho_s(x,dz_s)\bigg)^2\,\rho_s(x,dy_s)\notag\\
&=|\nabla v|^2(x) =|\nabla \int_{P(M)} F\,d\Gamma_x|^2(x)\, ,
\end{align}
from which the lemma follows.  In particular we have the estimate

\begin{align}
[dF^0]\leq e^{\frac{\kappa}{2}T}\int_{P(M)} |\nabla_0 F|^2+\int_0^T \frac{\kappa}{2}e^{\frac{\kappa}{2}s}|\nabla_s F|^2\, d\Gamma_x\, .
\end{align}

\end{proof}

In addition to proving $(R5)$ we will prove the following stronger pointwise versions.  These pointwise estimates will be especially important in the nonsmooth case.

\begin{theorem}\label{t:pointwise_r5}
The following are equivalent
\begin{enumerate}
\item The estimate $(R3)$.
\item For each $F\in L^2(P_x(M),\Gamma_x)$, $t\geq 0$ and $a.e.$ $\gamma\in P_x(M)$ we have the pointwise estimate
\begin{align}
[dF^t](\gamma)&\leq e^{\frac{\kappa}{2}(T-t)}\int_{P(M)} |\nabla_t F|^2(\gamma_{[0,t]}\circ\sigma)+\int_t^{T}\frac{\kappa}{2}e^{\frac{\kappa}{2}(s-t)}|\nabla_{s} F|^2(\gamma_{[0,t]}\circ\sigma)\,d\Gamma_{\gamma(t)}\, .
\end{align}\item For each $F\in L^2(P_x(M),\Gamma_x)$ and $t\geq 0$ we have the integral estimate $(R5)$:
\begin{align}\label{e:ricci_quadvar2}
\int_{P(M)}[dF^t]\, d\Gamma_x \leq e^{\frac{\kappa}{2}(T-t)}\int_{P(M)} |\nabla_t F|^2+\int_t^T \frac{\kappa}{2}e^{\frac{\kappa}{2}(s-t)}|\nabla_s F|^2\, d\Gamma_x\, ,
\end{align}
\end{enumerate}
\end{theorem}

\begin{proof}
We saw in Lemma \ref{l:r4_r5_t0} that (1) was equivalent to (2) at $t=0$.  In particular we need then to see that (1) implies (2) for all time.  Now let $F(\gamma)=u(\gamma(t_1),\ldots,\gamma(t_N))$ be a smooth cylinder function with $t>0$ fixed.  Let us fix $\gamma\in P_x(M)$, and let us consider the smooth cylinder function $F_{\gamma_t}\in L^2(P(M),\Gamma_{\gamma(t)})$ by
\begin{align}
F_{\gamma_t}(\sigma)\equiv F(\gamma_{[0,t]}\circ\sigma) = F(\gamma(t_1),\ldots, \gamma(t_k),\sigma(t_{k+1}-t),\ldots,\sigma(t_N-t))\, ,
\end{align}
where $t_k$ is the largest element of the partition such that $t_k\leq t$.  Now applying Lemma \ref{l:r4_r5_t0} to the function $F_{\gamma_t}$, and recalling that $F_{\gamma_t}$ is $\cF^{T-t}$-measurable, yields the estimate
\begin{align}
[dF^0_{\gamma_t}]&\leq e^{\frac{\kappa}{2}(T-t)}\int_{P(M)} |\nabla_0 F_{\gamma_t}|^2+\int_0^{T-t}\frac{\kappa}{2}e^{\frac{\kappa}{2}s}|\nabla_s F_{\gamma_t}|^2\,d\Gamma_{\gamma(t)}\, .
\end{align}
Now let us observe the following.  First we have the equality
\begin{align}
[dF^t](\gamma) = [dF^0_{\gamma_t}]\, ,
\end{align}
and then combining this with the previous estimate gives us
\begin{align}
[dF^t](\gamma) &\leq e^{\frac{\kappa}{2}(T-t)}\int_{P(M)} |\nabla_0 F_{\gamma_t}|^2+\int_0^{T-t}\frac{\kappa}{2}e^{\frac{\kappa}{2}s}|\nabla_s F_{\gamma_t}|^2\,d\Gamma_{\gamma(t)}\notag\\
&=e^{\frac{\kappa}{2}(T-t)}\int_{P(M)} |\nabla_t F|^2(\gamma_{[0,t]}\circ\sigma)+\int_t^{T}\frac{\kappa}{2}e^{\frac{\kappa}{2}(s-t)}|\nabla_{s} F|^2(\gamma_{[0,t]}\circ\sigma)\,d\Gamma_{\gamma(t)}\, ,
\end{align}
as claimed.

That $(2)\implies (3)$ is immediate through integration, and finally we then need to see that $(3)\implies (1)$, which by Lemma \ref{l:r4_r5_t0} it is enough to show $(3)\implies (2)$ at $t=0$.  Now $(3)$ at $t=0$ is the same as $(2)$ at $t=0$, and hence the Theorem is proved.
\end{proof}

Let us end by comparing the estimates of this Section to the lower Ricci curvature case.  Namely, as in the previous Sections, let us see how by applying the estimate to the simplest functions on path space we recover the lower Ricci curvature estimate:

\begin{theorem}\label{t:boundedricci_BE2_implies_lowerricci}
If (R5) holds for every $\cF^T$-measurable function $F$ and $t\leq T$, then for every smooth function $u:M\to \dR$ we have the inequality
\begin{align}
H_t|\nabla H_{T-t}u|^2(x)\leq e^{\kappa(T-t)}H_T|\nabla u|^2(x)\, ,
\end{align}
by applying (R5) to function $F(\gamma)\equiv u(\gamma(T))$.
\end{theorem}

\begin{proof}
Let us consider the smooth cylinder function given by $F(\gamma)\equiv u(\gamma(T))$ and let $t\leq T$.  Then we see that
\begin{align}
F^t(\gamma) = \int_{M}u(y)\rho_{T-t}(\gamma(t),dy) = H_{T-t}u(x)\, .
\end{align}
A computation then gives us
\begin{align}
\int_{P(M)}[dF^t]\,d\Gamma_x &= \lim\int_{P(M)}\frac{\big(F^{t}-F^{t-s}\big)^2}{s}\,d\Gamma_x = \lim\int_{P(M)}\frac{\big(F^{t}-(F^{t})^{t-s}\big)^2}{s}\,d\Gamma_x\notag\\
&=\lim\frac{1}{s}\int_M\bigg(\int_M \bigg(H_{T-t}u(y_2)-\int_M H_{T-t}u(z)\rho_s(y_1,dz)\bigg)^2\rho_s(y_1,dy_2)\bigg)\rho_{t-s}(x,dy_1)\notag\\
&=\int_M |\nabla H_{T-t}u|^2\rho_{t}(x,dy) = H_t|\nabla H_{T-t}u|^2(x)\, ,
\end{align}
while a computation like that in Lemma \ref{l:lower_ricci:1} gives us
\begin{align}
e^{\frac{\kappa}{2}(T-t)}\int_{P(M)} |\nabla_t F|^2&+\int_t^T \frac{\kappa}{2}e^{\frac{\kappa}{2}(s-t)}|\nabla_s F|^2\, d\Gamma_x\notag\\
&= e^{\frac{\kappa}{2}(T-t)}\int_{M} |\nabla u|^2(y)+\big(\int_t^T \frac{\kappa}{2}e^{\frac{\kappa}{2}(s-t)}\big)|\nabla u|^2(y)\,\rho_T(x,dy)\notag\\
&= e^{\kappa(T-t)}H_T|\nabla u|^2(x)\, .
\end{align}
Substituting these into (R5) proves the Lemma.
\end{proof}

\subsection{Proof that $(R2)\Leftrightarrow (R4)$}\label{ss:r2_r4}

The statements of this Section are the analogous statements for $(R2)$ and $(R4)$ as were stated in the previous Section for $(R3)$ and $(R5)$.  In particular we are interested in seeing that a Ricci curvature bound implies the estimate
\begin{align}\label{e:ricci_quadvar1}
\int_{P(M)}\sqrt{[dF^t]}\, d\Gamma_x \leq \int_{P(M)} |\nabla_t F|+\int_t^T \frac{\kappa}{2}e^{\frac{\kappa}{2}(s-t)}|\nabla_s F|\, d\Gamma_x\, .
\end{align}

We state the main results, but since the proofs are completely analogous, and indeed nearly verbatim, to those of Section \ref{ss:r3_r5} we do not do them.  The main result of this Section is the following:

\begin{theorem}\label{t:pointwise_r4}
The following are equivalent
\begin{enumerate}
\item The estimate $(R2)$.
\item For each $F\in L^2(P(M),\Gamma_x)$, $t\geq 0$ and $a.e.$ $\gamma\in P_x(M)$ we have the pointwise estimate
\begin{align}
\sqrt{[dF^t]}(\gamma)&\leq \int_{P(M)} |\nabla_t F|(\gamma_{[0,t]}\circ\sigma)+\int_t^{T}\frac{\kappa}{2}e^{\frac{\kappa}{2}(s-t)}|\nabla_{s} F|(\gamma_{[0,t]}\circ\sigma)\,d\Gamma_{\gamma(t)}\, ,
\end{align}

\item For each $F\in L^2(P(M),\Gamma_x)$ and $t\geq 0$ we have the integral estimate $(R4)$:
\begin{align}
\int_{P(M)}\sqrt{[dF^t]}\, d\Gamma_x \leq \int_{P(M)} |\nabla_t F|+\int_t^T \frac{\kappa}{2}e^{\frac{\kappa}{2}(s-t)}|\nabla_s F|\, d\Gamma_x\, ,
\end{align}
\end{enumerate}
\end{theorem}

Let us also observe that as in Theorem \ref{t:boundedricci_BE2_implies_lowerricci} that by applying $(R4)$ to the simplest functions on path space we can recover the analogous statements for lower Ricci curvature.

\begin{theorem}
If (\ref{e:ricci_quadvar1}) holds for every function $F$, then for every smooth function $u:M\to \dR$ we have the inequality
\begin{align}
H_t|\nabla H_{T-t}u|(x)\leq e^{\frac{\kappa}{2}(T-t)}H_T|\nabla u|(x)\, ,
\end{align}
by applying (\ref{e:ricci_quadvar1}) to function $F(\gamma)\equiv u(\gamma(T))$.
\end{theorem}

\subsection{Application to Continuity of Martingales}\label{ss:smooth_mart_cont}

In this Section we discuss applications of $(R4),(R5)$ to the time regularity of martingales.  In fact, we discuss two results in this Section, one of which applies to lower Ricci curvature bounds and the other applies to bounded Ricci curvature.  In particular, we see that moving from a lower Ricci curvature bound to an absolute bound results in increasing the regularity of martingales from continuous to H\"older continuous. We delay the proofs until the second part of the paper since we prove the same results in the much more complicated context of nonsmooth metric-measure spaces.  The importance of the results lies primarily in the nonsmooth case, where the proofs presented give the same martingale regularity in a context where nothing was previously understood.  

To talk about the time regularity of a martingale $F^t$ we need to address the issue that each $F^t$ is only well defined away from a set of measure zero.  In this direction we make the following definition:
\begin{definition}
Let $X^t,\tilde X^t:P_x(X)\to \dR$ be stochastic processes.  That is, each is a one parameter family of functions such that $X^t,\tilde X^t$ are $\cF^t$-measurable.  We say $X^t$ and $\tilde X^t$ are versions of each other if for each $t\geq 0$ we have that $X^t=\tilde X^t$ $\Gamma_x$-a.e.
\end{definition}

On $\dR^n$ it is a well known consequence of the Clark-Ocone theorem that every martingale $F^t$ on $P_0(\dR^n)$ has a pointwise time continuous version.  That is, for each $\gamma\in F^t$ we have that $F^t(\gamma)$ is a continuous function of time.  In the second part of the paper we will see how to prove this for metric-measure spaces with lower Ricci bounds, in fact, we will prove an effective version of it.  In particular we have the following result on smooth spaces (which can be proved by more standard means):

\begin{theorem}\label{t:smooth_martingale_lowerricci}
Let $(M^n,g,e^{-f}dv_g)$ be a complete smooth metric-measure space with $\Ric+\nabla^2 f\geq -\kappa g$, then every martingale $F^t$ has a pointwise time continuous version.
\end{theorem}
\begin{proof}
See Section \ref{s:ricci_martingales}.
\end{proof}

In fact, if $M^n$ has {\it bounded} Ricci curvature then there is an even stronger result.  In the second part of the paper we will see that metric-measure spaces with bounded Ricci curvature that martingales with bounded parallel gradients have H\"older continuous versions.  In particular, this is new even on smooth spaces:

\begin{theorem}\label{t:smooth_martingale_holder}
Let $(M^n,g,e^{-f}dv_g)$ be a complete smooth metric-measure space with $-\kappa g\leq \Ric+\nabla^2 f\leq \kappa g$.  Then if $F:P_x(M)\to \dR$ is $\cF^T$-measurable with $|\nabla_s F|$ uniformly bounded independent of $s$, then the induced martingale $F^t$ has a version such that for each $\gamma\in P_x(M)$ we have that $F^t(\gamma)$ is $C^\alpha$-H\"older continuous for every $\alpha<\frac{1}{2}$.
\end{theorem}
\begin{proof}
See Section \ref{s:ricci_martingales}.
\end{proof}
\vspace{.5 cm}

\section{Bounded Ricci Curvature and the Ornstein-Uhlenbeck Operator}\label{s:smooth_br_OU}

In this Section we discuss our third characterization of bounded Ricci curvature by studying the relationship between Ricci curvature and the analysis on path space $P_x(M)$.  Precisely, we prove the estimates $(R6),(R7)$ from Theorem \ref{t:smooth_bounded_ricci} for the Ornstein-Uhlenbeck operator and its twisted variations.  These operators act as infinite dimensional laplacians on path space.  The Ornstein-Uhlenbeck operator itself was first studied by Gross in \cite{Gross_LogSob} on $\dR^n$.  The definition and analysis of the Ornstein-Uhlenbeck operator on a smooth manifold took some time, the history of which was discussed some in the introduction.

We begin in Sections \ref{ss:H1_gradient} and \ref{ss:OU_smooth} by recalling the $H^1_x$-gradient and the construction of the Ornstein-Uhlenbeck operator, as well as some generalizations.  Then in Section \ref{ss:r5_r6} we prove the spectral gap estimate for these operators, and their equivalence to the martingale estimates of Section \ref{s:bounded_ricci_stoc_anal_smooth}.  In Section \ref{ss:r4_r7} we prove the log-Sobolev estimates under the assumption of $(R4)$.  Finally in Section \ref{ss:r7_r3} we will show that the log-Sobolev estimate $(R7)$ itself implies the gradient estimate $(R3)$.  This will be important in Section \ref{s:finish_maintheorem} when we finish the proof of Theorem \ref{t:smooth_bounded_ricci} by showing that the gradient estimate itself implies Ricci curvature bound $-\kappa g\leq \Ric+\nabla^2 f\leq \kappa g$. \\

\subsection{The $H^1_x$ Gradient on Path Space}\label{ss:H1_gradient}

In this Section we introduce the $H^1_x$-gradient on path space, sometimes also called the Malliavin gradient.  Although we will use other variations of this gradient as well, it seems worthwhile to introduce the $H^1_x$-gradient explicitly as similar constructions will be used to build other variations.  In essence, the definition of the $H^1_x$-gradient is the same as that for the parallel gradient, except one considers variations by $H^1$ vector fields, instead of parallel vector fields.

Throughout we will denote
\begin{align}
\cH\equiv H^1_0([0,\infty),T_xM)\, ,
\end{align}
as the collection of based $H^1_0$ paths in Euclidean space $T_xM$.  Now if $F:P_x(M)\to\dR$ is a smooth cylinder function then we define the Malliavin gradient $\nabla F:P_x(M)\to\cH$ for a.e. $\gamma\in P_x(M)$ as the unique element of $\cH$ such that
\begin{align}
\langle \nabla F, h\rangle_{H^1_0} = D_H F\text{ for all }h\in\cH\, ,
\end{align}
where $H=P_t^{-1}h(t)$ is the vector field along $\gamma$ associated to $h\in\cH$ under the stochastic parallel translation map.  It is a standard point now \cite{Hsu_book} that the operator $\nabla$ may be extended to a closed unbounded operator
\begin{align}
\nabla: L^2(P_{x}(M))\to L^2(P_{x}(M),\cH)\, ,
\end{align}
with the smooth cylinder functions as a dense subset of the domain $D(\nabla)$.  The proof of this point relies on Driver's integration by parts formula, see \cite{Driver_CM} and the related result of Lemma \ref{l:parallel_ibp}.  In Section \ref{ss:OU_smooth} we will use these ideas to construct the twisted Ornstein-Uhlenbeck operators.  

We will end this Section by remarking on the relationship between the Malliavin gradient and the parallel gradients.  In essence we see that the parallel gradients give a geometric interpretation of the time derivative of the Malliavian derivative, a point which is crucial in the nonsmooth case, and also plays a motivational role in the definition of the twisted Ornstein-Uhlenbeck operators. 

\begin{proposition}\label{p:parallel_H1_relation}
For $F$ a smooth cylinder function we have the following
\begin{enumerate}
\item For a.e. $\gamma\in P_x(M)$ we have that
\begin{align}
\frac{d}{dt}\nabla F(\gamma)=P_t^{-1}\nabla_t F(\gamma)\, ,
\end{align}
where $\nabla F$ is the $H^1_x$ gradient and $\nabla_t$ are the parallel gradients.
\item For a.e. $\gamma\in P_x(M)$ we have that
\begin{align}
&|\nabla F|^2(\gamma) = \int_0^\infty |\nabla_t F|^2\,dt\, .
\end{align}
\end{enumerate}
\end{proposition}
\begin{remark}
One can use this to give a new geometric interpretation of the Clark-Ocone formula.  This could be used to give some distinct proofs of a few of the results of this paper, though we focus on proofs whose morals carry over to the nonsmooth case.
\end{remark}

\begin{proof}

Let $F=e_\bt^*u$ be a smooth cylinder function on $P_x(M)$ and let $\nabla F:P_x(M)\to \cH$ be its Malliavin derivative.  Then we have that
\begin{align}
&\langle \nabla F,h\rangle_{H^1_0}(\gamma) = \int_0^\infty \big\langle\frac{d}{dt}{\nabla F(t)},\frac{d}{dt}h(t)\big\rangle\,dt = D_H F\, ,\notag\\
&\implies \int_0^\infty \big\langle-\frac{d^2}{dt^2}\nabla F, h\big\rangle\, dt +\langle \nabla F(0),h(0)\rangle = \sum \langle h(t_j),\nabla_j f\rangle\, .
\end{align}
By integrating both sides we get that
\begin{align}
\frac{d}{dt}\nabla F(t) = \sum_{t_j\geq t} P^{-1}_tP_{t_j}\nabla_j f\, .
\end{align}
Using Proposition \ref{p:parallel_gradient} this gives the first result.  To get the second result we take norms and integrate to get
\begin{align}
|\nabla F|^2_{H^1_x}&=\int_0^\infty |\frac{d}{dt}\nabla F|^2\,dt = \sum^{|\bt|}_{j=1} (t_j-t_{j-1})\big|\sum_{k\geq j}P_{t_j}\nabla_j f\big|^2\, ,\notag\\
&=\sum_j (t_j-t_{j-1}) |\nabla_{t_j}F|^2=\int_0^\infty |\nabla_t F|^2\, dt\, .
\end{align}
\end{proof}

\subsection{The twisted Ornstein-Uhlenbeck Operators on Smooth Metric-Measure Spaces}\label{ss:OU_smooth}

In this subsection we outline the structure needed to define the twisted Ornstein-Uhlenbeck operators on path space.  The constructions follow very similarly the construction of the classical Ornstein-Uhlenbeck operator.  The twisted Ornstein-Uhlenbeck operators are closed, nonnegative, self-adjoint operator on the Hilbert space $L^2(P_x(M),\Gamma_x)$ of $L^2$ functions on path space with respect to the Wiener measure $\Gamma_x$.  

The construction of the twisted Ornstein-Uhlenbeck operators $L^{t_1}_{t_0,\kappa}$ relies on the ability to write down the appropriate Dirichlet forms.  Since we have already discussed the construction of the Wiener measure and we have seen in Section \ref{ss:parallel_gradient} that the parallel gradients $\nabla_t$ are closed derivations, we are in a good position to do this.  First recall from Section \ref{sss:char_ricci_OU} that the operators $L^{t_1}_{t_0}=L^{t_1}_{t_0,0}$ are meant to represent the part of the Ornstein-Uhlenbeck operator restricted to the time interval $[t_0,t_1]$, and hence we are interested in the Dirichlet forms $E_{t_0}^{t_1}:L^2(P_{x}(M))\otimes L^2(P_{x}(M))\to \dR$ on path space defined by
\begin{align}
E_{t_0}^{t_1}[F,G]= E_{t_0,0}^{t_1}[F,G]\equiv \int_{P_{x}(M)} \int_{t_0}^{t_1}\langle \nabla_t F,\nabla_t G\rangle\, d\Gamma_{x}\, .
\end{align}
For the basics on Dirichlet forms see \cite{Fukushima_DirichletForms}.  One can now check using that $\nabla_t$ is a closed derivation that $E_{t_0}^{t_1}$ are closed Dirichlet forms with domain $\cD(E_{t_0}^{t_1})$.  In particular, associated with $E_{t_0}^{t_1}$ are unique self-adjoint operators $L_{t_0}^{t_1}$ such that for every $F\in \cD(E_{t_0}^{t_1})$ and $G\in \cD(L_{t_0}^{t_1})\subseteq \cD(E_{t_0}^{t_1})$ we have that
\begin{align}
E_{t_0}^{t_1}[F,G] = \int_{P_{x}(M)} \langle F, L_{t_0}^{t_1}G\rangle d\Gamma_{x}\, .
\end{align}
Let us remark on the following. If $\Delta_t:L^2(P_xM,\Gamma_x)\to L^2(P_xM,\Gamma_x)$ are the closed $t$-laplace operators associated to the parallel gradient $\nabla_t$, see Section \ref{ss:parallel_gradient}, then we have that
\begin{align}
L_{t_0}^{t_1} = \int_{t_0}^{t_1} \Delta_t\, ,
\end{align}
and in particular that we can write the classical Ornstein-Uhlenbeck operator as
\begin{align}
L_x = L_0^\infty = \int_{0}^{\infty} \Delta_t\, . \\ \notag
\end{align}

Finally, let us end by discussing the slightly more involved operators $L_{t_0,\kappa}^{t_1}$.  For $\kappa\neq 0$ one should simply interpret these as perturbations of the operators $L_{t_0}^{t_1}$.  We begin again by introducing the Dirichlet forms $E_{t_0,\kappa}^{t_1}:L^2(P_{x}(M))\otimes L^2(P_{x}(M))\to \dR$ defined on path space by
\begin{align}
E^{t_1}_{t_0,\kappa}[F,G]\equiv \int_{P_x(M)}\Bigg( \int_{t_0}^{t_1}\cosh\big(\frac{\kappa}{2}(s-t_0)\big)\langle \nabla_s F,\nabla_s G\rangle +\big(1-e^{-\frac{\kappa}{2}(t_1-t_0)}\big)\int^{\infty}_{t_1} e^{\frac{\kappa}{2}(s-t_1)}\langle \nabla_s F,\nabla_s G\rangle\Bigg)\,d\Gamma_x\, .
\end{align}
It is again fairly easy to check that these actually define Dirichlet forms because $\nabla_s$ are closed derivations.  Thus there exists a unique closed linear operator $L^{t_1}_{t_0,\kappa}: L^2(P_xM)\to L^2(P_xM)$ such that
\begin{align}
E^{t_1}_{t_0,\kappa}[F,G] = \int_{P_xM}\langle F,L^{t_1}_{t_0,\kappa} G\rangle \, d\Gamma_x\, ,
\end{align}
for all $F,G$ in the appropriate domains.  It is not hard to check that the operators $L^{t_1}_{t_0,\kappa}$ preserve $\cF^T$ measurable functions, and therefore define operators on time restricted path space $L^2(P_{x}^T(M),\Gamma_x)$ for every $T>0$.

\subsection{Proof that $(R5)\Longleftrightarrow (R6)$:}\label{ss:r5_r6}

In this Section we prove the spectral gap for the family of twisted Ornstein-Uhlenbeck operators is equivalent to the martingale estimate of $(R5)$.  As an application, we conclude from this spectral gap estimates on the classical Ornstein-Uhlenbeck operator which are sharper than those currently in the literature.  We will apply these estimates to the simplest functions on path space in order to recover some of the Bakry-Emery estimates.  Let us begin by proving the spectral gap of $(R6)$:

\begin{proof}[Proof that $(R5)\implies (R6)$]
Let $F$ be a $\cF^T$-measurable smooth cylinder function, and let $F^t$ be the martingale induced by letting $F^t$ be the function obtained by projecting $F$ to the $\cF^t$-measurable functions. Applying the Ito formula \cite{Kuo_book} to the function $|F^t|^2$ gives the formula 
\begin{align}\label{e:r5_r6:1}
\int_{P(M)} |F^{t_1}-F^{t_0}|^2\,d\Gamma_x = \int_{P(M)} |F^{t_1}|^2d\Gamma_x-\int_{P(M)} |F^{t_0}|^2\,d\Gamma_x = \int_{P(M)} \int_{t_0}^{t_1} [dF^t]\,d\Gamma_x\, .
\end{align}
Now let us recall that $(R5)$ gives us the estimate
\begin{align}
\int_{P(M)}[dF^t]\, d\Gamma_x \leq e^{\frac{\kappa}{2}(T-t)}\int_{P(M)} |\nabla_t F|^2+\int_t^T \frac{\kappa}{2}e^{\frac{\kappa}{2}(s-t)}|\nabla_s F|^2\, d\Gamma_x\, ,
\end{align}

Plugging this into (\ref{e:r5_r6:1}) gives the estimate
\begin{align}
\int_{P(M)} |F^{t_1}&-F^{t_0}|^2\,d\Gamma_x \leq \notag\\
&\leq e^{\frac{\kappa}{2}T}\int_{P(M)}\bigg(\int_{t_0}^{t_1} e^{-\frac{\kappa}{2}t}|\nabla_t F|^2+\int_{t_0}^{t_1}\int_t^T\frac{\kappa}{2}e^{\frac{\kappa}{2}(s-2t)}|\nabla_s F|^2ds\,dt\bigg)\,d\Gamma_x\notag\\
&= e^{\frac{\kappa}{2}T}\int_{P(M)}\bigg(\int_{t_0}^{t_1} e^{-\frac{\kappa}{2}t}|\nabla_t F|^2+\int_{t_0}^{t_1}\int_t^{t_1}\frac{\kappa}{2}e^{\frac{\kappa}{2}(s-2t)}|\nabla_s F|^2+\int_{t_0}^{t_1}\int_{t_1}^{T}\frac{\kappa}{2}e^{\frac{\kappa}{2}(s-2t)}|\nabla_s F|^2\bigg)\,d\Gamma_x \notag\\
&= e^{\frac{\kappa}{2}T}\int_{P(M)}\bigg(\int_{t_0}^{t_1} e^{-\frac{\kappa}{2}t}|\nabla_t F|^2+\frac{1}{2}\int_{t_0}^{t_1}\int_t^{t_1}\frac{\kappa}{2}e^{\frac{\kappa}{2}(s-2t)}|\nabla_s F|^2+\int_{t_0}^{t_1}\int_{t_1}^{T}\frac{\kappa}{2}e^{\frac{\kappa}{2}(s-2t)}|\nabla_s F|^2\bigg)\,d\Gamma_x \notag\\
&= e^{\frac{\kappa}{2}T}\int_{P(M)}\bigg(\int_{t_0}^{t_1} e^{-\frac{\kappa}{2}t}|\nabla_t F|^2+\int_{t_0}^{t_1}\int_{t_0}^{s}\frac{\kappa}{2}e^{\frac{\kappa}{2}(s-2t)}|\nabla_s F|^2dtds+\int_{t_1}^{T}\int_{t_0}^{t_1}\frac{\kappa}{2}e^{\frac{\kappa}{2}(s-2t)}|\nabla_s F|^2dtds\bigg)\,d\Gamma_x \notag\\
&= e^{\frac{\kappa}{2}(T-t_0)}\int_{P(M)}\bigg(\int_{t_0}^{t_1} e^{-\frac{\kappa}{2}(t-t_0)}|\nabla_t F|^2+\int_{t_0}^{t_1}\frac{1}{2}\big(e^{\frac{\kappa}{2}(s-t_0)}-e^{-\frac{\kappa}{2}(s-t_0)}\big)|\nabla_s F|^2\notag\\
&\;\;\;\;\;\;\;\;\;\;\;\;\;\;\;\;\;\;\;\;\;\;\;+\int_{t_1}^{T}\frac{1}{2}e^{\frac{\kappa}{2}(s-t_0)}\big(1-e^{-\kappa(t_1-t_0)}\big)|\nabla_s F|^2\bigg)\,d\Gamma_x \notag\\
&= e^{\frac{\kappa}{2}(T-t_0)}\int_{P(M)}\bigg(\int_{t_0}^{t_1} \cosh\big(\frac{\kappa}{2}(t-t_0)\big)|\nabla_t F|^2+\big(1-e^{-\kappa(t_1-t_0)}\big)\int_{t_1}^{T}\frac{1}{2}e^{\frac{\kappa}{2}(s-t_0)}|\nabla_s F|^2\bigg)\,d\Gamma_x\, ,\notag\\
&= e^{\frac{\kappa}{2}(T-t_0)} E_{t_0,\kappa}^{t_1}[F,F]\notag\\
&= e^{\frac{\kappa}{2}(T-t_0)}\int \langle F, L_{t_0,\kappa}^{t_1}\rangle \, d\Gamma_x\, ,
\end{align}
which proves the desired estimate.\\
\end{proof}

Let us now remark on the following corollary.  By letting $t_0=0$ and $t_1=T$, and making the observation that the Ornstein-Uhlenbeck operator $L_x$ satisfies the estimate
\begin{align}
L_{0,\kappa}^T \leq \cosh(\frac{\kappa}{2}T)L_x,
\end{align}
we immediately obtain the following:
\begin{corollary}\label{c:OU_spectral_gap}
If $(M^n,g,e^{-f}dv_g)$ is a smooth metric measure space with $|\Ric+\nabla^2 f|\leq \kappa$, then for the standard Ornstein-Uhlenbeck operator we have the spectral-gap estimate for all $\cF^T$-measurable functions:
\begin{align}
\int_{P_xM} \big|F-\int F\big|^2\, d\Gamma_x\leq \frac{1}{2}\big(e^{\kappa T}+1\big)\int_{P_xM} |\nabla_{H^1} F|^2\, d\Gamma_x\, .
\end{align}
\end{corollary}
\vspace{.5 cm}

Before continuing and proving the implication $(R6)\implies (R5)$ let us first use $(R6)$ to recover one of the classical Bakry-Emery estimates:

\begin{theorem}\label{t:bounded_ricci_implies_lower_ricci_spec_gap}
If $(R6)$ holds, then for each smooth $u:M\to\dR$ and $t>0$ we have the estimate
\begin{align}
\int_M |u-\int_M u\, \rho_t|^2\rho_t(x,dy) \leq  \kappa^{-1}\left(e^{\kappa t}-1\right)\int_M |\nabla u|^2 \rho_t(x,dy)\, .
\end{align}
\end{theorem}
\begin{proof}
The proof follows the same structure as Theorems \ref{t:boundedricci_BE_implies_lowerricci} and \ref{t:boundedricci_BE2_implies_lowerricci}.  Namely, if $F(\gamma)=u(\gamma(t))$ then we have that
\begin{align}
\int_{P(M)} F\, d\Gamma_x = \int_M u\rho_t(x,dy)\, ,
\int_{P(M)} F^2\,d\Gamma_x = \int_M u^2\rho_t(x,dy)\, ,
\end{align}
as well as
\begin{align}
|\nabla_s F| = |\nabla u|(\gamma(t))\, ,
\end{align}
for every $s\leq t$.  Plugging this into $(R7)$ gives
\begin{align}
\int_M |u-\int_M u\, \rho_t|^2\rho_t(x,dy) &= \int_{P(M)} |F-\int F|^2d\Gamma_x = \int_{P(M)} |F^T-F^0|^2d\Gamma_x \notag\\
&\leq e^{\frac{\kappa}{2}T}\int_{P(M)}\bigg(\int_0^T \cosh(\frac{\kappa}{2}t)|\nabla_t F|^2\, dt\bigg)\, d\Gamma_x \notag\\
&= e^{\frac{\kappa}{2}T}\int_0^T\cosh(\frac{\kappa}{2}t)dt\int_M |\nabla u|^2(y) \rho_t(x,dy)\notag\\
&=\kappa^{-1}(e^{\kappa T}-1)\int_M |\nabla u|^2(y) \rho_t(x,dy)\, ,
\end{align}
as claimed.
\end{proof}
\vspace{.5 cm}

Let us now finish this section by proving the converse relation $(R6)\implies (R5)$:

\begin{proof}[Proof that $(R6)\implies (R5)$]
Let $F$ be a $\cF^T$-measurable smooth cylinder function, and let $F^t$ be the martingale induced by letting $F^t$ be the function obtained by projecting $F$ to the $\cF^t$-measurable functions.  The spectral gap of $(R6)$ tells us that we can estimate
\begin{align}
\int_{P(M)} \int_{t_0}^{t_1}& [dF^t]\,d\Gamma_x = \int_{P(M)} |F^{t_1}-F^{t_0}|^2\notag\\
&\leq e^{\frac{\kappa}{2}(T-t_0)}\int_{P(M)}\bigg(\int_{t_0}^{t_1} \cosh\big(\frac{\kappa}{2}(t-t_0)\big)|\nabla_t F|^2+\big(1-e^{-\kappa(t_1-t_0)}\big)\int_{t_1}^{T}\frac{1}{2}e^{\frac{\kappa}{2}(t-t_0)}|\nabla_t F|^2\bigg)\,d\Gamma_x\, .
\end{align}
In particular, dividing both sides by $|t_1-t_0|$ and limiting $|t_1-t_0|\to 0$ we obtain
\begin{align}
\int_{P(M)} [dF^t]\,d\Gamma_x\leq e^{\frac{\kappa}{2}}\int_{P(M)}\Big(|\nabla_t F|^2 + \frac{\kappa}{2}\int_t^T e^{\frac{\kappa}{2}(s-t)}|\nabla_s F|^2\Big)\,d\Gamma_x\, , 
\end{align}
which is precisely the inequality $(R5)$, and proves the result.
\end{proof}
\vspace{.5 cm}

\subsection{Proof that $(R4)\implies (R7)$:}\label{ss:r4_r7}

In this Section we prove the log-Sobolev estimate for the family of twisted Ornstein-Uhlenbeck operators.  As an application, we conclude from this a log-Sobolev estimate on the classical Ornstein-Uhlenbeck operator which are sharper than those currently in the literature. We will also apply these estimates to recover the Bakry-Ledoux log-sobolev estimate for the heat kernel on a space with lower Ricci curvature bounds.  In particular, this will show us that these estimates on path space imply the correct lower Ricci curvature estimate.  Let us begin by deriving the log-Sobolev estimates:

\begin{proof}[Proof that $(R4)\implies (R7)$]
Let $F$ be a $\cF^T$-measurable smooth cylinder function, and let $H^t\equiv (F^2)^t$ be the martingale induced by projecting $F^2$ to the $\cF^t$-measurable functions.  Applying the Ito formula \cite{Kuo_book} to the function $H^t\ln H^t$ gives the formula
\begin{align}\label{e:r4_r7:1}
\int_{P(M)} H^{t_1}\ln H^{t_1}\,d\Gamma_x-\int_{P(M)} H^{t_0}\ln H^{t_0}\,d\Gamma_x = \frac{1}{2}\int_{P(M)} \int_{t_0}^{t_1} (H^t)^{-1}[dH^t]\,d\Gamma_x\, .
\end{align}
Now let us recall that $(R4)$, and in particular its consequence in Theorem \ref{t:pointwise_r4}, gives us the estimate
\begin{align}
\sqrt{[dH^t]}(\gamma)&\leq \int_{P(M)}|\nabla_t F^2|(\gamma_t\circ\sigma)+\int_t^T\frac{\kappa}{2}e^{\frac{\kappa}{2}(s-t)}|\nabla_s F^2|\,d\Gamma_{\gamma(t)}\notag\\
&= 2\int_{P(M)}F\bigg(|\nabla_t F|(\gamma_t\circ\sigma)+\int_t^T\frac{\kappa}{2}e^{\frac{\kappa}{2}(s-t)}|\nabla_s F|\bigg)\,d\Gamma_{\gamma(t)}\, .
\end{align}
Applying the Cauchy-Schwartz and the same computational scheme as in Section \ref{ss:r2_r3} we then have that
\begin{align}
[dH^t](\gamma)&\leq 4\int_{P(M)}F^2\,d\Gamma_{\gamma(t)}\cdot e^{\frac{\kappa}{2}(T-t)}\int_{P(M)}\bigg(|\nabla_t F|^2(\gamma_t\circ\sigma)+\int_t^T\frac{\kappa}{2}e^{\frac{\kappa}{2}(s-t)}|\nabla_s F|^2\bigg)\,d\Gamma_{\gamma(t)}\notag\\
&=4H^t(\gamma)\cdot e^{\frac{\kappa}{2}(T-t)}\int_{P(M)}\bigg(|\nabla_t F|^2(\gamma_t\circ\sigma)+\int_t^T\frac{\kappa}{2}e^{\frac{\kappa}{2}(s-t)}|\nabla_s F|^2\bigg)\,d\Gamma_{\gamma(t)}\, .
\end{align}

Plugging this into (\ref{e:r4_r7:1}) gives the estimate
\begin{align}
\int_{P(M)} H^{t_1}&\ln H^{t_1}\,d\Gamma_x-\int_{P(M)} H^{t_0}\ln H^{t_0}\,d\Gamma_x \leq \notag\\
&\leq 2\int_{P(M)}\bigg(\int_{t_0}^{t_1} e^{\frac{\kappa}{2}(T-t)}\int_{P(M)}\bigg(|\nabla_t F|^2(\gamma_t\circ\sigma)+\int_t^T\frac{\kappa}{2}e^{\frac{\kappa}{2}(s-t)}|\nabla_s F|^2\bigg)\,d\Gamma_{\gamma(t)}\bigg)\,d\Gamma_x\notag\\
&=2e^{\frac{\kappa}{2}T}\int_{P(M)}\bigg(\int_{t_0}^{t_1} e^{-\frac{\kappa}{2}t}|\nabla_t F|^2+\int_{t_0}^{t_1}\int_t^T\frac{\kappa}{2}e^{\frac{\kappa}{2}(s-2t)}|\nabla_s F|^2ds\,dt\bigg)\,d\Gamma_x\notag\\
&= 2e^{\frac{\kappa}{2}T}\int_{P(M)}\bigg(\int_{t_0}^{t_1} e^{-\frac{\kappa}{2}t}|\nabla_t F|^2+\int_{t_0}^{t_1}\int_t^{t_1}\frac{\kappa}{2}e^{\frac{\kappa}{2}(s-2t)}|\nabla_s F|^2+\int_{t_0}^{t_1}\int_{t_1}^{T}\frac{\kappa}{2}e^{\frac{\kappa}{2}(s-2t)}|\nabla_s F|^2\bigg)\,d\Gamma_x \notag\\
&= 2e^{\frac{\kappa}{2}T}\int_{P(M)}\bigg(\int_{t_0}^{t_1} e^{-\frac{\kappa}{2}t}|\nabla_t F|^2+\frac{1}{2}\int_{t_0}^{t_1}\int_t^{t_1}\frac{\kappa}{2}e^{\frac{\kappa}{2}(s-2t)}|\nabla_s F|^2+\int_{t_0}^{t_1}\int_{t_1}^{T}\frac{\kappa}{2}e^{\frac{\kappa}{2}(s-2t)}|\nabla_s F|^2\bigg)\,d\Gamma_x \notag\\
&= 2e^{\frac{\kappa}{2}T}\int_{P(M)}\bigg(\int_{t_0}^{t_1} e^{-\frac{\kappa}{2}t}|\nabla_t F|^2+\int_{t_0}^{t_1}\int_{t_0}^{s}\frac{\kappa}{2}e^{\frac{\kappa}{2}(s-2t)}|\nabla_s F|^2dtds+\int_{t_1}^{T}\int_{t_0}^{t_1}\frac{\kappa}{2}e^{\frac{\kappa}{2}(s-2t)}|\nabla_s F|^2dtds\bigg)\,d\Gamma_x \notag\\
&= 2e^{\frac{\kappa}{2}(T-t_0)}\int_{P(M)}\bigg(\int_{t_0}^{t_1} e^{-\frac{\kappa}{2}(t-t_0)}|\nabla_t F|^2+\int_{t_0}^{t_1}\frac{1}{2}\big(e^{\frac{\kappa}{2}(s-t_0)}-e^{-\frac{\kappa}{2}(s-t_0)}\big)|\nabla_s F|^2\notag\\
&\;\;\;\;\;\;\;\;\;\;\;\;\;\;\;\;\;\;\;\;\;\;\;+\int_{t_1}^{T}\frac{1}{2}e^{\frac{\kappa}{2}(s-t_0)}\big(1-e^{-\kappa(t_1-t_0)}\big)|\nabla_s F|^2\bigg)\,d\Gamma_x \notag\\
&= 2e^{\frac{\kappa}{2}(T-t_0)}\int_{P(M)}\bigg(\int_{t_0}^{t_1} \cosh\big(\frac{\kappa}{2}(t-t_0)\big)|\nabla_t F|^2+\big(1-e^{-\kappa(t_1-t_0)}\big)\int_{t_1}^{T}\frac{1}{2}e^{\frac{\kappa}{2}(s-t_0)}|\nabla_s F|^2\bigg)\,d\Gamma_x\, ,\notag\\
&= 2e^{\frac{\kappa}{2}(T-t_0)} E_{t_0,\kappa}^{t_1}[F,F]\notag\\
&= 2e^{\frac{\kappa}{2}(T-t_0)}\int \langle F, L_{t_0,\kappa}^{t_1}\rangle \, d\Gamma_x\, ,
\end{align}
which proves the estimate.\\
\end{proof}

Before continuing let us remark on the following corollary.  By letting $t_0=0$ and $t_1=T$, and making the observation that the Ornstein-Uhlenbeck operator $L_x$ satisfies the estimate
\begin{align}
L_{0,\kappa}^T \leq \cosh(\frac{\kappa}{2}T)L_x,
\end{align}
we immediately obtain the following log-sobolev estimate on the classical Ornstein-Uhlenbeck operator:
\begin{corollary}\label{c:OU_log_sob}
If $(M^n,g,e^{-f}dv_g)$ is a smooth metric measure space with $|\Ric+\nabla^2 f|\leq \kappa$, then for the standard Ornstein-Uhlenbeck operator we obtain the log-Sobolev estimate
\begin{align}
\int_{P_xM} F^2\ln F^2\, d\Gamma_x\leq \big(e^{\kappa T}+1\big)\int_{P_xM} |\nabla_{H^1} F|^2\, d\Gamma_x\, ,
\end{align}
for every $\cF^T$-measurable $F$ with $\int F^2 = 1$.
\end{corollary}
\vspace{.5 cm}

Let us end this subsection by observing how $(R7)$ implies the classical log-sobolev estimate for functions on $M$ with respect to the heat kernel estimate.  In particular, by theorem \ref{t:lower_ricci} we will then see that $(R7)$ implies the correct lower bound on the Ricci curvature of $M$:

\begin{theorem}\label{t:bounded_ricci_implies_lower_ricci_logsov}
If $(R7)$ holds, then for each smooth $u:M\to\dR$ and $t>0$ we have the estimate
\begin{align}
\int_M u^2\ln u^2 \rho_t(x,dy) \leq 2 \kappa^{-1}\left(e^{\kappa t}-1\right)\int_M |\nabla u|^2 \rho_t(x,dy)\, ,
\end{align}
if $\int_M u^2\,\rho_t=1$.
\end{theorem}
\begin{proof}
The proof follows the same structure as Theorems \ref{t:boundedricci_BE_implies_lowerricci} and \ref{t:boundedricci_BE2_implies_lowerricci}.  Namely, if $F(\gamma)=u(\gamma(t))$ then we have that
\begin{align}
\int_{P(M)} F^2\ln F^2 d\Gamma_x = \int_M u^2\ln u^2 \rho_t(x,dy)\, ,
\end{align}
as well as
\begin{align}
|\nabla_s F| = |\nabla u|(\gamma(t))\, ,
\end{align}
for every $s\leq t$.  Plugging this into $(R7)$ gives
\begin{align}
\int_M u^2\ln u^2 \rho_t(x,dy) &= \int_{P(M)} F^2\ln F^2 d\Gamma_x \notag\\
&\leq 2e^{\frac{\kappa}{2}T}\int_{P(M)}\bigg(\int_0^T \cosh(\frac{\kappa}{2}t)|\nabla_t F|^2\, dt\bigg)\, d\Gamma_x \notag\\
&= 2e^{\frac{\kappa}{2}T}\int_0^T\cosh(\frac{\kappa}{2}t)dt\int_M |\nabla u|^2(y) \rho_t(x,dy)\notag\\
&=2\kappa^{-1}(e^{\kappa T}-1)\int_M |\nabla u|^2(y) \rho_t(x,dy)\, ,
\end{align}
as claimed.
\end{proof}
\vspace{.5 cm}

\subsection{$(R7)\implies (R3)$}\label{ss:r7_r3}

Now let us end this Section by proving the relation $(R7)\implies (R3)$:\\

\begin{proof}[Proof that $(R7)\implies (R3)$]
Let $F$ be a $\cF^T$-measurable smooth cylinder function, and we can assume without loss that $\int F\,d\Gamma_x=0$.  Let $H^t$ be the martingale induced by projecting $F^2$ to the $\cF^t$-measurable functions.  The log-Sobolev of $(R7)$ tells us that we can estimate
\begin{align}
\frac{1}{2}\int_{P(M)} \int_{t_0}^{t_1} &(H^t)^{-1}[dH^t]\,d\Gamma_x = \int_{P(M)} H^{t_1}\ln H^{t_1}\,d\Gamma_x-\int_{P(M)} H^{t_0}\ln H^{t_0}\,d\Gamma_x\, ,\notag\\
&\leq 2e^{\frac{\kappa}{2}(T-t_0)}\int_{P(M)}\bigg(\int_{t_0}^{t_1} \cosh\big(\frac{\kappa}{2}(t-t_0)\big)|\nabla_t F|^2+\big(1-e^{-\kappa(t_1-t_0)}\big)\int_{t_1}^{T}\frac{1}{2}e^{\frac{\kappa}{2}(s-t_0)}|\nabla_s F|^2\bigg)\,d\Gamma_x\, .
\end{align}
In particular, dividing both sides by $|t_1-t_0|$ and limiting $|t_1-t_0|\to 0$ we obtain
\begin{align}
\int_{P(M)} (H^t)^{-1}[dH^t]\,d\Gamma_x\leq 4e^{\frac{\kappa}{2}T}\int_{P(M)}\Big(|\nabla_t F|^2 + \frac{\kappa}{2}\int_t^T e^{\frac{\kappa}{2}(s-t)}|\nabla_s F|^2\Big)\,d\Gamma_x\, .
\end{align}
Applying this to $t=0$ and using \eqref{e:r3_r5:equivalence} applied to $F^2$ gives us the estimate
\begin{align}\label{e:r6_r5:1}
\big(H^0\big)^{-1}\big|\nabla_x \int F^2\,d\Gamma_x\big|^2 = \big(H^0\big)^{-1}[dH^0]\leq 4e^{\frac{\kappa}{2}T}\int_{P(M)}\Big(|\nabla_0 F|^2 + \frac{\kappa}{2}\int_0^T e^{\frac{\kappa}{2}s}|\nabla_s F|^2\Big)\,d\Gamma_x\, .
\end{align}
Now let us choose a family of functions $F_\epsilon$ such that
\begin{align}
&F_\epsilon\equiv 1+\epsilon F+O(\epsilon^2)\, ,\notag\\
&\int F^2_\epsilon\equiv 1\, ,
\end{align}
which is possible because $\int F=0$.  Plugging $F_\epsilon$ into \eqref{e:r6_r5:1} we obtain
\begin{align}
\big|\nabla_x \int 2\epsilon F +O(\epsilon^2)\,d\Gamma_x\big|^2\leq 4\epsilon^2e^{\frac{\kappa}{2}T}\int_{P(M)}\Big(|\nabla_0 F|^2 + \frac{\kappa}{2}\int_0^T e^{\frac{\kappa}{2}s}|\nabla_s F|^2\Big)\,d\Gamma_x\, ,
\end{align}
which if we divide by $\epsilon^2$ and limit gives us
\begin{align}
\big|\nabla_x \int F\,d\Gamma_x\big|^2\leq e^{\frac{\kappa}{2}T}\int_{P(M)}\Big(|\nabla_0 F|^2 + \frac{\kappa}{2}\int_0^T e^{\frac{\kappa}{2}s}|\nabla_s F|^2\Big)\,d\Gamma_x\, ,
\end{align}
which is precisely $(R3)$.

\end{proof}

\vspace{1 cm}
 
\section{Finishing the Proof of Theorem \ref{t:smooth_bounded_ricci}}\label{s:finish_maintheorem}

Throughout the paper, with the goal of proving Theorem \ref{t:smooth_bounded_ricci}, we have shown the implications
\begin{align}
(R1)\implies (R2)\implies (R3)\Leftrightarrow (R5) \Leftrightarrow (R6)\, ,\notag\\
(R1)\implies (R2)\Leftrightarrow (R4)\implies (R7)\implies (R3)\, .
\end{align}
Hence, to finish the proof of Theorem \ref{t:smooth_bounded_ricci} we need the implication
\begin{align}
(R3)\implies (R1)\, .
\end{align}

The main goal of this Section is therefore to prove this implication.  At the end of the proofs of each of the estimates of Theorem \ref{t:smooth_bounded_ricci} we have compared the estimates to the lower Ricci curvature case.  This was done by picking test functions on path space that were particularly simple, and depend on only a single time $t>0$.  As a consequence we have recovered the classical Bakry-Emery-Ledoux estimates, and in particular we have shown that $(R3)$ implies the appropriate lower bound on the Ricci curvature in $(R1)$.  We will see in Section \ref{ss:r3_r1} how to recover the upper bound by picking test functions which depend on only two times, which will finish the proof of Theorem \ref{t:smooth_bounded_ricci}.  

\subsection{Proof that (R3) $\implies$ (R1)}\label{ss:r3_r1}

We wish to close the circle for our characterization of bounded Ricci curvature in this section.  We will describe a family of test functions which will allow us to recover the upper bound on the Ricci curvature from the gradient estimate.  As we have seen, we can recover the Bakry-Emery gradient estimate, and in particular lower Ricci curvature bound, from a cylinder function which depends on just one time.  We will see how to recover the upper Ricci curvature bound from a series of cylinder functions which depend on only two times.  

Let us begin with a little necessary computational background.  Let us fix $x\in M$ with $v\in T_xM$ a unit vector.  As in Lemma \ref{l:lower_ricci:2} let us consider a compactly supported function $u:M\to \dR$ such that
\begin{align}\label{e:R3_R1_test}
&u(x)=0\, ,\notag\\
&\nabla u(x) = v\, ,\notag\\
&\nabla^2 u(x) = 0\, .
\end{align}
Then let us consider the function on path space given by
\begin{align}\label{e:R3_R1_test_2}
F(\gamma) \equiv u(\gamma(0))+ cu(\gamma(t))\, ,
\end{align}
where $t>0$ is arbitrary but fixed and $c\in \dR$ is to be determined later.  Let us compute the following expansion:

\begin{lemma}\label{l:R3_R1_test_expansion}
Let $F:P(M)\to \dR$ be given by $F(\gamma) \equiv u(\gamma(0))+ cu(\gamma(t))$, where $u$ is as in \eqref{e:R3_R1_test}.  Then the following hold:
\begin{enumerate}
\item We have the following expansion:
\begin{align}
|\nabla \int F d\Gamma_x|^2(x) = |1+c|^2+(1+c)c\langle\nabla\Delta_f u(x),v\rangle t+O(t^2)\, .
\end{align}
\item We have the following expansion:
\begin{align}
\int |\nabla_s F|^2 d\Gamma_x = \begin{cases}
 &|1+c|^2+(1+c)c\langle\Delta_f\nabla u(x),v\rangle t+O(t^2)\, ,\text{ if }s=0\, ,\notag\\
 &c^2+O(t)\, ,\text{ if }0<s\leq t\, .
 \end{cases}
\end{align}
\end{enumerate}
\end{lemma}
\begin{proof}
Let us begin by proving $(1)$.  Indeed, for this for notice that
\begin{align}
\int F d\Gamma_x = u(x)+cH_tu(x)\, ,
\end{align}
and thus
\begin{align}
\nabla \int F d\Gamma_x = \nabla u(x)+c\nabla H_t u(x) = (1+c)\nabla u(x)+ \frac{1}{2}c\nabla\Delta_f u(x) t+O(t^2)\, .
\end{align}
Squaring leads to $(1)$.  To prove $(2)$ involves a little more structure, namely, we want to do the computations on the frame bundle.  Let us begin by noting that we can write the $s$-parallel gradient of $F$ by the formula
\begin{align}
\nabla_s F = \begin{cases}
 &\nabla u(x)+cP_t\nabla u(\gamma(t))\text{ if }s=0\, ,\notag\\
 &cP_t\nabla u(\gamma(t))\, ,\text{ if }0<s\leq t\, ,
 \end{cases}
\end{align}
where $P_{t}(\gamma):T_{\gamma(t)}M\to T_xM$ is the stochastic parallel translation map discussed in Section \ref{ss:stoc_par_trans}.  To write this in a computationally more friendly manner we proceed as follows.  Let $FM$ be the frame bundle over $M$ with $H_1,\ldots,H_n$ the canonical horizontal vector fields.  With $x\in M$ let $\tilde x\in F_x M$ be a fixed frame.  Note that $\tilde x$ gives us an isometric identification $\dR^n\equiv T_xM$. 

Now given $\gamma\in P_xM$ let us denote by $\tilde\gamma\in F_{\tilde x}M$ its (stochastic) horizontal lift.  Then we can identify the parallel translations of the gradients by
\begin{align}
\tilde x\circ P_{t}\nabla u_i(\gamma(t)) = H_\alpha u(\tilde\gamma(t))\in \dR^n\, ,
\end{align}
and hence we can rewrite
\begin{align}
\tilde x\circ\nabla_s F = \begin{cases}
 &H_\alpha u(\tilde x)+ c H_\alpha u(\tilde \gamma(t))\text{ if }s=0\, ,\notag\\
 &c H_\alpha u(\tilde \gamma(t))\, ,\text{ if }0<s\leq t\, .
 \end{cases}
\end{align}
In particular, we get that
\begin{align}
|\nabla_s F|^2(\gamma(t)) = \begin{cases}
 &|H_\alpha u|^2(\tilde x)+2c\langle H_\alpha u(\tilde x), H_\alpha u(\tilde\gamma(t))\rangle+ c^2 |H_\alpha u|^2(\tilde \gamma(t))\text{ if }s=0\, ,\notag\\
 &c^2 |H_\alpha u|^2(\tilde \gamma(t))\, ,\text{ if }0<s\leq t\, .
 \end{cases}
 \\ \notag\
\end{align}
Now let $\tilde\rho_t(\tilde x,d\tilde y)$ be the heat kernel on $FM$ with respect to $\frac{1}{2}\Delta_{H,f} = \frac{1}{2}\sum\big( H_\alpha H_\alpha - H_\alpha \tilde f\cdot H_\alpha\big)$.  In particular, if $\pi_{FM}:FM\to M$ is the projection map then we get that $\pi_{FM,*}\tilde\rho_t(\tilde x,d\tilde y) = \rho_t(x,dy)$, and thus if $\tilde \Gamma_{\tilde x}$ is the induced Wiener measure on $FM$, then $\pi_{FM,*}\tilde \Gamma_{\tilde x}=\Gamma_{x}$.  Then we can compute for $s=0$
\begin{align}
\int_{P_xM} &|\nabla_0 F|^2(\gamma(t))\,d\Gamma_x = \int_{P_{\tilde x}FM} |\nabla_0 F|^2(\tilde\gamma(t))\,d\tilde\Gamma_{\tilde x}\notag \\
&= \int_{FM}\Big( |H_\alpha u|^2(\tilde x)+2c\langle H_\alpha u(\tilde x), H_\alpha u(\tilde y)\rangle+ c^2 |H_\alpha u|^2(\tilde y)\Big)\tilde\rho_t(\tilde x,d\tilde y)\, ,\notag \\
&= |1+c|^2|H_\alpha u|^2(\tilde x)+2c\langle H_\alpha u(\tilde x), \frac{1}{2}\Delta_{H,f}H_\alpha u(\tilde x)\rangle t+\frac{1}{2}\Delta_{H,f}|H_\alpha u|^2(\tilde x)+O(t^2)\, ,\notag\\
& = |1+c|^2|\nabla u|^2(x)+\big(c(1+c)\langle \nabla u,\Delta_f\nabla u\rangle(x)+|\nabla^2 u|^2(x)\big)t+O(t^2)\, .
\end{align}
Using that $\nabla^2 u(x)=0$ we have shown $(2)$ for $s=0$.  A verbatim computation for $s>0$ proves the other estimate.
\end{proof}
\vspace{.5 cm}

With this in hand let us prove the implication $(R3)\implies (R1)$, and thus finish the proof of Theorem \ref{t:smooth_bounded_ricci}:

\begin{proof}[Proof of Theorem \ref{t:smooth_bounded_ricci}]

We have seen that $(R3)\implies \Ric+\nabla^2 f\geq -\kappa g$, and thus we need to show that $(R3)\implies \Ric+\nabla^2 f\leq \kappa g$ to close the circle.  Let $x\in M$ with $v\in T_xM$ a unit vector, and let $F:P(M)\to \dR$ be as in \eqref{e:R3_R1_test_2}.  Then $(R3)$ is the estimate
\begin{align}
\big|\nabla_x \int_{P(M)} F\, d\Gamma_x\big|^2 \leq e^{\frac{\kappa}{2}t}\int_{P(M)}\,|\nabla_0 F|^2+\int_0^t \frac{\kappa}{2}e^{\frac{\kappa}{2}s}|\nabla_s F|^2\, ds\cdot d\Gamma_x\, .
\end{align}

Now using Lemma \ref{l:R3_R1_test_expansion} we can expand both sides to get
\begin{align}
|1+c|^2+(1+c)c\langle\nabla\Delta_f u(x),v\rangle t+O(t^2)\leq (1+\frac{\kappa}{2}t)\big(|1+c|^2+(1+c)c\langle\Delta_f\nabla u(x),v\rangle t+O(t^2)\big)+\frac{\kappa}{2}tc^2+ O(t^2)\, ,\notag 
\end{align}
which by collecting terms gives us that
\begin{align}
(1+c)c\Big(\langle\nabla\Delta_f u(x),v\rangle-\langle\Delta_f\nabla u(x),v\rangle\Big)t \leq \frac{\kappa}{2}\Big(1+2c+2c^2\Big)t+O(t^2)\, ,
\end{align}
or that
\begin{align}
-(1+c)c\Big(Rc+\nabla^2 f\Big)(v,v) \leq \frac{\kappa}{2}\Big(1+2c+2c^2\Big)+O(t)\, .
\end{align}
Now let us choose $c=-\frac{1}{2}$ in our computation.  Then we arrive at the estimate
\begin{align}
\frac{1}{4}\Big(Rc+\nabla^2 f\Big)(v,v)\leq \frac{\kappa}{4}+O(t)\, ,
\end{align}
or that
\begin{align}
\Big(Rc+\nabla^2 f\Big)(v,v)\leq \kappa+O(t)\, .
\end{align}
By letting $t\to 0$ we arrive at the result.
\end{proof}
\vspace{1 cm}

\section{$d$-dimensional Ricci Curvature and the Proof of Theorem \ref{t:smooth_bounded_d_ricci}}\label{s:smooth_d_ricci}

In this Section we show how to extend the results of Theorem \ref{t:smooth_bounded_ricci} to the case where the $d$-dimensional Ricci curvature is bounded.  In fact, the proof of Theorem \ref{t:smooth_bounded_d_ricci} is essentially just a combination of Theorem \ref{t:smooth_bounded_ricci} and Theorem \ref{t:lower_d_ricci} once a few observations are made.

\begin{proof}[Proof of Theorem \ref{t:smooth_bounded_d_ricci}]

Let $F$ be a smooth cylinder function given by
\begin{align}
F(\gamma) = u(\gamma(t_1),\ldots,\gamma(t_N))\, .
\end{align}
Let us observe that if $\int_{P(M)} F\,d\Gamma_x$ is induced the function on $M$, then we have the equality
\begin{align}
H_t\int_{P(M)} F\,d\Gamma_x = \int_{P(M)} F_{+t}\,d\Gamma_x\, ,
\end{align}
where $F_{+t}$ is the smooth cylinder function given by
\begin{align}
F_{+t}(\gamma) = u(\gamma(t_1+t),\ldots,\gamma(t_N+t))\, .
\end{align}
Conversely, it is then clear that if $F$ is a smooth cylinder function which is $\cF_t^T$-measurable, then there exists a smooth cylinder function $F_{-t}$ which is $\cF^{T-t}_0$-measurable such that
\begin{align}
\int_{P(M)} F\,d\Gamma_x = H_t\int_{P(M)} F_{-t}\, d\Gamma_x\, .
\end{align}

Now let us assume the $d$-dimensional Ricci curvature bound
\begin{align}
-\kappa g+ \frac{1}{d-n}\nabla f\otimes \nabla f\leq \Ric+\nabla^2f\leq \kappa g\, .
\end{align}
In particular the lower bound gives us that Theorem \ref{t:lower_d_ricci} holds and the bound gives us that Theorem \ref{t:smooth_bounded_ricci} holds.  To prove Theorem \ref{t:smooth_bounded_d_ricci}.2 let $F$ be a $\cF^T_t$-measurable function and let us apply Theorem \ref{t:lower_d_ricci}.2 to the function $\int F_{-t}\,d\Gamma_x$ at time $t$ to get the inequality
\begin{align}
&|\nabla H_t \int F_{-t}\,d\Gamma_x|^2+\frac{e^{\kappa t}-1}{d\kappa}\big|\Delta_f H_t \int F_{-t}\,d\Gamma_x\big|^2 \leq e^{\kappa t}H_t|\nabla \int F_{-t}\,d\Gamma_x|^2\, \notag\\
&\implies |\nabla \int F\,d\Gamma_x|^2+\frac{e^{\kappa t}-1}{d\kappa}\big|\Delta_f \int F\,d\Gamma_x\big|^2 \leq e^{\kappa t}\bigg(e^{\frac{\kappa}{2}(T-t)}\int_{P(M)}|\nabla_0 F_{-t}|^2+\int_0^{T-t}\frac{\kappa}{2}e^{\frac{\kappa}{2}s}|\nabla_s F_{-t}|^2\,d\Gamma_x\bigg)\notag\\
&=e^{\kappa t}\bigg(e^{\frac{\kappa}{2}(T-t)}\int_{P(M)}|\nabla_0 F|^2+\int_t^{T}\frac{\kappa}{2}e^{\frac{\kappa}{2}(s-t)}|\nabla_s F|^2\,d\Gamma_x\bigg)\notag\\
&=e^{\frac{\kappa}{2}T}\int_{P(M)}e^{\frac{\kappa}{2}t}|\nabla_0 F|^2+\int_t^{T}\frac{\kappa}{2}e^{\frac{\kappa}{2}s}|\nabla_s F|^2\,d\Gamma_x\notag\\
&=e^{\frac{\kappa}{2}T}\int_{P(M)}|\nabla_0 F|^2+\int_0^{T}\frac{\kappa}{2}e^{\frac{\kappa}{2}s}|\nabla_s F|^2\,d\Gamma_x\, ,
\end{align}
as claimed.  To prove $(3)$ from $(2)$ we proceed as in the proof of $(R3)\implies (R5)$.  To prove $(4)$ we proceed as above, but use Theorem \ref{t:lower_d_ricci}.3 and the techniques of $(R5)\implies (R6)$.

To prove the converse direction let us assume that Theorem \ref{t:smooth_bounded_d_ricci}.2 holds, and see from this that $-\kappa g+ \frac{1}{d-n}\nabla f\otimes \nabla f\leq \Ric+\nabla^2f\leq \kappa g$.  The other implications are proved in a similar fashion.  Now if Theorem \ref{t:smooth_bounded_d_ricci}.2 holds, then in particular so does (R3).  It follows in particular from Theorem \ref{t:smooth_bounded_ricci} that $\Ric+\nabla^2f\leq \kappa g$.  Also by applying Theorem \ref{t:smooth_bounded_d_ricci}.2 to the function $F(\gamma)\equiv u(t)$, we see in a manner similar to Theorem \ref{t:boundedricci_BE_implies_lowerricci} that Theorem \ref{t:lower_d_ricci}.2 holds, and hence we have the lower bound $-\kappa g+ \frac{1}{d-n}\nabla f\otimes \nabla f\leq \Ric+\nabla^2f$, as claimed.

\end{proof}

\newpage

\part{The Case of Nonsmooth Metric-Measure Spaces}\label{part:nonsmooth}

In this part of the paper we focus on metric measure spaces $(X,d,m)$ and we make the basic assumptions

\begin{align}\label{e:mms_assumptions}
(X,&d,m) \text{ is a locally compact, complete length space such that}\notag\\ 
&\text{$m$ is a locally finite, $\sigma$-finite Borel measure with supp}\,m=X\, .
\end{align}

The primary objective of this part of the paper is to provide the necessary tools so that we can use Theorem \ref{t:smooth_bounded_ricci} in order to define the notion of bounded Ricci curvature on a metric measure space.  We will then spend the rest of this part of the paper analyzing the properties of such spaces.

In Section \ref{s:prelim_nonsmooth} we remark on some preliminaries.  Most of the notions in Section \ref{s:prelim_nonsmooth} have appeared elsewhere in one form or another, even if not so systematically or in the same context.  In Section \ref{s:lowerricci_nonsmooth} we recall the notion of a lower Ricci curvature bound for metric measure spaces.  In Section \ref{s:variation} we introduce the notion of a variation of a curve, and discuss many of their properties.  In particular we define in Section \ref{ss:parallel_variation_rect_curv} the notion of a parallel variation.  In Section \ref{s:parallel_gradient_nonsmooth} we use these notions in order to construct the parallel gradients of functions on path space, and then we spend some time discussing their properties and the properties of the associated energy functions.  

In Section \ref{s:bounded_ricci_mms} we use all of this to give the formal definition of a metric measure space with bounded Ricci curvature.  We begin by proving some basic structure on such spaces, and in particular Theorem \ref{t:boundedricci_implies_BE} and Theorem \ref{t:boundedricci_implies_LVS} that such spaces have lower Ricci curvature bounds.  

Section \ref{s:ricci_martingales} is dedicated to studying the relationship between bounded Ricci curvature and martingales in the nonsmooth case.  In particular Section \ref{ss:martingale_lowerricci} is dedicated to proving that spaces with lower Ricci curvature bounds have the continuous martingale property, and Section \ref{ss:martingale_holder} is dedicated to proving that spaces with bounded Ricci curvature have the H\"older martingale property.

In Section \ref{s:bounded_ricci_analysis} we use the structure of Section \ref{s:parallel_gradient_nonsmooth} to introduce the Ornstein-Uhlenbeck operator on the path space of a metric-measure space and prove Theorem \ref{t:OU_boundedricci} on the properties of this operator.

\section{Preliminaries on Nonsmooth Metric-Measure Spaces}\label{s:prelim_nonsmooth}

Let us record here a variety of basic notation and structure which will be used frequently.      Section \ref{ss:basic_notation} is dedicated to basic notation which is relatively commonplace.  In Section \ref{ss:weakly_riemannian} we introduce and discuss a little the notion of a weakly Riemannian and almost Riemannian space. The terminology is not completely standard, and some mild variations appear elsewhere in the literature.  For our purposes we will see that being weakly Riemannian is the minimum structure on a metric-measure space needed to make sense of a Wiener measure.

\subsection{Basic Notation}\label{ss:basic_notation}

The structure in this Section is all either common or slight adaptations of common notation.

\subsubsection{The $\Delta$-Simplex}\label{sss:simplex}  In the first part of the paper we used commonly partitions of intervals when discussing the evaluation maps and cylinder functions.  We discuss the collection of partitions in more detail here, since this structure will be important in the second part of the paper.  Let us begin by describing the standard simplex.  Let us denote by 
\begin{align}
\Delta^N[0,T]\equiv \{\bt=(t_1,\ldots,t_N):0\leq t_1\leq\cdots\leq t_N\leq T\}\, ,
\end{align}
the $N$-simplex of partitions of the interval $[0,T]$, and by
\begin{align}
\Delta[0,T] \equiv \bigcup \Delta^N[0,T]\, ,
\end{align}
the collection of all partitions of $[0,T]$, where $0<T\leq\infty$.  For an arbitrary partition $\bt\subseteq [0,T]$ we denote by $|\bt|$ the length of the partition.  That is, we say $|\bt|=N$ if $\bt\in \Delta^N[0,T]$.  

A particularly important structure on $\Delta[0,T]$ that will play a role is that it is a directed set.  Namely, we have a partial ordering on $\Delta[0,T]$ given by ${\bf s}\leq \bt$ iff $s_a\in \bt$ for each $s_a\in {\bf s}$, and further given any two partitions ${\bf s},\bt\in \Delta[0,T]$ there always exists a third partition ${\bf r}\in \Delta[0,T]$ such that ${\bf s}\leq {\bf r}$ and $\bt\leq {\bf r}$.  We call a function
\begin{align}
f:\Delta[0,T]\to \dR\, ,
\end{align}
a $\Delta$-net.  Since $\Delta[0,T]$ is a direct set it makes sense to ask if such a function $f$ has a limit.  If so we denote by
\begin{align}
\lim_{\bt\to \Delta} f(\bt)\, ,
\end{align}
the limit of $f$.  Similarly since $\Delta[0,\infty)$ is a directed set we may for every $\Delta$-net $f$ consider $\limsup_{\bt\to \Delta} f(\bt)$ and $\liminf_{\bt\to \Delta} f(\bt)$.


\subsubsection{Cylinder Functions}\label{sss:cylinder_function} As in most cases when one does analysis it is important to have a collection of well behaved functions which are dense in the various topologies.  In Section \ref{s:smooth_bounded_ricci_intro} we had described the collection of smooth cylinder functions on the path space of a smooth manifold.  On a general metric space this collection needs to be replaced by a similar but more appropriate collection.  Thus recall for each partition $\bt\in \Delta[0,T]$ that we have the associated evaluation map 
\begin{align}
e_\bt:P(X)\to X^{|\bt|}\, ,
\end{align}
given by
\begin{align}
e_\bt(\gamma)\equiv \big(\gamma(t_1),\ldots,\gamma(t_{|\bt|})\big)\, .
\end{align}

As in Part \ref{part:smooth} we denote by $\cF^T_t$ the bi-family of $\sigma$-algebras on $P(X)$ generated by the mappings $e_\bt$ with $\bt$ a partition of $[t,T]$.  Now we consider the collection of cylinder functions associated to the evaluation maps given by

\begin{align}
\Cyl(X)\equiv \bigg\{F:P(X)\to\dR:\, \exists\,\bt\in\Delta[0,\infty) \text{ and } u\in Lip(X^{|\bt|}) \text{ with } F\equiv e_\bt^*u\bigg\}\, ,
\end{align}
where $Lip(X^{|\bt|})$ is the space the Lipschitz functions with compact support.  

Note that the cylinder functions $\Cyl(X)\subseteq C^0(P(X))$ are continuous functions, and further it can be checked without too much difficulty that they form a subalgebra of $C^0(P(X))$.  It is clear that the cylinder functions also define continuous functions on based path spaces $P_x(X)$, and that a cylinder function $F\in\Cyl(X)$ is $\cF^T_t$-measurable if and only if we can write $F=e_\bt^*u$, where $\bt$ is a partition of $[t,T]$.

\subsection{Weakly Riemannian and Almost Riemannian Metric-Measure Spaces}\label{ss:weakly_riemannian}

We introduce in this Section two types of metric-measure spaces which are especially important to study.  The first are the weakly Riemannian spaces.  In short, we will see that these are precisely the metric-measure spaces whose laplace operator is a linear operator.  The second class we will introduce are the almost Riemannian spaces.  These are weakly Riemannian spaces whose energy structure and metric structure agree, see Section \ref{sss:almost_riemannian_spaces} for more.  

\subsubsection{The Cheeger Energy}\label{sss:cheeger_energy}
We begin in this subsection by recalling the Cheeger energy of a metric measure space $(X,d,m)$.  We use this to define when such a metric-measure space is weakly Riemannian, and prove some basic properties about such spaces.  We will study the diffusion measures on general metric-measure spaces, and in particular we will see in Section \ref{ss:Diff_Meas_weaklyRiemannian} that there exists diffusion measures on path space $P(X)$ if and only if $X$ is weakly Riemannian.  The basic sources which are most relevant for this section are \cite{Cheeger_DiffLipFun} ,\cite{Ambrosio_Calculus_Ricci},\cite{Fukushima_DirichletForms}.

As always we let $(X,d,m)$ be a metric-measure space which satisfies (\ref{e:mms_assumptions_intro}).  Following \cite{Cheeger_DiffLipFun} we define an upper gradient for a function by the following:

\begin{definition}\label{d:cheeger_gradient}
We define
\begin{enumerate}
\item Let $u,G:X\to\dR$ be Borel functions with $G$ bounded and nonnegative.  Then we say $G$ is an upper gradient for $u$ if we have that $|u(x)-u(y)|\leq \int_\gamma G|\dot\gamma|dt$ for all rectifiable curves connecting $x$ and $y$.
\item Given $u,G\in L^2(X,m)$ with $G$ nonnegative, we say that $G$ is a weak upper gradient if there exists a sequence $u_i,G_i$ with $u_i\to f$ in $L^2(X,m)$ and $G_i \rightharpoonup G$ weakly in $L^2(X)$.
\item Given $u\in L^2(X,m)$ we define its Cheeger gradient $|\nabla u|$ to be the (unique) weak upper gradient of $u$ with minimal $L^2$-norm. 
\end{enumerate}
\end{definition}

It is a consequence of \cite{Cheeger_DiffLipFun} that $(3)$ is well defined above.  Now given a function $u\in L^2(X)$ we define its Cheeger energy by the formula
\begin{align}
E_X[u]\equiv \frac{1}{2}\int_X |\nabla u|^2 dm_X\, ,
\end{align}

The fundamental result for our purposes is the following:

\begin{theorem}[\cite{Cheeger_DiffLipFun}]\label{t:classic_dirichlet_form}
The energy function $E_X:\cD(E_X)\subseteq L^2(X,m)\to \dR$ is convex, nonnegative, $2$-homogeneous and lower-semicontinuous.  Furthermore, the following hold:
\begin{enumerate}
\item (closed) The functional $||u||_1\equiv \sqrt{||u||_{L^2}+E_X[u]}$ defines a complete norm on $\cD(E)$.
\item (regular) The continuous functions with compact support $C_c(X)\subseteq \cD(E)$ form a dense subset of $\cD(E)$.
\item (strongly local) If $u,w\in \cD(E)$ are such that $g$ is a constant on $supp(u)\subseteq X$, then $E[u+w] = E[u]+E[w]$.
\end{enumerate}
\end{theorem}

The above allows one to apply standard techniques and ideas from the theory of convex functionals on Hilbert spaces to deduce the existence of densely defined mapping $\Delta_X\equiv \nabla E:\cD(\nabla E)\subseteq L^2(X)\to L^2(X)$ such that at each point $u\in \cD(\nabla E)$ of the domain, $\nabla E(u)$ is the unique element of $L^2(X)$ with minimal norm which satisfies the functional inequality
\begin{align}
E(u)+\langle\nabla E(u),v-u\rangle\leq E(v)\, ,
\end{align}
for each $v\in L^2(X)$.  Of course, where $E$ is differentiable we have that $\nabla E$ simply corresponds to the gradient.  Further we have for each $t> 0$ the induced gradient flow of $\frac{1}{2}\nabla E$ given by
\begin{align}
H_t:L^2(X)\to L^2(X)\, .
\end{align}
See \cite{Fukushima_DirichletForms} for a more complete introduction, but note in particular that $H_tu\to u$ as $t\to 0$ for each $u\in L^2(X)$, and $||H_t||\leq 1$ is a contraction mapping.

\subsubsection{Weakly Riemannian Spaces}\label{sss:weakly_riemannian}

Classically, we identify $\cD(E)$ with the Sobolev space $W^{1,2}(X)$.  If we are dealing with a smooth metric-measure space then it is well known that $W^{1,2}(X)$ is a Hilbert space.  In general, this can fail and $\cD(E)$ may only be a Banach space.  That is, the parallelogram law 
\begin{align}
2E[u]+2E[w] = E[u+w] + E[u-w]\, ,
\end{align}
may fail.  Equivalently, the laplace operator $\Delta_X$ defined in the previous section is not linear.  This brings us to the notion of a weakly Riemannian space:

\begin{definition}\label{d:weakly_riemannian}
We say a metric-measure space $(X,d,m)$ satisfying (\ref{e:mms_assumptions}) is {\it weakly Riemannian} if one of the following equivalent conditions is satisfied:
\begin{enumerate}
\item $(\cD(E),||\cdot||_1)$ is a Hilbert space.
\item For each $u,w\in \cD(E)$ we have the identity $2E[u]+2E[w] = E[u+w]+E[u-w]$.
\item $\Delta_{X}$ is a self-adjoint linear mapping.
\item $H_t$ are linear contractions.
\end{enumerate}
\end{definition}

In the case where $E_X$ satisfies the parallelogram law we can write 
\begin{align}
E_X[u,w]\equiv \frac{1}{2}\bigg(E_X[u+w]-E_X[u-w]\bigg)\, ,
\end{align}
and we see that $E_X[u,w]$ is a closed bilinear form with $E_X[u,u]=E_X[u]$.  In the rest of this subsection we will assume $(X,d,m)$ is weakly Riemannian and discuss ideas from \cite{Ambrosio_Ricci}.  

In particular, we wish to understand the energy measure $[u]$ from \cite{Fukushima_DirichletForms} and its relationship to the cheeger energy.  Beginning with the definition, we have for $u\in W^{1,2}(X,m)$ the measure $[u]$ defined in \cite{Fukushima_DirichletForms} by
\begin{align}
[u](\phi) \equiv 2E_X[u,u\phi]-E_X[u^2,\phi]\, .
\end{align}
Apriori the above is only well defined for sufficiently regular $\phi$, but it is seen in \cite{Fukushima_DirichletForms} it extends to a measure.  In comparison to the smooth manifold case one would hope for the equality $[u]\equiv |\nabla u|^2m$, and in particular that $[u]$ can be identified with an $L^1$ function.  It is an important result of \cite{Ambrosio_Ricci} that this can in fact be done, giving us the following.

\begin{theorem}[\cite{Ambrosio_Ricci}]\label{t:energy_measure}
Let $X$ be a weakly Riemannian space with $u,w\in W^{1,2}(X,m)$, then it holds that $[u]=|\nabla u|^2m$.
\end{theorem}

Let us end this subsection by remarking that using \cite{Fukushima_DirichletForms} one can also define the energy measure $[u]$ through the heat flow or laplace operator by
\begin{align}
[u]\equiv \frac{1}{2}\big(\Delta_X u^2-2u\Delta_X u\big)\, .
\end{align}

\subsubsection{Almost Riemannian Metric-Measure Spaces}\label{sss:almost_riemannian_spaces}


To motivate the definition of an almost Riemannian metric-measure space let us first illustrate a particular degeneracy which may occur with an example\footnote{The author is in debt to Luigi Ambrosio to many useful conversations on this issue.}:  

\begin{example}
Take $(X,d)\equiv \dR^n$ to be the standard geometry on $\dR^n$, and let $m\equiv\sum 2^{-j}\delta_{q_j}$ be the probability measure obtained given an enumeration $\{q_j\}$ of the rationals and their associated dirac-delta measures $\delta_{q_j}$.  It is trivially clear that as a metric space $X$ is a length space and satisfies any other criteria of 'nice' as a metric space.  However, it is also not difficult to check that given any lipschitz function $f:\dR^n\to \dR$ the cheeger gradient $|\nabla f|\equiv 0$ is identically zero.  In particular, the metric measure space $(X,d,m)$ is weakly Riemannian, and even satisfies the Bakry-Emery criteria $|\nabla H_t u|\leq H_t|\nabla u|$ for nonnegative Ricci curvature.  On the other hand, $(X,d,m)$ does not satisfy the criteria for a lower Ricci curvature bound in the sense of Lott-Villani-Sturm.
\end{example}

The above example illustrates that for a given metric-measure space it is possible that the energy function $E_X$ is not compatible with the underlying geometry of the space.  To make this more precise we define:

\begin{definition}\label{d:distance_energy}
We define the energy distance on $X$ by
\begin{align}
d_{E}(x,y)\equiv \sup\{\big|u(x)-u(y)\big|:u\in C^0(X)\text{ with }|\nabla u|\leq 1 \text{ a.e.}\}\, .
\end{align}
\end{definition}

Of course on a smooth metric-measure space it is standard that the above distance function agrees with the underlying distance function.  On a general metric-measure space this may not be the case.  The following is a relatively simple and follows from just playing with the definitions.

\begin{theorem}\label{t:almost_riemannian_space}
Let $(X,d,m)$ satisfy (\ref{e:mms_assumptions}), then the following are all equivalent:
\begin{enumerate}
\item The energy distance function agrees with the standard distance function on $X$, that is, $d_E(x,y)=d(x,y)$.
\item A function $f\in W^{1,2}(X)$ satisfies $|\nabla f|\leq 1$ a.e. iff $Lip f\leq 1$.
\item For a lipschitz function $u$ we have for $a.e. x\in X$ the equality
\begin{align}
|\nabla u|(x)=|\text{Lip}\, u|(x)\equiv \limsup_{y\to x}\frac{|u(y)-u(x)|}{d(x,y)}\, .
\end{align}
\end{enumerate}
\end{theorem}

We therefore end up with the following definition of an almost Riemannian space:

\begin{definition}\label{d:almost_riemannian}
We call a weakly Riemannian metric-measure space $(X,d,m)$ an almost Riemannian space if any of the equivalent conditions of Theorem \ref{t:almost_riemannian_space} hold.
\end{definition}

This is a well studied notion, and we end this Section with some examples.  We begin with the following, which is a result of \cite{Cheeger_DiffLipFun}:

\begin{example}[Doubling+Poincare]\label{e:almost_riemannian_doubling}
Let $(X,d,m)$ be a weakly Riemannian metric-measure space such that $m$ satisfies a doubling condition 
\begin{align}
m(B_{2r}(x))\leq C\, m(B_r(x))\, ,
\end{align}
and a local weak Poincare
\begin{align}
\int_{B_{r}(x)}|u-\int_{B_r}u|\,dm\leq Cr^{-2}\int_{B_{2r}(x)}|\nabla u|^2\, dm\, .
\end{align}
Then $(X,d,m)$ is an almost Riemannian metric-measure space.  See \cite{Cheeger_DiffLipFun}.
\end{example}

We also have the following, which is a result of \cite{Ambrosio_BE_vs_LVS}:

\begin{example}[Dirichlet Forms]\label{e:almost_riemannian_dirichlet}
Let $(X,m)$ be a Polish measure space with $\text{supp}\,m=X$, and let $E$ be a regular strongly local dirichlet form on $L^2(X,m)$.  Let us also assume that the induced distance
\begin{align}
d_{E}(x,y)\equiv \sup\{\big|u(x)-u(y)\big|:u\in C^0(X)\text{ with }[u]\leq m \text{ a.e.}\}\, ,
\end{align}
where $[u]$ is the energy measure defined by $[u](f)\equiv 2E(u,uf)-E(u^2,f)$, induces the same topology on $X$.  Then the triple $(X,d_E,m)$ is an almost Riemannian space if and only if $E$ satisfies the additional upper semicontinuity property
\begin{align}
&\text{For every }f\in \cD(E)\text{ }\exists\, f_j\in\cD(E)\cap C(X)\text{ and upper semicontinuous } g_j:X\to\dR \\ 
&\text{such that }[f_j]\leq g_j^2\, m\, ,\, f_j\to f\text{ in }L^2(X,m)\text{, and }  \limsup \int_X g_n^2\,dm\leq E(f,f)\, .
\end{align}
\end{example}

Now we end with a final example, which is an application of the above and the standard properties of the cheeger energy:

\begin{example}[Nondegenerate Weakly Riemannian Space]\label{e:almost_riemannian_nondegenerate}
Let $(X,d,m)$ be a weakly Riemannian space, and let us assume that the energy distance $d_E$ from Definition \ref{d:distance_energy} is nondegenerate in that it induces the same topology on $X$.  Then the triple $(X,d_E,m)$ is an almost Riemannian space.
\end{example}

\subsection{Diffusion Measures on Weakly Riemannian Spaces}\label{ss:Diff_Meas_weaklyRiemannian}

In the previous Section we saw that a weakly Riemannian metric space is one for which the heat flow map $H_t$ is linear.  In fact, once it is known that the energy functional $E_X$ is quadratic, then the content of Theorem \ref{t:classic_dirichlet_form} is that $E_X$ defines a regular, strongly local Dirichlet form on $L^2(X,m)$.  We can obtain from this a good deal more information than just linearity of the heat flow.  

To begin with, associated with the heat flow has a kernel \cite{Fukushima_DirichletForms}.  More precisely, for each $x\in X$ and $t>0$ there exists a measure $\rho_t(x,dy)$ with the property that for every continuous function $u\in C^0(X)$ we have that
\begin{align}
H_tu(x)=\int_{X} u(y)\rho_t(x,dy)\, .
\end{align}
The kernel may be viewed as a function $\rho_t:X\times \cB(X)\to \dR^+$ is such that $\rho_t(x,\cdot)$ is a measure for each $x\in X$ and $\rho_t(\cdot,U)$ is a measurable function for each Borel set $U\in \cB(X)$.  Note that using this we can extend the heat flow to a contraction mapping $H_t:C^0(X)\to C^0(X)$ on the bounded continuous functions, that is $||H_t||_{C^0}\leq 1$.

Using the theory of Dirichlet forms these ideas may be pushed further.  As in Section \ref{ss:diffusion_measures} one would like to build the diffusion measures on $P(X)$.  Namely, for each measure $\mu$ on $X$ we would like there to be a corresponding measure $\Gamma_\mu$ such that for every partition $\bt\in \Delta[0,\infty)$ we have that the pushforward measure $e_{\bt,*}\Gamma_\mu$ on $X^{|\bt|}$ is given by
\begin{align}\label{e:diffusion_measure}
e_{\bt,*}\Gamma_\mu = \int_M \rho_{t_1}(x,dy_1)\rho_{t_2-t_1}(y_1,dy_2)\cdots\rho_{t_k-t_{k-1}}(y_{k-1},dy_k) d\mu(x)\, .
\end{align}
Using \cite{Fukushima_DirichletForms} and Theorem \ref{t:classic_dirichlet_form} we see that such a measure does exist, and given that the evaluation maps $e_\bt$ generate the standard $\sigma$-algebra on $P(X)$ it is clear that it is unique.  Conversely, if such measures exist then for each $x\in X$ we can consider the diffusion measure $\Gamma_x\equiv \Gamma_{\delta_x}$.  Using (\ref{e:diffusion_measure}) we see that there exists a kernel for $H_t$, and in particular that $H_t$ is linear.  Hence, in this case we have that $X$ is weakly Riemannian.  Summarizing we have the following

\begin{theorem}
Let $(X,d,m)$ satisfy (\ref{e:mms_assumptions}), then there exists for each Borel measure $\mu$ on $X$ a diffusion measure $\Gamma_\mu$ on $P(X)$ satisfying (\ref{e:diffusion_measure}) if and only if $X$ is weakly Riemannian.
\end{theorem}

\subsubsection{Stochastic Completeness}\label{sss:stochastic_completeness} We end this Section by having a brief discussion of stochastic completeness.  To describe this let us denote by
\begin{align}
X^*\equiv X\cup\{*\}\, ,
\end{align}
the one point compactification of $X$.  In the case where $X$ is already compact we simply let $X^*\equiv X$.  The point $*$ is often referred to as the cemetery in the Dirichlet form literature.  In the general case, even for a complete smooth manifold, if $\mu$ is a probability measure on $X$ then the diffusion measure $\Gamma_\mu$ need not be a probability measure on $P(X)$.  However, it turns out that $\Gamma_\mu$ extends uniquely to a probability measure on $P(X^*)$ such that for $\Gamma_\mu$-$a.e.$ $\gamma\in P(X^*)$ we have that if $\gamma(t)=*$ for some $t\geq 0$, then $\gamma(t')=*$ for all $t'\geq t$.

Using this it is natural to define for $\gamma\in P(X^*)$ the hitting map
\begin{align}
T^*(\gamma)\equiv \min\big\{t>0:\gamma(t)\in *\big\}\, .
\end{align}
The mapping $T^*$ is often referred to as the lifetime of $\gamma$.  

The following characterizes the notion of stochastic completeness, it tells us when the diffusion measures $\Gamma_\mu$ are probability measures on $P(X)$:

\begin{definition}\label{d:stochastic_completeness}
We say that $X$ is stochastically complete if one of the following equivalent conditions is satisfied:
\begin{enumerate}
\item $H_t1=1$ $a.e.$
\item $\rho_t(x,\cdot)$ is a probability measure on $X$ for each $t>0$ and $x\in X$.
\item For each measure $\mu$ on $X$ we have that $T^*(\gamma)=\infty$ for $\Gamma_\mu$-a.e. $\gamma\in P(X^*)$.
\item For each probability measure $\mu$ on $X$ we have that $\Gamma_\mu$ a.e. $\gamma\in P(X^*)$ satisfies $\gamma\in P(X)$. 
\end{enumerate}
\end{definition}
\begin{remark}
One also says the Dirichlet form $E_X$ is conservative if any of the above conditions is satisfied.
\end{remark}

The metric-measure space $X$ need not be stochastically complete, even for a complete Riemannian manifold.  It is known on a complete Riemannian manifold with Ricci curvature bounded from below that $X$ is stochastically complete, and similarly in Section \ref{ss:stochastic_completeness} we will see that a metric-measure space with bounded Ricci curvature is stochastically complete.

\section{Lower Ricci Curvature on Metric-Measure Spaces}\label{s:lowerricci_nonsmooth}

As in the smooth case we give a brief introduction to lower Ricci curvature in the context of nonsmooth metric-measure spaces.  Primarily, this gives us an excuse to introduce some terminology and notation which will be used later.  There are many possible notions of lower Ricci curvature bounds in the metric-measure setting.  The three that will play the most important role for us are the Bakry-Emery conditions introduced in \cite{BakryEmery_diffusions}, the curvature dimension $CD(n,\kappa)$ condition introduced in \cite{LV_OptimalRicci}, \cite{Sturm_GeomMetricMeasSpace}, \cite{Sturm_KE_LVS}, and the Riemannian curvature dimension $RCD(n,\kappa)$ condition introduced in \cite{Ambrosio_Ricci},\cite{Sturm_KE_LVS}.

The Bakry-Emery criteria for a lower Ricci curvature bound can be summarized with the results of Theorem \ref{t:boundedricci_implies_BE} and Theorem \ref{t:lower_implies_bounded_d}.  The one observation we make is that to truthfully equate this with the notion of a lower Ricci curvature bound as introduced in \cite{BakryEmery_diffusions} one needs Theorem \ref{t:energy_measure}, namely that the energy measure $[u]\equiv \frac{1}{2}\big(\Delta_X u^2 -2u\Delta_X u\big)$ can be identified with the measure $|\nabla u|^2 m$.

To discuss the curvature dimension or Riemannian curvature dimension criteria for a lower Ricci curvature bound we recall first the Wasserstein distance and the space of probability measures.  Recall that $\cP_2(X)$ denotes the space of probability measures on $X$ with finite second moments.  That is, $\mu\in \cP_2(X)$ if $\mu$ satisfies $\int_X d^2(x_0,y)\,d\mu(y)<\infty$.  On the space $\cP_2(X)$ we denote the Wasserstein distance by
\begin{align}
W_2(\mu,\nu)\equiv \inf_{\pi} \int_{X\times X} d^2(x,y)\,d\pi\, ,
\end{align}
where the infimum is over all probability measures $\pi$ on $X\times X$ whose marginals are $\mu$ and $\nu$.  There are many other characterizations of the Wasserstein distance, but we will not discuss them here.  Recall that since $X$ is a separable complete length space, so is $\cP_2(X)$.

The notion of a lower Ricci curvature bound is now tied in with the entropy functional defined by
\begin{align}
\Ent_m(\rho m)\equiv \int_X \rho\ln\rho \,dm\, 
\end{align}
on measures $\rho m$ which are absolutely continuous with respect to $m$, and $Ent_m\equiv \infty$ otherwise.  To understand the connection between Ricci curvature and the entropy functional recall that a real valued function $u:I\to \dR$ defined on an interval $I\subseteq \dR$ is called $\kappa$-convex if $u''\geq \kappa$.  Following \cite{Sturm_KE_LVS} we also call the function $(d,\kappa)$-convex if $u''\geq \kappa+\frac{1}{d}(u')^2$.  Similarly, given a function $u$ on a length space we call $u$ weakly $(d,\kappa)$-convex if for any two points there exists {\it some} unit speed minimizing geodesic $\gamma(t)$ connecting the points such that $u(\gamma(t))$ is $(d,\kappa)$-convex.  We call $u$ strongly $(d,\kappa)$-convex if for {\it every} minimizing geodesic $\gamma(t)$ we have that $u(\gamma(t))$ is $(d,\kappa)$-convex.  Now following \cite{LV_OptimalRicci},\cite{Sturm_GeomMetricMeasSpace}, \cite{Ambrosio_Ricci},\cite{Sturm_KE_LVS} we define the following:
\begin{definition}
Given a metric-measure space $(X,d,m)$ satisfying (\ref{e:mms_assumptions}) we say:
\begin{enumerate}
\item $X$ satisfies the curvature dimension $CD(d,\kappa)$ criteria if $Ent_m$ is weakly $(d,\kappa)$-convex on $P_2(X)$ with respect to the Wasserstein geometry.
\item $X$ satisfies the Riemannian curvature dimension $RCD(d,\kappa)$ criteria if $X$ is weakly Riemannian and $Ent_m$ is strongly $(d,\kappa)$-convex on $P_2(X)$ with respect to the Wasserstein geometry.
\end{enumerate}
\end{definition}
\begin{remark}
The criteria $CD(d,\kappa)$ was defined slightly differently in \cite{Sturm_GeomMetricMeasSpace}, and strictly speaking the notion defined is called the entropic curvature dimension condition in \cite{Sturm_KE_LVS}.  In this paper we will primarily be interested in the condition $RCD(d,\kappa)$, which is stronger than any of the other notions anyway.  See \cite{Ambrosio_BE_vs_LVS}, \cite{Sturm_KE_LVS} for more on that.
\end{remark}

\section{Variations of a Curve}\label{s:variation}

In this section we consider a complete metric space $(X,d)$ and are interested in finding replacements for the notion of a vector field along a continuous curve $\gamma$, we will call these objects variations of $\gamma$.  Once these are introduced and some basic structure is proven, we will define the notion of a parallel variation, which will of course take the place of a parallel translation invariant vector field along a curve.

We begin in Section \ref{ss:variation_point} by discussing some preliminaries, and in particular we will consider the space of point variations $\Sigma X$ on $X$.  In essence this is nothing more than the space of Cauchy sequences on $X$, however it will be useful to consider and describe a variety of structure on this space.  In particular, we will describe an equivalence relation on the space which will be particularly important later when describing variations of curves.  On a general metric space we will view $\Sigma X$ as a replacement for a tangent space.

In Section \ref{ss:variation_curve} we extend the notion of a variation of a point to a variation of a curve.  These variations have apriori little or no regularity, and can be viewed in the smooth case as corresponding, up to equivalence, to measurable vector fields along a curve.  We will again consider an equivalence classes of such variations, a point which will be important for the regularity theory of variations as well as for comparisons in the smooth case and seeing that up to equivalence the parallel variations may be identified with the parallel translation invariant vector fields along a nice curve.

In Section \ref{ss:parallel_variation_rect_curv} we introduce the notion of a parallel variation of a rectifiable curve.  There will be several structural theorems which we will prove about such variations, and we will end the subsection with a discussion of the smooth case.  In the smooth case it is important to extend this to more general continuous curves, however we will see how to avoid in this in Section \ref{s:parallel_gradient_nonsmooth}.

\subsection{Variations of a Point}\label{ss:variation_point}

A variation of a point is meant to replace the notion of a tangent vector at a point.  The natural replacement of such a notion on a metric space is a Cauchy sequence.  In the same manner that two vectors at a point are the same iff their induced directional derivatives act identically on all smooth functions, we will want to say two Cauchy sequences are equivalent if their induced actions on all lipschitz functions are equivalent.  More specifically, we start with the following:

\begin{definition}\label{d:variations_point}
If $X$ is a complete metric space, then we make the following definitions:
\begin{enumerate}
\item  We denote by $\Sigma X$ the space of all Cauchy sequences $\bv \equiv \{x_j\}$ on $X$ such that $x_j=x_\infty$ for at most a finite number of $j$.
\item  We let $\Sigma_x X\subseteq \Sigma X$ be the subset of Cauchy sequences $\bv\equiv \{x_j\}$ such that $x_j\to x$.
\item  If $f:X\to \dR$ is a Lipschitz function and $\bv\in \Sigma_x X$ then we denote the directional derivative by
\begin{align}
|d_\bv f|\equiv \limsup \frac{\big|f(x_j)-f(x)\big|}{d(x_j,x)}\, .
\end{align}
\item We say two Cauchy sequences $\bv,\bw\in \Sigma_x X$ are equivalent and write $\bv\sim \bw$ iff for every Lipschitz function $f:X\to \dR$ we have that $|d_\bv f|=|d_\bw f|$. 
\end{enumerate}
\end{definition}

Since we will generally only be interested in equivalence classes of variations, one could easily have defined $\Sigma X$ as the equivalence classes of such variations.  However, though this is possible it adds little to the discussion while making each proof more convoluted than is necessary, so we avoid this.

Let us remark on a few properties of $\Sigma X$.  Let $\bv=\{x_j\}$ and $\bw=\{y_j\}$ be Cauchy sequences and assume that $y_j\in \bv$ for all but a finite number of $j$, and conversely that $x_k\in \bw$ for all but a finite number of $k$.  Then we have that $\bv\sim \bw$.  In particular, up to equivalence the elements of $\Sigma X$ only depend on the asymptotic behavior of the sequence, and only up to rearrangement and repetition.  Note also that there is a canonical mapping
\begin{align}
\Sigma X\to X\, ,
\end{align}
given by $\bx\to x_\infty$, whose fiber above $x\in X$ is $\Sigma_x X$.  We call elements of $\Sigma_x X$ variations of $x$.  

The following gives us a basic characterization of when two variations are equivalent.    

\begin{lemma}\label{l:CS_equivalence_properties}
For a complete metric space $X$ and two variations $\bv=\{x_j\}, \bw=\{y_j\}\in \Sigma_x X$, we have that the following are equivalent:
\begin{enumerate}
\item The sequences $\bv\sim \bw$ define the same equivalence class.
\item There exists variations $\bv'\leq \bv$ and $\bw'\leq \bw$ with $\bv'\sim \bv$ and $\bw'\sim \bw$, such that
\begin{align}\label{e:CS_equiv2}
\lim_{j\to\infty}\frac{d(x'_{j},y'_{j})}{d(y'_j,x)}= 0\, .
\end{align}
\end{enumerate}
\end{lemma}
\begin{remark}
The second statement tells us that two Cauchy sequences satisfy $\bv\sim \by$ iff up to rearrangement they differ by an error which decays faster than they converge.
\end{remark}

\begin{proof}
Let us first prove $(2)\implies (1)$.  Specifically let $\{x\}$ and $\{y\}$ be two Cauchy sequences which satisfy $(2)$ and let $f$ be a Lipschitz function, then we have that 
\begin{align}
\bigg|\frac{\big|f(x'_{j})-f(x)\big|}{d(x'_j,x)}-\frac{\big|f(y'_j)-f(x)\big|}{d(y'_j,x)}\bigg| &\leq \frac{\big|f(x'_{j})-f(y'_j)\big|}{d(x'_{j},x)}+\frac{\big|f(y'_j)-f(x)\big|}{d(x'_{j},x)}\bigg(1-\frac{d(x'_{j},x)}{d(y'_j,x)}\bigg)\, ,\notag\\
&\leq 2\,Lip(f)\frac{d(x'_{j},y'_j)}{d(x'_{j},x)}\to 0\, ,
\end{align}
as claimed.



Let us prove $(1)\implies (2)$.  Define a mapping $I:\dN\to\dN$ where $I(k)$ is defined to be the integer $j$ which minimizes $\min_{j}d(x_k,y_j)$.  Let us first see that the variation $\bv'\equiv\{x_{I(j)}\}$ satisfies
\begin{align}\label{e:CS_equiv3}
\lim_{j\to\infty}\frac{d(x_{I(j)},y_{j})}{d(y_j,x)}= 0\, .
\end{align}
In particular, this implies that $\bv'\sim\bw$ and hence $\bv'\sim\bv$.  So assume (\ref{e:CS_equiv3}) fails, so that we can find a subsequence $y_{j_k}$ such that
\begin{align}
\liminf_{k\to\infty}\frac{d(x_{I(j_k)},y_{j_k})}{d(y_{j_k},x)}>\epsilon\, ,
\end{align}
for some $\epsilon>0$. Now let $r_k\equiv \frac{\epsilon}{10}d(y_{j_k},x)$, and let $\varphi$ be the $L^\infty$ function such that $\varphi(y)=1$ if for some $k$ we have that $y\in B_{r_k}(y_{j_k})$ and $\varphi\equiv 0$ otherwise.  Then we can consider the lipschitz function $f(y)\equiv d(x,y)\varphi(y)$.  Note that for every $j$ we have that 
\begin{align}
\frac{|f(x_{j})-f(x)|}{d(x_{j},x)}=0\, ,
\end{align}
while for every $k$ we have that
\begin{align}
\frac{|f(y_{j_k})-f(x)|}{d(y_{j_k},x)}=1\, ,
\end{align}
which is a contradiction to the equivalence of $\bv$ and $\bw$, and hence shows (\ref{e:CS_equiv3}) and proves the Lemma with $\bw'\equiv \bw$.
\end{proof}

\subsection{Variations of a Curve}\label{ss:variation_curve}

We want to introduce the notion of a variation of a curve.  Such a variation of a continuous curve $\gamma$ is our replacement for a vector field along $\gamma$, and is nothing more than the assignment to each point $\gamma(t)$ a Cauchy sequence in $X$ which converges to $\gamma(t)$.  That is, a variation $V$ is a section of the bundle $\Sigma X$ above $\gamma$.  Let us begin with the definition:

\begin{definition}\label{d:variation}
If $\gamma\in P(X)$ is a continuous curve, then a variation of $\gamma$ is a mapping $V:[0,\infty)\to \Sigma X$ such that $\lim V_j(t)=\gamma(t)$ for each $t$.  We denote by $\Sigma_\gamma X$ the collection of all variations of $\gamma$.  
\end{definition}

Equivalently, a variation $V$ is a sequence of mappings $\{V_j\}:[0,\infty)\to X$ such that $\lim V_j(t)=\gamma(t)$.  Notice from this point of view that {\it apriori} we are not even assuming the mappings $V_j$ are measurable.  In principle, we do not want to force too much regularity on the mappings $V_j$, for instance continuity, as this will not be the case for $s$-parallel variations.  On the other hand, it will be not so hard to see that for reasonable variations, for instance parallel variations, there will exist an {\it equivalent} variation which is continuous.


It will turn out to be a useful observation that for any partition $\bt\in \Delta[0,\infty)$ we have that $V(\bt)\in \Sigma_{\gamma(\bt)} X^{|\bt|}$.  That is, we may view $V(\bt)$ as a point variation in $X^{|\bt|}$.  With this in mind let us quickly address the correct notion of equivalence for variations:

\begin{definition}\label{d:variation_equivalence}
We say two variations $V,V'\in \Sigma_\gamma X$ of a continuous curve $\gamma$ are equivalent, and write $V\sim V'$ if for each $\bt\in \Delta[0,\infty)$ there exists $\bt\leq \bt'$ such that $V_j(\bt')\sim V'_j(\bt')$ as elements of $\Sigma_{\gamma(\bt')} X^{|\bt'|}$.
\end{definition}

It will be convenient when studying a variation to consider its pointwise length, namely we have the simple notation that if $V=\{V_j\}$ is a variation then we denote
\begin{align}
|V_j|(t)\equiv d(V_j(t),\gamma(t))\, .
\end{align}

Now let us consider the prototypical example:

\begin{example}\label{e:parallel_variation_smooth}
Let $X$ be a smooth manifold with $\gamma\in P(X)$ and let $v_j(t)$ a sequence of vector fields, not necessarily continuous, along $\gamma$ such that $v_j(t)\to 0$ pointwise.  Then we have that
\begin{align}
V_j(t)\equiv \exp_{\gamma(t)}\big(v_j(t)\big)\, ,
\end{align}
is a variation of $\gamma$.  In particular, if $v_j(t)$ are all parallel translation invariant vector fields then we might call $V$ a parallel variation of $\gamma$.  This notation will be made more rigorous and clear in the next subsection.
\end{example}

\subsection{Parallel Variations of Rectifiable Curves}\label{ss:parallel_variation_rect_curv}

In Section \ref{ss:variation_curve} we introduced the notion of a variation of a curve.  In this subsection we discuss the notion of a parallel variation $V$ over a rectifiable curve $\gamma\in P(X)$.  We will be particularly interested in applying this to piecewise geodesics, and in Section \ref{s:parallel_gradient_nonsmooth} we will see how the results of this Section can be used to help define the parallel gradient in the nonsmooth context.

The notion of a parallel variation is not completely well defined on a nonsmooth space, and there are various conditions, some more restrictive and some less, which could be used.  For the purposes of this paper we will want to stick with a definition which assumes as little as possible.  In fact, although there are many properties one might expect to hold for a parallel variation, there are only two conditions that {\it must} be satisfied for a parallel variation $V\equiv \{V_j\}$ of a rectifiable curve in order for the Theorems of the remainder of the paper to hold.  These are the following:
\begin{enumerate}[(A)]
\item (parallel norm)  For any $s,t\geq 0$ we have $$\lim \frac{\big|\, |V_j(t)|-|V_j(s)|\big|}{|V_j(s)|}=0\, .$$
\item (reduction in smooth case)  If $X$ is a smooth manifold then up to equivalence a parallel variation $V$ is equivalent to a variation of $\gamma$ induced by a parallel translation invariant vector field, see Example \ref{e:parallel_variation_smooth} and Theorem \ref{t:parallel_smooth_rect}.
\end{enumerate}

In Section \ref{sss:parallelogram} we discuss some elementary properties of parallelograms in $\dR^n$.  This gives us a geometric way of identifying parallel vectors in $\dR^n$.  Using the properties discussed there we will define the notion of a parallel translation invariant variation in Section \ref{sss:parallel_variation_rect_curv}.

\subsubsection{Parallelograms and Parallel Translation}\label{sss:parallelogram}

Given points $x,y\in \dR^n$ and a variation $v_x=\{x_i\}$ of $x$, it is clear that up to equivalence the only variation of $y$ which could reasonably be considered the parallel translation of $v_x$ is the variation $v_y= \{y_i\}\equiv\{x_i-x+y\}$.  One can also identify the point $y_i$ as the unique element of $\dR^n$ such that the quadruple $(x,y,y_i,x_i)$ is a parallelogram.  To generalize this to more complicated situations let us begin with the following very classical statement:

\begin{theorem}[Parallelogram Law]\label{t:parallelogram_law}
Let $\bx= (x_1,x_2,x_3,x_4)$ be a quadruple in $\dR^n$.  Then the quadruple forms a parallelogram if and only if for any $x_j$ we have that 
\begin{align}
e_j(\bx)\equiv 2|x_{j+1}-x_j|^2+2|x_j-x_{j-1}|^2 -|x_3-x_1|^2-|x_4-x_2|^2=0\, .
\end{align}
\end{theorem}

Hence at least in $\dR^n$ the numbers $e_j(\bx)$ give a quantitative measurement of how close $\bx$ is to a parallelogram.  To exploit this in the context of a metric space let us record the following, which gives a more complete understanding the error functions $e_j$:\\

\begin{lemma}\label{l:parallelogram_basics}
Let $\bx= (x_1,x_2,x_3,x_4)$ be a quadruple in $\dR^n$, and let $v_1\equiv x_4-x_1$, $v_2\equiv x_3-x_2$.  Then the following hold:
\begin{enumerate}
\item $e_1+e_2=4\langle v_1-v_2, x_2-x_1\rangle$.
\item $e_3+e_4=4\langle v_2-v_1, x_3-x_4\rangle$.
\item $e_2-e_1=2\big(|v_2|^2-|v_1|^2\big)$.
\item $e_3-e_4=2\big(|v_2|^2-|v_1|^2\big)$.
\item $e_1+e_3=e_2+e_4=2|v_2-v_1|^2$.
\end{enumerate}
\end{lemma}

Now with the above in hand we will make sense of a parallelogram in a metric space $X$.  More completely, we would like to define a quantitative version of a parallelogram in a metric space.  Thus let us consider a quadruple $\bx=(x_1,x_2,x_3,x_4)$ in a metric space $X$, and let us denote the error functions by
\begin{align}
e_j(\bx)\equiv 2d(x_{j+1},x_j)^2+2d(x_j,x_{j-1})^2 -d(x_3,x_1)^2-d(x_4,x_2)^2\, ,
\end{align}
as well as the perimeter functions
\begin{align}
&P_v\equiv \max\{d(x_1,x_4),d(x_2,x_3)\}\, ,\notag\\
&P_x\equiv \max\{d(x_1,x_2),d(x_3,x_4)\}\, .
\end{align}

Then motivated by Lemma \ref{l:parallelogram_basics} and the constructions of the next subsection we make the following definition:

\begin{definition}\label{d:almost_parallelogram}
Let $\bx=(x_1,x_2,x_3,x_4)$ be a quadruple in a metric space $X$, then we say that $\bx$ is an $\epsilon$-parallelogram if the following inequalities hold:
\begin{enumerate}
\item $|e_1+e_2|\leq \epsilon\, P_x\cdot P_v$.
\item $|e_3+e_4|\leq \epsilon\, P_x\cdot P_v$.
\item $|e_2-e_1|\leq \epsilon\, P_v^2$.
\item $|e_4-e_3|\leq \epsilon\, P_v^2$.
\item $|e_1+e_3|=|e_2+e_4|\leq \epsilon\, P_v^2$.
\end{enumerate}
\end{definition}

\subsubsection{Parallel Variation along Rectifiable Curves}\label{sss:parallel_variation_rect_curv}

Now let us use the notion of an $\epsilon$-parallelogram to define a parallel variation along a finite length rectifiable curve.  Since we will mainly be interested in applying this to piecewise geodesics, one could easily restrict to this set as well.  First we recall that a curve $\gamma:[0,T]\to X$ with $T<\infty$ rectifiable if we have that
\begin{align}
\lim_{\bt\in \Delta[0,T]} \sum d(\gamma(t_j),\gamma(t_{j+1}))<\infty \, .
\end{align}
Since this is a monotone function on the directed set $\Delta[0,T]$ there therefore exists a limit, which we denote by $\ell(\gamma)$.  

Now we define the notion a parallel variation $V$ of a rectifiable curve $\gamma$.  Roughly, it is just the statement that the quadrilaterals $(\gamma(t),\gamma(t),V_j(t),V_j(s))$ are converging toward parallelograms.  Precisely:

\begin{definition}\label{d:parallel_variation_rect_curve}
Let $\gamma:[0,T]\to X$ with $T<\infty$ be a rectifiable curve with $V$ a variation of $\gamma$.  Then we say that $V$ is a parallel variation if for every $\epsilon>0$  there exists a partition $\bt'\in \Delta[0,T]$ such that for all partitions $\bt'\leq \bt$ and all $j\geq J(\bt,\epsilon)$ sufficiently large we have that $\big(\gamma(t_a),\gamma(t_{a+1}),V_j(t_{a+1}),V_j(t_a)\big)$ is a $\epsilon\cdot d(\gamma(t_a),\gamma(t_{a+1}))$-parallelogram.
\end{definition}

Now we spend the rest of this section exploring properties of a parallel variation.  First we see the following, which is almost immediate from the definition

\begin{lemma}
Let $\gamma:[0,T]\to X$ be a rectifiable curve with $V$ a parallel variation of $\gamma$.  Then if $V'$ is a variation of $\gamma$ which is equivalent to $V$, then $V'$ is also a parallel variation.
\end{lemma}

Hence, we see that the notion of a parallel variation is independent of equivalence class.  Now let us prove the parallel norm property from the introduction:

\begin{theorem}\label{t:parallel_norm_rect}
Let $\gamma:[0,T]\to X$ be a rectifiable curve with $V$ a parallel variation of $\gamma$.  Then for any $s,t\geq 0$ we have
\begin{align}
\lim \frac{\big|\, |V_j(t)|-|V_j(s)|\big|}{|V_j(s)|}=0\, .
\end{align}
\end{theorem}
\begin{proof}
Fix $\epsilon>0$ and let $\bt$ be a partition with $s,t\in \bt$ such that for $j$ sufficiently large we have, as in Definition \ref{d:parallel_variation_rect_curve}, that $\big(\gamma(t_a),\gamma(t_{a+1}),V_j(t_{a+1}),V_j(t_a)\big)$ are $\epsilon\cdot d(\gamma(t_a),\gamma(t_{a+1}))$-parallelograms.  In particular, using Definition \ref{d:almost_parallelogram} we have that
\begin{align}
\big| |V_j(t_{a+1})|^2-|V_j(t_a)|^2\big|\leq \epsilon \max\{|V_j(t_{a})|^2,|V_j(t_{a+1})|^2\}\,d(\gamma(t_{a}),\gamma(t_{a+1}))\, ,
\end{align}
which gives us that
\begin{align}\label{e:propertyA:1}
\big| |V_j(t_{a+1})|-|V_j(t_a)|\big|\leq \epsilon |V_j(t_{a+1})|\cdot d(\gamma(t_{a}),\gamma(t_{a+1}))\, .
\end{align}
Now first let $t_{max}\in\bt$ be such that $|V_j(t_{max})|\equiv\max\{|V_j(t_a)|\}$.  Then for any other element of the partition we have that
\begin{align}
\frac{\big| |V_j(t_{max})|-|V_j(t_a)|\big|}{|V_j(t_{max})|}\leq \epsilon \sum  d(\gamma(t_{a}),\gamma(t_{a+1}))\leq \epsilon \ell(\gamma)\, .
\end{align}
In particular, for $\epsilon$ sufficiently small and $j$ sufficiently large we have that for all $|V_j(t_a)|$ that
\begin{align}
\frac{1}{2}|V_j(t_{max})|< |V_j(t_a)| \leq |V_j(t_{max})|.
\end{align}
Now returning to (\ref{e:propertyA:1}) and summing between all elements of the partition between $s$ and $t$ we have the estimate
\begin{align}
\big| |V_j(t)|-|V_j(s)|\big|\leq \epsilon |V_j(t_{max})| \sum  d(\gamma(t_{a}),\gamma(t_{a+1}))\leq 2\epsilon |V_j(s)|\cdot \ell(\gamma)\, ,
\end{align}
or that
\begin{align}
\frac{\big| |V_j(t)|-|V_j(s)|\big|}{|V_j(s)|}\leq 2\epsilon\cdot \ell(\gamma)\, ,
\end{align}
for all $j$ sufficiently large.  Since $\epsilon>0$ was arbitrary, we have proved the result.

\end{proof}

Now we end this subsection by studying a parallel variation along a rectifiable curve in a smooth manifold:

\begin{theorem}\label{t:parallel_smooth_rect}
Let $(M^n,g)$ be a smooth manifold with $\gamma:[0,T]\to M$ a piecewise smooth curve.  Then
\begin{enumerate}
\item If $v_j\in T_{\gamma(0)}M$ is any sequence of tangent vectors with $v_j\to 0$, then the variation $V_j(t)\equiv \exp_{\gamma(t)}(P_t^{-1}v_j)$, where $P_t$ is the parallel translation map, is a parallel variation.
\item If $V$ is any parallel variation of $\gamma$ then there exists $v'_j\in T_{\gamma(0)}M$ such that the induced parallel variation $V'$ as above is equivalent to $V$.
\end{enumerate}
\end{theorem}
\begin{proof}
To begin let $x\in M$ and let us consider exponential coordinates centered at $x$ on the ball $B_{\iota_x}(x)$, where $\iota_x\leq \min\{\frac{1}{2}\inj(x)\}$ is such that $B_{\iota_x}(x)$ is a convex set.  Standard computations tell us that the metric $g_{ij}$ is such coordinates may be written
\begin{align}\label{e:parallel_variation:1}
&\big|g_{ij}(y)-\delta_{ij}\big| \leq C\, d(x,y)^2\, ,\notag\\
&\big|\partial_k g_{ij}\big|(y)\leq Cd(x,y)\, ,\notag\\
&\big|\partial_k\partial_\ell g_{ij}\big|(y)\leq C\, ,
\end{align}
where in general the $C$ depends on the full curvature tensor bounds of $M$ in the neighborhood of $x$.  Now using (\ref{e:parallel_variation:1}) let us observe the following properties.  First if $(x_1,x_2,x_3,x_4)\subseteq B_r(x)\subseteq B_{\iota_x}(x)$ then let us denote by $v_1\equiv x_4-x_1$ and $v_2\equiv x_3-x_2\in \dR^n$ the coordinate difference, and by $v'_1\in T_{x_1}M$, $v'_2\in T_{x_2}M$ the vector difference defined by $x_4\equiv \exp_{x_1}(v'_1)$, $x_3\equiv \exp_{x_2}(v'_2)$.  Then if $P:T_{x_2}M\to T_{x_1}M$ is the isometry defined by parallel translation along the unique geodesic connecting $x_1,x_2$, then using (\ref{e:parallel_variation:1}) we have the estimates
\begin{align}\label{e:parallel_variation:2}
&C^{-1}r^2 \max\{|v_1|,|v_2|\}\leq \big||v_1-v_2|-|v'_1-Pv'_2|\big|\leq Cr^2 \max\{|v_1|,|v_2|\}\, ,\notag\\
&1-Cr^2\leq \frac{|v_1|}{|v'_1|},\frac{|v_2|}{|v'_2|}\leq 1+Cr^2\, .
\end{align}
Note in particular that by using this and Lemma \ref{l:parallelogram_basics} this then tells us that if $\bx\subseteq B_r(x)\subseteq B_{\iota_x}(x)$ is an $\epsilon$-parallelogram, then we have the estimate
\begin{align}\label{e:parallel_variation:3}
|v'_1-Pv'_2|\leq (\epsilon+Cr^2)\max\{|v'_1|,|v'_2|\}\, .
\end{align}
 
Similarly, we have from (\ref{e:parallel_variation:1}) and (\ref{e:parallel_variation:2}) that if $v(t)$ is a parallel translation invariant vector field along $\gamma$ and $s,t\in [0,T]$, then the quadruple $\bx=(\gamma(s),\gamma(t),\exp_{\gamma(t)}(v),\exp_{\gamma(s)}(v))$ is a $C\big(|t-s|^2+|v|^2\big)$-parallelogram.  In particular, if $v_j\in T_{\gamma(0)}M$ with $|v_j|\to 0$ then for every partition $\bt$ we see that by letting $j$ be sufficiently large, namely such that $|v_j|<\max |t_{a+1}-t_a|$, then this gives us that the variation $V$ given by $V_j(t)=\exp_{\gamma(t)}(P_t^{-1}v_j)$ is a parallel variation, as claimed.

To prove the second claim let $V$ be a variation of $\gamma$ and let $v'_j(t)\in T_{\gamma(t)}M$ be defined by $\exp_{\gamma(t)}(v'_j)=V_j(t)$.  Using (\ref{e:parallel_variation:3}) we therefore see that for all partitions $\bt$ that if $j$ is sufficiently large then
\begin{align}\label{e:parallel_variation:4}
|v'_j(t_a)-Pv'_j(t_{a+1})|\leq C|t_{a+1}-t_a|^2\max\{|v'_j|(t_a),|v'_j|(t_{a+1})\}\, .
\end{align}
Now by Theorem \ref{t:parallel_norm_rect} we know that $\frac{|v'_j|(t)}{|v'_j|(0)}\to 1$, which in combination with the above tells us that if $v_j(t)\equiv P_t^{-1}v_j$ is the parallel translation invariant vector field along $\gamma$ with $v_j(0)=v'_j(0)$, then
\begin{align}
\lim_j \frac{\big|v_j(t)-v'_j(t)\big|}{|v_j|(t)}\to 0\, ,
\end{align}
which precisely proves that the variation $V'$ given by $V'_j(t)=\exp_{\gamma(t)}(v_j(t))$ is equivalent to $V$, as claimed.
\end{proof}

\section{The Parallel Gradient on Path Space}\label{s:parallel_gradient_nonsmooth}

One of the key purposes of the the structure of the previous Sections has been to build a geometric structure on path space $P(X)$.  Even in the smooth case this required some work since the geometry of interest is not compatible with the structure of the underlying curves.  In this Section we give a construction of the parallel gradients of a function on $P(X)$ that will work on an arbitrary metric-measure space $X$, without the need for a smooth structure.  We will show in Section \ref{ss:parallel_grad_comparison} that the constructions of this Section and those of Section \ref{ss:parallel_gradient} give rise to the same gradients on the path space of a smooth manifold.  These constructions will be used in Section \ref{s:bounded_ricci_mms} to define the notion of bounded Ricci curvature on a metric-measure space, and they will be generalized in Section \ref{s:bounded_ricci_analysis} to define the $H^1$-gradient and Ornstein-Uhlenbeck operators on path space.

In Section \ref{ss:parallel_slope} we introduce the parallel slope and discuss its basic properties.  As in the case of the cheeger gradient on a metric-measure space, one must first define the slope operator, and then take the lower semicontinuous refinement in order to define the gradient, which is done in Section \ref{ss:parallel_gradient_nonsmooth}.  On a smooth metric-measure space this definition of gradient is {\it apriori} quite different than the one given in Section \ref{ss:parallel_gradient}.  However, in Section \ref{ss:parallel_grad_comparison} we show the two definitions agree on a smooth metric-measure space.

\subsection{The Slope on Path Space}\label{ss:parallel_slope}

In this section we introduce the parallel slopes for a cylinder function on path space and prove some basic properties about them.  As was previously remarked, we will take the lower semicontinuous refinement in order to define the gradient.  On a sufficiently nice metric-measure space, for instance a smooth metric-measure space, the slope and gradient will coincide, but apriori this may not be the case and cannot be assumed.

We use heavily the notation and constructions of Section \ref{s:prelim_nonsmooth} and Section \ref{s:variation}.  We begin by introducing some terminology, in particular, it will be important to consider approximations of continuous curves by piecewise geodesics.  Precisely:

\begin{definition}
We make the following definitions:
\begin{enumerate}
\item Given a partition $\bt\in \Delta[0,\infty)$ we call a curve $\gamma:[0,t_{|\bt|}]\to X$ a $\bt$-geodesic if the restriction of $\gamma$ to each interval $[t_a,t_{a+1}]$ is a minimizing geodesic.  If $\gamma$ is a $\bt$-geodesic with respect to some partition we may just call $\gamma$ a piecewise geodesic.
\item Given a curve $\gamma\in P(X)$ and a partition $\bt\in \Delta[0,\infty)$ we call a curve $\gamma_{\bt}:[0,t_{|\bt|}]\to X$ a $\bt$-approximation of $\gamma$ if $\gamma_{\bt}$ is a $\bt$-geodesic with $\gamma_{\bt}(t_a)=\gamma(t_a)$ for each $t_a\in \bt$.
\end{enumerate}
\end{definition}

To define the parallel slopes we will need to define the directional derivative of a cylinder function with respect to a parallel variation.  We will be primarily interested in studying these along $\bt$-geodesics.  Specifically, let $F\in\Cyl(X)$ be a cylinder function with $\gamma\in P(X)$ a $\bt$-geodesic and $V=\{V_j\}$ a $s$-parallel variation of $\gamma$.  We define the (normalized) directional derivative of $F$ in the direction $V$ by the formula
\begin{align}\label{e:variation_directional_derivative}
|D_V F|(\gamma)\equiv \limsup_{j\to\infty} \frac{|F(\gamma)-F(V_j)|}{|V_j|(s)}\, .
\end{align}

Now we are in a position to use the above to define the parallel slopes of a cylinder function:

\begin{definition}
Let $F\in \Cyl(X)$ be a cylinder function on path space and $\gamma\in P(X)$ a continuous curve.  Then we define the parallel slope $|\partial_s F|:P(X)\to \dR$ by the formula
\begin{align}\label{e:parallel_slope_nonsmooth}
|\partial_s F|(\gamma)\equiv \limsup_{\bt\to\Delta}\{|D_{V_\bt} F|(\gamma_\bt):\gamma_\bt\text{ is a $\bt$-approximation of $\gamma$, and }V_\bt\text{ is a $s$-parallel variation of $\gamma_\bt$}\}.
\end{align}
\end{definition}

Recall in the above definition that $\Delta[0,\infty)$ is a directed set, and therefore we may consider limits with respect it, see Section \ref{sss:simplex}.  

Let us begin with a few simple estimates that will be useful throughout.  Among other things they tell us that a cylinder function $F$ is a lipschitz function with respect to any of the parallel slopes.

\begin{lemma}\label{l:slope_lipschitz}
Let $F=e_\bt^*u\in \Cyl(X)$ be a cylinder function with $\bt\in\Delta[0,T]$ and $u\in Lip_c(X^{|\bt|})$.  Then the following hold:
\begin{enumerate}
\item For all $0\leq s\leq T$ we have that $|\partial_s F|(\gamma)\leq\sqrt{|\bt|}\cdot |\Lip\, u|(\gamma(\bt))$. 
\item For all $s>T$ we have that $|\partial_s F| = 0$.
\item If $\bt = \{0\leq t_1<\ldots<t_{|\bt|}\leq T\}$ and if for some $k$ we have that $t_k< s<s'\leq t_{k+1}$, then we have that $|\partial_s F|(\gamma)\leq |\partial_{s'} F|$.
\end{enumerate}
\end{lemma}
\begin{remark}
The first property tells us in particular that if $F$ is a cylinder function then $|\partial_s F|(\gamma)$ is uniformly bounded independent of $\gamma$ and $s$.
\end{remark}

\begin{proof}
Let $\gamma\in P(X)$ be a piecewise geodesic with $V\equiv \{V_j\}$ a $s$-parallel variation of $\gamma$.  Note now that if $t_k\in\bt$ is an element of the partition with $t_k<s$ then $V_j(t_k)=\gamma(t_k)$.  On the other hand, if $t_k\geq s$ then by Theorem \ref{t:parallel_norm_rect} we have that
\begin{align}
\frac{|V_j|(t_k)-|V_j|(s)}{|V_j|(s)}\stackrel{j\to\infty}{\longrightarrow} 0\, .
\end{align}
In particular, let $k$ be the largest integer such that $t_k<s$.  Then the above observation gives us that
\begin{align}
\lim_{j\to\infty} \frac{\sqrt{\sum_a d(\gamma(t_a),V_{j}(t_a))^2}}{d(\gamma(s),V_j(s))}\to \sqrt{|\bt|-k}\, ,
\end{align}
and in particular we have
\begin{align}
|D_V F|(\gamma)&\equiv \lim_{j\to\infty}\frac{|F(\gamma)-F(V_j)|}{|V_j|(s)}=\lim_{j\to\infty} \frac{|u(\gamma(\bt))-u(\gamma_j(\bt))|}{(|\bt|-k)^{-1/2}d_{X^{|\bt|}}(\gamma(\bt),V_j(\bt))}\leq \sqrt{|\bt|-k}\cdot |\Lip\, u|(\gamma(\bt))\, .
\end{align}
By estimating $|\bt|-k\leq |\bt|$ and observing that this holds for an arbitrary $\bt'$-geodesic we obtain the first claim, while if $T<s$ and so $k=|\bt|$ we obtain the second claim.  

To prove the third claim we require the following observation.  Let $V_s\equiv {V_j}$ be a $s$-parallel variation.  Then the it follows that the variation $V_{s'}\equiv \{V'_j\}$ defined by $V_j'(t)=0$ if $t< s'$ and $V'_j(t)=\gamma_j(t)$ if $t\geq s$ is a $s'$-parallel variation.  Further using Theorem \ref{t:parallel_norm_rect} once again we have that
\begin{align}
|D_{V_s} F|(\gamma) &=\limsup\frac{|F(\gamma)-F(V_j)|}{|V_j|(s)}\notag\\
&= \limsup\frac{|F(\gamma)-F(V'_j)|}{|V'_j|(s)} = |D_{V_{s'}} F|(\gamma)\, .
\end{align}
Since this held for every piecewise geodesic $\gamma$ and any $s$-parallel variation of $\gamma$ we immediately get $|\partial_s F|\leq |\partial_{s'} F|$ as claimed.

\end{proof}

Now that we've seen a few basic estimates on the parallel slope, and that in particular the cylinder functions are well behaved with respect to it, let us now discuss a few more refined properties.

\begin{theorem}\label{t:slope_properties}
The following properties hold for the parallel slopes:
\begin{enumerate}

\item (Convexity) If $F,G\in \Cyl(X)$ are cylinder functions, then we have the convexity estimates
\begin{align}
&|\partial_s (F+G)|(\gamma)\leq |\partial_s F|_{}(\gamma)+|\partial_s G|_{}(\gamma)\, .
\end{align}

\item (Strongly Local) If $F,G\in \Cyl(X)$ are cylinder functions with $F=\text{const}$ on a neighborhood of the support of $G$, then
\begin{align}
|\partial_s (F+G)|(\gamma) = |\partial_s F|(\gamma)+|\partial_s G|(\gamma)\, .
\end{align}

\item (Stability under Lipschitz Calculus) If $F\in \Cyl(X)$ is a cylinder function and $\phi:\dR\to\dR$ is lipschitz, then
\begin{align}
|\partial_s\big( \phi\circ F\big)|(\gamma)\leq ||\phi||_{Lip}\cdot |\partial_s F|(\gamma)\, .
\end{align}

\item (Strong Convexity) If $F,G,\chi\in \Cyl(X)$ are cylinder functions with $0\leq \chi\leq 1$, then we have the pointwise convexity estimate
\begin{align}
|\partial_s (\chi F+(1-\chi)G)|\leq \chi|\partial_s F|+(1-\chi)|\partial_s G|+|\partial_s\chi|\cdot|F-G|\, .
\end{align}
\end{enumerate}
\end{theorem}

\begin{proof}
Throughout we let  $\gamma\in P(X)$ be a piecewise geodesic with $V=\{\gamma_j\}$ a $s$-parallel variation of $\gamma$.  The first statement follows easily from the triangle inequality
\begin{align}
|D_V (F+G)|=\limsup\frac{\big|(F+G)(\gamma)-(F+G)(\gamma_j)\big|}{d(\gamma(s),\gamma_j(s))}&\leq \limsup\frac{\big|F(\gamma)-F(\gamma_j)\big|}{d(\gamma(s),\gamma_j(s))}+\limsup\frac{\big|G(\gamma)-G(\gamma_j)\big|}{d(\gamma(s),\gamma_j(s))}\notag\\
&\leq |\partial_s F|+|\partial_s G|\, .
\end{align}

To prove the second statement let us first note that there is no harm in assuming $F=e_{\bt}^*f$ and $G=e_{\bt}^*g$ for some common $\bt\in \Delta[0,T]$.  Of course, this can always be forced by taking a common refinement of partitions.  If $F$ is constant on a neighborhood of the support of $G$, then this implies that $f$ is constant on a neighborhood of the support of $g$.  Hence for each variation $V=\{\gamma_j\}$ we have that for $j$ sufficiently large that for each $t_a\in \bt$ that either $F(\gamma(t_a))=F(\gamma_{j}(t_a))=\text{const}$ or $G(\gamma(t_a))=G(\gamma_{j}(t_a))=0$.  The result then easily follows.\\

The third statement is proved in the same manner as the first with the pointwise estimate
\begin{align}
\frac{\big|\phi\circ F(\gamma)-\phi\circ F(\gamma_j)\big|}{d(\gamma(s),\gamma_j(s))}\leq ||\phi||_{Lip}\frac{\big|F(\gamma)-F(\gamma_j)\big|}{d(\gamma(s),\gamma_j(s))}\, .
\end{align}

The fifth statement is also proved in the same manner as the first with the estimate
\begin{align}
&\frac{\big|\big(\chi F+(1-\chi)G\big)(\gamma)-\big(\chi F+(1-\chi)G\big)(\gamma_j)\big|}{d(\gamma(s),\gamma_j(s))}\notag\\
&=\frac{\big|\chi(\gamma)\big(F(\gamma)-F(\gamma_j)\big)+(1-\chi(\gamma))\big(G(\gamma)-G(\gamma_j)\big)  +\big(\chi(\gamma)-\chi(\gamma_j)\big)\big(F(\gamma_j)-G(\gamma_j)\big)\big|}{d(\gamma(s),\gamma_j(s))}\notag\\
&\leq \frac{\big|\chi(\gamma)\big|\cdot\big|F(\gamma)-F(\gamma_j)\big|}{d(\gamma(s),\gamma_j(s))}  +\frac{\big|1-\chi(\gamma)\big|\cdot\big|G(\gamma)-G(\gamma_j)\big|}{d(\gamma(s),\gamma_j(s))}   +\frac{\big|\chi(\gamma)-\chi(\gamma_j)\big|\cdot\big|F(\gamma_j)-G(\gamma_j)\big|}{d(\gamma(s),\gamma_j(s))}\, .
\end{align}
\end{proof}

\subsection{The Parallel Gradients}\label{ss:parallel_gradient_nonsmooth}

Section \ref{ss:parallel_slope} was dedicated to defining the slope and proving some basic properties about it.  This will be done by taking the lower semicontinuous refinement of the slope, see \cite{Cheeger_DiffLipFun} and \cite{Ambrosio_Calculus_Ricci}.

In addition to the standard assumptions about the metric measure space $(X,d,m)$, in this Section we assume that $X$ is weakly Riemannian.  That is, the laplacian $\Delta_X$ of the metric measure space is linear, see Section \ref{ss:weakly_riemannian}.  As was shown this condition is equivalent to the existence of the diffusion measures $\Gamma_\mu$ on $P(X)$, where $\mu$ is a measure on $X$.  In particular, this condition is equivalent to the existence of the Wiener measures $\Gamma_x$ on $P(X)$.  Throughout this Section we will be pairing path space $P(X)$ with the diffusion measure $\Gamma_m$.  Let us begin by defining the upper parallel gradient.

\begin{definition}
Given $F\in L^2(P(X),\Gamma_m)$ we say that $G\in L^2(P(X),\Gamma_m)$ is a upper $s$-parallel gradient for $F$ if there exists a sequence of cylinder functions $F_i\in \Cyl(X)$ such that $F_i\to F$ strongly in $L^2(P(X))$ and $|\partial_s F_i|\rightharpoonup G'$ weakly in $L^2(P(X))$ with $G'\leq G$ $a.e.$
\end{definition}

We will want to define the parallel gradient of $F$ as the unique {\it minimal} upper parallel gradient for $F$.  First we must study some basic properties of the upper gradients, and in particular using Theorem \ref{t:slope_properties} we arrive at the following:

\begin{lemma}\label{l:upper_grad_prop1}
For $F\in L^2(P(X),\Gamma_m)$ and each $s\geq 0$ the following hold:
\begin{enumerate}
\item The collection of upper $s$-parallel gradients for $F$ is a closed convex subset of $L^2(P(X))$.
\item If $G_1$, $G_2$ are upper $s$-parallel gradients for $F$ then so is $G(x)\equiv \min\{G_1(x),G_2(x)\}$.
\end{enumerate}
\end{lemma}
\begin{proof}
To prove the first statement let us remark on the convexity first.  Namely assume $G_1$ and $G_2$ are upper $s$-parallel gradients for $F$ and let $F_{1i}$, $F_{2i}\in \Cyl(X)$ be cylinder functions with 
\begin{align}
|\partial_s F_{1i}|\rightharpoonup G_1'\, ,\notag\\
|\partial_s F_{2i}|\rightharpoonup G_2'\, .
\end{align}
For any $0\leq t\leq 1$ if we consider the sequence $F_i\equiv tF_{1i}+(1-t)F_{2i}$, then we clearly still have that $F_i\to F$ strongly.  Further, by Theorem \ref{t:slope_properties} we have that
\begin{align}
|\partial_s F_i|\leq t|\partial_s F_{1i}|+(1-t)|\partial_s F_{2i}|\rightharpoonup tG'_1+(1-t)G'_2\leq tG_1+(1-t)G_2\, ,
\end{align}
which proves the convexity claim.  To prove that the set is closed let $G_j$ be a sequence of upper $s$-parallel gradients for $F$ with $G_j\rightharpoonup G$.  If $F_{ij}\in \Cyl(X)$ are cylinder functions with $F_{ij}\stackrel{i\to\infty}{\rightarrow} F$ and $|\partial_s F_{ij}|\stackrel{i\to\infty}{\rightharpoonup} G_j'\leq G_j$.  Now we can use the usual diagonalization procedure to pick a subsequence $F_i=F_{ij_i}$ such that $F_{i}\rightarrow F$ and $|\partial_s F_{i}|\rightharpoonup G'\leq G$ as claimed.

To prove the second claim we prove the following stronger statement.  Namely, let $\cB\subseteq P(X)$ be a Borel set, then if $G_1,G_2$ are upper $s$-parallel gradients for $F$, then so is $\chi_{\cB}G_1+\chi_{P(X)\setminus\cB}G_2$, where $\chi_\cB$ is the characteristic function of the set $\cB$.  To prove this it is enough, by the closed property of the set of upper gradients, to show this for cylinder sets $e_{\bt}^*\cB$, where $\cB\subseteq X^{|\bt|}$ is a compact subset.  That is, since the collection of cylinder sets is a an algebra of sets which generates the Borel $\sigma$-algebra on $P(X)$, if $\cB$ is a Borel subset of $P(X)$ then there exists compact Borel cylinder sets $\cB_j$ which converge in measure to $\cB$.  In particular, we have that $\chi_{\cB_j}G_1+\chi_{\cB_j^c}G_2\rightharpoonup \chi_{\cB}G_1+\chi_{\cB^c}G_2$, so it is enough to prove that for each $j$ that $\chi_{\cB_j}G_1+\chi_{\cB_j^c}G_2$ is an upper $s$-parallel gradient.

Now let $\cB$ be a compact cylinder set, and for each $\epsilon>0$ let $\chi_\epsilon$ be a lipschitz cutoff function on $X^{|\bt|}$ with $\chi_\epsilon\equiv 1$ on $\cB$ and $\chi_\epsilon\equiv 0$ outside of $B_\epsilon(\cB)$ and $Lip(\chi_\epsilon)<2\epsilon^{-1}$.  We will show for each $\epsilon>0$ that $\chi_\epsilon G_1+(1-\chi_\epsilon)G_2$ is a weak upper gradient for $F$.  Again, by using the closed property of the upper gradients this then proves the claim.  To see this let $F_{1i}\to F$ and $F_{2i}\to F$ such that $|\partial_s F_{1i}|\rightharpoonup G_1'\leq G_1$ and $|\partial_s F_{2i}|\rightharpoonup G_2'\leq G_2$.  Let us consider the sequence $F_i\equiv \chi_\epsilon F_{1i}+(1-\chi_\epsilon)F_{2i}$.  Clearly we have $F_i\to F$, and further by using Theorem \ref{t:slope_properties}.5 we have that 
\begin{align}
|\partial_s F_i|&\leq \chi_\epsilon \,|\partial_s F_{1i}|+(1-\chi)|\partial_s F_{2i}|+|\partial_s \chi_\epsilon|\cdot|F_{1i}-F_{2i}|\notag\\
&\rightharpoonup \chi_\epsilon G_1'+(1-\chi_\epsilon)G_2'\leq \chi_\epsilon G_1+(1-\chi_\epsilon)G_2\, ,
\end{align}
where we have used from Lemma \ref{l:slope_lipschitz} that $|\partial_s \chi_\epsilon|$ is uniformly bounded, which proves the claim.
\end{proof}

Let us write down the primary application of the above:

\begin{theorem}\label{t:minimal_upper_gradient}
Let $F\in L^2(P(X),\Gamma_m)$, then there exists a unique upper $s$-parallel gradient $G\in L^2(P(X),\Gamma_m)$ such that for any other upper $s$-parallel gradient $G'$ we have that $G\leq G'$ $a.e.$  Further, there exists a sequence of cylinder functions $F_i\to F$ such that $|\partial_s F_i|\to G$ strongly.
\end{theorem}
\begin{proof}
To prove the first statement we note that since the set of upper $s$-parallel gradients is a closed convex subset by Lemma \ref{l:upper_grad_prop1}, there exists an element $G$ with minimal $L^2$ norm.  If $G'$ is any other upper gradient, then since $\min\{G',G\}$ is also a upper gradient, we must have that $\min\{G,G'\}=G$ a.e. 

The second statement is a standard application of Mazur's theorem.  Namely, let $F_i\to F$ be any sequence such that $|\partial_s F_i|\rightharpoonup G$.  By Mazur's theorem we can find convex combinations $\sum_{j=i}^{N(i)} c^i_j |\partial_s F_j|$ with $c^i_j\stackrel{j\to\infty}{\longrightarrow}0$ which converge strongly to $G$.  In particular if we define the new sequence
\begin{align}
F'_i\equiv \sum_{j=i}^{N(i)} c^i_j F_j\, ,
\end{align}
then clearly $F'_i\to F$ strongly still, while by using Theorem \ref{t:slope_properties} we have that
\begin{align}
|\partial_s F'_i| \leq \sum_{j=i}^{N(i)} c^i_j |\partial_s F_j|\to G\, .
\end{align}
Because $G$ is minimal, we also have that $\liminf |\partial_s F'_i|\geq G$, and hence $|\partial_s F'_i|\to G$.
\end{proof}

Using the above we can make rigorous the notion of the parallel gradient.

\begin{definition}\label{d:parallel_gradient}
Given $F\in L^2(P(X),\Gamma_m)$ we define the $s$-parallel gradient $|\nabla_s F|$ as the unique minimal upper parallel gradient of $F$ as in Theorem \ref{t:minimal_upper_gradient}.
\end{definition}

Let us observe the following simple estimate:

\begin{lemma}\label{l:slope_grad_est}
For any cylinder function $F\in\Cyl(X)$ we have that the gradient satisfies the estimate
\begin{align}
|\nabla_s F|\leq |\partial_s F|\, ,
\end{align}
\end{lemma}
\begin{proof}
Note that $|\partial_s F|$ is the upper $s$-parallel gradient obtained by taking the constant sequence $F_i\equiv F$.  
\end{proof}

Now using Theorem \ref{t:slope_properties} we immediately have the following important properties of the parallel gradients:

\begin{theorem}\label{t:parallel_grad_properties}
The following properties hold for the $s$-parallel gradients:
\begin{enumerate}
\item (Convexity) Let $F,G\in L^2(P(X),\Gamma_m)$, then we have the convexity estimate
\begin{align}
|\nabla_s (F+G)|\leq |\nabla_s F|+|\nabla_s G|\, .
\end{align}

\item (Strongly Local) If $F,G\in L^2(P(X),\Gamma_m)$ with $F=\text{const}$ on a neighborhood of the support of $G$, then
\begin{align}
|\nabla_s (F+G)| = |\nabla_s F|+|\nabla_s G|\, .
\end{align}

\item (Stability under Lipschitz Calculus) If $F\in L^2(P(X),\Gamma_m)$ and $\phi:\dR\to\dR$ is lipschitz, then
\begin{align}
|\nabla_s\big( \phi\circ F\big)|\leq ||\phi||_{Lip}\cdot |\nabla_s F|\, .
\end{align}

\item (Leibnitz) If $F,G\in L^2(P(X),\Gamma_m)$ then we have the estimate
\begin{align}
|\nabla (F\cdot G)|_{H^1}\leq |F|\cdot |\nabla G|_{H^1}+ |G|\cdot |\nabla F|_{H^1}\, .
\end{align}
\end{enumerate}
\end{theorem}

\subsubsection{Expressions of the Parallel Gradients}\label{sss:expressions_parallel}

In this Section we discuss a notational convention of the paper.  We will often consider expressions of the form
\begin{align}
\int_0^\infty |\nabla_s F|\,d\mu(s)\, ,
\end{align}
or
\begin{align}
\int_0^\infty |\nabla_s F|^2\,d\mu(s)\, ,
\end{align}
where $\mu(s)$ is a measure on $\dR^+$.  By definition we mean this to be the lower semicontinuous refinement of the corresponding slope expressions.  That is, in the spirit of the previous section let us call $G$ a $(s,\mu)$-upper gradient for $F$ if there exists $F_j\to F$ with $\int_0^\infty |\partial_s F_j|\,d\mu(s)\rightharpoonup G'$ and such that $G\leq G'$.  Then following the verbatim techniques as the last Section we end up with the following:

\begin{theorem}
There exists for $F\in L^2(P(X),\Gamma_m)$ a unique $(s,\mu)$-upper gradient, which we denote by $\int |\nabla_s F|\,d\mu(s)$, such that for any other $(s,\mu)$-upper gradient $G$ we have that $\int_0^\infty |\nabla_s F|\,d\mu(s)\leq G$ a.e.
\end{theorem}

The above defines for us the $(s,\mu)$-parallel gradient of $F$.  A verbatim statement may be made for $\int_0^\infty |\nabla_s F|^2\,d\mu(s)$.

\subsubsection{The $L_s$-Laplace operator}
We end this Section with the following construction of the Dirichlet energy associated with the parallel gradients and a listing of its basic properties, most of which are immediate from Theorem \ref{t:parallel_grad_properties}.  We begin with a definition:

\begin{definition}
Let $\cD(E_s)\subseteq L^2(P(X),\Gamma_m)$ be the subset of functions $F$ with upper $s$-parallel gradients.  We define the path space energy functional $E_s:\cD(E_s)\to \dR$ by
\begin{align}
E_s[F]\equiv \int_{P(X)} |\nabla_s F|^2\,d\Gamma_x\, .
\end{align}
\end{definition}

As an immediate consequence of Theorem \ref{t:parallel_grad_properties} we have the following:

\begin{theorem}\label{t:sDirichlet_form_pathspace}
The energy function $E_s:\cD(E_s)\to \dR$ is convex, nonnegative, $2$-homogeneous and lower-semicontinuous.  Furthermore, the following hold:
\begin{enumerate}
\item (closed) The functional $||F||_s\equiv \sqrt{||F||_{L^2}+E_s(F)}$ defines a complete norm on $\cD(E_s)$.
\item (stability under lipschitz calculus) Given a $1$-lipschitz function $\phi:\dR\to\dR$ with $\phi(0)=0$ we have that $E_s[\phi\circ F]\leq E_s[F]$.
\item (strongly local) If $F,G\in \cD(E)$ are such that $G$ is a constant on $\text{supp}(F)\subseteq P(X)$, then $E(F+G) = E(F)+E(G)$.
\end{enumerate}
\end{theorem}

Now we can apply standard techniques from the theory of convex functionals on Hilbert spaces \cite{Fukushima_DirichletForms} to build a laplace operator on path space associated to the $s$-parallel gradients.  Namely, we can define the subgradient of $E_s$ at a point in the usual manner
\begin{align}
\partial E_s [F]\equiv \{G:E_s(F)+\langle G, H-F\rangle\leq E_s(H)\text{ for every }H\in L^2(P(X),\Gamma_m)\}\, .
\end{align}

Theorem \ref{t:sDirichlet_form_pathspace} tells us, among other things, that the set $\partial E_s[F]$ is a convex subset, and thus there exists a unique element of minimal $L^2$ norm, which we define as the gradient $\nabla E_s[F]\equiv L_s F$.  Using the standard theory of convex functionals on a Hilbert space we therefor obtain the following:

\begin{theorem}
There exists a densely defined operator $L_{s}:\cD(L_{s})\subseteq L^2(P(X),\Gamma_m)\to\dR$ .
\end{theorem}

Let us remark that if $F$ is $\cF^{s-}$-measurable, then we have $L_sF=0$.

\subsection{The Parallel Gradient on a Smooth Manifold}\label{ss:parallel_grad_comparison}

Let us now address the issue of the parallel gradient on a smooth metric-measure space.  In particular, we will see that the $s$-parallel gradient as defined in Section \ref{ss:parallel_gradient} and as defined in this Section agree.  For the sake of this Section let us denote by $|\nabla_s F|^*$ the $s$-parallel gradient as defined in Section \ref{ss:parallel_gradient}.  Then, we prove the following:

\begin{theorem}\label{t:parallel_gradient_smooth_vs_nonsmooth}
Let $(X,d,m)\equiv (M^n,g,e^{-f}dv_g)$ be a smooth metric-measure space with $F\in L^2(P(M),\Gamma_m)$.  Then for a.e. $\gamma\in P(M)$ we have that $|\nabla F|(\gamma)\equiv |\nabla F|^*(\gamma)$.
\end{theorem}

In fact, the main estimate of this Section will be to see that for a smooth cylinder function $F$ that
\begin{align}
|\partial_s F| = |\nabla_s F|^*\, ,
\end{align}
where $|\partial_s F|$ is the parallel slope as defined in Section \ref{ss:parallel_slope}.  To see this requires several steps.  To begin with, let us use Theorem \ref{t:parallel_smooth_rect} and Theorem \ref{t:parallel_norm_rect} in order to see the following connection between the smooth parallel gradient and the parallel slope along a piecewise geodesic:

\begin{lemma}\label{l:parallel_gradient_smooth_vs_nonsmooth}
Let $F$ be a smooth cylinder function and $\gamma\in P(M)$ a piecewise geodesic in $M$, then we have that
\begin{align}
|\nabla_s F|^*(\gamma) = |\partial_s F|(\gamma)\, .
\end{align}
\end{lemma}
\begin{proof}
Let us first observe by Theorem \ref{t:parallel_smooth_rect} we have that for a piecewise geodesic $\gamma\in P(M)$ that if $v\in T_{\gamma(s)}M$ is a vector and $s_j\to 0$ is any sequence, then the variation $V\equiv\{V_j\}$ defined by $V_j(t)=0$ for $t<s$ and $V_j(t)\equiv \exp_{\gamma(t)}(s_j P_{t}^{-1}P_s v)$ is a $s$-parallel variation of $\gamma$, where $P_t$ is the usual parallel translation map.  In particular, we have by (\ref{e:parallel_slope_smooth}) and (\ref{e:variation_directional_derivative}) that
\begin{align}
\sup_{V_s}\{|D_{V_s}F|(\gamma):V_s\text{ is a $s$-parallel variation}\}\geq |\nabla_s F|^*(\gamma)\, .
\end{align}

Conversely, let $V$ be a $s$-parallel variation of a piecewise geodesic $\gamma$.  Then by again by Theorem \ref{t:parallel_smooth_rect} we have that there exists a Cauchy sequence $v'_j\in T_{\gamma(s)} M$ with $v'_j\to 0$ such that if we consider the $s$-parallel variation $V'$ given by $V'_{j}(t)=\exp_{\gamma(t)}(P^{-1}_tP_s v'_j)$, then $V'$ is equivalent to $V$ in the sense of Definition \ref{d:variation_equivalence}.  In particular, we get easily from this, (\ref{e:parallel_slope_smooth}), and (\ref{e:variation_directional_derivative}) the reverse inequality from above, and hence
\begin{align}
\sup_{V_s}\{|D_{V_s}F|(\gamma):V_s\text{ is a $s$-parallel variation}\} = |\nabla_s F|^*(\gamma)\, .
\end{align}

Now the above holds for any piecewise geodesic.  In particular, if we fix a piecewise geodesic $\gamma$ and a partition $\bt\in \Delta[0,\infty)$ we can apply the above to the $\bt$-approximation $\gamma_\bt$ of $\gamma$ to obtain 
\begin{align}
\sup_{V_s}\{|D_{V_s}F|(\gamma_\bt):V_s\text{ is a $s$-parallel variation}\} = |\nabla_s F|^*(\gamma_\bt)\, .
\end{align}
Now in the case of a piecewise geodesic $\gamma$, or indeed any piecewise smooth curve, we have that 
\begin{align}
\lim_{\bt\to\Delta}|\nabla_s F|^*(\gamma_\bt)\to |\nabla_s F|^*(\gamma)\, .
\end{align}
In particular, combining this with the above gives
\begin{align}
\limsup_{\bt\to \Delta}|D_{V_s}F|(\gamma_\bt) = \limsup_{\bt\to\Delta}|\nabla_s F|^*(\gamma_\bt)=\lim_{\bt\to\Delta}|\nabla_s F|^*(\gamma_\bt)=|\nabla_s F|(\gamma)\, ,
\end{align}
which proves the Lemma.
\end{proof}

On a smooth curve $\gamma$ in $M$ with $v\in T_{\gamma(0)}M$ it is completely clear that if one considers any sequence of $\bt$-approximations $\gamma_\bt\to\gamma$ which converge to $\gamma$, then the parallel vector fields $P_t^{-1}v$ along $\gamma_\bt$ converge uniformly to $P_t^{-1}v$ along $\gamma$.  Much less clear apriori is that there is a set of curves $\gamma\in P(M)$ of full measure along with vector fields $V(t)$ along these curves such that if one again considers any sequence of $\bt$-approximations $\gamma_\bt\to\gamma$ which converge to $\gamma$, then the parallel vector fields $P_t^{-1}v$ along $\gamma_\bt$ converge uniformly to $V$ along $\gamma$.  Given this, it is maybe less surprising that $V$ agrees with the stochastic parallel translation $P_t^{-1}v$ of $v$ along $\gamma$.  This is a key point in the proof of Theorem \ref{t:parallel_gradient_smooth_vs_nonsmooth}.  Precisely we have the following: 

\begin{proposition}\label{p:stochastic_parallel_translation_pointwise}
Let $(M^n,g,e^{-f}dv_g)$ be a smooth metric-measure space with $x\in M$.  Then for a.e. $\gamma\in P_x(M)$ and every $v\in T_xM$ we have that for every increasing dense sequence of partitions $\bt\in \Delta[0,\infty)$ that the sequence of vector fields $V_\bt\equiv P_t^{-1}v$ along the $\bt$-approximations $\gamma_\bt$ converges uniformly to the vector field $V\equiv P_t^{-1} v$ along $\gamma$, where $P_t$ is usual parallel translation map along $\gamma_\bt$ and the stochastic parallel translation map along $\gamma$.
\end{proposition}
\begin{proof}

Let us recall some basics of the construction of the stochastic parallel translation map.  In the construction one considers a sequence of increasingly dense partitions $\bt^m\subseteq \Delta[0,\infty)$ and the corresponding space $P^{\bt^m}_x(M)\subseteq P_x(M)$ of piecewise geodesics with vertices given by $\bt^m$.  On this space one can consider the horizontal lifting map $P^{\bt^m}_x(M)\to P_x(FM)$ to the frame bundle, which itself can be extended to a mapping $H^{\bt^m}:P_x(M)\to P_x(FM)$ by composing with the projection map $P_x(M)\to P_x^{\bt^m}(M)$ which takes a curve $\gamma$ to its $\bt^m$-approximation (which is unique away from a set of measure zero).  The basic result is that $H^{\bt^m}$ converges in measure to a mapping $H:P_x(M)\to P_x(FM)$, see \cite{Stroock_book}.  In particular, for every sequence of partitions $\bt^m$ then there exists a subsequence such that $H^{\bt^m}$ converges pointwise a.e. in $P_x(M)$.  The stochastic parallel translation map is nothing more than the identification of frames given by this lifting map. 

Now let us consider sequences of increasingly dense partitions $\bt^m\in \Delta[0,\infty)$ such that each element $t^m_a$ is rational.  The collection of all such sequences of rational partitions is itself a countable set.  To see this let $\bt^m$ be such a sequence and denote by $|\bt^m|\equiv N^m$.  Then we see that the sequence $\bt^m$ defines an element of $\dQ^{N^1}\times \dQ^{N^2}\times\cdots$, which is a countable set.  Therefore the collection of all such partitions is contained in the countable union of countable sets given by $\bigcup_{\vec N}\dQ^{N^1}\times \dQ^{N^2}\times\cdots$ with $\vec N\in \dN\times\dN\times\cdots$, and thus is itself countable.

Applying the above tells us that there is a set of full measure $\cS\subseteq P_x(M)$ such that for every sequence of rational partitions $\bt^m$ there exists a subsequence which converges pointwise on $\cS$ to the limit $H(\gamma)$.  We will now see that on $\cS$ we must therefore have the stronger statement that for {\it every} sequence of increasingly dense partitions $\bt^m$ we have that $H^{\bt^m}(\gamma)\to H(\gamma)$, without necessarily passing to a subsequence.  This will of course prove the Theorem.  

First let us apply a standard argument to make the following claim, namely that on $\cS$ we have that for every sequence of rational partitions $\bt^m$ that the sequence $H^{\bt^m}(\gamma)\to H(\gamma)$ converges, without passing to a subsequence.  Indeed, imagine this were not the case for some sequence $\bt^m$, then we can pick a subsequence $\bt^{',m}$ such that $d_{C^0}(H^{\bt^{',m}}(\gamma),H(\gamma))>\epsilon>0$ for all $m$.  However since $\bt^{',m}$ is itself a sequence of rational partitions, there exists a subsequence which contradicts this. 

Now let $\bt\in \Delta[0,\infty)$ be an arbitrary partition.  Note that because the evaluation maps are continuous and $\Gamma_x$ is a Borel probability measure that if $\bt'\to\bt$ then $e_{\bt'}\to e_{\bt}$ in measure.  In particular, for each $\epsilon>0$ we can find a rational partition $\bt'$ with $|\bt|=|\bt'|$ such that away from a set of measure $\epsilon$ we have for every curve $\gamma\in P_x(M)$ that $d_{C^0}(\gamma_\bt,\gamma_{\bt'})<\epsilon$, where $\gamma_\bt,\gamma_{\bt'}$ are the respective piecewise geodesic approximations of $\gamma$.  Applying this to $\delta>0$ sufficiently small gives us that  we can pick a rational partition such that away from a set of measure $\epsilon>0$ we have that $d_{C^0}(H^{\bt}(\gamma),H^{\bt'}(\gamma))<\epsilon$.

Now let $\bt^m$ be an arbitrary sequence of increasingly dense partitions.  Applying the previous paragraph for each $m$ tells us that we can find a sequence of rational partitions $\bt^{',m}$ with $|\bt^m|=|\bt^{',m}|$ such that for each $m$ we have that away from a set of measure $2^{-m}$ we have that for each $\gamma\in P_x(M)$ that 
$$
d_{C^0}(H^{\bt^m}(\gamma),H^{\bt^{',m}}(\gamma))<2^{-m}\, .
$$ 
In particular, away from a set of measure $0$ in $\cS$ we have that every curve satisfies the above for all but at most a finite number of $m$.  Since $H^{\bt^{',m}}(\gamma)\to H(\gamma)$ for every $\gamma\in \cS$ we must therefore have that away from a set of measure $0$ in $\cS$ that $H^{\bt^m}(\gamma)\to H(\gamma)$, as claimed.
\end{proof}

Our main application of the above is the following semi-continuity result, which will be the main lemma allowing us to prove the main Theorem of the subsection:

\begin{corollary}\label{c:parallel_grad_semicontinuity}
Let $(M^n,g,e^{-f}dv_g)$ be a smooth metric-measure space with $F$ a smooth cylinder function.  Then for a.e. $\gamma\in P(M)$ we have that
\begin{align}
|\nabla_s F|^*(\gamma) = \lim_{\bt\to \Delta} |\nabla_s F|^*(\gamma_{\bt})\, ,
\end{align}
where $\gamma_\bt$ is the $\bt$-approximation of $\gamma$.
\end{corollary}
\begin{remark}
Note that this does not claim that $|\nabla F|^*$ is a continuous function on $P(M)$, which need not be true.
\end{remark}
\begin{proof}
For a smooth cylinder function $F=e_\bt^*u$ let $\gamma_m$ be any sequence of curves with $\gamma_m\to \gamma$ in $P(M)$, and let $V_m$ be any sequence of vector fields with $V_m\to V$ converging uniformly to a vector field $V$ along $\gamma$.  Then in particular $V_m(\gamma_m(\bt))$ is converging uniformly to $V(\gamma(\bt))$ and thus
\begin{align}
D_{V_m} F(\gamma_{\bt^m})\to D_V F(\gamma)\, .
\end{align}

Now by Proposition \ref{p:stochastic_parallel_translation_pointwise} we have that for a.e. $\gamma\in P_x(M)$ and any sequence $\bt^m\in \Delta[0,\infty)$ with $v_m\to v\in T_xM$, we have that the $s$-parallel vector fields $V_m(t)$ defined by $V_m(t)=P_t^{-1}v_m$ for $s\leq t$ converge uniformly to the $s$-parallel vector field $V(t)$ defined by $V(t)\equiv P_t^{-1} v$ along $\gamma$.  In particular, let $v_m\equiv \frac{\nabla_s F(\gamma_{\bt^m})}{|\nabla_s F(\gamma_{\bt^m})|}$ if $\nabla_s F(\gamma_{\bt^m})\neq 0$, with $v_m\equiv 0$ otherwise.  By the first paragraph, and choosing an appropriate $\limsup$ subsequence so that $v_m\to v'\in T_xM$, we therefore have that 
$$
\limsup |\nabla_s F|^*(\gamma_{\bt^m})\leq D_{V'}F(\gamma)\leq |\nabla_s F|^*(\gamma)\, .
$$
On the other hand, let $v\equiv \frac{\nabla_s F(\gamma)}{|\nabla_s F(\gamma)|}$ if $\nabla_s F(\gamma)\neq 0$, with $v\equiv 0$ otherwise.  Then by considering the $s$-parallel vectorfields $V_m(t)\equiv P_t^{-1}v$ along $\gamma_{\bt_m}$ and again applying the first paragraph we have that
\begin{align}
|\nabla_s F|^*(\gamma)=\lim D_{V_m}F(\gamma_{\bt^m})\leq \liminf |\nabla_s F|^*(\gamma_{\bt^m})\, ,
\end{align}
where the $\liminf$ is obtained by passing to the appropriate subsequence.  Thus we have proved 
\begin{align}
|\nabla_s F|^*(\gamma)=\lim_m |\nabla_s F|^*(\gamma_{\bt^m})\, .
\end{align}
Since this held for any increasingly dense sequence $\bt^m$ this proves the Corollary.
\end{proof}

We are now in a position to prove the main Theorem of this subsection:

\begin{proof}[Proof of Theorem \ref{t:parallel_gradient_smooth_vs_nonsmooth}]
Let $F$ be a smooth cylinder function. In Lemma \ref{l:parallel_gradient_smooth_vs_nonsmooth} we proved the Theorem for any piecewise geodesic $\gamma\in P(M)$.  Now let $\gamma\in P(M)$ satisfy the conditions of Corollary \ref{c:parallel_grad_semicontinuity}.  For any such curve we therefore have
\begin{align}
|\partial_s F|(\gamma)=\limsup_{\bt\to \Delta}|\partial_s F|(\gamma_\bt)=\limsup_{\bt\to \Delta} |\nabla_s F|^*(\gamma_\bt) = |\nabla_s F|^*(\gamma)\, .
\end{align}
Since this is a set of full measure this finishes the Theorem.
\end{proof}

\section{Bounded Ricci Curvature on a Metric-Measure Space}\label{s:bounded_ricci_mms}

In this Section we introduce the notion of bounded Ricci curvature on a metric-measure space, and study some of its basic properties.  Specifically, let us recall from Section \ref{s:nonsmooth_bounded_ricci_intro} the following:

\begin{definition}\label{d:bounded_ricci_mms}
Let $(X,d,m)$ be a metric measure space which satisfies (\ref{e:mms_assumptions_intro}) and which is weakly Riemannian.  Then we say that $X$ is a $BR(\kappa,\infty)$ space if for every function $F\in L^2(P(X),\Gamma_m)$ we have the inequality
\begin{align}\label{e:bounded_ricci_mms}
|\Lip_x \int_X F\,d\Gamma_x|\leq \int_{P(X)} |\nabla_0 F|+\int_0^\infty \frac{\kappa}{2}e^{\frac{\kappa}{2}s}|\nabla_s F|\,d\Gamma_x\, ,
\end{align}
for $m-a.e.$ $x\in X$, where $|\Lip_x\cdot|$ is the lipschitz slope as in Section \ref{ss:weakly_riemannian} and the expression on the right hand side is as in Section \ref{sss:expressions_parallel} .
\end{definition}

One consequence is that if $F$ is a cylinder function on path space, then using Lemma \ref{l:slope_lipschitz} we see that the induced function $\int_{P(M)} F\,d\Gamma_x$ on $X$ is a lipschitz function.

Now having made rigorous sense of (\ref{e:bounded_ricci_mms}) in Section \ref{s:parallel_gradient_nonsmooth}, the goal of this Section is to prove the basic properties of such spaces.  We begin in Section \ref{ss:bounded_implies_lower} by proving that spaces with bounded Ricci curvature in the sense of (\ref{e:bounded_ricci_mms}) have Ricci curvature bounded from below in the sense of Bakry-Emery.  With the help of some estimates from this Section we prove in Section \ref{ss:stochastic_completeness} that a metric measure space $X$ with bounded Ricci curvature is stochastically complete, and therefore using the results of \cite{Ambrosio_BE_vs_LVS} we will see that spaces with bounded Ricci curvature are $RCD(-\kappa,\infty)$-spaces, and in particular have lower Ricci bounded from below in the sense of Lott-Villani-Sturm.  In Section \ref{ss:existence_parallel_variations} we discuss the relationship between bounded Ricci curvature and parallel translation invariant variations on $P(X)$.

\subsection{$BR(\kappa,\infty)$ $\implies$ Bakry-Emery}\label{ss:bounded_implies_lower}

In this Section we prove Theorem \ref{t:boundedricci_implies_BE}.  That is, we prove that a metric-measure space with Ricci curvature bounded by $\kappa$ has lower Ricci curvature bounded from below in the sense of Bakry-Emery.  The proofs follow the same moral lines as Theorems \ref{t:boundedricci_BE_implies_lowerricci}, \ref{t:boundedricci_BE2_implies_lowerricci} in Part \ref{part:smooth} of the paper, however there are technical issues that must be addressed in the nonsmooth cases.

We begin with the following important structural result.

\begin{lemma}\label{l:parallel_grad_comp}
Given a cylinder function $F(\gamma)\equiv u(\gamma(t))$ we have for every $0\leq s\leq t$ and $a.e.$ $\gamma\in P(X)$ that
\begin{align}
|\nabla_s F|(\gamma)\leq |\nabla u|(\gamma(t))\, .
\end{align}
\end{lemma}
\begin{proof}
Let us begin with the estimate
\begin{align}\label{e:exist_paral:1}
|\partial_s F|(\gamma)\leq |\Lip\, u|(\gamma(t))\, ,
\end{align}
for a.e. $\gamma$.  To see this let $V=\{V_j\}$ be a $s$-parallel variation of a piecewise geodesic $\gamma$ with $v=\{V_j(t)\}$ the associated variation of $\gamma(t)$, then we can compute
\begin{align}
|D_V F|&\equiv \lim\sup_j \frac{|F(V_j)-F(\gamma)|}{d(V_j(s),\gamma(s))} \notag\\
&= \lim\sup_j \frac{|u(V_j(t))-u(\gamma(t))|}{d(V_j(s),\gamma(s))}\notag\\
&=\lim\sup_j \frac{|u(V_j(t))-u(\gamma(t))|}{d(V_j(t),\gamma(t))}\notag\\
&= |D_v u|(\gamma(t))\leq |\Lip\, u|(\gamma(t))\, ,
\end{align}
where we have used Theorem \ref{t:parallel_norm_rect} in the third line.  Since the variation $V$ and piecewise geodesic $\gamma$ was arbitrary we get the claimed estimate $|\partial_s F|(\gamma)\leq |\Lip\, u|(\gamma(t))$.  By definition $|\nabla_s F|$ is the lower semicontinuous refinement of $|\partial_s F|$, and using \cite{Ambrosio_Calculus_Ricci} we have similarly that $|\nabla u|$ is the lower semi-continuous refinement of $|Lip\, u|$.  Thus, we can let $u_a$ be a sequence of lipschitz functions on $X$ such that
\begin{align}
&u_a\to u \text{ in } L^2(X,m)\, ,\notag\\
&|\Lip\, u_a|\to |\nabla u| \text{ in } L^2(X,m)\, .
\end{align}

Recall from Section \ref{s:parallel_gradient_nonsmooth} that to define $|\nabla_s F|$ we consider all sequences of cylinder functions $F_a\to F$, which converge in $L^2(P(X),\Gamma_m)$ to $F$, and then we consider upper gradient defined as the weak limit $G\equiv \lim |\partial_s F_a|$, if it exists.   We have from Theorem \ref{t:minimal_upper_gradient} that $|\nabla_s F|$ is the unique upper gradient with $|\nabla_s F|(\gamma)\leq G(\gamma)$ for every other weak upper gradient.  In particular, consider the sequence of cylinder functions $F_a\equiv u_a(\gamma(t))$.  Then after passing to subsequences we have by using (\ref{e:exist_paral:1}) that for $a.e.$ $\gamma\in P(X)$ 
\begin{align}
|\nabla_s F|(\gamma)\leq \lim |\partial_s F_a|(\gamma) \leq \lim |\Lip\, u_a|(\gamma(t)) = |\nabla u|(\gamma(t))\, ,
\end{align}
for a.e. $\gamma\in P(X)$, which proves the Lemma.
\end{proof}

This enables us to take the first step and prove a strong version of Theorem \ref{t:boundedricci_implies_BE}.2, which also proves Theorem \ref{t:br_basic_properties}.1:

\begin{theorem}
Let $(X,d,m)$ be a $BR(\kappa,\infty)$ space, then for every lipschitz $u\in L^2(X,m)$ we have that
\begin{align}
|\nabla H_t u|(x)\leq |\Lip\, H_t u|(x)\leq e^{\frac{\kappa}{2}t}H_t|\nabla u|\, ,
\end{align}
for a.e. $x\in X$.  In particular, using the above for $t=0$ gives the equality $|\Lip\, u|(x) = |\nabla u|(x)$ for a.e. $x\in X$ and proves Theorem \ref{t:br_basic_properties}.1.
\end{theorem}

\begin{proof}

Let us begin with the $t=0$ case by applying (\ref{e:bounded_ricci_mms}) to the test function $F(\gamma)\equiv u(\gamma(0))$.  Note that
\begin{align}\label{e:bounded_implies_lower:1}
\int_{P(X)} F\,d\Gamma_x = u(x)\, ,
\end{align}
and hence by (\ref{e:bounded_ricci_mms}) we have the estimate
\begin{align}
|\nabla u|\leq |\Lip\, u|\leq \int_{P_x(X)} |\nabla_0 F|\,d\Gamma_x\, ,
\end{align}
for a.e. $x\in X$.  Using Lemma \ref{l:parallel_grad_comp} therefore gives us for a.e. $x\in X$ that
\begin{align}
|\nabla u|\leq |\Lip\, u|\leq \int_{P(X)} |\nabla u|(\gamma(0))\, d\Gamma_x=|\nabla u|(x)\, .
\end{align}
In particular we see that $X$ is an almost Riemannian space, which proves Theorem \ref{t:br_basic_properties}.1.

Now we will apply (\ref{e:bounded_ricci_mms}) to the test function $F(\gamma)\equiv u(\gamma(t))$ where $t\geq 0$ in order to prove the remaining part of the Theorem.  Note that
\begin{align}\label{e:bounded_implies_lower:1}
\int_{P(X)} F\,d\Gamma_x = \int_X u(y)\rho_t(x,dy)=H_tu(x)\, ,
\end{align}
and hence by (\ref{e:bounded_ricci_mms}) we have the estimate
\begin{align}
|\nabla H_tu|(x)\leq |\Lip\, H_t u|(x)=|\Lip\, \int_{P(X)} F\,d\Gamma_x|\leq \int_{P(X)} |\nabla_0 F|+\int_0^\infty \frac{\kappa}{2}e^{\frac{\kappa}{2}s}|\nabla_s F|\,d\Gamma_x\, .
\end{align}
Now as in Lemma \ref{l:parallel_grad_comp} we further have for a.e. $x\in X$ the estimate
\begin{align}
\int_{P(X)} |\nabla_0 F|+\int_0^\infty \frac{\kappa}{2}e^{\frac{\kappa}{2}s}|\nabla_s F|\,d\Gamma_x &\equiv \inf_{F_j\to F}\int_{P(X)} |\partial_0 F_j|+\int_0^\infty \frac{\kappa}{2}e^{\frac{\kappa}{2}s}|\partial_s F_j|\,d\Gamma_x\notag\\
&\leq \inf_{u_j\to u}\int_{P(X)} |\partial_0 e_t^*u_j|+\int_0^\infty \frac{\kappa}{2}e^{\frac{\kappa}{2}s}|\partial_s e_t^*u_j|\,d\Gamma_x\, ,
\end{align}
where in the second infinimum we are considering all sequences $u_j\in L^2(X,m)$ be such that $u_j\to u$.  Then as in Lemma \ref{l:slope_lipschitz} we have for $s>t$ that $|\partial_s e_t^*u_j|=0$, while for $s\leq t$ we have the estimate $|\partial_s e_t^* u_j|\leq |\Lip\, u|(\gamma(t))$.  Plugging this in and using that we have now proved that $X$ is almost Riemannian yields for a.e. $x\in X$ that
\begin{align}
|\nabla H_tu|= |\Lip\, H_t u|&\leq \inf_{u_j\to u} \int_{P(X)}|\Lip\,u_j|(\gamma(t))+\int_0^\infty \frac{\kappa}{2}e^{\frac{\kappa}{2}s}|\Lip\,u_j|(\gamma(t))\,d\Gamma_x\, ,\notag\\
&\leq e^{\frac{\kappa}{2}t} H_t|\Lip\, u| =e^{\frac{\kappa}{2}t}H_t |\nabla u|\, ,
\end{align}
which proves the Theorem.
\end{proof}

To prove Theorem \ref{t:boundedricci_implies_BE}.3 and Theorem \ref{t:boundedricci_implies_BE}.4 we rely on the following corollary of Theorem \ref{t:boundedricci_implies_BE}.2:

\begin{lemma}\label{l:heat_estimates}
Let $X$ be a $BR(\kappa,\infty)$ space, then the following estimates hold for every $u\in W^{1,2}(X,m)$:
\begin{enumerate}
\item $|\nabla H_t u|^2\leq e^{\kappa t}H_t|\nabla u|^2$.
\item $H_s|\nabla H_{t-s}u|^2\leq e^{\kappa (t-s)}H_t|\nabla u|^2$.
\item $H_s|\nabla H_{t-s}u|^2\geq e^{-\kappa s}|\nabla H_t u|^2$.
\end{enumerate}
\end{lemma}
\begin{proof}
The first estimate is simply an application of H\"olders inequality.  That is, by Theorem \ref{t:boundedricci_implies_BE}.2 we have
\begin{align}
|\nabla H_t u|^2&\leq \bigg(e^{\frac{\kappa}{2}t}H_t|\nabla u|\bigg)^2\notag\\
&=e^{\kappa t}\bigg(\int_X |\nabla u|\rho_t(x,dy)\bigg)^2\leq e^{\kappa t}\int_X\rho_t(x,dy)\int_X |\nabla u|^2 \rho_t(x,dy)\notag\\
&\leq e^{\kappa t} H_t|\nabla u|^2\, ,
\end{align}
for a.e. $x\in X$.  There is only a little care needed since recall we have not yet proved $X$ is stochastically complete.  That is, the heat kernel is not {\it apriori} a probability measure.  However we always have the estimate $\int_X \rho_t(x,dy)\leq 1$, which is sufficient.  For $(2)$ we directly apply $(1)$ to the function $H_{t-s}u$.  Similarly, for $(3)$ we use $(1)$ and the convolution property of the heat flow to conclude
\begin{align}
|\nabla H_t u|^2 = |\nabla H_sH_{t-s}u|^2\leq e^{\kappa s}H_s|\nabla H_{t-s}u|^2\, ,
\end{align}
for a.e. $x\in X$,as claimed.
\end{proof}

Now we can prove Theorem \ref{t:boundedricci_implies_BE}.3-\ref{t:boundedricci_implies_BE}.5:

\begin{proof}[Proof of Theorem \ref{t:boundedricci_implies_BE}.3-\ref{t:boundedricci_implies_BE}.5]

The argument follows closely the original homotopy argument of \cite{BakryEmery_diffusions}, however a key point from \cite{Ambrosio_Calculus_Ricci} is the ability to identify the energy measure $[u]$ of a function $u\in W^{1,2}(X,m)$ with the Dirichlet energy.  Specifically, let $u\in W^{1,\infty}(X,m)\cap \cD(\Delta_X)$ and $t>0$ be fixed.  Then for $s\in [0,t]$ we can consider the family of functions
\begin{align}
H_s(H_{t-s}u)^2\, .
\end{align}
The restriction on $u$ gives us that the family of functions is differentiable in $s$ and a simple computation gives that
\begin{align}
\frac{d}{ds}H_s(H_{t-s}u)^2=\frac{1}{2}H_s\bigg(\Delta_X(H_{t-s}u)^2-2H_{t-s}u\Delta_X H_{t-s}u\bigg)\, .
\end{align}
Now by Theorem \ref{t:energy_measure} we can then write 
\begin{align}
\frac{d}{ds}H_s(H_{t-s}u)^2 = H_s[H_{t-s}u] = H_s|\nabla H_{t-s}u|^2\, ,
\end{align}
which by integration gives us
\begin{align}
H_tu^2-\big(H_tu\big)^2 = \int_0^tH_s|\nabla H_{t-s}u|^2\, .
\end{align}
Now by applying Lemma \ref{l:heat_estimates} we get the estimates
\begin{align}
&H_tu^2-\big(H_tu\big)^2\leq \kappa^{-1}(e^{\kappa t}-1)H_t|\nabla u|^2\, ,\notag\\
&H_tu^2-\big(H_tu\big)^2\geq \kappa^{-1}(1-e^{-\kappa t})|\nabla H_t u|^2\, ,
\end{align}
which completes the proofs of Theorem \ref{t:boundedricci_implies_BE}.3 and Theorem \ref{t:boundedricci_implies_BE}.4.  To prove Theorem \ref{t:boundedricci_implies_BE}.5 one argues in the same way with respect to the family of functions $H_s\big(H_{t-s}u\ln H_{t-s}u\big)$.
\end{proof}

Now let us remark that the almost Riemannian property of Theorem \ref{t:br_basic_properties}.1 and Theorem \ref{t:boundedricci_implies_BE}.4 implies the strong Feller property of Theorem \ref{t:br_basic_properties}.3:

\begin{corollary}\label{c:feller}
If $(X,d,m)$ is a $BR(\kappa,\infty)$ space, then the heat flow to any $L^2$ function immediately makes the function become lipschitz.
\end{corollary}
\begin{remark}
Of course all that is being used in the above are the implied lower Ricci curvature bounds given by the previous results of the Section.
\end{remark}

\subsection{$BR(\kappa,\infty)$ $\implies$ $RCD(-\kappa,\infty)$ $\implies$ Lott-Villani-Sturm}\label{ss:stochastic_completeness}

In this Section we begin by proving Theorem \ref{t:br_basic_properties}.2, that a metric measure space with bounded Ricci curvature is stochastically complete.  We will get as a corollary, when combined with the results of Section \ref{ss:bounded_implies_lower}, that $X$ has the lower Ricci curvature bound $RCD(-\kappa,\infty)$.  In particular, $X$ has a lower Ricci curvature bound of $-\kappa$ in the sense of Lott-Villani-Sturm.  \\

We begin with the following

\begin{theorem}\label{t:stochastic_completeness}
Let $(X,d,m)$ be a $BR(\kappa,\infty)$ space.  Then the following, equivalent, conditions all hold:
\begin{enumerate}
\item For $x\in X$ and $t\geq 0$ we have that the heat kernel measures $\rho_t(x,dy)$ on $X$ are probability measures.  That is, $\int_X \rho_t(x,dy)=1$.
\item For $\Gamma_x$ a.e. $\gamma\in P_x(X^*)$ we have that the lifetime $T(\gamma)=\infty$ is not finite.
\item For each probability measure $\mu$ on $X$ we have that the corresponding diffusion measure $\Gamma_\mu$ is a probability measure on $P(X)$.
\end{enumerate}
\end{theorem}
\begin{remark}
We refer to Section \ref{sss:stochastic_completeness} for the terminology.
\end{remark}

\begin{proof}
Having proved Theorem \ref{t:boundedricci_implies_BE} in Section \ref{ss:bounded_implies_lower} we have in particular for every $u\in W^{1,2}(X,m)$ the estimate
\begin{align}\label{e:stoc_comp:1}
H_t u^2-\big(H_t u\big)^2 \leq \kappa^{-1}(e^{\kappa t}-1)H_t|\nabla u|^2\, .
\end{align}

If $x_0\in X$ is a fixed point let us denote by $d_0(x)\equiv d(x_0,x)$ the distance function to some point $x_0$.  Now let $u_r(y)\equiv\phi_r(d_0(y))$ where $\phi_r(s)$ is a cutoff function on $\dR$ with $\phi_r=1$ on $[-r,r]$, $\phi_r=0$ outside of $[-2r,2r]$, and $|\dot \phi_r|\leq r^{-1}$.  Therefore we have the pointwise estimate
\begin{align}
|\nabla u_r|(x)\leq |\Lip\, u_r|(x)\leq r^{-1}\, .
\end{align}

Now if we apply the test functions $u_r$ to (\ref{e:stoc_comp:1}) we obtain the estimate
\begin{align}
H_t u_r^2-\big(H_t u_r\big)^2 \leq \kappa^{-1}(e^{\kappa t}-1)r^{-2}\, .
\end{align}
Note the following estimates
\begin{align}
\lim_{r\to\infty} H_t u_r(x) = \int_X\rho_t(x,dy)=\rho_t(x,X)\, ,\notag\\
\lim_{r\to\infty} H_t u^2_r(x) = \int_X\rho_t(x,dy)=\rho_t(x,X)\, ,
\end{align}
and therefore by letting $r$ tend to infinity we get
\begin{align}
\rho_t(x,X)\bigg(1-\rho_t(x,X)\bigg)=0\, .
\end{align}
Finally, since $\rho_t(x,X)$ is continuous in time and $\rho_0(x,X)=1$ \cite{Fukushima_DirichletForms}, we have for all $t\geq 0$ and $x\in X$ that
\begin{align}
\int_X \rho_t(x,dy) = 1\, ,
\end{align}
as claimed.  To get $(2)$ we recall that the diffusion measure satisfies the condition 
\begin{align}
\rho_t(x,dy)\equiv e_{t,*}\Gamma_x\, .
\end{align}
Hence if $(2)$ failed then for some $t>0$ we would have that $\Gamma_x$, as a measure on $P(X)$, is not a probability measure, and hence $\rho_t(x,dy)$ is not a probability measure, which is a contradiction.  To get $(3)$ we just note that
\begin{align}
\Gamma_\mu =\int_X \Gamma_x\,d\mu(x)\, ,
\end{align}
and hence since we just argued that $\Gamma_x$ are all probability measures, then so is $\Gamma_\mu$ .
\end{proof}

Now by applying Theorem \ref{t:boundedricci_implies_BE}, Theorem \ref{t:br_basic_properties}.2 and \cite{Ambrosio_Ricci} we can immediately conclude Theorem \ref{t:boundedricci_implies_LVS}.

\subsection{Bounded Ricci Curvature and the Existence of Parallel Translation Invariant Variations}\label{ss:existence_parallel_variations}

On a smooth manifold $M$ if one were to pick a smooth curve $\gamma$ and a vector $v\in T_{\gamma(0)}M$ then there exists a parallel translation invariant vector field $V(t)$ along $\gamma$ for which $V(0)=v$.  More generally, the stochastic parallel translation map shows us this is still true for at least $a.e.$ continuous $\gamma\in P(M)$.  

One could pose a similar question on a metric-measure space $X$.  Given a continuous curve $\gamma\in P(X)$ and a variation $v$ of $\gamma(0)$, does there exist arbitrarily dense partitions $\bt$ such that there exists a parallel translation invariant variation $V$ of $\gamma_\bt$ such that $V(0)\equiv v$.  For a general metric-measure space there is likely quite uncommon.  However, in analogy with the smooth manifold case, we will see in this Section that at least for $a.e.$ $\gamma\in P(X)$ such an extension exists for $a.e.$ variation of $\gamma(0)$.  The key point is to make rigorous the meaning of $a.e.$ variation of $\gamma(0)$.  The only natural way to do this is to use $W^{1,2}(X,m)$ functions.  More precisely we prove Theorem \ref{t:br_basic_properties}.4:

\begin{theorem}
Given any $u\in W^{1,2}(X,m)$ let us consider the function $F(\gamma)\equiv u(\gamma(0))$.  Then for $a.e.$ $\gamma\in P(X)$ we have that 
\begin{align}
|\nabla u|(\gamma(0))=|\nabla_0 F|(\gamma(0))\, .
\end{align}
\end{theorem}

\begin{proof}[Proof of Theorem \ref{t:br_basic_properties}.4]
Note that for $F(\gamma)\equiv u(\gamma(0))$ we have that
\begin{align}
\int_{P(X)} F\, d\Gamma_x = u(x)\, ,
\end{align}
and by using Lemma \ref{l:slope_lipschitz} we have for $s>0$ the estimate
\begin{align}
|\nabla_s F|=0\, .
\end{align}
Thus by using (\ref{e:bounded_ricci_mms}) we have for $a.e.$ $x\in X$ the estimate
\begin{align}
|\nabla u|(x)\leq \int_{P(M)} |\nabla_0 F|\,d\Gamma_x\, .
\end{align}
However using Lemma \ref{l:parallel_grad_comp} we have for $a.e.$ $x\in X$ and a.e. $\gamma\in P_x(M)$ the pointwise estimate
\begin{align}
|\nabla_0 F|(\gamma)\leq |\nabla u|(x)\, .
\end{align}
Combining these yields for $a.e.$ $x\in X$ and a.e. $\gamma\in P_x(M)$ the estimate
\begin{align}
|\nabla u|(x)=|\nabla_0 F|(\gamma)\, ,
\end{align}
as claimed.
\end{proof}

\section{Ricci Curvature and Martingales on $P(X)$}\label{s:ricci_martingales}

In this Section we explore the relationship between Ricci curvature and martingales on $P(X)$.  In particular, we will see that there is a strong connection between both the lower and bounded Ricci curvature of a metric-measure space, and the regularity of a martingale on $P(X)$.  As was discussed in the introduction, in the smooth case we focued on martingales on the spaces $P_x(X)$ with respect to the measure $\Gamma_x$.  Because we are working with metric-measure spaces it is more convenient not to focus on based path space and instead study to full path space.  That is, we will instead study martingales on $P(X)$ with respect to the measure $\Gamma_m$.  In fact, there is little difference as if we take a continuous martingale $F^t$ on $P(X)$ with respect to $\Gamma_m$, then it is easy to check that its restriction to each $P_x(X)$ is a martingale on $P_x(X)$ with respect to $\Gamma_x$.  That is, we are in effect studying martingales on all the $P_x(X)$ simultaneously.  With a lower Ricci curvature bound in effect, and hence a certain regularity control on the Wiener measure, it is easy to revert back to the study of $P_x(X)$ martingales by the same methods, however we find it more natural to consider the full path space so that such assumptions are not apriori necessary.

To understand the relationship between martingales and ricci curvature let us recall the following.  The one parameter family of $\sigma$-algebras $\cF^t$ on $P(X)$ gives rise to a method for decomposing functions $F\in L^2(P(X),\Gamma_m)$ on path space into a one parameter family of functions $F^t$ by projecting $F$ to the closed subspace of $\cF^t$-measurable functions.  The family $F^t$ is a martingale, see Section \ref{ss:martingale_lowerricci} for a more general definition.  In general, it is a reasonable question to ask about the time-regularity of the evolution of the family $F^t$.  The first results of this Section are for lower Ricci curvature, and tell us that on a metric-measure space with a lower Ricci curvature bound that {\it every} martingale $F^t$ is pointwise continuous in time.  That is, for $\gamma\in P(X)$ we have that $F^t(\gamma)$ is a continuous function of time, see Section \ref{ss:martingale_lowerricci} for a precise statement and proof.

In Section \ref{ss:martingale_quad} we study more refined regularity properties of a martingale $F^t$ associated with {\it bounds} on the ricci curvature.  Specifically, recall that associated with every martingale is its quadratic variation $[F^t]$ and its infinitesmal $[dF^t]$.  The family $F^t$ is highly nondifferentiable in $t$, see Part \ref{part:smooth} for a discussion of this, and in essence $[dF^t]$ is the appropriate replacement for the time derivative of $F^t$.  The first connection between martingales and bounded Ricci curvature is given in Section \ref{ss:martingale_quad}, where Theorem \ref{t:ricci_quad_nonsmooth} is proved and it is shown that a bound on the Ricci curvature is {\it equivalent} to estimates on $[dF^t]$ by the parallel gradients of $F$.  In particular, a consequence of this is that for a nice martingale $F^t$, in particular those generated by cylinder functions, we have as a mapping $[0,\infty)\to L^2(P(X),\Gamma_m)$ that $F^t$ is $C^{\frac{1}{2}}$-H\"older continuous.

In Section \ref{ss:martingale_holder} we see that the results of Section \ref{ss:martingale_lowerricci} for spaces with lower Ricci curvature may be refined for spaces with bounded Ricci curvature.  Namely, as was discussed it is shown in Section \ref{ss:martingale_lowerricci} that for a general martingale $F^t$ that $F^t(\gamma)$ is a continuous function of time for $\gamma\in P(X)$.  In Section \ref{ss:martingale_holder} we refine this on spaces with bounded Ricci curvature and prove Theorem \ref{t:br_basic_properties}.4.  That is, for martingales $F^t$ generated by functions $F$ which are lipschitz with respect to the parallel gradient, in particular cylinder functions, we have that $F^t(\gamma)$ is $C^\alpha$-H\"older continuous for every $\alpha<\frac{1}{2}$.

\subsection{Lower Ricci Curvature and the Continuous Martingale Property}\label{ss:martingale_lowerricci}

Let us first recall from Section \ref{ss:function_spaces} that for a function $F\in L^1(P(X),\Gamma_m)$ on path space that the $\cF^t$-expectation of $F$ may be written
\begin{align}\label{e:expectation:1}
F^t(\gamma) = \int_{P(X)} F_{\gamma_t}(\sigma) \,d\Gamma_{\gamma(t)}\equiv \int_{P(X)} F(\gamma_{[0,t]}\circ\sigma) \,d\Gamma_{\gamma(t)}\, .
\end{align}
We call $F^t$ the martingale generated by $F$.  If $F\in L^2(P(X),\Gamma_m)$ then this is the same as the projection of $F$ to the closed subspace $\cF^t$-measurable functions, and we call $F^t$ a $L^2$-martingale.  More generally, a family $F^t$ of $\cF^t$-measurable functions in $L^1(P(X),\Gamma_m)$ is called a martingale if for each $t<T$ we have that $F^t=(F^T)^t$ a.e.  We will be especially interested in Theorem \ref{t:martingale_continuous_lowerricci} in the time regularity of such a family.  Because each $F^t$ is only defined up to a set of measure zero, let us begin with a formal definition that will be convenient when discussing this issue:
\begin{definition}
Let $X^t,\tilde X^t:P(X)\to \dR$ be stochastic processes, that is, each is a one parameter family of functions such that $X^t,\tilde X^t$ are $\cF^t$-measurable.  We say $X^t$ and $\tilde X^t$ are versions of each other if for each $t\geq 0$ we have that $X^t=\tilde X^t$ $\Gamma_m$-a.e.
\end{definition}

We will often be interested in the pointwise time regularity of a martingale $F^t$, and so by this we mean of some version.  An interesting and very general theorem, which in particular applies to all weakly Riemannian metric-measure spaces $X$, is that for every martingale $F^t$ there exists a {\it cadlag} version \cite{Kuo_book}.  That is, for $\gamma\in P(X)$ we have that $F^t(\gamma):[0,\infty)\to \dR$ is right continuous with left limits.  In general this is the most regularity one could hope for when studying a martingale on a general metric-measure space. 

It turns out that a good deal of stochastic analysis focuses on {\it continuous} martingales, and that they enjoy many additional structural properties not satisfied generally. That is, a continuous martingale is one such that there exists a version such that for each $\gamma\in P(X)$ we have that $F^t(\gamma)$ is a continuous function.  The following generalization of Theorem \ref{t:br_basic_properties}.3 is the continuous martingale property.  Namely, it tells us that on a space with a lower Ricci curvature bound, a martingale always has a continuous version:

\begin{theorem}\label{t:martingale_continuous_lowerricci}
Let $(X,d,m)$ be a $RCD(\kappa,\infty)$ space.  That is, let $X$ be a weakly Riemannian space with lower Ricci curvature bounded from below in the sense of \cite{Ambrosio_Ricci}.   If $F^t$ is a martingale on $X$, then there exists a continuous version of $F^t$.
\end{theorem}

To prove the above we begin with the following well understood estimates, see for instance \cite{LV_OptimalRicci}, \cite{Sturm_GeomMetricMeasSpace}, \cite{Ambrosio_Ricci}.

\begin{lemma}\label{l:heatkernel_continuity}
Let $(X,d,m)$ be a weakly Riemannian space with Ricci curvature bounded from below in the $RCD(-\kappa,\infty)$ sense.  Then the following hold:
\begin{enumerate}
\item For each $t\geq 0$ fixed we have that the mapping $\rho_t(\cdot,dy):X\to \cP_2(X)$ is $e^{\kappa t}$-lipschitz with respect 2-Wasserstein distance on $\cP_2(X)$.
\item The mapping $\rho_{\cdot}(\cdot,dy):\dR^+\times X\to \cP_2(X)$ is continuous with respect 2-Wasserstein distance on $\cP_2(X)$.
\end{enumerate}
\end{lemma}
\begin{proof}
The first estimate follows from two points.  First from \cite{Ambrosio_Calculus_Ricci} we see that for $x\in X$ fixed we have that $\rho_t(x,dy)$, as a function of $t$, is gradient flow of $\delta_x$ by the entropy functional with respect to the distance function $W_2$ on $\cP_2(X)$.  In particular, since the entropy functional is $-\kappa$-convex one has the contraction property
\begin{align}
W_2(\rho_t(x_0,dy),\rho_t(x_1,dy))\leq e^{\kappa t}W_2(\delta_{x_0},\delta_{x_1})=e^{\kappa t}d(x_0,x_1)\, ,
\end{align}
as claimed.

To prove the second claim note that the flow of $\rho_t(x,dy)$ in the $t$ variable, with $x$ held fixed, is always continuous since it may be identified with the gradient flow of the entropy functional.  Since $\cP_2(X)$ is a complete metric space it is enough to prove sequential continuity.  Thus let $(t_j,x_j)\to (t,x)$ and fix $\epsilon>0$.  By the previous statement we know for all $j$ sufficiently large that
\begin{align}
W_2\bigg(\rho_{t_j}(x,dy),\rho_{t}(x,dy)\bigg)<\frac{\epsilon}{2}\, .
\end{align}
Further by the first part of the Lemma we have for all $j$ sufficiently large that
\begin{align}
W_2\bigg(\rho_{t_j}(x_j,dy),\rho_{t_j}(x,dy)\bigg)<e^{\kappa t_j}d(x_j,x)<\frac{\epsilon}{2}\, .
\end{align}
Combining these two with a triangle inequality gives the desired result.
\end{proof}

Now we can finish the proof of Theorem \ref{t:martingale_continuous_lowerricci}:

\begin{proof}[Proof of Theorem \ref{t:martingale_continuous_lowerricci}]
Let us begin by assuming that $F=e_\bt^*u\in\Cyl(X)$ is a cylinder function and consider the martingale $F^t$ generated by $F$.  Note then that using (\ref{e:expectation:1}) we may explicitly write a version of $F^t$ by

\begin{align}\label{e:ricci_quad_proof:1}
F^t(\gamma) &= \int_{X^{|\bt|-k-1}} u(\gamma(t_1),\ldots,\gamma(t_k),y_{k+1},\ldots,y_{|\bt|})\rho_{t_{k+1}-t}(\gamma(t),dy_{k+1})\cdots\rho_{t_{|\bt|}-t_{|\bt|-1}}(y_{|\bt|-1},dy_{|\bt|})\notag\\
&= \int_X\bigg(\int_{X^{|\bt|-k-2}} u(\gamma(t_1),\ldots,\gamma(t_k),y_{k+1},\ldots,y_{|\bt|})\cdots\rho_{t_{k+2}-t_{k+1}}(y_{k+1},dy_{k+2})\bigg)\rho_{t_{k+1}-t}(\gamma(t),dy_{k+1})\, ,
\end{align}
where $t_{k+1}$ is the smallest element of the partition with $t<t_{k+1}$.  Note that by the Feller property of Corollary \ref{c:feller} that 
\begin{align}
v(y_{k+1})\equiv \int_{X^{|\bt|-k-2}} u(\gamma(t_1),\ldots,\gamma(t_k),y_{k+1},\ldots,y_{|\bt|})\cdots\rho_{t_{k+1}-t_k}(y_{k+1},dy_{k+2})\, ,
\end{align}
is a continuous function of $y_{k+1}$ with compact support.  Thus, it is reasonable to write
\begin{align}
F^t(\gamma)=\int_X v(y) \rho_{t_{k+1}-t}(\gamma(t),dy)\, .
\end{align}
Now let us fix $\gamma\in P(X)$ and view $F^t(\gamma)$ as a function of $t$.  Now by Lemma \ref{l:heatkernel_continuity} we have that $\rho_{t_{k+1}-t}(\gamma(t),dy)$ is continuous in $t$ as a family of measures with respect to the $2$-Wasserstein distance.  On the other hand, since $v(y)$ is a continuous function with compact support, it therefore follows that $\int_X v(y) \rho_{t_{k+1}-t}(\gamma(t),dy)$ is a continuous function of $t$.  In particular, for {\it every} $\gamma\in P(X)$ we have that $F^t(\gamma)$ is a continuous function of $t$, as claimed.

Now let $F^t$ be a martingale and let $T>0$ be fixed.  Note that there always exists a {\it cadlag} version \cite{Kuo_book}, which we assume we are working with.  Let us choose a sequence of cylinder functions $F_j\in \Cyl(X)$ such that $F_j\to F^T$ in $L^1(X,\Gamma_m)$.  Note that $G^t_j\equiv F_j^t-F^t$ is a martingale with $G^T_j\to 0$ in $L^1$.  Let $S^T_j\equiv \sup_{0,T}|G^t_j|$, then we can apply the Doob inequality \cite{Kuo_book} to get that
\begin{align}
\Gamma_m\big\{S^T_j\geq \epsilon\big\} \leq \frac{||G^T_j||_{L^1}}{\epsilon}\to 0\, .
\end{align}
Thus for a.e. $\gamma\in P(X)$ we have that $S^T_j(\gamma)\to 0$, and in particular that $F^t_j(\gamma)\to F^t(\gamma)$ uniformly.  Since $F^t_j(\gamma)$ are continuous, therefore so is $F^t(\gamma)$.
 
\end{proof}

An important corollary of Theorem \ref{t:martingale_continuous_lowerricci} is the following:

\begin{corollary}\label{c:quad_continuous_lowerricci}
Let $(X,d,m)$ be a $RCD(\kappa,\infty)$ space.  If $F^t$ is a martingale on $X$ and $[F^t]$ is its quadratic variation, then there exists a continuous version of $[F^t]$.  That is, for a.e. $\gamma\in P(X)$ we have that $[F^t](\gamma)$ is a continuous function of time.
\end{corollary}
\begin{proof}
Recall that if $F^t$ is {\it cadlag} then $[F^t]$ can be defined as the unique right continuous process with $[F^0]=0$, $(F^t)^2-[F^t]$ is a martingale, and such that the jumps of $[F^t]$ satisfy $\Delta [F^t]=\Delta (F^t)^2$.  In particular, if $F^t(\gamma)$ is continuous then we see that $[F^t](\gamma)$ is as well. 
\end{proof}

\subsection{Bounded Ricci Curvature and Quadratic Variation}\label{ss:martingale_quad}

Recall from Part \ref{part:smooth} of the paper that the quadratic variation of a $L^2$-martingale $F^t$ is defined by
\begin{align}\label{e:quad_variation}
[F^t]\equiv \lim_{\ell(\bt)\to 0}\sum \big(F^{t_{a+1}}-F^{t_a}\big)^2\, ,
\end{align}
where $\bt\in \Delta[0,t]$ is a partition of $[0,t]$ and $\ell(\bt)\equiv\sup|t_{a+1}-t_a|$.  Note as in \cite{Kuo_book} that this limit exists at least in measure.  From this one can define the infinitesmal quadratic variation by
\begin{align}
[dF^t]\equiv \lim\frac{\big(F^{t+s}-F^{t}\big)^2}{s}\, ,
\end{align}
which is nonnegative and exists for a.e. time.  The main result of this Section is the proof of Theorem \ref{t:ricci_quad_nonsmooth}.  More specifically, we will prove the following, which we will see is a slight generalization.

\begin{theorem}\label{t:ricci_quad_equivalence}
Let $(X,d,m)$ be a weakly Riemannian metric-measure space.  Then the following are equivalent:
\begin{enumerate}
\item For every $F\in L^2(P(X),\Gamma_m)$ we have the estimate $$|\nabla_x \int_{P(X)}F\,d\Gamma_x|\leq \int_{P(X)} |\nabla_0 F|+\int_0^{\infty}\frac{\kappa}{2}e^{\frac{\kappa}{2}s}|\nabla_{s} F|\,d\Gamma_{\gamma(t)}\, .$$
\item For every $F\in L^2(P(X),\Gamma_m)$ we have the estimate $$\sqrt{[dF^t]}(\gamma)\leq \int_{P(X)} |\nabla_t F|(\gamma_{[0,t]}\circ \sigma)+\int_t^{\infty}\frac{\kappa}{2}e^{\frac{\kappa}{2}(s-t)}|\nabla_{s} F|(\gamma_{[0,t]}\circ \sigma)\,d\Gamma_{\gamma(t)}\, .$$
\end{enumerate}
\end{theorem}

Let us first see that this implies Theorem \ref{t:ricci_quad_nonsmooth}.  
\begin{proof}[proof of Theorem \ref{t:ricci_quad_nonsmooth} given Theorem \ref{t:ricci_quad_equivalence}]

Let us first assume that $X$ is a $BR(\kappa,\infty)$ space.  Then we have that Theorem \ref{t:ricci_quad_equivalence}.1 holds, and hence that Theorem \ref{t:ricci_quad_equivalence}.2 holds.  Further, since we proved Theorem \ref{t:br_basic_properties}.1 in the previous Section we know that $X$ is an almost Riemannian space.  Hence we have proved one direction of Theorem \ref{t:ricci_quad_nonsmooth}.  

On the other hand,  if $X$ is an almost Riemannian space and Theorem \ref{t:ricci_quad_equivalence}.2 holds, then by Theorem \ref{t:ricci_quad_equivalence} we have that Theorem \ref{t:ricci_quad_equivalence}.1 holds.  Further, using that $X$ is almost Riemannian we have that $|\Lip_x \int_{P(X)}F\,d\Gamma_x|=|\nabla_x \int_{P(X)}F\,d\Gamma_x|$, and hence $X$ is a $BR(\kappa,\infty)$ space, which finishes the proof of Theorem \ref{t:ricci_quad_nonsmooth}.
\end{proof}

To prove Theorem \ref{t:ricci_quad_equivalence} we will focus on cylinder functions $F\in \Cyl(X)$.  From this one can extend in the usual fashion to $L^2(P(X),\Gamma_m)$.  The first point to address is to compute the quadratic variation of a martingale induced by a cylinder function.

\begin{lemma}\label{l:quad_var_cylinder}
Let $F=e_\bt^*u\in \Cyl(X)$ be a cylinder function, then for $t\geq 0$ we have that
\begin{align}
[dF^t](\gamma) = \big|\nabla_x \int_{P(X)} F_{\gamma_t}\,d\Gamma_{x}\big|^2(\gamma(t))\, , 
\end{align}
for a.e. $\gamma\in P(X)$.
\end{lemma}

\begin{proof}

Using (\ref{e:expectation:1}) we can compute that if $F=e_{\bt}^* u$ is a cylinder function then for all $t\geq 0$ and $s$ sufficiently small so is $F^{t+s}$ with
\begin{align}
F^{t+s}(\gamma) &= v_s(\gamma(t_1),\ldots,\gamma(t_{k}),\gamma(t+s)) \notag\\
&= \int_{X^{|\bt|-k-1}} u(\gamma(t_1),\ldots,\gamma(t_k),y_{k+1},\ldots,y_{|\bt|})\rho_{t_{k+1}-t-s}(\gamma(t+s),dy_{k+1})\cdots\rho_{t_{|\bt|}-t_{|\bt|-1}}(y_{|\bt|-1},dy_{|\bt|})\, ,
\end{align}
where $t_{k}$ is the smallest element of the partition $\bt$ with $t_{k}\leq t$.  From this we can write
\begin{align}
F^t(\gamma)\equiv \int_X v_s(\gamma(t_1),\ldots,\gamma(t_k),y)\rho_s(\gamma(t),dy)\, .
\end{align}
Now combining these and observing that 
$$v_0(\gamma(t_1)\ldots,\gamma(t_k),\gamma(t))=\int_{P(X)}F_{\gamma_t}\,d\Gamma_{\gamma(t)}\, ,$$ 
allows us to compute
\begin{align}
[dF^t](\gamma)&= \lim\int_{P(X)}\frac{\big(F^{t+s}-F^{t}\big)^2}{s}\,d\Gamma_{\gamma(t)}\notag\\ 
&=\lim\frac{1}{s}\int_X\bigg(v(\gamma(t_1),\ldots,\gamma(t_k),y_s)-\int_M v(\gamma(t_1),\ldots,\gamma(t_k),z_s)\rho_s(\gamma(t),dz_s)\bigg)^2\,\rho_s(\gamma(t),dy_s)\, ,\notag\\
&=\lim\frac{1}{s}\int_X\bigg(w(y_s)-\int_X w(z_s)\rho_s(\gamma(t),dz_s)\bigg)^2\,\rho_s(\gamma(t),dy_s)\, ,
\end{align}
where with $\gamma$ fixed we have written $w_s(y)\equiv v_s(\gamma(t_1),\ldots,\gamma(t_k),y)$.  To deal with the limit on the last line we see from \cite{Fukushima_DirichletForms} and Section \ref{ss:weakly_riemannian} that
\begin{align}
\lim\frac{1}{s}\int_X\bigg(w_s(y_s)-\int_X w_s(z_s)\rho_s(\gamma(t),dz_s)\bigg)^2\,\rho_s(\gamma(t),dy_s) = [w_0](\gamma(t))\, ,
\end{align}
where $[w_0]$ is the energy measure of $w_0$, see Section \ref{ss:weakly_riemannian}, and we have used that $w_s$ is differentiable in $s$ a.e because $t_{k+1}-t>0$.  Using \cite{Ambrosio_Ricci} we are therefore able to compute for $m$-a.e. $\gamma(t)$, and hence $\Gamma_m$-a.e. $\gamma\in P(X)$, that
\begin{align}
[dF^t](\gamma)=|\nabla w_0|^2(\gamma(t)) =|\nabla_x \int_{P(X)} F_{\gamma_t}\,d\Gamma_{x}|^2(\gamma(t))\, ,
\end{align}
as claimed.

\end{proof}

Now let us prove Theorem \ref{t:ricci_quad_equivalence}:

\begin{proof}[proof of Theorem \ref{t:ricci_quad_equivalence}]
Note that if $F\in \Cyl(X)$ is a cylinder function then by applying Lemma \ref{l:quad_var_cylinder} to $F$ at $t=0$ we have that $[dF^0]$ is a function on $X$ given by
\begin{align}
\sqrt{[dF^0]}(x) = |\nabla \int_{P(X)} F\,d\Gamma_x|\, .
\end{align}
In particular, we see that Theorem \ref{t:ricci_quad_equivalence}.2 at $t=0$ is equivalent to Theorem \ref{t:ricci_quad_equivalence}.1.  We need therefore to see that Theorem \ref{t:ricci_quad_equivalence}.1 implies Theorem \ref{t:ricci_quad_equivalence}.2 for $t>0$. 

Let us first use Lemma \ref{l:quad_var_cylinder} as well as the assumed bound on the Ricci curvature in order to conclude
\begin{align}
|\nabla \int_{P(X)} F_{\gamma_t}\,d\Gamma_{\gamma(t)}|\leq\int_{P(X)} |\nabla_0 F_{\gamma_t}|+\int_0^{\infty}\frac{\kappa}{2}e^{\frac{\kappa}{2}s}|\nabla_{s} F_{\gamma_t}|\,d\Gamma_{\gamma(t)} \, .
\end{align}
In the smooth case one could conclude $|\nabla_s F_{\gamma_t}|=|\nabla_{t+s} F|$ and therefore the main result.  In the nonsmooth case things are more subtle.  We clearly still have the estimate $|\partial_s F_{\gamma_t}|=|\partial_{t+s} F|$, however the procedure of taking the lower semi-continuous refinement may apriori destroy the equality of$|\nabla_s F_{\gamma_t}|=|\nabla_{t+s} F|$.  Thus, let us begin by seeing that the above implies the seemingly weaker inequality
\begin{align}\label{e:ricci_quad_weak}
\sqrt{[dF^t]}\leq\int_{P(X)} |\partial_t F|+\int_t^{\infty}\frac{\kappa}{2}e^{\frac{\kappa}{2}(s-t)}|\partial_s F|\,d\Gamma_{\gamma(t)} \, ,
\end{align}
for all cylinder functions $F$.  Recall that by definition
\begin{align}
|\nabla_t F|+\int_t^{\infty}\frac{\kappa}{2}e^{\frac{\kappa}{2}(s-t)}|\nabla_s F| = \liminf_{F_j\to F}\bigg( |\partial_t F_j|+\int_t^{\infty}\frac{\kappa}{2}e^{\frac{\kappa}{2}(s-t)}|\partial_s F_j|\bigg)\, ,
\end{align}
see Section \ref{s:parallel_gradient_nonsmooth}.  Now to finish let $F\in L^2(P(X),\Gamma_m)$ with $F_j\in L^2(P(X),\Gamma_m)$ cylinder functions be such that $F_j\to F$, and such that $|\partial_t F_j|+\int_t^{\infty}\frac{\kappa}{2}e^{\frac{\kappa}{2}(s-t)}|\partial_s F_j|\to |\nabla_t F|+\int_t^{\infty}\frac{\kappa}{2}e^{\frac{\kappa}{2}(s-t)}|\nabla_s F|$ strongly in $L^2$ and pointwise a.e, see Theorem \ref{t:minimal_upper_gradient}.  After passing to another subsequence we can assume by the usual methods that for a.e. $t\geq 0$ that $[dF^t_j]\to [dF^t]$ a.e.  Plugging this into (\ref{e:ricci_quad_weak}) gives for a.e. $\gamma\in P(X)$ that
\begin{align}
\sqrt{[dF^t]}\leq\int_{P(X)} |\nabla_t F|+\int_t^{\infty}\frac{\kappa}{2}e^{\frac{\kappa}{2}(s-t)}|\nabla_s F|\,d\Gamma_{\gamma(t)} \, ,
\end{align}
as claimed.
\end{proof}

\subsection{Bounded Ricci Curvature and H\"older Continuous Martingale Property}\label{ss:martingale_holder}

The main Theorem of this Section is to study the time regularity of a martingale on $P(X)$.  That is, we wish to prove Theorem \ref{t:br_basic_properties}.4, which is restated below:

\begin{theorem}\label{t:martingale_holder}
Let $X$ be a $BR(\kappa,\infty)$ space with $F\in L^2(P(X),\Gamma_m)$ a $\cF^T$-measurable function for some $T<\infty$, which satisfies the uniform lipschitz condition $\sup |\nabla_s F|< A$.  Then the induced martingale $F^t$ has a representative such that for a.e. $\gamma\in P(X)$ we have that $F^t(\gamma)$ is a $C^\alpha$-Holder continuous for all $\alpha<\frac{1}{2}$.
\end{theorem}

Note in particular that by Lemma \ref{l:slope_lipschitz} and Lemma \ref{l:slope_grad_est} that the above holds for all cylinder functions $F\in \Cyl(X)$.

The proof of the Theorem requires the various structure that has been built throughout Section \ref{s:ricci_martingales}.  Additionally we require the following H\"older version of the Kolmogorov continuity theorem.  

\begin{lemma}[Kolmogorov H\"older Continuity Theorem]\label{l:Kolmogorov_continuity}
Let $X^t:P(X)\to \dR$ be a {\it cadlag} process, and assume there exists $a,b,C>0$ such that
\begin{align}
\int_{P(X)} |X^t-X^s|^a \leq C|t-s|^{1+b}\, ,
\end{align}
then for a.e. $\gamma\in P(X)$ we have that $X^t(\gamma)$ is $C^\alpha$-H\"older continuous for every $\alpha<\frac{b}{a}$.
\end{lemma}

We refer the reader to \cite{Kuo_book} for a proof of the lemma.  Strictly speaking, a weaker version of the Kolmogorov continuity theorem is stated in \cite{Kuo_book}, however it can be checked that the proof actually gives the stated result.

Let us now prove Theorem \ref{t:martingale_holder}:

\begin{proof}[Proof of Theorem \ref{t:martingale_holder}:]

If $F\in L^2(P(X),\Gamma_m)$ then by Theorem \ref{t:ricci_quad_nonsmooth} we have that there exists $C(\kappa,A,T)\equiv e^{\kappa T}A^2$, where $\kappa$ is the Ricci bound of the metric-measure space, such that we have the uniform estimate

\begin{align}
[dF^t]\leq C\, ,
\end{align}
for all $t$ and a.e. $\gamma\in P(X)$.

Our main estimate required to prove the Theorem is the following.  We will show for each integer $k\in \dN$ that 
\begin{align}
\int_{P(X)}\big|F^t-F^s\big|^{2k}\,d\Gamma_m\leq \frac{(2k)!}{2^k} C^{k}\,|t-s|^k\, .
\end{align}
Then by applying lemma \ref{l:Kolmogorov_continuity} to all $a=2k$ and $b=k-1$ we will have proved the Theorem.  We will prove the estimate by induction, so let us begin with the $k=1$ base case.  Here we can use the defining properties of the quadratic variation to compute
\begin{align}
\int_{P(X)}\big|F^t-F^s\big|^{2}\,d\Gamma_m &= \int_{P(X)}[F^t]-[F^s]\,d\Gamma_m\notag\\
&=\int_s^t\int_{P(X)} [dF^u]\,d\Gamma_m \leq C|t-s|\, ,
\end{align}
as claimed.

Now to prove the induction step, let us fix $s<t$ and consider the martingale generated by $F^t-F^s$.  Thus for $u\in (s,t]$ we have that $\big(F^t-F^s\big)^u=F^u-F^s$ and for $u\leq s$ we have that $\big(F^t-F^s\big)^u=0$.  Further, we have for $u\in (s,t]$ that that the quadratic variation is given by $[\big(F^t-F^s\big)^u]=[F^u]$.  

Let us note from Theorem \ref{t:martingale_continuous_lowerricci} and Corollary \ref{c:quad_continuous_lowerricci} that both the martingale $\big(F^t-F^s\big)^u$ and its quadratic variation are continuous in time.  In particular, we can apply the Ito formula \cite{Kuo_book} with respect to a smooth function $f:\dR\to\dR$ to obtain 
\begin{align}
\int_{P(X)}f(F^t-F^s)\, d\Gamma_m=\frac{1}{2}\int_{P(X)} \int_s^tf''(F^u-F^s)[dF^u]\,d\Gamma_m\, .
\end{align}
Now assume we have proved the claim for some $k$, and we wish to show the claim for $k+1$.  Then applying the Ito formula to $f(x)\equiv x^{2(k+1)}$ we obtain
\begin{align}
\int_{P(X)} |F^t-F^s|^{2(k+1)}\,d\Gamma_m &= \frac{(2k+2)(2k+1)}{2}\int_{P(X)} \int_s^t|F^u-F^s|^{2k}[dF^u]\,d\Gamma_m\notag\\
&\leq \frac{(2k+2)(2k+1)}{2}C|t-s|\int_{P(X)}|F^t-F^s|^{2k}\,d\Gamma_m \notag\\
&\leq \frac{(2k+2)(2k+1)}{2}C^{k+1}\frac{(2k)!}{2^k} \,|t-s|^{k+1} = \frac{(2k+2)!}{2^{k+1}}C^{k+1}\, |t-s|^{k+1}\, ,
\end{align}
as claimed, which therefore finishes the proof of the Theorem.

\end{proof}

\section{Bounded Ricci Curvature and Analysis on Path Space}\label{s:bounded_ricci_analysis}

In this Section we introduce the Ornstein-Uhlenbeck operator on the path space of a metric-measure space, and discuss some of its properties.  Classically, on a smooth manifold, the Ornstein-Uhlenbeck operator is defined on based path space.  As in our discussion of martingales, it is better on nonsmooth spaces to begin by considering all the based path spaces simultaneously and to discuss the Ornstein-Uhlenbeck operator as an operator on $L^2(P(X),\Gamma_m)$.  Also as with the martingale case, once some basic structure is in place, for instance a lower Ricci curvature bound, one can argue by methods similar to this section to then consider the operator on each based path space individually if one is interested.

In Section \ref{ss:H1_grad_nonsmooth} we introduce the $H^1_0$-gradient on a general metric space, and prove some of its basic properties.  In particular, we show that on a smooth metric-measure space the definition agrees with the standard one.  Much of the work for this construction is analogous to Section \ref{s:parallel_gradient_nonsmooth} when we constructed variations and the parallel gradient.  In Section \ref{sss:OU_construction} we use the $H^1_0$-gradient to build a Dirichlet form on $L^2(P(X),\Gamma_m)$, which allows us to define the Ornstein-Uhlenbeck operator and discuss some of its basic points.  In Section \ref{ss:boundedricci_OU_nonsmooth} we show, as in the smooth case, that on a metric-measure space with bounded Ricci curvature one has a spectral gap and log-sobolev inequalities for the Ornstein-Uhlenbeck operators.

\subsection{The $H^1_0$-Gradient}\label{ss:H1_grad_nonsmooth}

In Section \ref{s:parallel_gradient_nonsmooth} we introduced the parallel gradient norm $|\nabla_s F|$ of a function on path space.  Following a structural pattern similar to Section \ref{s:parallel_gradient_nonsmooth} we begin in Section \ref{sss:H1_slope} by defining the $H^1_0$-slope of a cylinder function.  In Section \ref{sss:H10_gradient} we use the slope to define the $H^1_0$-gradient by taking the lower semicontinuous refinement.  We end the Section by showing that on a smooth metric-measure space the definition of the $H^1_0$-gradient given in this Section agrees with the standard definition.

\subsubsection{The $H^1_0$-Slope}\label{sss:H1_slope}
The definition of the parallel gradient began with the introduction of the parallel slopes $|\partial_s F|$ given in Section \ref{ss:parallel_slope}.  It is possible to use the structure of Section \ref{s:variation} to define the $H^1_0$-slope analogously by supping over a larger class of variations.  However, it will instead be more useful to use Proposition \ref{p:parallel_H1_relation} from Part \ref{part:smooth} of the paper as motivation for the definition.  Thus we define the $H^1_0$-slope of a cylinder function $F\in\Cyl(X)$ by
\begin{align}\label{e:H1_slope}
|\partial F|^2_{H^1_0}(\gamma)\equiv \int_0^\infty |\partial_s F|^2(\gamma)\, .
\end{align}

Before using this to define the $H^1_0$-gradient, let us explore some properties of the slope.  The first is a simple estimate which tells us that cylinder functions have bounded $H^1_0$-slope:

\begin{lemma}\label{l:H1_slope_lipschitz}
Let $F=e_\bt^*u\in \Cyl(X)$ be a cylinder function with $\bt\in\Delta[0,T]$ and $u\in Lip_c(X^{|\bt|})$.  Then we have that
\begin{align}
|\partial F|_{H^1_x}(\gamma)\leq \sqrt{|\bt|\, T}\cdot ||u||_{Lip}\, .
\end{align}
\end{lemma}

\begin{proof}
Using Lemma \ref{l:slope_lipschitz} we have for $0\leq s\leq T$ that $|\partial_s F|\leq \sqrt{|\bt|}\cdot||u||_{Lip}$, while for $s>T$ we have that $|\partial_s F|=0$.  Plugging this into (\ref{e:H1_slope}) gives the desired estimate.
\end{proof}

Now in analogy with Section \ref{ss:parallel_slope} we discuss the convexity properties of the $H^1_0$-slope.

\begin{theorem}\label{t:H1_slope_properties}
The following properties hold for the $H^1_0$-slope:
\begin{enumerate}

\item (Convexity) If $F,G\in \Cyl(X)$ are cylinder functions, then we have the convexity estimates
\begin{align}
&|\partial (F+G)|_{H^1_0}\leq |\partial F|_{H^1_0}+|\partial G|_{H^1_0}\, .
\end{align}

\item (Strongly Local) If $F,G\in \Cyl(X)$ are cylinder functions with $F=\text{const}$ on a neighborhood of the support of $G$, then
\begin{align}
|\partial (F+G)|_{H^1_0}= |\partial F|_{H^1_0}+|\partial G|_{H^1_0}\, .
\end{align}

\item (Stability under Lipschitz Calculus) If $F\in \Cyl(X)$ is a cylinder function and $\phi:\dR\to\dR$ is lipschitz, then
\begin{align}
|\partial\big( \phi\circ F\big)|_{H^1_0}\leq ||\phi||_{Lip}\cdot |\partial F|_{H^1_0}\, .
\end{align}

\item (Strong Convexity) If $F,G,\chi\in \Cyl(X)$ are cylinder functions with $0\leq \chi\leq 1$, then we have the pointwise convexity estimate
\begin{align}
|\partial(\chi F+(1-\chi)G)|_{H^1_0}\leq \chi|\partial F|_{H^1_0}+(1-\chi)|\partial G|_{H^1_0}+|\partial\chi|_{H^1_0}\cdot|F-G|\, .
\end{align}
\end{enumerate}
\end{theorem}
\begin{proof}
The first statement follows immediately from Theorem \ref{t:slope_properties}.1.  That is,
\begin{align}
|\partial (F+G)|_{H^1_x} &= \sqrt{\int_0^\infty |\partial_s(F+G)|^2}\leq \sqrt{\int_0^\infty \big(|\partial_s F|+|\partial_s G|\big)^2}\notag\\
&\leq \sqrt{\int_0^\infty|\partial_s F|^2}+\sqrt{\int_0^\infty|\partial_s G|^2}=|\partial F|_{H^1_x}+|\partial G|_{H^1_x}\, .
\end{align}
For the second statement we observe, as in the proof of Theorem \ref{t:slope_properties}.2, that for every $\gamma\in P(X)$ we have the stronger statement that either $|\partial_s (F+G)|=|\partial_s F|$ for every $s$ or $|\partial_s (F+G)|=|\partial_s G|$ for every $s$.  Combining with (\ref{e:H1_slope}) this gives the second statement.  The third and fourth statements follow immediately from Theorem \ref{t:slope_properties}.3, (\ref{e:H1_slope}) and arguing as above.
\end{proof}

\subsubsection{The $H^1_0$-gradient}\label{sss:H10_gradient}

Completely analogous to Section \ref{ss:parallel_gradient_nonsmooth} we now introduce the $H^1_0$-gradient.  In particular, the proofs of the corresponding theorems are almost verbatim, so we will not include all the details.

As in Section \ref{ss:parallel_gradient_nonsmooth} we will assume throughout this Section that $(X,d,m)$ is a weakly Riemannian space, so that we may equip path space $P(X)$ with the Wiener measure $\Gamma_m$.

\begin{definition}
Given $F\in L^2(P(X),\Gamma_m)$ we say that $G\in L^2(P(X),\Gamma_m)$ is a upper $H^1_0$-gradient for $F$ if there exists a sequence of cylinder functions $F_j\in \Cyl(X)$ such that $F_j\to F$ strongly in $L^2(P(X),\Gamma_m)$ and $|\partial F_j|_{H^1_0}\rightharpoonup G'$ weakly in $L^2(P(X),\Gamma_m)$ with $G'\leq G$ $a.e.$
\end{definition}

We will define the $H^1_0$-gradient of $F$ as the unique {\it minimal} upper parallel gradient for $F$.  First we must study some basic properties of the upper gradients, and in particular using Theorem \ref{t:H1_slope_properties} and arguing as in Section \ref{ss:parallel_gradient_nonsmooth} we arrive at the following:

\begin{lemma}\label{l:H1_upper_grad_prop1}
For $F\in L^2(P(X),\Gamma_m)$ the following hold:
\begin{enumerate}
\item The collection of upper $H^1_0$-gradients for $F$ is a closed convex subset of $L^2(P(X),\Gamma_m)$.
\item If $G_1$, $G_2$ are upper $H^1_0$-gradients for $F$ then so is $G(\gamma)\equiv \min\{G_1(\gamma),G_2(\gamma)\}$.
\end{enumerate}
\end{lemma}

The primary application of the above is the following, which is obtained by arguing as in Section \ref{ss:parallel_gradient_nonsmooth}

\begin{theorem}\label{t:H1_minimal_upper_gradient}
Let $F\in L^2(P(X),\Gamma_m)$, then there exists a unique upper $H^1_0$-gradient $G\in L^2(P(X),\Gamma_m)$ such that for any other upper $H^1_0$-gradient $G'$ we have that $G\leq G'$ $a.e.$  Further, there exists a sequence of cylinder functions $F_j\to F$ such that $|\partial F_j|_{H^1_0}\to G$ strongly.
\end{theorem}

Thus we can define the $H^1_0$-gradient:

\begin{definition}\label{d:H1_gradient}
Given $F\in L^2(P(X),\Gamma_m)$ we define its $H^1_0$-gradient $|\nabla F|_{H^1_0}$ as the unique minimal upper $H^1_0$-gradient of $F$ as in Theorem \ref{t:H1_minimal_upper_gradient}.
\end{definition}

Let us list some basic properties of the $H^1_0$-gradient.  These follow immediately from Theorem \ref{t:H1_slope_properties}, Theorem \ref{t:H1_minimal_upper_gradient} and the techniques of Theorem \ref{t:slope_properties}:

\begin{theorem}\label{t:H1_grad_properties}
The following properties hold for the $H^1_0$ gradient:
\begin{enumerate}
\item (Convexity) Let $F,G\in L^2(P(X),\Gamma_m)$, then we have the convexity estimate
\begin{align}
|\nabla (F+G)|_{H^1_0}\leq |\nabla F|_{H^1_0}+|\nabla G|_{H^1_x}\, .
\end{align}

\item (Strongly Local) If $F,G\in L^2(P(X),\Gamma_m)$ with $F=\text{const}$ on a neighborhood of the support of $G$, then
\begin{align}
|\nabla (F+G)|_{H^1_0} = |\nabla F|_{H^1_0}+|\nabla G|_{H^1_0}\, .
\end{align}

\item (Stability under Lipschitz Calculus) If $F\in L^2(P(X),\Gamma_m)$ and $\phi:\dR\to\dR$ is lipschitz, then
\begin{align}
|\nabla\big( \phi\circ F\big)|_{H^1_0}\leq ||\phi||_{Lip}\cdot |\nabla F|_{H^1_0}\, .
\end{align}

\item (Leibnitz) If $F,G\in L^2(P(X),\Gamma_m)$ then we have the estimate
\begin{align}
|\nabla (F\cdot G)|_{H^1_0}\leq |F|\cdot |\nabla G|_{H^1_0}+ |G|\cdot |\nabla F|_{H^1_0}\, .
\end{align}
\end{enumerate}
\end{theorem}

Let us remark on the following, which is an obvious consequence of the definition:

\begin{lemma}\label{l:H1_slope_grad_est}
For any cylinder function $F\in\Cyl(X)$ we have that the $H^1_0$-gradient satisfies the pointwise estimate
\begin{align}
|\nabla F|_{H^1_0}\leq |\partial F|_{H^1_0}\, .
\end{align}
\end{lemma}

We end this Section by seeing the $H^1_0$-gradient defined in this Section for a weakly Riemannian metric-measure space $X$ agrees with the standard definition when $X$ is a smooth metric-measure space.  Recall that on a smooth metric-measure space we have classically defined the $H^1_0$-gradient as living on based path space $P_x(M)$.  Therefore, it will be convenient notation in the next Theorem to denote by $|\nabla F|_{H^1_x}$ the $H^1_0$ gradient as originally defined in \cite{Malliavin_StocAnal} on based path space $P_x(M)$ of a smooth manifold.  Our result is the following:

\begin{theorem}\label{t:H10_grad_comparison}
Let $(X,d,m)\equiv (M^n,g,e^{-f}dv_g)$ be a smooth metric-measure space with $F\in L^2(P(M),\Gamma_m)$.  Then for a.e. $\gamma\in P(M)$ we have that $|\nabla F|_{H^1_0}(\gamma) = |\nabla F|_{H^1_{\gamma(0)}}(\gamma)$.
\end{theorem}
\begin{proof}
The majority of the work for this Theorem was accomplished in Section \ref{s:variation} and with Proposition \ref{p:parallel_H1_relation}.  Namely, let us assume $F\in L^2(P(X),\Gamma_m)$, then in this case we have that $|\partial_s F|=|\nabla_s F|$ for each $s$.  Hence we have that $|\nabla F|_{H^1_{\gamma(0)}}(\gamma)=|\partial F|_{H^1_0}(\gamma)$.  On the other hand, by Lemma \ref{l:H1_slope_grad_est} and Proposition \ref{p:parallel_H1_relation} we have that 
\begin{align}
|\nabla F|^2_{H^1_0}&\leq |\partial F|^2_{H^1_0}\equiv \int_0^\infty |\partial_s F|^2\notag\\
& = \int_0^\infty |\nabla_s F|^2 \leq |\nabla F|^2_{H^1_0}\, ,
\end{align}
and hence $|\nabla F|_{H^1_0}=|\partial F|_{H^1_0}$.  Combining these gives the desired result.
\end{proof}

\subsection{The Energy Functional and the Ornstein-Uhlenbeck Operator}\label{ss:OU_nonsmooth_definition}

We use the results of Section \ref{ss:H1_grad_nonsmooth}, and in particular the $H^1_0$ gradient, to define the associated energy functional and Sobolev space in Section \ref{sss:energy_functional}.  In Section \ref{sss:OU_construction} we will use this to construction the Ornstein-Uhlenbeck operator on a general metric-measure space.

\subsubsection{The Energy Functional}\label{sss:energy_functional}

Having defined the $H^1_0$-gradient in Section \ref{ss:H1_grad_nonsmooth} let us begin by defining the associated energy:

\begin{definition}\label{d:dirichlet_pathspace}
Let $\cD(E)\subseteq L^2(P(X),\Gamma_m)$ be defined as the subset of functions for which there exists a $H^1_0$-upper gradient.  Then given $F\in \cD(E)$ we define the path space energy function $E:\cD(E)\to \dR^+$ by
\begin{align}
E[F]\equiv \frac{1}{2}\int_{P(X)} |\nabla F|_{H^1_0}^2 \,d\Gamma_m\, .
\end{align}
\end{definition}

Note from Lemma \ref{l:H1_slope_grad_est} that we clearly have that $\Cyl(X)\subseteq \cD(E)$, and so in particular the domain is dense.  The basic structure theorem about the energy function is the following:

\begin{theorem}\label{t:Dirichlet_form_pathspace}
The energy function $E:\cD(E)\to \dR$ is convex, nonnegative, $2$-homogeneous and lower-semicontinuous.  Furthermore, the following hold:
\begin{enumerate}
\item (closed) The functional $||F||_1\equiv \sqrt{||F||_{L^2}+E(F)}$ defines a complete norm on $\cD(E)\subseteq L^2(P(X),\Gamma_m)$.
\item (stability under lipschitz calculus) Given a $1$-lipschitz function $\phi:\dR\to\dR$ with $\phi(0)=0$ we have that $E[\phi\circ F]\leq E[F]$.
\item (strongly local) If $F,G\in \cD(E)$ are such that $G$ is a constant on $supp(F)\subseteq P(X)$, then $E(F+G) = E(F)+E(G)$.
\end{enumerate}
\end{theorem}
\begin{proof}
The proof of (1) is standard given the lower semicontinuity of $E$, see for instance \cite{Cheeger_DiffLipFun}.  The proofs of (2) and (3) follow easily from Theorem \ref{t:H1_grad_properties}.
\end{proof}

Given the above it makes sense to make the following definition:

\begin{definition}
We identify the Sobolev class $W^{1,2}(P(X),\Gamma_m)\equiv \cD(E)$.
\end{definition}

\subsubsection{The Construction of the Ornstein Uhlenbeck Operator}\label{sss:OU_construction}

Proceeding as in the standard theory of convex functionals on a Hilbert space, let us define the subgradient of $E$ at a point by
\begin{align}
\partial E [F]\equiv \{G:E(F)+\langle G, H-F\rangle\leq E(H)\text{ for every }H\in L^2(P(X),\Gamma_m)\}\, .
\end{align}

Theorem \ref{t:Dirichlet_form_pathspace} tells us, among other things, that the set $\partial E[F]$ is a convex subset, and thus there exists a unique element of minimal $L^2$ norm, which we define as the gradient $\nabla E[F]\equiv L F$.  We define this gradient to be the Ornstein-Uhlenbeck operator on path space.  Using Theorem \ref{t:Dirichlet_form_pathspace} and the standard theory of convex functionals on a Hilbert space we obtain the following:

\begin{theorem}\label{t:OU_definition}
There exists a densely defined operator $L\equiv \nabla E:\cD(\Delta)\subseteq L^2(P(X),\Gamma_m)\to\dR$ which is densely defined and preserves the $\cF^T$-measurable functions.
\end{theorem}

Let us list a couple of interesting properties of the Ornstein-Uhlenbeck operator.  The following are related to our defining the Ornstein-Uhlenbeck operator on $P(X)$ as opposed to based path space $P_x(X)$, and gives a direct way to recapture the based operators:

\begin{lemma}\label{l:OU_properties}
The following hold:
\begin{enumerate}
\item Let $F$ be $\cF^0$-measurable, then $F\in \cD(E)$ and $LF=0$.
\item Let $F$ be $\cF^0$-measurable and $G\in \cD(E)$.  Then $L(F\cdot G)=F\cdot LG$
\end{enumerate}
\end{lemma}
\begin{proof}
Note that if $F=u(\gamma(0))$ is $\cF^0$-measurable, then by Lemma \ref{l:slope_lipschitz} we have that $|\partial_s F|=0$ for all $s>0$.  In particular, by (\ref{e:H1_slope}) we have that $|\partial F|_{H^1_0}=0$, and hence by the definition of the gradient that $|\nabla F|_{H^1_0}\equiv 0$.  It follows from this that $0$ is a subgradient of $E$ at $F$, and since it is clearly the minimal subgradient we have that $LF\equiv 0$.  To prove the second part observe that $|\partial_s (F\cdot G)|=|F|\cdot|\partial_s G|$ for $s>0$, and arguing as above we can conclude the second part of the Lemma.
\end{proof}

So in particular the above tells us that the $\cF^0$-measurable functions are in the kernel of $L$.  We will see in the next Section that on spaces with bounded Ricci curvature that these are the only elements in the kernel of $L$.

Let us end this Section with the following, which tells us that on a smooth metric-measure space the Ornstein-Uhlenbeck operator as defined above is the same as the standard definition.  As with the $H^1_0$-gradient recall we have defined Ornstein-Uhlenbeck operator on total path space $P(X)$.  Therefore in the next Theorem we will denote by $L_x:L^2(P_x(M),\Gamma_x)\to L^2(P_x(M),\Gamma_x)$ the classical Ornstein-Uhlenbeck operator on a smooth space as defined in \cite{DriverRockner_Diff}.  The proof of the next Theorem is an immediate application of Theorem \ref{t:H10_grad_comparison}:

\begin{theorem}\label{t:OU_comparison}
Let $(M^n,g,e^{-f}dv_g)$ be a smooth complete metric-measure space, and let $F\in L^2(P(M),\Gamma_m)$.  Then for a.e. $\gamma\in P(M)$ we have that $L F(\gamma) = L_{\gamma(0)} F(\gamma)$.   
\end{theorem}

\subsection{Bounded Ricci Curvature and the Ornstein-Uhlenbeck Operator}\label{ss:boundedricci_OU_nonsmooth}

In this Section we will prove that bounds on the Ricci curvature of $X$ imply the appropriate Poincare and log-Sobolev estimates on $L$.  That is, we will prove Theorem \ref{t:OU_boundedricci}.

The based path space spectral gap and log-sobolev inequalities of Theorem \ref{t:OU_boundedricci} will be proved in two steps.  We will first prove the following global version, which is a true poincare and log-sobolev for the operator $L$, viewed as the Ornstein-Uhlenbeck operator on global path space $P(X)$.  We will then use this to prove the based pathspace estimates of Theorem \ref{t:OU_boundedricci}.

\begin{theorem}\label{t:OU_global_logsob}
Let $(X,d,m)$ be a $BR(\kappa,\infty)$ space.  Then for every $F\in L^2(P(X),\Gamma_m)$ which is $\cF^T$-measurable we have the Poincare estimate
\begin{align}
\int_{P(X)}\bigg|F - \int_{P_x(X)}F\,d\Gamma_x\bigg|^2 d\Gamma_m \leq e^{\frac{\kappa}{2}T}\int_{P(X)}\bigg(\int_0^T \cosh(\frac{\kappa}{2}t)|\nabla_t F|^2\, dt\bigg)\, d\Gamma_m\, ,
\end{align}
as well as the log-sobolev estimate
\begin{align}
\int_{P(X)}F^2\ln F^2 d\Gamma_m - \int_X \bigg(\int F^2\,d\Gamma_x\cdot\ln\big(\int F^2\,d\Gamma_x\big)\bigg)\,dm\leq 2e^{\frac{\kappa}{2}T}\int_{P(X)}\bigg(\int_0^T \cosh(\frac{\kappa}{2}t)|\nabla_t F|^2\, dt\bigg)\, d\Gamma_m\, .
\end{align}
\end{theorem}

In particular, the above proves the claim after Lemma \ref{l:OU_properties} that the Ornstein-Uhlenbeck operator $L$ satisfies $\ker L \equiv L^2(X,m)$.  We will only focus on proving the log-sobolev inequality, as the spectral gap may be proved either by the same methods or as in Section \ref{ss:r4_r7}.

\begin{proof}
The overall structure of the proof is essentially verbatim as in Section \ref{ss:r4_r7}.  The primary difficulty is in seeing that we are allowed to make the same arguments, which is often times a subtle point, and requires the various structure built in Section \ref{s:bounded_ricci_analysis} and Section \ref{s:ricci_martingales}.  So let $F\in \Cyl(X)$ be a $\cF^T$-measurable function on path space $P(X)$, and let $H^t\equiv (F^2)^t$ be the martingale induced by projecting $F^2$ to the $\cF^t$-measurable functions.  By Theorem \ref{t:martingale_continuous_lowerricci} the martingale $H^t$ as well as its quadratic variation are continuous, and thus we may apply the Ito formula \cite{Kuo_book} to the function $H^t\ln H^t$ arrive at
\begin{align}\label{e:global_logsob:1}
\int_{P(X)} F^2\ln F^2\,d\Gamma_m - \int_X \bigg(\int F^2\,d\Gamma_x\cdot\ln\big(\int F^2\,d\Gamma_x\big)\bigg)\,dm = \int_{P(X)} \int_0^T (H^t)^{-1}[dH^t]\,d\Gamma_m\, .
\end{align}
Now using Theorem \ref{t:ricci_quad_nonsmooth} we have the estimate
\begin{align}
\sqrt{[dH^t]}(\gamma)&\leq \int_{P(X)}|\nabla_t F^2|(\gamma_t\circ\sigma)+\int_t^T\frac{\kappa}{2}e^{\frac{\kappa}{2}(s-t)}|\nabla_s F^2|\,d\Gamma_{\gamma(t)}\notag\\
&= 2\int_{P(X)}F\bigg(|\nabla_t F|(\gamma_t\circ\sigma)+\int_t^T\frac{\kappa}{2}e^{\frac{\kappa}{2}(s-t)}|\nabla_s F|\bigg)\,d\Gamma_{\gamma(t)}\, .
\end{align}
Thus we have that
\begin{align}
[dH^t](\gamma)&\leq \int_{P(X)}F^2\,d\Gamma_{\gamma(t)}\cdot e^{\frac{\kappa}{2}(T-t)}\int_{P(X)}\bigg(|\nabla_t F|^2(\gamma_t\circ\sigma)+\int_t^T\frac{\kappa}{2}e^{\frac{\kappa}{2}(s-t)}|\nabla_s F|^2\bigg)\,d\Gamma_{\gamma(t)}\notag\\
&=H^t(\gamma)\cdot e^{\frac{\kappa}{2}(T-t)}\int_{P(X)}\bigg(|\nabla_t F|^2(\gamma_t\circ\sigma)+\int_t^T\frac{\kappa}{2}e^{\frac{\kappa}{2}(s-t)}|\nabla_s F|^2\bigg)\,d\Gamma_{\gamma(t)}\, .
\end{align}
Plugging this into (\ref{e:global_logsob:1}) gives the estimate
\begin{align}
\int_{P(X)} F^2\ln F^2\,d\Gamma_m &- \int_X \bigg(\int F^2\,d\Gamma_x\cdot\ln\big(\int F^2\,d\Gamma_x\big)\bigg)\,dm \notag\\
&\leq 2\int_{P(X)}\bigg(\int_0^T e^{\frac{\kappa}{2}(T-t)}\int_{P(M)}\bigg(|\nabla_t F|^2(\gamma_t\circ\sigma)+\int_t^T\frac{\kappa}{2}e^{\frac{\kappa}{2}(s-t)}|\nabla_s F|^2\bigg)\,d\Gamma_{\gamma(t)}\bigg)\,d\Gamma_m\notag\\
&=2\int_{P(X)}\bigg(\int_0^T e^{\frac{\kappa}{2}(T-t)}|\nabla_t F|^2+\int_0^T\int_t^T\frac{\kappa}{2}e^{\frac{\kappa}{2}(T+s-2t)}|\nabla_s F|^2ds\,dt\bigg)\,d\Gamma_m\notag\\
&=2e^{\frac{\kappa}{2}T}\int_{P(X)}\bigg(\int_0^T e^{-\frac{\kappa}{2}t}|\nabla_t F|^2dt+\int_0^T\int_0^se^{\frac{\kappa}{2}(s-2t)}|\nabla_s F|^2dt\,ds\bigg)\,d\Gamma_m\notag\\
&=2e^{\frac{\kappa}{2}T}\int_{P(X)}\bigg(\int_0^T e^{-\frac{\kappa}{2}t}|\nabla_t F|^2dt+\int_0^T e^{\frac{\kappa}{2}t}\big(\int_0^t \frac{\kappa}{2}e^{-\kappa s}ds\big)|\nabla_t F|^2\,dt\bigg)\,d\Gamma_m\notag\\
&=2e^{\frac{\kappa}{2}T}\int_{P(X)}\bigg(\int_0^T e^{-\frac{\kappa}{2}t}|\nabla_t F|^2dt+\int_0^T\frac{1}{2}\big(e^{\frac{\kappa}{2}t}-e^{-\frac{\kappa}{2}t}\big)|\nabla_t F|^2ds\,dt\bigg)\,d\Gamma_m\notag\\
&=2e^{\frac{\kappa}{2}T}\int_{P(X)}\bigg(\int_0^T \cosh(\frac{\kappa}{2}t)|\nabla_t F|^2\bigg)d\Gamma_m\, ,
\end{align}
as claimed.
\end{proof}

Now let us use this to finish the proof of Theorem \ref{t:OU_boundedricci}:

\begin{proof}[Proof of Theorem \ref{t:OU_boundedricci}]
We again focus on the log-sobolev inequality, as we can either prove the spectral gap by the same methods or as in Section \ref{ss:r4_r7}.  Let us begin by letting $\varphi:X\to \dR^+$ be an arbitrary lipschitz function and considering the function on path space given by
\begin{align}
G(\gamma)=F(\gamma)\varphi(\gamma(0))\, ,
\end{align}
where $\int_{P_x(X)} F^2\,d\Gamma_x = 1$ for each $x$.  Note then that 
\begin{align}
\int_{P(X)} G^2\ln G^2\,d\Gamma_m &- \int_X \bigg(\int G^2\,d\Gamma_x\cdot\ln\big(\int G^2\,d\Gamma_x\big)\bigg)\,dm\notag\\
&= \int_X\Bigg[\int_{P_x(X)}F^2\ln F^2 +F^2\ln\varphi^2\,d\Gamma_x\Bigg]\varphi^2(x)\,dm -\int_X \varphi^2\ln\varphi^2\,dm \notag\\
&=\int_X\Bigg[\int_{P_x(X)} F^2\ln F^2\Bigg]\varphi^2(x)\,dm\, .
\end{align}
Note for $t>0$ that we have that $|\partial_t G|=\varphi|\partial_t F|$, and thus for every such $\varphi$ we can compute that
\begin{align}
\int_X\Bigg[\int_{P_x(X)}F^2\ln F^2 d\Gamma_x- 2e^{\frac{\kappa}{2}T}\int_{P_x(X)}\bigg(\int_0^T \cosh(\frac{\kappa}{2}t)|\nabla_t F|^2\, dt\bigg)\, d\Gamma_x\Bigg]\varphi^2(x)\,dm\leq 0\, .
\end{align}
Thus by letting $\varphi$ approximate the characteristic functions of open sets $U$ we get that
\begin{align}
\int_U\Bigg[\int_{P_x(X)}F^2\ln F^2 d\Gamma_x- 2e^{\frac{\kappa}{2}T}\int_{P_x(X)}\bigg(\int_0^T \cosh(\frac{\kappa}{2}t)|\nabla_t F|^2\, dt\bigg)\, d\Gamma_x\Bigg]\varphi^2(x)\,dm\leq 0\, .
\end{align}
Since the set $U$ is arbitrary we can therefore conclude that for a.e. $x\in X$ that
\begin{align}
\int_{P_x(X)}F^2\ln F^2 d\Gamma_x- 2e^{\frac{\kappa}{2}T}\int_{P_x(X)}\bigg(\int_0^T \cosh(\frac{\kappa}{2}t)|\nabla_t F|^2\, dt\bigg)\, d\Gamma_x\leq 0\, ,
\end{align}
which finishes the proof of the Theorem.
\end{proof}


\bibliographystyle{plain}

\end{document}